\documentclass{amsart}

\usepackage{geometry}
\usepackage{physics}
\usepackage{times}
\usepackage{longtable}

\usepackage{enumerate}
\usepackage[all]{xy} 

\usepackage{amsmath, amssymb, amsfonts, latexsym, mdwlist, amsthm,amscd}
\usepackage{caption, subcaption}

\usepackage{hologo} 
\usepackage{booktabs} 
\usepackage{tablefootnote} 
\usepackage{tikz}
\usetikzlibrary{calc,trees,arrows,chains,shapes.geometric,%
	decorations.pathreplacing,decorations.pathmorphing,shapes,%
	matrix,shapes.symbols}

\tikzset{
	>=stealth',
	punktchain/.style={
		rectangle,
		rounded corners,
		draw=black, thick,
		minimum height=3em,
		text centered,
		on chain},
	line/.style={draw, thick, <-},
	element/.style={
		tape,
		top color=white,
		bottom color=blue!50!black!60!,
		minimum width=8em,
		draw=blue!40!black!90, very thick,
		text width=10em,
		minimum height=3.5em,
		text centered,
		on chain},
	every join/.style={->, thick,shorten >=1pt},
	decoration={brace},
	tuborg/.style={decorate},
	tubnode/.style={midway, right=2pt},
}
\usepackage{cite}
\usepackage{url}

\usepackage{verbatim}
\usepackage{mathtools}

\usepackage{pgf,tikz}
\usepackage{tikz-cd}
\usetikzlibrary{calc,matrix,patterns,shadings,backgrounds,fit}
\pgfdeclarelayer{myback}
\pgfsetlayers{myback,background,main}
\usepackage{stmaryrd}
\usepackage{diagbox}
\usepackage{extarrows}

\usepackage[bookmarks, colorlinks, breaklinks, pdftitle={StaR manifold},
pdfauthor={Chunyi Li}]{hyperref}
\hypersetup{linkcolor=blue,citecolor=blue,filecolor=black,urlcolor=blue}

\usepackage{paralist}
\setdefaultenum{(a)}{(i)}{}{}

\usepackage{pgfplots}
\pgfplotsset{compat=1.18}
\usepackage{mathrsfs}
\usetikzlibrary{arrows}

\numberwithin{equation}{section}

\def\C{\ensuremath{\mathbb{C}}}

\def\H{\ensuremath{\mathbb{H}}}

\def\Q{\ensuremath{\mathbb{Q}}}
\def\R{\ensuremath{\mathbb{R}}}
\def\Z{\ensuremath{\mathbb{Z}}}

\def\bv{\ensuremath{\mathbf{v}}}

\def\Aut{\mathop{\mathrm{Aut}}\nolimits}
\def\ch{\mathop{\mathrm{ch}}\nolimits}

\def\Coh{\mathop{\mathrm{Coh}}\nolimits}

\def\deg{\mathop{\mathrm{deg}}}

\def\dim{\mathop{\mathrm{dim}}\nolimits}

\def\ev{\mathop{\mathrm{ev}}\nolimits}

\def\Eff{\mathop{\mathrm{Eff}}}

\def\Gr{\mathop{\mathrm{Gr}}\nolimits}
\def\GL{\mathop{\mathrm{GL}}\nolimits}

\def\Hom{\mathop{\mathrm{Hom}}\nolimits}
\def\id{\mathop{\mathrm{id}}\nolimits}
\def\Id{\mathop{\mathrm{Id}}\nolimits}

\def\Ker{\mathop{\mathrm{Ker}}\nolimits}
\def\Kn{\mathop{\mathrm{K}_{\mathrm{num}}}}

\def\min{\mathop{\mathrm{min}}\nolimits}

\def\Nef{\mathop{\mathrm{Nef}}\nolimits}
\def\NS{\mathop{\mathrm{NS}}\nolimits}

\def\Pic{\mathop{\mathrm{Pic}}}

\def\rk{\mathop{\mathrm{rk}}}

\def\supp{\mathop{\mathrm{supp}}}

\def\tgr{\ensuremath{{\mathrm{Gr}}}}
\def\Stb{\ensuremath{S_3}}
\def\link{\ensuremath{\bowtie}}
\def\linkless{\ensuremath{\triangleleft}}
\def\Root{\ensuremath{\mathsf{roots}}}

\def\Stab{\mathop{\mathrm{Stab}}\nolimits}

\def\glr{\ensuremath{\mathrm{GL}^+(2,\mathbb R)}}

\def\Db{\mathrm{D}^{b}}

\def\glt{\widetilde{\mathrm{GL}}^+(2,\R)}

\def\RHom{\mathrm{RHom}}

\def\ts{\ensuremath{\tilde{\sigma}}}
\def\ttau{\ensuremath{\tilde{\tau}}}

\def\spant#1#2{\ensuremath{\mathrm{Span}_{#1}(#2)}}
\def\spa#1{\ensuremath{\mathrm{Span}(#1)}}
\def\sbr{\ensuremath{SB}}

\makeatletter
\newtheorem*{rep@theorem}{\rep@title}
\newcommand{\newreptheorem}[2]{%
\newenvironment{rep#1}[1]{%
 \def\rep@title{#2 \ref{##1}}%
 \begin{rep@theorem}}%
 {\end{rep@theorem}}}
\makeatother

\theoremstyle{definition}
\newtheorem{Thm}{Theorem}[section]
\newreptheorem{Thm}{Theorem}
\newtheorem{Prop}[Thm]{Proposition}
\newtheorem{PropDef}[Thm]{Proposition and Definition}
\newtheorem{Lem}[Thm]{Lemma}

\newtheorem{Cor}[Thm]{Corollary}
\newreptheorem{Cor}{Corollary}
\newtheorem{Con}[Thm]{Conjecture}

\newtheorem{Asp}[Thm]{Assumption}
\newtheorem{SubLem}[Thm]{Sublemma}

\newtheorem{thm-int}{Theorem}

\newtheorem{Def-s}[Thm]{Definition}
\newtheorem{Def}[Thm]{Definition}
\newtheorem{Rem}[Thm]{Remark}

\newtheorem{Ex}[Thm]{Example}

\newtheorem{Not}[Thm]{Notation}

\def\C{\ensuremath{\mathbb{C}}}

\def\H{\ensuremath{\mathbb{H}}}

\def\Q{\ensuremath{\mathbb{Q}}}
\def\R{\ensuremath{\mathbb{R}}}
\def\Z{\ensuremath{\mathbb{Z}}}

\def\cA{\ensuremath{\mathcal A}}

\def\cC{\ensuremath{\mathcal C}}

\def\cE{\ensuremath{\mathcal E}}

\def\cH{\ensuremath{\mathcal H}}

\def\cL{\ensuremath{\mathcal L}}

\def\cO{\ensuremath{\mathcal O}}
\def\cP{\ensuremath{\mathcal P}}

\def\cT{\ensuremath{\mathcal T}}

\def\cW{\ensuremath{\mathcal W}}

\def\bP{\ensuremath{\mathbf P}}

\def\ut{\ensuremath{\underline{t}}}
\def\um{\ensuremath{\underline{m}}}
\def\us{\ensuremath{\underline{s}}}
\def\vd{\ensuremath{\mathsf{B}}}
\def\vq{\ensuremath{\mathsf{Q}}}

\def\BBB{\mathfrak B}

\def\PPP{\mathfrak P}

\def\UUU{\mathfrak U}

\def\iff{\; \Longleftrightarrow \;}

\usepackage{todonotes}

\def\sab{\sigma_{\alpha, \beta}}

\def\kn{\mathrm{K}_{\mathrm{num}}}
\def\Cone{\mathrm{Cone}}

\def\hom{\mathrm{hom}}

\def\HN{\mathrm{HN}}
 \def\nega{\mathrm{neg}}

\def\sb{\mathrm{Sb}}
\def\ta{\ensuremath{\mathrm{Ta}}}

\def\Forg{\ensuremath{\mathrm{Forg}}}
\def\lbdd{\ensuremath{(\Lambda_{\mathbb R})^*}}

\def\sbr{\ensuremath{\mathrm{Sbr}}}
\def\hut{\ensuremath{\hat{\underline{t}}}}

\def\sc{\ensuremath{\mathrm{SC}}}

\def\sep{\ensuremath{\mathrm{sep}}}

\def\be{\mathbf e}

\def\lsm{\ensuremath{\lesssim}}
\def\lsmeq{\ensuremath{\lessapprox}}

\def\vperp{\ensuremath{v^\perp}}
\def\wperp{\ensuremath{w^\perp}}

\title{A Real Reduction of the Manifold of  Bridgeland Stability Conditions}
\author{Chunyi Li}
\address{C. L.:
Mathematics Institute, University of Warwick,
Coventry, CV4 7AL,
United Kingdom}
\email{C.Li.25@warwick.ac.uk}
\urladdr{https://sites.google.com/site/chunyili0401/}

\begin{document}
\begin{abstract}
Let $\mathcal{T}$ be a $k$-linear triangulated category. The space of Bridgeland stability conditions on $\mathcal{T}$, denoted by $\mathrm{Stab}(\mathcal{T})$, forms a complex manifold. In this paper, we introduce an equivalence relation $\sim$ on $\mathrm{Stab}(\mathcal{T})$ and study the quotient space $\mathrm{Sb}(\mathcal{T}) \coloneq \mathrm{Stab}(\mathcal{T})/\sim$, which parametrizes what we call reduced stability conditions. We show that $\mathrm{Sb}(\mathcal{T})$ admits the structure of a real (possibly non-Hausdorff) manifold of half the dimension of $\mathrm{Stab}(\mathcal{T})$. The space $\mathrm{Sb}(\mathcal{T})$ preserves the wall-and-chamber structure of $\mathrm{Stab}(\mathcal{T})$, but in a significantly simpler form. Moreover, we define a relation $\lesssim$ on $\mathrm{Sb}(\mathcal{T})$, and show that the full stability manifold $\mathrm{Stab}(\mathcal{T})$ can be reconstructed from the space $\mathrm{Sb}(\mathcal{T})$ together with the additional data $\lesssim$.

We then focus on the case where $\mathcal{T} = \mathrm{D}^b(X)$, the bounded derived category of coherent sheaves on a smooth polarized variety $(X, H)$. By explicitly describing $\mathrm{Sb}(X)$ for varieties $X$ of small dimension, we formulate two equivalent conjectures concerning a family of stability conditions $\mathrm{Stab}_H^*(X)$ and their reduced counterparts $\mathrm{Sb}_H^*(X)$ on $\mathrm{D}^b(X)$. We establish some desirable properties for both families. In particular, using a version of the restriction theorem formulated in terms of $\lesssim$, we show that the existence of $\mathrm{Stab}_H^*(X)$ implies the existence of stability conditions on every smooth subvariety of $X$.
\end{abstract}


\date{\today}

\maketitle

\setcounter{tocdepth}{1}
\tableofcontents

\maketitle

\section{Introduction}
Stability conditions on triangulated categories were introduced by Bridgeland in \cite{Bridgeland:Stab}. Despite progress in several special cases, the existence of stability conditions on the bounded derived category of a smooth projective variety remains an open problem. In particular, there is currently no precise conjectural framework that applies uniformly to smooth projective varieties in all dimensions.

One goal of this paper is to propose such a conjecture for a specific family of stability conditions on smooth projective varieties.

\subsection{Standard slice}
Let $(X,H)$ be an $n$-dimensional irreducible smooth polarized variety over $\C$. Define the $H$-polarized lattice $\Lambda_H$  via the map:
\begin{align*}
    \lambda_H:K(X)\to \Lambda_H:[E]\mapsto(H^n\ch_0(E),H^{n-1}\ch_1(E),\dots,\ch_n(E)).
\end{align*}

A stability condition $\sigma=(\cP,Z')$ on $\Db(X)$ consists of a slicing $\cP$ and a  group homomorphism $Z':K(X)\to \C$, called the central charge, satisfying certain compatibility assumptions and the support property with respect to a finite rank lattice $\lambda:K(X)\to \Lambda$. 

Let $\Stab_H(X)$ denote the set of all stability conditions on $\Db(X)$ with respect to the fixed lattice $\Lambda_H$. In particular, for each such stability condition, the central charge factors through $\lambda_H$: 
\begin{align*}
    Z':K(X)\xrightarrow{\lambda_H}\Lambda_H\xrightarrow{Z}\C.
\end{align*} 
By Bridgeland's seminal result, the space $\Stab_H(X)$, whenever non-empty, is a complex manifold of dimension $n+1$, with local coordinates given by the forgetful map to the space of central charges.

We formulate a conjecture describing a specific family of stability conditions on $\Db(X)$.
  
\begin{Con}\label{conj:intro}
   There exists a family of stability conditions $\Stab_H^*(X)$ on $\Db(X)$, defined  with respect to the $H$-polarized lattice $\Lambda_H$,  satisfying the following properties:
    \begin{enumerate}
        \item The forgetful map 
        \begin{align*}
            \Forg: \Stab_H^{*}(X)\to \Hom(\Lambda_H,\C): \sigma=(\cP,Z)\mapsto Z
        \end{align*}
        is a homeomorphism onto $\UUU_n$.
        \item The space $\Stab_H^*(X)$ is invariant under the $\otimes \cO_X(H)$-action. In other words,  for every $\sigma\in \Stab_H^{*}(X)$,   the stability condition $\sigma\otimes \cO_X(H)$ is  in $\Stab_H^{*}(X)$.
    \end{enumerate}
\end{Con}
\begin{Not}\label{not:intro}
Here the subspace $\UUU_n$ in $\Hom(\Lambda_H,\C)$ is defined as follows:
  \begin{align*}
      \UUU_n\coloneq \left\{c\vd_{\us}+id\vd_{\ut} \;\middle| \;\us,\ut\in\sbr_n,d>0;\;\begin{aligned}
          &\ut<\us<\ut[1] \text{ and } c<0\\
          \text{ or }&\us<\ut<\us[1]\text{ and } c>0
      \end{aligned}\right\},
  \end{align*}
where the parameter space $\sbr_n$ is given by:
\begin{align*}
     \sbr_n& \coloneq \{(t_1,t_2,\dots,t_n):t_1<t_2<\dots<t_n,t_i\in \R\text{ when }i\leq n-1, t_n\in\R\cup\{+\infty\}\}.
\end{align*}
For $\us=(s_1,\dots,s_n)$, $\ut=(t_1,\dots,t_n)$, we write
\begin{align*}
     \us<\ut<\us[1]& :\iff s_1<t_1<s_2<t_2<\dots<s_n<t_n. 
\end{align*}
 
Given $\ut\in\sbr_n$ with $t_n\neq+\infty$ and $\bv=(v_0,\dots,v_n)\in\Lambda_H$, the real-valued linear function $\vd_{\ut}$ is defined by the determinant:
\begin{align}\label{eq:Bvddefintro}
        \vd_{\ut}(\bv)\coloneq C_{\ut} \det\begin{vmatrix}
            1 & t_1 &\dots & \frac{t_1^n}{n!} \\ \dots & \dots & \dots & \dots \\ 1 & t_n &\dots & \frac{t_n^n}{n!}\\ v_0 & v_1 & \dots & v_n
        \end{vmatrix},
    \end{align}
where $C_{\ut}>0$ is a normalizing constant chosen so that the coefficient of $v_n$ in $\vd_{\ut}(\bv)$ is $1$; see its explicit definition in equation \eqref{eq:BvdC}. 

If $t_n=+\infty$, the function $\vd_{\ut}$ is defined inductively by $\vd_{\ut}(\bv)\coloneq -\vd_{t_1,\dots,t_{n-1}}(v_0,v_1,\dots,v_{n-1})$.

\noindent In particular, up to  scaling, the function $\vd_{\ut}$ is uniquely characterized by the vanishing condition  $\vd_{\ut}(\gamma_n(t_i))=0$ for all $t_i$'s, where
\begin{align*}
\gamma_n:\R\cup\{+\infty\}&\to\Lambda_H\otimes\R:
    \R\ni t \mapsto (1,t,\frac{t^2}{2!},\dots,\frac{t^n}{n!})\; ;\; +\infty \mapsto (0,\dots,0,1).
\end{align*}
\end{Not}

\noindent Rather than proving the conjecture for a specific variety, this paper focuses on establishing foundational properties of the space $\Stab^*_H(X)$. One key result is a restriction theorem, which shows that the existence of stability conditions on smooth subvarieties of $\bP^n$ can be deduced from Conjecture \ref{conj:intro} for $\bP^n$ itself.

\begin{Thm}\label{thm:intro}
    Assume Conjecture \ref{conj:intro} holds for  $(X,H)$. Then the following statements hold.
    \begin{enumerate}[(1)]
        \item (Uniqueness) The family of stability conditions described in Conjecture \ref{conj:intro} is unique up to a homological shift $[2k]$ for some $k\in \Z$.
        \item (Restriction) Let $Y$ be a smooth  subvariety of $X$. Then there exist stability conditions on $Y$. \\
        
    \noindent For the next three statements,  let $\sigma\in \Stab_H^*(X)$ and $E,F$ be two $\sigma$-stable objects. 
        \item (Geometric) Skyscraper sheaves are $\sigma$-stable. Conversely, if $\lambda_H([E])=(0,0,\dots,0,*)$, then $E=\cO_p[k]$ for some $p\in X$ and $k\in\Z$. 
        \item (Bayer Vanishing Lemma) Assume that $\phi_\sigma(E)\geq \phi_\sigma(F)$ and $E,F$ are not skyscraper sheaves up to a homological shift, then $\Hom(E\otimes \cO_X(mH),F)=0$ for every $m\in \Z_{\geq 1}$.
        \item (Bound on polarized character) There exists a unique $\ut\in\sbr_n$ satisfying $\vd_{\ut}(E)=0$. Moreover, the $H$-polarized character of $E$ satisfies  \begin{align}\label{eq11}
         \lambda_H([E])= \sum_{i=1}^n (-1)^ia_i\gamma_n(t_i)
     \end{align}
     where the  coefficients $a_i$ are either all non-negative or all non-positive.
    \end{enumerate}
\end{Thm}

\begin{Rem}\label{rem:intro}
We include some remarks related to Conjecture \ref{conj:intro} and Theorem \ref{thm:intro}.
\begin{enumerate}[(1)]
 \item Conjecture \ref{conj:intro} is known to hold when $X$ is a curve or a surface. In the case where $X$ is a threefold,  \cite[Conjecture 4.1]{BMS:stabCY3s} implies Conjecture \ref{conj:intro}. In particular, the conjecture is verified for $\bP^3$, abelian threefolds, most  Fano threefolds, and is known to hold or nearly hold for certain Calabi--Yau threefolds. See \cite{BMSZ:Fano3stab,Naoki:stabdoublesolids,Naoki:stabproduct,Naoki:Stabnefbundle,LC:Fano3,quintic,Shengxuan:thesis,Macri:P3,Benjamin:BGquadric,sun2021stabilityconditionsthreefoldsvanishing, Toda:BGCY3} for examples. As a consequence, Theorem \ref{thm:intro} holds for all of these cases. Moreover, in certain special cases, for example, when $X$ is a curve with genus non equal to zero, a surface with finite Albanese map, or an abelian threefold, we have the whole stability manifold given by $\Stab_H(X)=\coprod_{k\in\Z}\Stab_H^*(X)[k]$.
\item For threefolds, Conjecture \ref{conj:intro} is actually `equivalent' to \cite[Conjecture 4.1]{BMS:stabCY3s}, which is known to be too strong for some specific threefolds. Consequently, we do NOT expect  Conjecture \ref{conj:intro}  holds for ALL polarized varieties $(X,H)$. To address this issue,  we introduce in the main text  a modified version $\Stab^d$ Conjecture \ref{conj:stabd}, which is parametrized by a real variable $d\in\R_{\geq 0}$. The $\Stab^d$ Conjecture becomes weaker as $d$ increases, with $\Stab^0$ Conjecture \ref{conj:stabd} coinciding with the original Conjecture \ref{conj:intro}. Assuming any of the weaker conjectures, the statements in Theorem \ref{thm:intro} still holds with appropriate modifications. As the definitions involved are more technical, we defer further details to Section \ref{sec:sm}.

 Nevertheless, we  expect Conjecture \ref{conj:intro} to hold for certain important varieties, such as the projective space $\bP^n$ and abelian varieties. 

\item The restriction statement in Theorem \ref{thm:intro}.{(2)} is made more precise in Proposition \ref{prop:reststab}, Corollary \ref{cor:restrictionall}, and Remark \ref{rem:restcentralcharge}. Specifically, there exists an integer $m_Y\in\Z_{\geq 1}$ such that 
if the central charge of a stability condition in $\Stab_H^*(X)$ is of the form $\vd + i \vd_{\ut}$, with $\ut = (t_1, \dots, t_n)$ satisfying $t_{i+1} - t_i > m_Y$ for all $i$, then the corresponding stability condition restricts to a stability condition on $Y$. Moreover, the restricted stability condition also satisfies the geometric and vanishing properties stated in  Theorem \ref{thm:intro}.{(3)} and {(4)}.

\item Regarding Theorem \ref{thm:intro}.(5), for any object $E$ with $\vd_{\ut}(E) = 0$, by basic linear algebra, its $H$-polarized Chern character $\lambda_H([E])$ necessarily takes the form as that in \eqref{eq11} for some real coefficients $a_i$. The statement asserts that all $a_i \geq 0$, or all $a_i \leq 0$.

When $n = 2$, this condition is equivalent to the classical $H$-polarized Bogomolov inequality: $\Delta_H(E) \geq 0$. When $n = 3$, it is equivalent to a family of Bogomolov--Gieseker type inequalities \eqref{eq117} as that in \cite[Theorem 8.7 and Lemma 8.5]{BMS:stabCY3s}. We discuss this in more detail in Sections \ref{sec:3fold} and \ref{sec:sub:bounds}.
\end{enumerate}
\end{Rem}

\subsection{Real reduction of the stability manifold}
The rationale behind Conjecture \ref{conj:intro} and the proof of Theorem \ref{thm:intro} rely on a new concept within the theory of Bridgeland stability conditions. We introduce this framework in a general setting applicable to any $k$-linear triangulated category.
\subsubsection{Reduced stability conditions}
Let $\cT$ be a $k$-linear triangulated category and fix a lattice $\lambda:K(\cT)\to \Lambda$ of finite rank. Denote  by $\Stab_\Lambda(\cT)$  the space of stability conditions with respect to the lattice $\Lambda$; see Definitions \ref{def:prestab} and \ref{def:supportproperty} for a detailed review of these concepts.

\begin{Def}\label{def:eqintro}
We define an equivalence relation $\sim$ on $\Stab_\Lambda(\cT)$ as follows: two stability conditions $\sigma \sim \tau$ if
\begin{align*}    
\Im Z_\sigma = \Im Z_\tau \quad \text{and} \quad d(\mathcal{P}_\sigma, \mathcal{P}_\tau) < 1,
\end{align*}
where $d(-,-)$ denotes the standard metric on slicings (see \eqref{eq:defslicing}).
\end{Def}

Although this is not immediate from the definition, we will show in Proposition \ref{prop:convex} that $\sim$ indeed defines an equivalence relation.

We denote the resulting quotient space by

$$
\sb_\Lambda(\mathcal{T}) \coloneq \Stab_\Lambda(\mathcal{T}) / \sim,
$$

and equip it with the quotient topology induced from $\Stab_\Lambda(\mathcal{T})$.

\noindent An element $\ts \in \sb_\Lambda(\mathcal{T})$ will be referred to as a \emph{reduced stability condition}. By definition, the imaginary part of the central charge, $\Im Z_\sigma$, is well-defined for $\ts$, independent of the choice of representative $\sigma$. We denote this part by $B_{\ts}$ and call it a \emph{reduced central charge}.

The space $\sb_\Lambda(\cT)$ also admits a nice local topological structure.
\begin{Prop}\label{prop:intro2}
   The forgetful map 
\begin{align}
 \label{eq19}   \Forg: \sb_\Lambda(\cT) &\to \Hom_\Z(\Lambda,\R)\\
 \notag   \ts & \mapsto B_{\ts}
\end{align}
is a local homeomorphism.
\end{Prop}

\subsubsection{The $\lsm$ relation}
In addition to the reduced central charge $B_{\ts}$, we will show that both the heart $\cA_{\ts}\coloneq\cP_{\ts}((0,1])$ and the slice $\cP_{\ts}(1)$ do not rely on the representative of $\ts$. This allows us to define a binary relation `$\lsm$' on reduced stability conditions via
\begin{align*}
    \ts\lsm\ttau:\iff \cA_{\ts}\subset \cP_{\ttau}(<1).
\end{align*}

In particular, whenever $\ts\lsm\ttau$, it follows that $\cA_{\ts}\subset \cA_{\ttau}[\leq 0]$. One of the main structural results is that the stability manifold $\Stab_\Lambda(\cT)$ can be recovered from $\sb_\Lambda(\cT)$ together with the relation $\lsm$. As topological spaces, we have the identification
\begin{align}\label{eq112}
    \Stab_\Lambda(\cT)\simeq \ta\sb_\Lambda(\cT),
\end{align}
where $\ta\sb_\Lambda(\cT)$ denotes a certain subspace of the tangent bundle of $\sb_\Lambda(\cT)$ determined by $\lsm$;  see \eqref{eq41226} for the formal definition. 
\subsubsection{Wall and chamber structure}
Given a character $v\in \Lambda$ and a reduced stability condition $\ts$ satisfying $B_{\ts}(v)=0$, the stability of an object $E$ with character $\lambda([E])=v$ is independent of the choice of  representative $\sigma\in\Stab_\Lambda(\cT)$ corresponding to $\ts$. Let $\sb_{\Lambda,v}(\cT)\subset \sb_\Lambda(\cT)$ denote the subspace consisting of reduced stability conditions $\ts$ with $B_{\ts}(v)=0$. Then the moduli space $M(v)$ of $v$ admits a wall and chamber structure over $\sb_{\Lambda,v}(\cT)$, given by a locally finite decomposition: 
\begin{align*}
        \sb_{\Lambda,v}(\cT)=\left(\bigcup_i \tilde\cW_i(v)\right)\coprod\left(\coprod_j \tilde\cC_j(v)\right),
\end{align*}
where the $\tilde{\mathcal{W}}_i(v)$ are walls  and the $\widetilde{\mathcal{C}}_j(v)$ are chambers (within which $M(v)$ is constant).

Under the local chart of $\sb_\Lambda(\cT)$ given by \eqref{eq19}, the subspace $\sb_{\Lambda,v}(\cT)$ is locally homeomorphic to the hyperplane $\vperp\subset \Hom(\Lambda,\R)$. Each wall $\Forg(\tilde\cW_i(v))$ lies in a real codimension-one linear subspace of $\vperp$. We have the following slogan-style statement.

\begin{Prop}\label{prop:intro}
The natural map $\pi_\sim:\Stab_\Lambda(\cT)\to \sb_\Lambda(\cT)$ has convex fibers and preserves all wall and chamber structures.
\end{Prop}

More details on wall and chamber structures are provided in Section \ref{sec:wc}. Here, we highlight one immediate corollary. When the rank of $\Lambda$ is small, the wall and chamber structure on $\sb_{\Lambda,v}(\cT)$, and hence on $\Stab_\Lambda(\cT)$, is sufficiently simple to describe explicitly. For instance, when $\rk\Lambda = 3$, Proposition \ref{prop:intro} recovers the Bertram Nested Wall Theorem: all walls $\cW_i(v)$ in $\Stab_\Lambda(\cT)$ are pairwise disjoint. When $\rk\Lambda = 4$, the wall and chamber structure can be visualized on a real plane, with the walls corresponding to real lines; in particular, any two walls intersect at most once. For the case $\cT = \Db(\bP^3)$, we refer the reader to \cite{JM:walls3fold, JMM:wall3fold, Benjamin:wall3fold} for more details on this topic.

\subsubsection{Examples}
\begin{Ex}[Reduced stability conditions on curves]\label{ex:introcurvesb}
   Let $C$ be a smooth irreducible curve with $g\geq 1$. We have
    \begin{align*}
        \sb^*(C)=\left\{\ts_{t}\cdot c=(\Coh^{\sharp t}(C),e^{-c}\vd_{t})\mid c\in \R,t\in\R\cup\{+\infty\}\right\}, \text{ and }
        \sb(C)=\coprod_{n\in\Z} \sb^*(C)[n],
    \end{align*}
    where $\Coh^{\sharp t}(C)\coloneq\langle\Coh^{>t}(C),\Coh^{\leq t}(C)[1]\rangle$  for $t\neq +\infty$, and $\Coh^{\sharp +\infty}(C)\coloneq \Coh(C)[1]$. It is easy to observe that $\ts_s\lsm \ts_t$ when $s<t$.
\end{Ex}

\begin{Ex}[Reduced stability conditions on a polarized surface]\label{ex:introsurf}
Let $(S,H)$ be a smooth polarized surface. There is a family of reduced stability conditions given as follows:
 \begin{align*}
        &\sb^*_H(S)=\left\{\ts_{t_1,t_2}\cdot c=(\cA_{t_1,t_2},e^{-c}\vd_{t_1,t_2})\;\middle|\; c\in\R, t_1<t_2,t_2\in\R\cup\{+\infty\}
        \right\}.
    \end{align*}

   When $t_2=+\infty$,  the heart is given by $\cA_{t_1}\coloneq \Coh_H^{\sharp t_1}(S)[1]$.  When $t_2\neq +\infty$, the heart is defined as $\cA_{t_1,t_2}\coloneq (\cA_t[-1])^{\sharp0}_{\vd_{t_1,t_2}}$ as that in Notation \ref{not:tilt0heart} for any $t\in(t_1,t_2)$. In particular, it does not rely on the choice of $t$.

For the relation $\lsm$, we have 
\begin{align*}
    \ts_{t_1,t_2}\lsm \ts_{s_1,s_2} \text{ whenever }t_1<s_1\text{ and }t_2<s_2.
\end{align*}
\end{Ex}
Ignoring the scalar coefficient $c$, we may draw $\sb^*_H(S)$ using the linear coordinates $(t_1+t_2,t_1t_2)$ as that in Figure \ref{Figure:sbS}.

The image $\Forg(\sb^*_H(S))$ is the area strictly below the parabola.  For each character $v$ with $\Delta_H(v)\geq 0$, the subspace $\sb^*_{H,v}(S)$ is identified as $\vperp\cap \Forg(\sb^*_H(S))$  which is the union of two rays in this coordinate system.  One can describe the wall and chamber structure of $M(v)$ on $\vperp$. Each wall corresponds to a point given by the intersection $\vperp\cap \wperp$ for some character $w$. 
\begin{figure}[h]
    \centering

\scalebox{0.8}{
\begin{tikzpicture}[scale=1.8]
  \draw[->] (0,-1) -- (0,3.4) node[above] {$t_1t_2$};
  \draw[->] (-2.5,0) -- (2.5,0) node[right] {$t_1+t_2$};

  \draw[thick, domain=-2.5:2.5, smooth, variable=\x] 
    plot ({\x}, {0.5*\x*\x}) node[above]{$\{t_1=t_2\}$};

  \draw[thick] (-2.55,2.05) -- (0.5,-1) node[right]{$\sb^*_{H,v_2}(S)$}; 

  \draw[dashed] (-1.8,1.7) -- (1.8,1.7);
  \draw (-2.5,1.7)--(-1.8,1.7);
  \draw (1.8,1.7)--(2.8,1.7) node[below] {$\sb_{H,v_1}^*(S)$};

  \filldraw (0,0) circle (0.02) node[below right] {$O$};

  \node at (-2.3,0.5) {\Forg$(\sb_{H}^*(S))$};

  \node at (0,-1.4) {$v_1 = (1, 0, -c)$; $v_2 = (1, -H, \frac{H^2}{2})$};
\end{tikzpicture}
}

    \caption{Space of reduce central charges on a polarized surface.}
    \label{Figure:sbS}
\end{figure}
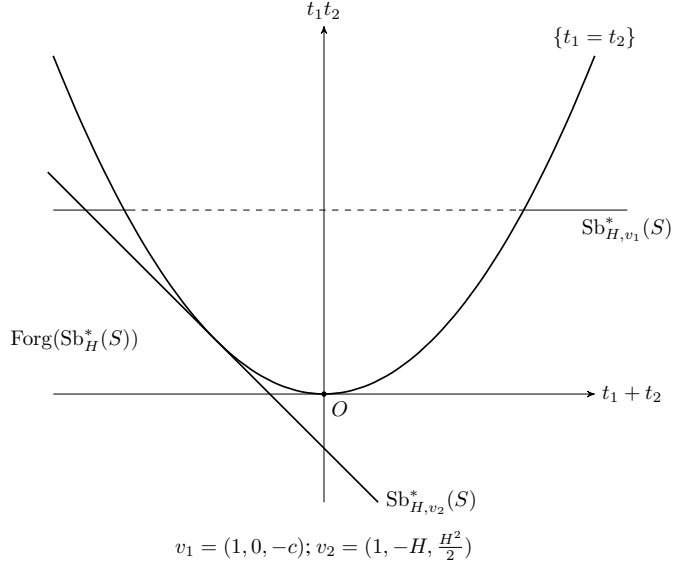
\begin{Rem}[Bousseau's scattering diagram]  
When $S$ is the projective plane $\bP^2$, the visualization described above has been used by Bousseau in \cite[Section 0.2.3]{Bousseau:P2} to interpret the scattering diagram in terms of of Bridgeland stability conditions. In Bousseau's framework, a stability condition $\sigma$ is reduced to its real part of the central charge $\Re Z_\sigma$, which corresponds to considering the image $\pi_\sim(\sigma[-\tfrac{1}{2}])$ in our terminology. The diagram in \cite[Figure 1 and 2]{Bousseau:P2} appears upside down compared of ours.

This perspective suggests a potential direction for generalizing the construction to threefolds via reduced stability conditions. 
\end{Rem}

\begin{Ex}[Theorem \ref{thm:sb3}, reduced stability conditions on a polarized threefold]\label{eg:introsb3}
     Let $(X,H)$ be a polarized smooth threefold satisfying \cite[Conjecture 4.1]{BMS:stabCY3s}. Then by \cite[Theorem 8.2]{BMS:stabCY3s}, there exists a family of stability conditions  $\tilde \PPP_3(X)$  on $X$. The corresponding family of reduced stability conditions on $\Db(X)$ can be parametrized by
     \begin{align*}
         \sb_H^*(X)\coloneq \left\{\ts_{\ut}\cdot c=(\cA_{\ut},e^{-c}\vd_{\ut})\mid c\in\R,\ut=(t_1,t_2,t_3)\in\sbr_3\right\}
     \end{align*}
  with the relation $\pi_\sim(\tilde {\mathfrak P}_3(X))=\coprod_{n\in\Z}\sb_H^*(X)[n]$. As that in the cases of curves and surfaces, we have the relation $\ts_{\us}\lsm \ts_{\ut}$, whenever $\us<\ut$.
\end{Ex}
\subsubsection{Bogomolov type inequality}\label{subsubsec:BG} Let $E$ be a $\sigma$-semistable with respect to some stability condition $\sigma\in\tilde \PPP_3(X)$. By \cite[Theorem 8.7]{BMS:stabCY3s}, its $H$-polarized character satisfies a family of quadratic Bogomolov-type inequalities
 \begin{align}\label{eq117}
     Q^\beta_K(E) \coloneq K\Delta_H(E)+\nabla^\beta_H(E)\geq 0
 \end{align} for any parameter $K$ in a certain interval $I$.
 
 Under the framework of reduced stability conditions, this can be reformulated as follows. The object $E\in\cP_{\ts_{\ut}}(1)$, where $\ts_{\ut}\in\sb^*_H(X)$ is the reduced stability condition from $\sigma[\theta]$. In particular, $\vd_{\ut}(E)=0$, the $H$-polarized character of $E$ satisfies
 \begin{align}\label{eq1018}
             \lambda_H(E)= \sum_{i=1}^3(-1)^ia_i\gamma_3(t_i)
 \end{align}
 for some real coefficients $a_i\in \R$. The family of inequalities in \eqref{eq117} is equivalent to the condition that  all $a_i$ are either non-negative or non-positive.

\subsubsection{Conjecture on reduced stability conditions}
Motivated by Examples \ref{ex:introcurvesb}, \ref{ex:introsurf} and \ref{eg:introsb3}, we define
\begin{align}\label{eq:defBn}
    \BBB_n\coloneq \{c\vd_{\ut}:c>0,\ut\in\sbr_n\}\subset\Lambda_n^*, \quad\pm\BBB_n\coloneq \{c\vd_{\ut}:c\neq0,\ut\in\sbr_n\},
\end{align}
and formulate the following conjecture on reduced stability conditions.
\begin{Con}\label{conj:mainsb}
   There exists a family of reduced stability conditions $\sb_H^{*}(X)$ with the following properties:
    \begin{enumerate}
        \item The forgetful map 
        \begin{align*}
            \Forg: \sb_H^{*}(X)\to \Hom(\Lambda_H,\R): \ts=(\cA,B)\mapsto B
        \end{align*}
        is a homeomorphism onto $\BBB_n^{}$. Moreover, the extended map $\Forg:\coprod_{n\in\Z}\sb_H^{*}(X)[n]\to \pm\BBB_n$ is a universal cover.
        \item The space $\sb_H^{*}(X)$ is preserved under the $\otimes \cO_X(H)$ action.
        \item For any $\ts_{\us},\ts_{\ut}\in\sb_H^{*}(X)$ with $\us<\ut<\us[1]$, the relation $\ts_{\us}\lsm\ts_{\ut}\lsm\ts_{\us}[1]$ holds.
    \end{enumerate}
\end{Con}

\noindent As we will show in Theorem \ref{thm:equivalentconjs}, for any polarized smooth variety, Conjecture \ref{conj:mainsb} is equivalent to Conjecture \ref{conj:intro}. In particular, Conjecture \ref{conj:mainsb} also implies  Theorem \ref{thm:intro}. The statement Theorem \ref{thm:intro}.{(5)} can be stated under the setup of reduced stability conditions as follows. 

For every reduced stability condition $\ts_{\ut}\in\sb^*_H(X)$ and $\ts_{\ut}$-semistable object $E$, its $H$-polarized character of $E$ is in the form of $\lambda_H([E])= \sum_{i=1}^n (-1)^ia_i\gamma_n(t_i)$ with coefficients $a_i\geq 0$ (or $\leq 0$) for all $i$. When $X$ is a threefold,  this is exactly the same as \eqref{eq1018} which is equivalent to the family of Bogomolov-type inequalities given in \eqref{eq117}.

To show that  Conjecture \ref{conj:mainsb} implies Conjecture \ref{conj:intro}, the key ingredient is the reconstruction formula as stated in equation \eqref{eq112}. 

Conversely, to  deduce  Conjecture \ref{conj:mainsb} from Conjecture \ref{conj:intro}, the main missing property is the comparison between reduced stability conditions as described in Conjecture \ref{conj:mainsb}.{(c)}. The proof relies on general structural results of reduced stability conditions and the relation $\lsm$. As a consequence, the Bayer Vanishing Lemma in Theorem \ref{thm:intro}.{(4)} essentially follows from the properties  that `$\ts_{\ut}\otimes \cO_X(H)=\ts_{\ut+1}$' in Conjecture \ref{conj:mainsb}.{(b)} and  `$\ts_{\ut}\lsm\ts_{\ut+1}$' (when $t_n\neq +\infty$) in Conjecture \ref{conj:mainsb}.{(c)}.

\subsection{Restriction of stability conditions}
To prove Theorem \ref{thm:intro}.{(2)}, we need a general result concerning the restriction of stability conditions from a variety to its hypersurfaces. To this end, we introduce an analogue of the relation `$\lsm$' on $\Stab(\cT)$, similar to the one defined on $\sb(\cT)$.

Given two stability conditions $\sigma, \tau \in \Stab(\cT)$, we define:

\begin{align*}
\sigma\lsm \tau & :\iff \cP_{\sigma}(\theta)\subset\cP_{\tau}(<\theta) \text{ for every }\theta\in \R;\\
\sigma\lsmeq \tau & :\iff \cP_{\sigma}(\theta)\subset\cP_{\tau}(\leq\theta) \text{ for every }\theta\in \R.
\end{align*}
For example, for every stability condition $\sigma$ in  $\Stab^*_H(X)$, we have the Bayer property $\sigma\lsmeq \sigma\otimes \cO_X(H)$.

The restriction theorem is stated as follows:
 \begin{Prop}[{\cite[Corollary 2.2.2]{Polishchuk:families-of-t-structures}}]\label{prop:introreststab}
     Let $\iota:Y\hookrightarrow X$ be an inclusion of smooth projective varieties, with $Y\in|D|$.  Let $\sigma=(\cP,Z)$ be a stability condition on $\Db(X)$ satisfying \begin{align}\label{eq:121}
         \sigma\otimes \cO_X(D)\lsm \sigma[1].
     \end{align} Then the datum
      \begin{align*}
          \sigma|_{\Db(Y)}\coloneq (\cP|_{\Db(Y)},Z\circ[\iota_*])
      \end{align*}
      defines a stability condition on $\Db(Y)$. Here the slicing is given by
    \begin{align*}
        \cP|_{\Db(Y)}(\theta)\coloneq \{E\in\Db(Y)\colon\iota_*E\in\cP_\sigma(\theta)\}\text{ for every $\theta\in \R$. }
    \end{align*} 
\end{Prop}
As we will see in the main text, when $D=mH$,  a stability condition  $\sigma$ in $\Stab^*_H(X)$ with central charge $\vd_{\us}+i\vd_{\ut}$ satisfies the condition \eqref{eq:121} $\sigma\otimes \cO_X(mH)\lsm \sigma[1]$ whenever
\begin{align}\label{eq:condm}
   \text{`every $\vd_{\ut'}$ in the pencil spanned by $\vd_{\us}$ and $\vd_{\ut}$ satisfies $t_{i+1}'-t_i'>m$ for every $i$.' }
\end{align}

For example, in the surface case, by definition of the restricted slicing, Proposition \ref{prop:introreststab} recovers the estimation of the first wall for $\iota_*E$.\\

\noindent The restricted stability condition preserves the property \eqref{eq:121}. In particular, we may keep restricting the stability condition to subvarieties in $|D_Y|$ on $Y$. To conclude Theorem \ref{thm:intro}.{(2)}, we need to deal with subvarieties that are not  complete intersections with respect to $D$. The following result enables us to modify the polarization accordingly and complete the argument.

\begin{Prop}\label{prop:introbayerall}
    Let $(X,H)$ be an irreducible smooth polarized variety over $\C$. Then for every divisor $D$ on $X$, there exists an integer $m(D)$ such that for every geometric stability condition $\sigma$ satisfying $\sigma\lsmeq \sigma\otimes \cO_X(H)$, we have
    \begin{align*}
        \sigma\lsmeq \sigma\otimes \cO_X\left(m(D)H+D\right).
    \end{align*}
\end{Prop}

As a direct corollary of Propositions \ref{prop:introreststab} and \ref{prop:introbayerall}, if for every $m\in\Z_{\geq 1}$ there exists a stability condition $\sigma_m$ on $X$ satisfying $\sigma_m\lsmeq \sigma_m\otimes \cO_X(H)$ and $\sigma_m\otimes \cO_X(mH)\lsm \sigma_m[1]$, then every smooth subvariety of $X$ admits stability conditions. Theorem \ref{thm:intro}.{(2)} is a special case of this corollary. \\

\noindent For a stability condition satisfying \eqref{eq:121} with $D$ a very ample divisor, we can successively restrict it down to points. In particular, all skyscraper sheaves become stable. Together with the Bayer property $\sigma\lsmeq \sigma\otimes \cO_X(H)$, the condition  \eqref{eq:121} can be viewed as a strong geometric assumption on $\sigma$ with respect to $X$. Another consequence of Proposition \ref{prop:introreststab} is that such stability conditions are entirely determined by their central charges.

\begin{Prop}\label{prop:introgeom}
    Let $(X,H)$ be an irreducible smooth variety with a very ample divisor $H$. Let $\sigma_1$, $\sigma_2$ be stability conditions on $X$ satisfying $\sigma_i\lsmeq\sigma_i\otimes \cO_X(H)\lsm\sigma_i[1]$ and  $Z_{\sigma_1}=Z_{\sigma_2}$. Then $\sigma_1=\sigma_2[2m]$ for some $m\in\Z$.
\end{Prop}

\subsection*{Organization of the paper}The main theoretical content of this paper is presented in the even-numbered sections. Readers interested primarily in the core arguments may safely skip the odd-numbered sections, which serve mainly to provide supplementary context and examples. Below is a detailed overview of the structure of the paper and its appendices.

Section \ref{sec:nested} introduces the concept of reduced stability conditions and defines the space $\sb(\cT)$. Additional technical details, including results on local topology and degeneracy loci, are developed in Appendices \ref{sec:degenloci} and \ref{sec:qudform}. Section \ref{sec:wc} investigates the wall and chamber structure of $\sb(\cT)$, with an example on the interpretation of the Bayer--Macr\`i divisor.

Section \ref{sec:order} introduces the notion $\lsm$ and studies its basic properties on both $\sb(\cT)$ and $\Stab(\cT)$. In Sections \ref{eg:cs}, \ref{sec:3fold}, and Appendix \ref{sec:appDunpolaized}, we study the family $\sb^*_H$ of reduced stability conditions in the cases of curves, surfaces, and polarized threefolds. We also explain that the classical conjecture concerning the existence of stability conditions on threefolds, which appears in \cite{BBMT:Fujita}, \cite{BMT:3folds-BG}, and \cite[Conjecture 4.1]{BMS:stabCY3s}, implies Conjecture \ref{conj:intro}.

Section \ref{sec:reslem} is devoted to the proof of Propositions \ref{prop:introreststab} and \ref{prop:introbayerall}, both of which are formulated and established using the $\lsm$ relation. In the final main section, Section \ref{sec:sm}, we state the conjectures on $\sb_H(X)$ and $\Stab_H(X)$, prove their equivalence, and complete the proof of Theorem \ref{thm:intro}. The arguments in this section rely on foundational, though nontrivial, results from linear algebra and interlaced polynomials, which are collected in Appendix \ref{sec:basicalgeb}.

\subsection*{Acknowledgements} The author would like to thank Arend Bayer, Yiran Cheng, Naoki Koseki, Wanmin Liu, Zhiyu Liu, Chunkai Xu, Qizheng Yin and Xiaolei Zhao for enlightening discussions. The author is supported by the Royal Society URF$\backslash$R1$\backslash$201129 “Stability condition and application in algebraic geometry”.\\

\section{Bertram Nested Wall Theorem} \label{sec:nested}
\subsection{Notions and definitions}\label{sec:recap}

We briefly recall some notions of Bridgeland stability conditions on a triangulated category.

\begin{Def}\label{def:slicing}
Let $\cT$ be a $k$-linear triangulated category. A \emph{slicing} $\cP$ on $\cT$ is a collection of full additive subcategories $\cP(\theta)\subset \cT$  indexed by $\theta\in \R$, satisfying the following conditions:
\begin{enumerate}
    \item For any $\theta\in \R$, we have $\cP(\theta)[1]=\cP(\theta+1)$.
    \item If $\theta_1>\theta_2$ and $F_i\in \operatorname{Obj}(\cP(\theta_i))$ for $i=1,2$, then $\Hom(F_1,F_2)=0$;
    \item \label{c}Every non-zero object $E \in \cT$ admits a finite sequence of distinguished triangles 
    \begin{equation*}
        \begin{tikzcd}[column sep=tiny]
0=E_0  \arrow[rr]& & E_1 \arrow[dl] \arrow[rr]& & E_2\arrow[r]\arrow [dl]& \cdots\arrow[r] &E_{m-1}\arrow [rr]& & E_m=E\arrow[dl]\\
& A_1 \arrow[ul,dashed, "+1"] && A_2\arrow[ul,dashed, "+1"] &&&& A_m \arrow[ul,dashed, "+1"]
\end{tikzcd}
    \end{equation*}
such that  each nonzero $A_i=\mathrm{Cone}(E_{i-1}\rightarrow E_i)$ belongs to $\cP(\theta_i)$ with real numbers $\theta_1>\dots >\theta_m$.
\end{enumerate}
\label{def:slicing}
\end{Def}

\begin{Not}\label{rem:slicing}
We call nonzero objects in $\cP(\theta)$ \emph{semistable} of \emph{phase} $\theta$, and refer to the simple objects in $\cP(\theta)$ as \emph{stable}. The sequence of triangles in Definition \ref{def:slicing}.\eqref{c} is unique up to isomorphism and is called the \emph{Harder–Narasimhan (HN) filtration} of an object $E\in\cT$. Each object $A_i$  in the filtration is called an \emph{HN factor} of $E$. We denote:
\[
\HN_\cP^+(E) \coloneqq A_1,\quad \HN_\cP^-(E) \coloneqq A_m,\quad \phi_\cP^+(E) \coloneqq \theta_1,\quad \phi_\cP^-(E) \coloneqq \theta_m.
\]
In particular, if $0\neq E\in \cP(\theta)$, its phase is written as $\phi_\cP(E)\coloneqq\theta$.

Given an interval $I \subset \R$, we define $\cP(I)$ to be the extension closure of the subcategories: $\{\cP(\theta):\theta\in I\}$. That is, $\cP(I)$ is the smallest full additive subcategory of $\cT$ containing all objects whose HN factors have phases in $I$. In particular, the slicing $\cP$ defines a bounded $t$-structure on $\cT$ whose heart is $\cP((0,1])$.

Given integers $a \leq b$ and a heart $\cA$ of a bounded $t$-structure on $\cT$, we denote by $\cA[a,b]$ the extension closure of  $\{\cA[k] : k \in [a,b]\}$. In particular, we have the equivalence
$\cA \subset \cA'[a,b] \iff\cA' \subset \cA[-b, -a]$.
\end{Not}

The distance between two slicings $\cP$ and $\cP'$ on $\cT$ is defined by
\begin{align}\label{eq:defslicing}
    d(\cP,\cP')\coloneq \sup_{0\neq E\in\cT}\{|\phi^+_\cP(E)-\phi^+_{\cP'}(E)|,|\phi^-_\cP(E)-\phi^-_{\cP'}(E)|\}\in[0,+\infty].
\end{align}

An equivalent expression, useful in applications, is the following (see \cite[Lemma 6.1]{Bridgeland:Stab}):
\begin{equation*}
    d(\cP,\cP')=\sup\{\phi^+_{\cP}(E)-\phi_{\cP'}(E),\phi_{\cP'}(E)-\phi^-_{\cP}(E)\colon 0\neq E\in\cP'(\theta)\text{ for some }\theta\in\R \}.
\end{equation*}
We denote by $\mathrm{K}(\cT)$ the Grothendieck group of $\cT$.

\begin{Def}
A \emph{Bridgeland pre-stability condition} on $\cT$ is a pair $\sigma=(\cP,Z)$, where 
\begin{itemize}
    \item $\cP$ is a slicing of $\cT$;
    \item $Z:\mathrm{K}(\cT)\rightarrow \C$ is a group homomorphism, called the \emph{central charge};
\end{itemize}
such that for any non-zero object $E$ in $\cP(\theta)$, we have \text{$Z([E])=m(E)e^{i\pi \theta}$ for some $m(E)\in \mathbb{R}_{>0}$.}
\label{def:prestab}
\end{Def}

Given a pre-stability condition $\sigma=(\cP_\sigma,Z_\sigma)$, we write $\HN^*_{\sigma}$ (resp. $\phi^*_{\sigma}$) for $\HN^*_{\cP_\sigma}$ (resp. $\phi^*_{\cP_\sigma}$).  We denote by $\cA_\sigma\coloneq \cP_\sigma((0,1])$ the heart of the associated bounded t-structure on $\cT$. The central charge satisfies  $Z_\sigma(\cA_\sigma\setminus\{0\})\subset \H\coloneq \{a+bi:b>0\text{ or } b=0>a\}$. A pre-stability condition $\sigma$ is uniquely determined by the datum $(\cA_\sigma,Z_\sigma)$, and we may freely refer to $\sigma$ as  $(\cA_\sigma,Z_\sigma)$ throughout the paper. \\

 Let  $\Lambda$ be a free abelian group of finite rank, and let $\lambda\colon \mathrm{K}(\cT)\twoheadrightarrow \Lambda$ be  a surjective  group homomorphism.
\begin{Def}[\!\cite{Bridgeland:Stab, Kontsevich-Soibelman:stability}]
A pre-stability condition $(\cP,Z')$ is said to  satisfy the \emph{support property} (with respect to $\Lambda$, or rather to $\lambda:\mathrm{K}(\cT)\twoheadrightarrow \Lambda$) if:
\begin{itemize}
\item the central charge $Z'$ factors through $\Lambda$, in other words, there exists a group homomorphism $Z\colon \Lambda \to \C$ such that $Z' = Z \circ \lambda$;
\item there exists a quadratic form $Q_\Lambda$ on $\Lambda_\R \coloneqq \Lambda \otimes \R$ such that:
\begin{enumerate}
\item  the kernel $\Ker Z \subset \Lambda_\R$ is negative definite with respect to $Q_\Lambda$;
\item  for every semistable object $E \in \cT$, we have $Q_\Lambda(\lambda([E])) \geq 0$.
\end{enumerate}
\end{itemize}
A pre-stability condition that satisfies the support property is called a \emph{(Bridgeland) stability condition} (with respect to $\Lambda$), and the collection of all such stability conditions is denoted by $\Stab_\Lambda(\cT)$.
\label{def:supportproperty}
\end{Def}

For every quadratic form $Q$ on $\Lambda_\R$, we define its negative cone as \begin{align*}
    \nega(Q)\coloneq \{v\in\Lambda_\R\;|\; Q(v,v)<0\}\cup \{0\}.
\end{align*}
\begin{Rem}[Support property]
We will also use the following equivalent formulation of support property, which is  the original version introduced in \cite{Bridgeland:Stab}:
\begin{center}
\emph{$\exists\; C > 0$ such that for every $\sigma$-semistable object $E \in \cT$, we have} $
    \|\lambda([E])\| \leq C \cdot |Z(\lambda([E]))|.
$
\end{center}

\noindent Here $||\bullet||$ is a norm on $\Lambda_\R$. The existence of $C$ does not rely on the choice of the norm.  By \cite[Lemma A.4]{BMS:stabCY3s}, these two definitions of support properties are equivalent.
\end{Rem}

When the lattice $\Lambda$ is not important in the statement, we will omit it and write $\Stab(\cT)$ for simplicity.

The set of all stability conditions carries a natural topology induced by the following generalized metric:
\[\mathrm{dist}(\sigma_1,\sigma_2)\coloneqq \max\left\{d(\cP_1,\cP_2), ||Z_1-Z_2||\right\}\in [0,+\infty].\]

\begin{Thm}[\!{\cite[Theorem 7.1]{Bridgeland:Stab}, \cite[Theorem 1.2]{Arend:shortproof}}]
 The map forgetting the slicing
\[\Forg_Z: \mathrm{Stab}_\Lambda(\cT) \rightarrow \mathrm{Hom}_{\mathbb Z}(\Lambda, \mathbb C)\colon \ \ \ \sigma =(\cP,Z) \mapsto Z \]
is a local isomorphism at every point of $\Stab_\Lambda(\cT)$.

In particular, whenever non-empty, the space $\Stab_\Lambda(\cT)$ is a complex manifold of dimension $\mathrm{rank}(\Lambda)$.
\label{thm:spaceasamfd}
\end{Thm}

\begin{Not}[$\glt$-action]\label{not:glt}
We will frequently use the $\glt$-action on (pre-)stability conditions throughout this paper. It is  worthwhile to recall some details of this action.

Let $\glr\coloneq \{M\in\mathrm{GL}(2,\mathbb R)\colon \det(M)>0\}$, and  let $\glt$ denote the universal cover of $\glr$.  We adopt the standard presentation of  $\glt$ as follows: an element $\widetilde{g}=(g,M) \in \glt$ consists of an element $M \in \glr$ together with a strictly increasing function $g\colon  \mathbb{R} \to \mathbb{R}$ satisfying \begin{align*}
    g(\phi +1 )= g(\phi) +1\text{ and }\begin{pmatrix}
    \cos g(\theta)\pi \\ \sin g(\theta)\pi
\end{pmatrix}= c_\theta M\begin{pmatrix}
    \cos \theta \pi \\ \sin \theta\pi
\end{pmatrix}
\end{align*}
for some $c_\theta\in\R_{>0}$.

There is a natural right group action of $\glt$ on the space of (pre-)stability conditions  defined by 
\[\sigma \cdot \widetilde{g} = (\mathcal{P}_\sigma((g(0), g(1)]), M^{-1} \circ Z_\sigma).\]

In particular, the new slicing $\cP_{\sigma\cdot \tilde g}(\theta)=\cP_\sigma(g(\theta))$. This action preserves any fixed lattice $\lambda:K(\cT)\to \Lambda$, and acts continuously on the space $\Stab_\Lambda(\cT)$.

The subgroup $\C=\R\oplus i\R\subset\glt$, corresponding to scaling and rotation, acts freely on $\Stab(\cT)$. For any $a+bi\in\C$ and stability condition $\sigma=(\cA,Z)$, the stability condition $\sigma\cdot(a+bi)$ is given by $\left(\cP_\sigma((b,b+1]),e^{-a-b\pi i}Z_\sigma\right)$. To simplify notation, for  $\theta\in \R$, we will write
\begin{align}\label{eq2310}
    \sigma[\theta]\coloneq \sigma\cdot (i\theta)=(\cP_\sigma((\theta,\theta+1]),e^{-\theta\pi i}Z_\sigma).
\end{align}
In particular, for $n\in \Z$, this gives $\sigma[n]=(\cA_\sigma[n],(-1)^nZ_\sigma$), which is consistent with the standard convention for shifts in triangulated categories.

Finally, for a stability condition whose central charge is written as $Z=g+if$ for some $f,g\in \Hom(\Lambda,\R)$, we will frequently use the notation \begin{align*}
    \sigma[\tfrac{1}{2}]=(\cP_\sigma((\tfrac{1}{2},\tfrac{3}{2}]),f-ig)
\end{align*} to denote the effect of half shift. \label{not:Caction} 
\end{Not}

\begin{Not}[Aut$(\cT)$-action]\label{not:actiononstab}
Let $\Phi$ be an exact autoequivalence of $\cT$, and denote by $\Phi_*:K(\cT)\to K(\cT)$ the induced isomorphism on the Grothendieck group. For a (pre-)stability condition $\sigma=(\cA,Z)$ on $\cT$, we define the action of $\Phi$ on $\sigma$ as  $\Phi\cdot \sigma\coloneq (\Phi(\cA),Z\circ \Phi^{-1}_*)$.

In general, this action does not preserve the fixed lattice $\lambda:K(\cT)\to \Lambda$. That is, even if $\sigma=(\cA,Z(\lambda(-)))\in\Stab_\Lambda(\cT)$, the transformed stability condition $\Phi\cdot \sigma=(\Phi(\cA),Z((\lambda\circ \Phi^{-1}_*)(-)))$ lies in $\Stab_{\Lambda'}(\cT)$, where the new lattice is given by $\lambda'=\lambda\circ \Phi^{-1}_*:K(\cT)\to \Lambda$. 

Assume now that the action of $\Phi$ is compatible with the lattice $\lambda$, meaning that there exists an isomorphism $\Phi_{\Lambda *}:\Lambda\to\Lambda$ such that $\Phi_{\Lambda *}\circ\lambda=\lambda\circ \Phi_*$. In this case,  $\Phi\cdot \sigma$ defines a stability condition in $\Stab_\Lambda(\cT)$ with central charge given by $Z\circ (\Phi_{\Lambda*})^{-1}\circ \lambda$. 
\end{Not}

\subsection{Nested wall theorem}
The following result, referred to as the nested wall theorem, serves as the starting point for taking a meaningful quotient of the stability manifold $\Stab(\cT)$. This phenomenon has been previously observed in the study of wall-crossing on stability conditions on surfaces; see, for example, \cite[Theorem 3.1]{Maciocia:walls}, as well as \cite{AB:Ktrivial, ABCH:MMP} for related developments. We will also present a version of this result in Corollary \ref{cor:nested} that closely aligns with the classical formulation.

\begin{Lem}[Bertram nested wall theorem]\label{lem:nested}
Let  $V\subset \Stab(\cT)$ be a path-connected subset such that  $\Im Z_\sigma=\Im Z_{\sigma'}\equiv\Im Z$ for every $\sigma,\sigma'\in V$. Then $\cP_\sigma(1)=\cP_{\sigma'}(1)$ for every $\sigma,\sigma'\in V$.
\end{Lem}
\begin{proof}
It suffices to show that for any object $E \in \cT$ with nonzero class $v \in \Ker(\Im Z)$, the stability type of $E$ - stable, strictly semistable, or unstable - is the same with respect to any two stability conditions $\sigma, \sigma' \in V$.

Both subsets
\[
\{\tau \in \Stab(\cT) \mid E \text{ is } \tau\text{-stable}\} \quad \text{and} \quad \{\tau \in \Stab(\cT) \mid E \text{ is } \tau\text{-unstable}\}
\]
are open, and hence remain open when restricted to $V$. Therefore, it remains to show that the set
\begin{align}\label{eq:tauinV}
\{\tau \in V \mid E \text{ is strictly } \tau\text{-semistable}\}
\end{align}
is also open in $V$.

Assume that $E$ is strictly $\sigma$-semistable for some $\sigma \in V$. In particular, $Z_\sigma(E) \neq 0$, and since $\Im Z(E) = 0$, we may apply a shift to assume $E \in \cP_\sigma(1)$. By the support property, $E$ admits a Jordan--H\"older filtration with $\sigma$-stable factors $E_1, \dots, E_m$, each lying in $\cP_\sigma(1)$.

Stability of each $E_i$ is an open condition on $\Stab(\cT)$, so there exists an open neighborhood $U \subset V$ of $\sigma$ such that every $E_i$ remains $\tau$-stable for all $\tau \in U$. Since $\Im Z(E_i) = 0$ and $Z_\tau(E_i) \neq 0$, we also have $E_i \in \cP_\tau(1)$. Thus, for every $\tau \in U$, the object $E$ is strictly $\tau$-semistable.

Hence, the subset \eqref{eq:tauinV} is open in $V$. As $V$ is path-connected, the stability type of $E$ with $\Im Z(E)=0$ is constant across all $\sigma \in V$. This implies that the slices $\cP_\sigma(1)$ and $\cP_{\sigma'}(1)$ coincide for all $\sigma, \sigma' \in V$, completing the proof.
\end{proof}

\begin{Lem}\label{lem:sameImimpliesheart}
    Let $V\subset \Stab(\cT)$ be a path-connected subset such that  $\cP_\sigma(1)=\cP_{\sigma'}(1)$ for every $\sigma,\sigma'\in V$. Then $\cA_\sigma=\cA_{\sigma'}$ for every $\sigma,\sigma'\in V$. 
\end{Lem}
\begin{proof}
    Let $\Gamma\subset V$ be a path from $\sigma$ to $\sigma'$. Then for every object $E\in\cP_\sigma((0,1))$, the function 
    $$f:\Gamma\to \R: \tau\mapsto \phi^+_\tau(E)$$ 
    is continuous by the definition of the topological structure on $\Stab(\cT)$.\\
    
    Suppose, for contradiction, that $f(\sigma')\geq 1$. Then by continuity, there exists $\tau\in \Gamma$ such that $f(\tau)=1$. By the assumption, we have
    \begin{align*}
        \HN^+_\tau(E)\in\cP_\tau(1)=\cP_\sigma(1).
    \end{align*} 
    Note that $E\in\cP_\sigma((0,1))$, we have $\Hom(\HN^+_\tau(E),E)= 0$. 
    
    On the other hand, the object $\HN^+_\tau(E)$ is the first HN-factor of $E$ with respect to $\tau$, so $\Hom(\HN^+_\tau(E),E)\neq0$. This lead to the contradiction. So we must have $f(\sigma')<1$, in other words, we have $\phi^+_{\sigma'}(E)<1$.\\

    By the same argument, we have $\phi^-_{\sigma'}(E)>0$. Therefore, we have $E\in\cP_{\sigma'}((0,1))$ for every $E\in \cP_\sigma((0,1))$. It follows that $\cP_\sigma((0,1))\subseteq\cP_{\sigma'}((0,1))$. 
    
   Reversing the rule of $\sigma$ and $\sigma'$, the same argument yields the reverse inclusion, so $\cP_\sigma((0,1))\supseteq\cP_{\sigma'}((0,1))$. Therefore, the heart $\cA_\sigma=\cA_{\sigma'}$ for every $\sigma,\sigma'\in V$.
\end{proof}

  \subsection{Reduced stability conditions}
  Let $\Forg_{\Im}: \Stab_\Lambda(\cT)\to \Hom(\Lambda,\R):(\cA,Z)\mapsto \Im Z$  be the natural forgetful map to the imaginary part of the central charge. We define an equivalent relation  on $\Stab_\Lambda(\cT)$ as follows.
\begin{Def}\label{def:sb}
    Two stability conditions $\sigma=(\cP,Z)$ and $\sigma'=(\cP',Z')$ are said to be equivalent, written as $\sigma\sim \sigma'$ if
\begin{enumerate}
    \item [(1)\;] $\Im Z=\Im Z'$;
    \item [(2)\;] $\sigma$ and $\sigma'$ lie in the same path-connected component of the fiber $(\Forg_{\Im})^{-1}(\Im Z)$.
\end{enumerate}
\end{Def}
It is clear from the definition that $\sim$ defines an equivalent relation on $\Stab_\Lambda(\cT)$. We define the quotient space  
\begin{align*}
    \sb_\Lambda(\cT)\coloneq \Stab_\Lambda(\cT)/\sim
\end{align*} 
equipped with the quotient topology induced from $\Stab_\Lambda(\cT)$. We call an element $\ts$ in $\sb_\Lambda(\cT)$ a \emph{reduced stability condition} on $\cT$.

We denote the natural quotient map by
\begin{align*}
    \pi_\sim:\Stab_\Lambda(\cT)\to \sb_\Lambda(\cT).
\end{align*}

Any stability condition $\sigma\in (\pi_\sim)^{-1}(\ts)$ is refer to as a representative of $\ts$.

\noindent By Lemmas \ref{lem:nested} and \ref{lem:sameImimpliesheart}, for any reduced stability condition $\ts \in \sb_\Lambda(\cT)$, the following data are independent of the choice of representative $\sigma \in \pi_\sim^{-1}(\ts)$:
\begin{itemize}
    \item the imaginary part of the central charge, denoted by $B_{\ts} \coloneq \Im Z_{\sigma}$,
\item the heart of the t-structure, $\cA_{\ts} \coloneq \cA_{\sigma}$, and
\item  the slice $\cP_{\ts}(1) \coloneq \cP_{\sigma}(1)$.
\end{itemize}
We refer to $B_{\ts}$ as the reduced central charge of $\ts$. Accordingly, other notions such as $\cP_{\ts}(<1)$, $\cP_{\ts}(>0)$, and $\cA_{\ts}[\leq 1]$ are also well-defined for reduced stability conditions.

We will show in Proposition \ref{prop:convex} that  Definition \ref{def:sb}.(2) can be replaced by other equivalent conditions. For example, one may require that $d(\cP_\sigma,\cP_{\sigma'}) < 1$, or that every linear combination $aZ + bZ'$ (with $a,b \in \R_{>0}$) defines a stability condition on a fixed heart $\cA$. We adopt condition (2) as the definition because  it is an equivalent relation directly and the quotient topology is easy to describe.

\subsection{Local chart for $\sb(\cT)$} 
 
\begin{Prop}\label{prop:localisom}
    The forgetful map 
    \begin{align*}
        \Forg:\sb(\cT)\to \lbdd: \ts\mapsto B_{\ts}
    \end{align*}
    is a local homeomorphism.
\end{Prop}
The argument is by basic point set topology.
\begin{proof}
Consider the commutative diagram
    \begin{equation}
        \begin{tikzcd}\label{diag:1}
            \Stab(\cT)\arrow{r}{\pi_\sim}\ar{d}{\Forg'} & \sb(\cT) \ar{d}{\Forg}\\
           \begin{array}{cc}
                 \Hom(\Lambda,\C) \\
               Z= Z_R+iZ_I
           \end{array} \ar{r}{\pi_{\Im}}& \begin{array}{cc}
              \lbdd \\
              Z_I
           \end{array}.
        \end{tikzcd}
    \end{equation}
    Let $\ts_0\in \sb(\cT)$ with a representative $\sigma_0$. By \cite[Theorem 1.2]{Bridgeland:Stab}, there exists an open neighborhood $U$ of $\sigma_0$ such that $\Forg'|_U$ is a homeomorphism onto its image. Shrinking $U$ if necessary, we may assume that
     \begin{align*}
        \Forg'(U)=W\times V\subset \Hom(\Lambda,\R)\times i\Hom(\Lambda,\R)=\Hom(\Lambda,\C)
    \end{align*} 
   with $W$ and $V$ both open and path-connected.\\
    
\noindent  We first show that $\pi_\sim(U)$ is open. As $\sb(\cT)$ adopts the quotient topology, that is just to show that $(\pi_\sim)^{-1}(\pi_\sim(U))$ is open in $\Stab(\cT)$.
    
    For any $\sigma\in (\pi_\sim)^{-1}(\pi_\sim(U))$, by definition there exists $\tau\in U$ with $\tau\sim \sigma$. Let $\gamma$ be a path in $(\Forg_{\Im})^{-1}(\Im Z_\sigma)$ connecting $\sigma$ and $\tau$. Then for every point $\sigma_t\in \gamma$, by \cite[Theorem 1.2]{Bridgeland:Stab}, there exists an  open neighborhood $U_t$ of $\sigma_t$ for which $\Forg'|_{U_t}$ is a homeomorphism. We may shrink $U_t$ so that 
    \begin{align*}
        \Forg'(U_t)=W_t\times V_t\subset \Hom(\Lambda,\R)\times i\Hom(\Lambda,\R)=\Hom(\Lambda,\C)
    \end{align*} 
     with $W_t$ and $V_t\subset V$ both open and path-connected. Moreover, we may assume that the open neighborhood of $\tau$ is contained in $U$.
     
   As $\gamma$ is compact, the curve can be covered by finitely many $U_{t_i}$ with $U_{t_0}\ni \sigma$ and $U_{t_n}\ni \tau$. In particular, the subset $V'\coloneqq\cap V_{t_i}$ is open and it contains $\Im Z_\sigma$. 

    For every $\sigma'\in (\Forg')^{-1}(W_{t_0}\times V')\cap U_{t_0}$, by the construction, we have $\sigma'\sim \tau'$ for some $\tau'\in U_{t_n}\subset U$. So $(\Forg')^{-1}(W_{t_0}\times V')\cap U_{t_0}$ is open and contained in $ (\pi_\sim)^{-1}(\pi_\sim(U))$. So $(\pi_\sim)^{-1}(\pi_\sim(U))$ is open.\\
    
  \noindent We then show that the  map $\Forg|_{\pi_\sim(U)}:\pi_\sim(U)\to V$ is a homeomorphism.
  \begin{itemize}
      \item The map  $\Forg$ is continuous since both $\Forg'$ and $\pi_{\Im}$ are continuous and $\pi_\sim$ is a quotient map.
      \item The map $\Forg|_{\pi_\sim(U)}$ is onto $V$ since $V=\pi_{\Im}(\Forg'(U))=\Forg(\pi_\sim(U))$.
      \item  For every $\ts,\ttau\in \pi_\sim(U)$ with $\Forg(\ts)=\Forg(\ttau)$, we may choose $\sigma$ and $\tau$ in $U$ being representatives of them respectively. In particular, we have $B\coloneq \Im Z_\sigma=\Im Z_\tau$ and $\sigma,\ttau\in (\Forg'|_U)^{-1}(W\times \{B\})$. As $W$ is assumed to be path-connected, we have $\sigma\sim \tau$ by definition. So the map $\Forg|_{\pi_\sim(U)}$ is one-to-one.
      \item For every open subset $X\subset \pi_\sim(U)$, the subset $(\pi_\sim)^{-1}(X)\cap U$ is open. Since $\Forg'|_U$ is a homeomorphism from  $U$ to $\Forg'(U)$ and $\Forg|_{\pi_\sim(U)}$ is one-to-one, the subset $(\pi_{\Im}|_{\Forg'(U)})^{-1}(\Forg(X))=\Forg'((\pi_\sim)^{-1}(X)\cap U)$ is open. By the choice of $U$, the topology on $V$ is also the quotient topology induced from $\Forg'(U)$, so the subset $\Forg(X)$ is open.   
\end{itemize}
 
To sum up, the set $\pi_\sim(U)$ is an open neighborhood of $\ts_0$ in $\sb(\cT)$ with the map $\Forg|_{\pi_\sim(U)}$ being a homeomorphism onto $V$. The statement holds.
 \end{proof}

\begin{Rem}[Actions on $\sb$]\label{rem:Ractiononsb}
    There is an $\R$-action on $\sb(\cT)$, defined by $\ts\cdot c\coloneqq \pi_\sim(\sigma\cdot c)$. The action is just to scaling the reduced central charge by $e^{-c}$.

    The group $\Aut(\cT)$  acts from the left on $\sb(\cT)$ as $\Phi\cdot \ts\coloneq \pi_\sim(\Phi\cdot \sigma)$ which does not rely on the choice of $\sigma$. In particular, we have $\cA_{\Phi\cdot\ts}=\Phi(\cA_{\ts})$ and  $B_{\Phi\cdot\ts}=B_{\ts}\circ( \Phi_*)^{-1}$.
\end{Rem}

 The following statement is clear from the proof of Proposition \ref{prop:localisom} and the diagram \eqref{diag:1}.
 \begin{Cor}\label{cor:localchart}
     Let $\sigma_1\sim \sigma_2$, then there exist open neighborhoods $U_i$ of $\sigma_i$ in $\Stab(\cT)$ and homeomorphism $f:U_1\to U_2$ such that
     \begin{itemize}
         \item for each $i$, the map $\Forg'|_{U_i}$ is a homeomorphism from $U_i$ to $\Forg'(U_i)$;
         \item $f(\sigma_1)=\sigma_2$;
         \item $\cA_\tau=\cA_{f(\tau)}$ and $Z_\tau-Z_{f(\tau)}\equiv Z_{\sigma_1}-Z_{\sigma_2}$ for every $\tau\in U_1$.
     \end{itemize}
     
     In other words, there is an open neighborhood $U$ of $0$ in $\Hom_{\Z}(\Lambda,\C)$ commutes the following diagram of homeomorphisms:
     \begin{equation*}
         \begin{tikzcd}
             (U_1,\sigma_1)\ar{r}{f}\ar[d,"\Forg-Z_{\sigma_1}"'] & (U_2,\sigma_2)\subset \Stab(\cT) \arrow{dl}{\Forg-Z_{\sigma_2}}\\
              (U,0).
         \end{tikzcd}
     \end{equation*}
 \end{Cor}

 \begin{Lem}\label{lem:d<1}
     Let $\sigma$ and $\tau$ be two stability conditions satisfying  $\Im Z_\sigma=\Im Z_\tau$ and $d(\cP_\sigma,\cP_\tau)<1$.     Then $\cA_\sigma=\cA_\tau$.
 \end{Lem}
\begin{proof}
    As $d(\cP_\sigma,\cP_\tau)<1$, we have $\cA_\sigma\subset \cA_\tau[-1,0,1]$ and $\cA_\tau\subset \cA_\sigma[-1,0,1]$. \\
    
   We first show that $\cA_\sigma\subset \cA_\tau[-1,0]$. Suppose $\cA_\sigma\not\subset\cA_\tau[-1,0]$, then there exists an object $F\in\cA_\sigma$ fitting into the distinguished triangle
    \begin{equation}\label{eq25}
        F_+\to F\to F_-\xrightarrow{+}
    \end{equation}
    for some non-zero $F_+\in \cA_\tau[1]$ and $F_-\in\cA_\tau[-1,0]$. As $d(\cP_\sigma,\cP_\tau)<1$, we have $F_+\notin \cP_\tau(2)$. It follows that $\Im Z_\tau(F_+)<0$. 

    Since $F_+\in \cA_\tau[1]\subset \cA_\sigma[0,1,2]$ and $\Im Z_\tau(F_+)<0$, the object $F'\coloneq \HN^+_\sigma(F_+)$ is in $\cA_\sigma[1,2]$. As $F'$ is the first HN-factor of $F$, we have $\Hom(F',F)\neq0$. 
    
    Applying $\Hom(F',-) $ to \eqref{eq25}, we get the long exact sequence:
    \begin{equation*}
        \dots \to \Hom(F',F_-[-1])\to \Hom(F',F_+)\to \Hom(F',F)\to \dots.
    \end{equation*}
    
    As $F_-[-1]\in\cA_\tau[-2,-1]\subset\cA_\sigma[-3,-2,-1,0]$, we have $\Hom(F',F_-[-1])=0$.
    
    As $F\in \cA_\sigma$, we have $\Hom(F',F)=0$.
    
    This leads to the contradiction with $\Hom(F',F)\neq0$. \\
    
    So we must have $\cA_\sigma\subset \cA_\tau[-1,0]$. Due to the same reason, we have $\cA_\tau\subset \cA_\sigma[-1,0]$.

   Given an object $F\in\cA_\sigma$, then it fits into the distinguished triangle as that of \eqref{eq25} for some $F_+\in \cA_\tau$ and $F_-\in \cA_\tau[-1]$. However, as $\cA_\tau[-1]\subset\cA_\sigma[-2,-1]$, we have $\Hom(F,F_-)=0$. It follows that $F_-=0$. So $F\in \cA_\tau$. 
    
    For the same reason $\cA_\tau\subset \cA_\sigma$. The statement holds. 
\end{proof}

\subsection{Convexity of the fiber of $\pi_\sim$}
\begin{Prop}\label{prop:convex}
    Let $\sigma=(\cA,Z)$ and $\sigma'=(\cA',Z')$ be two stability conditions with $\Im Z=\Im Z'$, then the following statements are equivalent:
    \begin{enumerate}[(1)]
        \item $\cA=\cA'$ and the pair of datum $(\cA,aZ+bZ')$ is a stability condition for every $a,b\in\R_{>0}$.
        \item $d(\cP_\sigma,\cP_{\sigma'})<1$.
        \item $\exists$ open neighborhoods $U$ and $U'$ of $\sigma$ and $\sigma'$ respectively in $\Stab(\cT)$ and homeomorphism $f:U\to U'$ satisfying $f(\sigma)=\sigma'$ and $\cA_{f(\tau)}=\cA_\tau$ for every $\tau\in U$.
        \item $\sigma\sim \sigma'$.
    \end{enumerate}
\end{Prop}
\begin{proof}
{(1)}$\implies${(4)}: The path $\gamma:[0,1]\to \Stab(\cT): t\mapsto (\cA,(1-t)Z+tZ')$ connects $\sigma$ and $\sigma'$ satisfying $\Im Z_{\gamma(t)}=(1-t)\Im Z+t\Im Z'=\Im Z$. By definition, we have $\sigma\sim \sigma'$.

\noindent {(4)}$\implies${(3)}: This follows directly from Corollary \ref{cor:localchart}.\\

   \noindent {(3)}$\implies${(2)}: Since $\cA=\cA'$, we have $d(\cP_\sigma,\cP_{\sigma'})\leq 1$.
   
   Suppose for contradiction that the equality  holds, in other words, $d(\cP_\sigma,\cP_{\sigma'})= 1$. Note that for any subobject (resp. quotient object) $F$ of $E$ in $\cA$, we have $\phi^+_{\sigma'}(E)\geq \phi^+_{\sigma'}(F)$ (resp. $\phi^-_{\sigma'}(E)\leq \phi^-_{\sigma'}(F)$). So by taking the Harder--Narasimhan or Jordan--H\"older factors if necessary, there exists an infinite sequence of $\sigma$-stable object $E_1,\dots, E_n,\dots$ with $\lim\phi_\sigma(E_n)=1$ (or resp. $=0$) and $\lim \phi^+_{\sigma'}(E_n)=0$ (resp. $\phi^-_{\sigma'}(E_n)=1$). Without loss of generality, we only prove that the $\lim\phi_\sigma(E_n)=1$ case will lead to a contradiction. 
   
   By definition, we have  \begin{equation}\label{eq:21}
        \lim_{n\to +\infty}\arg(Z(E_n))=\pi \text{ and }\lim_{n\to +\infty} \arg( Z'(E_n))=0.
    \end{equation}
    We may choose open neighborhoods $U$ and $U'$ of $\sigma$ and $\sigma'$ respectively satisfying the assumption as that in {(3)}. In addition, we may require that for every $\tau\in U$ and $\tau'\in U'$, the distance $d(\cP_\sigma,\cP_\tau)<\tfrac{1}{4}$ and $d(\cP_{\sigma'},\cP_{\tau'})<\tfrac{1}{4}$.

    When $n$ is sufficiently large, for every $\tau\in U$, we have $\phi^\pm_\tau(E_n)\in (\tfrac{1}{2},\tfrac{3}{2})$. For every $\tau'\in U'$, $\phi^\pm_{\tau'}(E_n)\in (-\tfrac{1}{2},\tfrac{1}{2})$. It follows that 
    \begin{align*}
        E_n\in\langle \cA_\tau,\cA_\tau[1]\rangle\cap\langle \cA_{\tau'}[-1],\cA_{\tau'}\rangle.
    \end{align*}
    Let $\tau'=f(\tau)$. It follows that
    \begin{align}\label{eq:EnincA}
        E_n\in\cA_\tau\text{ for every }\tau\in U.
    \end{align} 

\noindent    On the other hand, by \eqref{eq:21}, we have $ \lim_{n\to +\infty}\frac{\Re Z(E_n)}{\Im Z(E_n)}=-\infty$. Note that there exists $\delta>0$ such that $Z+i\delta\Re Z\in \Forg(U)$. We may let $\tau\in U$  with $\Forg(\tau)=Z+i\delta \Re Z$. Then when $n$ is sufficiently large, we have $\Im Z_\tau(E_n)=\Im Z(E_n)+\delta \Re Z(E_n)<0$. 
    
    This leads to the contradiction with \eqref{eq:EnincA} that $E_n\in \cA_\tau$. Therefore, we must have $d(\cP_\sigma,\cP_{\sigma'})<1$.\\

\noindent    {(2)}$\implies${(1)}: By Lemma \ref{lem:d<1}, we have $\cA=\cA'$.
    
  We first deal with the degenerate cases.  If $\Im Z=0$, then $\cA=\cP_\sigma(1)$. In addition, an object is $\sigma$-semistable $\iff$ $\sigma'$-semistable $\iff$ non-zero in $\cA$.
  
  By the support property, there exists a constant $C>0$ such that for every object $0\neq E\in \cA$, $C|Z(E)|>||\lambda(E)||$ and $C|Z'(E)|>||\lambda(E)||$.  Note that $\Re Z(E)<0$ and $\Re Z'(E)<0$. So $\Re(aZ+bZ')(E)<0$ and $C|(aZ+bZ')(E)|>\min\{a,b\}||\lambda(E)||$ for every $a,b>0$. Therefore, the pair of datum $(\cA,aZ+bZ')$ is a stability condition.
    
    If $\Im Z\neq 0$ and the rank of $\{\Re Z,\Re Z',\Im Z\}$ is not $3$, then for every $a,b>0$, the pair of datum $(\cA,aZ+bZ')$ is on the $\glt$-orbit of $(\cA,Z)$ and is a stability condition. \\

\noindent    Now we may assume that $\Re Z,\Re Z',\Im Z$ are linearly independent. 
    
For every $v\in \Lambda_\R\setminus \Ker Z$, there is a unique $\phi(v)\in(-1,1]$ satisfying $Z(v)\in \R_{>0}\cdot e^{\pi i\phi(v)}$. Similarly, we may define $\phi'(v)$ with respect to $Z'$. Assume that $d(\cP_\sigma,\cP_{\sigma'})=1-\delta$ for some $\delta>0$. Denote by
    \begin{align}\label{eq:2M}
        M\coloneq \{0\neq v\in\Lambda_\R\setminus(\Ker Z\cup\Ker Z')\;:\; |\phi(v)-\phi'(v)|>1-\delta\}.
    \end{align} 
    As $\Re Z,\Re Z',\Im Z$ are linearly independent,  for every $a,b>0$ and $v\in \Ker(aZ+bZ')\setminus(\Ker Z\cup\Ker Z')$, $|\phi(v)-\phi'(v)|=1$. In particular, we have $M\supset \Ker(aZ+bZ')\setminus(\Ker Z\cup\Ker Z')$ .

     We may apply Lemma \ref{lem:B10} by setting $\Im Z=\Im Z'=h$, $\Re Z=f_1$, $\Re Z'=f_2$. There is then a value $d>0$ only depending on $\delta$ so that the $M_d$ as that in Lemma \ref{lem:B10} is contained in the set $M$ as that in \eqref{eq:2M}.\\

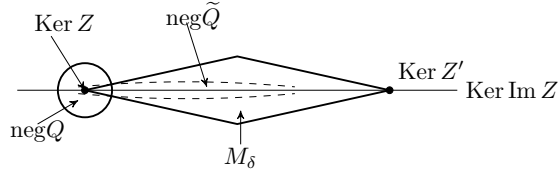
\begin{figure}[h]
    \centering
\scalebox{0.9}{ 
\begin{tikzpicture}[scale=1, transform shape]

  \draw (-2.5,0) -- (4,0) node[right] {$\Ker \Im Z$};

  \filldraw (-1.5,0) circle (0.05) node[below left] {};

  \filldraw (3,0) circle (0.05);
  \node[right] at (3,0.3) {$\Ker Z'$};

  \draw[thick] (-1.5,0) -- (0.75,0.5)--(3,0);
  \draw[thick] (-1.5,0) -- (0.75,-0.5)--(3,0);
  \node at (0.8,-1) {$M_\delta$};
  \draw[->](0.8,-0.8)--(0.8,-0.3);


  \draw[dashed] (-1.6,0.05) .. controls (-0.5,0.15) and (0.5,0.15) .. (1.6,0.05);
  \draw[dashed] (-1.6,-0.05) .. controls (-0.5,-0.15) and (0.5,-0.15) .. (1.6,-0.05);

  \draw[thick] (-1.5,0) circle (0.4);
  \node at (-2.2,-0.6) {$\nega Q$};
  
  \draw[->] (-2.2,-0.5)  -- (-1.6,-0.15);

  \node at (-1.8,1) {$\Ker Z$};
  \draw[->] (-2,0.8)  -- (-1.5,0);

  \draw[->] (0,1) -- (0.3,0.05);
  \node at (0.1,1.1) {$\nega \widetilde{Q}$};

\end{tikzpicture}
}

    \caption{Deform the kernel of central charge.}
    \label{fig:2}
\end{figure} 
     
     Let $Q$ (resp. $Q'$) be a quadratic form with signature $(2,\rho-2)$ for the support properties of $\sigma$ (resp. $\sigma'$). Then by Lemma \ref{lem:B10}, there exists a quadratic form $\tilde Q$ (resp. $\tilde Q'$) such that 
     \begin{align} \label{eq299}
         & \cup_{a\in \R,0\leq t\leq 2}\Ker (Z+tZ') \subset \nega(\tilde Q)\subset M\cup \nega(Q);\\
         & \cup_{a\in \R,0\leq t\leq 2}\Ker (Z'+tZ) \subset \nega(\tilde Q')\subset M\cup \nega(Q').\notag
     \end{align}

\noindent \textit{Claim}:    The quadratic form $\tilde Q$ gives the support property for  $\sigma$ as that in Definition \ref{def:supportproperty}.
    
\begin{proof}[Proof of the claim] (a) As $\Ker Z\subset \nega(\tilde Q)$, the restricted quadratic form $\tilde Q|_{\Ker Z}$ is negative definite.
    
\noindent   (b) For a $\sigma$-semistable object $E\in \cA$, suppose $\tilde Q(E)<0$, then by \eqref{eq299}, the character $\lambda(E)\in M$ or $\nega(Q)$. As $Q$ is for the support property of $\sigma$, we have $E\in M$. Note that $\cA=\cA'$, it follows that $\phi_{\sigma'}^-(E)\leq \phi'(\lambda(E))\leq \phi_{\sigma'}^+(E)$. Therefore, we have
    \begin{align*}
        d(\cP_\sigma,\cP_{\sigma'})\geq \max\{|\phi_\sigma(E)-\phi_{\sigma'}^-(E)|,|\phi_\sigma(E)-\phi_{\sigma'}^+(E)|\}\geq |\phi(\lambda(E))-\phi'(\lambda(E))|>1-\delta.
    \end{align*}
This leads to the contradiction. So for every $\sigma$-semistable $0\neq E\in \cA$, we have $\tilde Q(E)\geq 0$. 
\end{proof}

\noindent    Now by \eqref{eq299}, the restricted quadratic form $\tilde Q|_{\Ker(Z+tZ')}$ is negative definite for every $t\in[0,2]$. By \cite[Proposition A.5]{BMS:stabCY3s}, the stability condition $\sigma$ deforms to stability conditions with central charges $Z+tZ'$. By Lemma \ref{lem:nested} and \ref{lem:sameImimpliesheart}, the heart structures are the same as $\sigma$. By rescaling the central charges, we get stability conditions $(\cA,aZ+taZ')$ for all $a>0$ and $t\in[0,2]$.

    Repeat the above argument for $Q'$,  we get stability conditions $(\cA,tbZ+bZ')$ for all $b>0$ and $t\in[0,2]$. The statement holds.
\end{proof}

 \begin{Lem}\label{lem:B10}
     Let $h,f_1,f_2\in\lbdd$ be linearly independent elements and $d>0$. Let $Q$ be a quadratic form with signature $(2,\rho-2)$ and negative definite on $\Ker h\cap \Ker f_1$. Let 
     \begin{align*}
        M_d\coloneq \left\{v\in\Lambda_\R: f_1(v)f_2(v)<0,h(v)^2-df_1(v)^2<0,h(v)^2-df_2(v)^2<0\right\}.
    \end{align*}
       Then for every $N>0$, there exists a quadratic form $\tilde Q$  with signature $(2,\rho-2)$ such that 
          \begin{align} \label{eqB299}
        \Ker h\bigcap \left(\cup_{0\leq t\leq N}\Ker (f_1+tf_2)\right) \subset \nega(\tilde Q)\subset M_d\cup \nega(Q).
     \end{align}
     \end{Lem}
     \begin{proof}
   By the assumption, we may choose basis $\{\be_1,\dots,\be_\rho\}$ for $\Lambda_\R$ with dual basis $\{\be^*_1,\dots,\be^*_\rho\}$ such that $h=\be^*_1$, $f_1=\be^*_2$, $f_2=\be^*_3$. The set $M_d$ is then given as
   \begin{align*}
        M_d= \left\{\sum x_i\be_i: x_2x_3<0,x_1^2-dx_2^2<0,x_1^2-dx_3^2<0\right\}.
    \end{align*}
   
   By shrinking $\nega(Q)$ if necessary, we may assume the quadratic form $Q=Dx_1^2+Dx_2^2-x^2_3-\dots-x^2_\rho$ for some large $D>1$. 

    We may consider \begin{align*}
        \tilde Q=\tilde D_1x^2_1+\tilde D_2x_2(x_2+Nx_3)-x_4^2-\dots-x_\rho^2-\epsilon(x^2_2+(x_2+Nx_3)^2)
    \end{align*} for some $\tilde D_i>0$ and $0<\epsilon\ll 1$. When $\epsilon$ is sufficiently small, the form $\tilde Q$ is with signature $(2,\rho-2)$. 

    For $0\neq v\in\Ker h\cap\Ker(f_1+tf_2)$ with $0\leq t\leq N$, we have $v=(0,x,y,x_4,\dots,x_\rho)$ for some $x+yt=0$. It is clear that $\tilde Q(v)= (\tilde D_2t(t-N)-\epsilon t^2-\epsilon(t-N)^2)y^2-\sum_{i=4}^\rho x_i^2<0$. The first `$\subset$' in \eqref{eqB299} holds. \\

  \noindent  To show the second `$\subset$' in \eqref{eqB299}, we consider any $v\in\nega (\tilde Q)\setminus \nega(Q)$, then $\tilde Q(v)-Q(v)<0$, which implies
    \begin{align}\label{eqB2133}
        (\tilde D_1-D)x^2_1+(\tilde D_2-D-2\epsilon)x^2_2+(1-N^2\epsilon)x^2_3+N(\tilde D_2-2\epsilon)x_2x_3<0.
    \end{align}
     We may set $\epsilon$ sufficiently small so that $1>N^2\epsilon$; set $\tilde D_2=D+1>D+2\epsilon$, and $\tilde D_1>D$. It is then followed by \eqref{eqB2133} that $x_2x_3<0$. 
     
     Ignoring the $\epsilon$'s, the inequality \eqref{eqB2133} implies \begin{align*}
       &  (\tilde D_1-D)x_1^2+(x_3+\tfrac{1}{2}N(D+1)x_2)^2-(\tfrac{1}{4}N^2(D+1)^2-1)x_2^2<_\epsilon 0;\\
      \text{ and } &  (\tilde D_1-D)x_1^2+(x_2+\tfrac{1}{2}N(D+1)x_3)^2-(\tfrac{1}{4}N^2(D+1)^2-1)x_3^2<_\epsilon 0.
     \end{align*}
     By further letting $\tilde D_1>_\epsilon (\tfrac{1}{4}N^2(D+1)^2-1)/d+D$, it is then clear that $x_1^2-dx^2_2<0$ and $x_1^2-dx^2_3<0$. It follows that $v\in M_d$. Therefore, we have $\nega(\tilde Q)\subset M_d\cup \nega(Q)$.
    \end{proof}

\begin{Rem}[Convex Hull]
    One can apply Proposition \ref{prop:convex} to construct new stability conditions from old ones. For every subset of stability conditions $U\subset\Stab(\cT)$, we may define its \emph{convex hull} $\spa U$ as the smallest subset in $\Stab(\cT)$ closed under taking the $\glt$-action and operation as that in Proposition \ref{prop:convex}.{(1)}. More precisely, it can be defined as follows:
\begin{equation*}
    \spant1U\coloneq \left\{\sigma=(\cA,Z)\;\middle|\;\begin{array}{cc}
        Z=tZ_1+(1-t)Z_2\text{ for some }t\in[0,1], \\
         \text{ and }\sigma_i=(\cA,Z_i)\in U\cdot\glt,\sigma_1\sim\sigma_2
    \end{array} \right\}.
\end{equation*}
Define $\spant{n+1}U\coloneq \spant1{\spant nU}$ and $\spa U\coloneq \cup_{n=1}^{+\infty}\spant nU$. 
\end{Rem}

\begin{Lem}\label{lem:213}
    Assume that $E$ is $\sigma$-(semi)stable for every $\sigma\in U$, then $E$ is $\sigma$-(semi)stable for every $\sigma\in\spa U$.  
\end{Lem}
\begin{proof}
    It is clear that the $\glt$-action does not affect the stability of any object. So we only need to show that if $E$ is $\sigma_i$-(semi)stable, then it is $\tau$-(semi)stable with respect to  $\tau=(\cA,tZ_1+(1-t)Z_2)$ for every $t\in(0,1)$. By shifting $E$ if necessary, we may assume $E\in\cA$. 

    If $\Im Z(E)=0$, then the statement follows from Lemma \ref{lem:nested}. Otherwise, for every $0\neq F\hookrightarrow E$ in $\cA$, we have
    \begin{align*}
        \frac{\Re (tZ_1+(1-t)Z_2)}{\Im Z}(F)=t\frac{\Re Z_1}{\Im Z}(F)+(1-t)\frac{\Re Z_2}{\Im Z}(F)< (\leq) \frac{\Re (tZ_1+(1-t)Z_2)}{\Im Z}(E).
    \end{align*}
    So $E$ is $\tau$-(semi)stable.
\end{proof}

\begin{Ex}[Beilinson quiver stability]\label{eg:beilinsonquiver}
Let $\bP^n$ be the $n$-dimensional projective space, then we may consider the stability conditions offered by the Beilinson quiver, see \cite{Beilinson:EquivalencePn}. More precisely, for every $m\in \Z$, there is a   heart of bounded t-structure given by the extension closure:
\begin{align*}
    \cA_m\coloneq \langle\cO_{\bP^n}(m)[n],\cO_{\bP^n}(m+1)[n-1],\dots,\cO_{\bP^n}(m+n)\rangle.
\end{align*}

For every $(n+1)$-tuple of complex numbers $\underline{v}=(z_0,z_1,\dots, z_n)$ with every $z_i\in \H$, there is a unique central charge $Z_{\underline{v}}$ on $\Kn(\bP^n)$ by assigning $Z_{\underline{v}}(\cO_{\bP^n}(m+i)[n-i])=z_i$. The pair of datum $(\cA_m,Z_{\underline{v}})$ is a stability condition for every $m\in\Z$ and $\underline{v}\in\H^{n+1}$. We may consider the following set of stability conditions
\begin{align*}
    U\coloneq \{(\cA_m,Z_{\underline{v}}):m\in\Z, \underline{v}\in\H^{n+1}\}\cap \Stab^{\text{Geo}}(\bP^n),
\end{align*}
where $\Stab^{\text{Geo}}(\bP^n)$ stands for the space of geometric stability conditions as that in Definition \ref{def:geostab}.

Then when $n\leq 2$, by Lemma \ref{lem:213}, it is not difficult see that $\spa U=\Stab^{\text{Geo}}(\bP^n)$. When $n\geq 3$, the space $\spa U$ is strictly larger than $U$ but, unfortunately, is a proper subset of  $\Stab^{\text{Geo}}(\bP^n)$. To get the full family of stability conditions as that in Conjecture \ref{conj:intro}, one needs more tools, for instance, Proposition \ref{prop:lifttostab}, to extend $\spa U$. We leave this direction to a future project. 
\end{Ex}

\section{Wall and Chamber structure}\label{sec:wc}
In this section, we set up some notions for the wall and chamber structure on $\sb(\cT)$. The first difference from $\Stab(\cT)$ is that the $\ts$-stability depends on the representatives of $\sigma$ in general. However, by Lemma \ref{lem:nested}, the slice $\cP_{\ts}(1)$, or more generally $\cP_{\ts}(m)$ with $m\in \Z$, does not depend on the representatives. So for a given character $v$, it makes sense to define the $\ts$-stability for objects with character $v$ when $B_{\ts}(v)=0$. This leads to the following notion.
\begin{Def}\label{def:redstab}
Let $E$ be an object in $\cT$ with $B_{\ts}(E)=0$. We say that $E$ is $\ts$-\emph{(semi)stable} if it is $\sigma$-(semi)stable for a representative of $\ts$.

In particular, by Lemma \ref{lem:nested}, the object $E$ is $\ts$-(semi)stable if it is $\sigma$-(semi)stable for one representative of $\ts$.
\end{Def}
 
For every non-zero character $v\in\Lambda$, we define 
\begin{align*}
    \sb_v(\cT)\coloneq \{\ts\in\sb(\cT)\;|\; B_{\ts}(v)=0\}.
\end{align*}
Denote by 
\begin{align}\label{eq:defvstar}
    \vperp\coloneq \{f\in \lbdd\mid f(v)=0\}.
\end{align}
It is clear that the forgetful map  $\Forg:\sb_v(\cT)\to \vperp$ is a local homeomorphism.

We denote $M_{\ts}(v)$ the moduli space parametrizing $\ts$-semistable objects in $\cA_{\ts}$ with class $v$. By Lemma \ref{lem:nested} and \ref{lem:sameImimpliesheart}, the space $M_{\ts}(v)=M_\sigma(v)$ for every representative $\sigma$ of $\ts$.

\subsection{Removing the  locus with empty moduli} 

To relate the wall and chamber structures on $\Stab$ and $\sb$, for every $v$, we need a quotient map from $\Stab$ to $\sb_v$. However, $\sb_v$ is not a quotient space of $\sb$ in general. To solve this issue, we notice that the homological shift $\R$-action (resp. $\Z$) on $\Stab$ (resp. $\sb_v$) does not affect the stability of objects at all. This leads to an `expected map' from  $\Stab/\R$ to $\sb_v/\Z$. However, such a map still does not exist in general as it is not well-defined on $\sigma$ with $Z_\sigma(v)=0$. On the other hand, for such a stability condition, the moduli space $M_\sigma(v)$ is always empty. So removing them does not cause much problem. Accordingly,  we also remove the locus on $\sb_v$  where $M_{\ts}(v)$ is for sure to be empty.

More precisely, we denote\begin{align*}
    \sb_v^{\emptyset}(\cT)\coloneq \{\ts\in\sb_v(\cT) \colon \exists \text{ a representative }\sigma \text{ with } Z_\sigma(v)=0\}.
\end{align*} 

By the support property, there is no $\sigma$-semistable object with class $v$. In other words, the space $M_{\ts}(v)=\emptyset$ for every $\ts\in\sb_v^{\emptyset}(\cT)$.

Both spaces $\sb_v(\cT)$ and $\sb_v^{\emptyset}(\cT)$ are invariant under the homological shift $\Z$-action.

For every element $\ts$ in $\sb_v(\cT)\setminus \sb_v^{\emptyset}(\cT)$, by Proposition \ref{prop:convex}.{(1)}, the sign of $\Re Z_\sigma(v)$ does not rely on the choice of the representative $\sigma$. It follows that on each connected component of $(\pi_\sim)^{-1}(\sb_v(\cT)\setminus \sb_v^{\emptyset}(\cT))$, the sign of $\Re Z_\sigma(v)$ does not change. We denote $(\sb_v(\cT)\setminus \sb_v^{\emptyset}(\cT))^-$ as the component where $\Re Z_\sigma(v)<0$.

\begin{Not}\label{not:3.3}
    We denote
    \begin{align*}
    \sb^\dag_v(\cT)\coloneq (\sb_v(\cT)\setminus \sb_v^{\emptyset}(\cT))^-/(2\Z),
\end{align*} 
where $2\Z$ stands for the homological shift action $[2n]$ with an even degree. It is clear that the forgetful map $\Forg$ to the reduced central charge is well-defined on $\sb^\dag_v(\cT)$ and is a local isomorphism to $\vperp$.
\end{Not}

For every $\sigma\in \Stab(\cT)$ with $Z_\sigma(v)\neq 0$, there is a unique $\theta\in \R/\Z$ with $e^{-i\pi \theta}Z_\sigma(v)\in \R$. In particular, we have $\pi_\sim(\sigma[\theta])\in \sb_v(\cT)$. Denote 
\begin{align*}
    \Stab_v(\cT)\coloneq \{\sigma\in \Stab(\cT):Z_\sigma(v)\neq0, \pi_\sim(\sigma[\theta])\not\in \sb_v^{\emptyset}(\cT)\}.
\end{align*}

\begin{Rem}\label{rem:stabv}
Note that the $\glt$-action preserves the stability of objects. In particular, an object is $\sigma$-semistable if and only if it is $\sigma[\theta]$-semistable. So for every $\sigma\not\in\Stab_v(\cT)$, the space $M_\sigma(v)=\emptyset$.   
\end{Rem}
\begin{Def}\label{def:piv}
We define the map $\pi_v$ as:
 \begin{align*}
    \pi_v: \Stab_v(\cT) & \to \sb_v^\dag(\cT) \\
    \sigma & \mapsto \pi_\sim(\sigma[\theta]),
\end{align*}
where $\theta$ is the unique element in $\R/2\Z$ with $e^{-i\pi \theta}Z_\sigma(v)\in \R_{<0}$.   
\end{Def}

\subsection{Wall and Chamber structure}\label{sec:wcstructure}
For every non-zero character $v\in\Lambda$, one may consider the set of $\sigma$-semistable objects $E\in\cT$  with class $v$ as $\sigma$ varies.   The manifold $\Stab(\cT)$  admits a wall and chamber decomposition such that for every chamber $\cC_j(v)$, the space $M_\sigma(v)$ is independent of the choice of $\sigma\in\cC_j(v)$. 

We recall the following proposition/definition for walls and chambers. More details can be found in \cite[Section 9]{Bridgeland:K3}, \cite[Proposition 2.8]{Toda:K3},  \cite[Proposition 3.3]{localP2}, 
\cite{MYY2}, and \cite{MYY:wallcrossingK3}.

\begin{Prop}[{\cite[Proposition 2.3]{BM:projectivity}}]\label{prop:wallchambonstab}
    There exists a locally finite set of walls, real codimension one submanifolds $\cW_i(v)$'s with boundary, in $\Stab(\cT)$, depending only on $v$:
    \begin{align}\label{eq:wall}
        \Stab(\cT)=\left(\bigcup_i \cW_i(v)\right)\coprod\left(\coprod_j \cC_j(v)\right)
    \end{align}
    with the following properties:
    \begin{enumerate}
        \item Each chamber $\cC_j$ is open and path-connected. The space $M_\sigma(v)$ is independent with the  \textbf{generic} choice $\sigma$ within $\cC_j$.
\item When $\sigma$ lies on a single wall $\cW_i$, then there is a $\sigma$-semistable object that
is unstable in one of the adjacent chambers, and semistable in the other adjacent
chamber.
\item When we restrict to an intersection of finitely many walls $\cW_1,\dots,\cW_k$, we obtain a
wall-and-chamber decomposition on  $\cW_1\cap\dots\cap\cW_k$ with the same properties, where
the walls are given by the intersections $\cW\cap \cW_1\cap \dots \cap \cW_k$  for any of the walls $\cW\subset \Stab(\cT)$ with respect to v.
\end{enumerate}
\end{Prop}
\begin{Rem}[Isolated strictly semistable objects]
    For the sake of accuracy, we add the `generic' assumption on $\sigma$ in Proposition \ref{prop:wallchambonstab}.{(a)}, because the statement will fail otherwise in many cases of $\cT$. 
    
    For example, we may consider the category $\cT=\Db(\bP^2)$ and the heart structure $\cA $ generated by $\cO[4]$, $\cO(1)[2]$, and $\cO(2)$. In particular, an object is in $\cA$ if and only if it is the direct sum of these three generators. Consider all the stability conditions $\sigma$ on $\cA$ and the character $v=[\cO]+[\cO(1)]+[\cO(2)]$. It is clear that $M_\sigma(v)\neq \emptyset$ when and only when $\phi_\sigma(\cO[4])=\phi_\sigma(\cO(1)[2])=\phi_\sigma(\cO(2))$, which is a real codimension two condition. 

    On the other hand, these `isolated' strictly semistable objects do not affect any of the wall-crossing procedures. In particular, if an object is $\sigma$-semistable for generic $\sigma\in \cC$, then it is $\sigma$-semistable for all $\sigma\in\cC$.
\end{Rem}
\begin{Not}\label{not:Mchamber}
For every chamber $\cC$ as that in \eqref{eq:wall}, we denote by the set $M_\cC(v)\coloneq M_\sigma(v)$ for a generic $\sigma\in\cC$.
\end{Not}

\begin{Prop}(Wall and chamber structure on $\sb_v^\dag(\cT)$)\label{prop:wallchambonsb}
    The map $\pi_v:\Stab_v(\cT)\to \sb^\dag_v(\cT)$ preserves the wall and chamber structure and all chambers with non-empty moduli as that in Proposition \ref{prop:wallchambonstab}. More precisely, we set $\tilde\cW_i(v)\coloneq \pi_v(\cW_i(v)\cap\Stab_v(\cT))$ and $\tilde\cC_j(v)\coloneq \pi_v(\cC_j(v)\cap\Stab_v(\cT))$. Then 
    \begin{align}\label{eq:wallsb}
        \sb_v^\dag(\cT)=\left(\bigcup_i \tilde\cW_i(v)\right)\coprod\left(\coprod_j \tilde\cC_j(v)\right)
    \end{align}
    with the following properties:
    \begin{enumerate}
        \item Each wall $\tilde \cW_i(v)$ is a non-empty real codimension one submanifold with boundary. On each open local chart \[\sb_v^\dag(\cT)\supset U\xhookrightarrow{\Forg}\vperp\subset \lbdd,\]
        the image of the wall $\Forg(\tilde \cW_i(v)\cap U)$ is a subset of real codimension one linear subspace $\vperp\cap \wperp\subset \vperp$ for some $w\in \Kn(\cT)$.
        \item For every chamber $\cC_j(v)$ with $M_{\cC_j(v)}(v)\neq \emptyset$, the chamber $\tilde \cC_j(v)$ is non-empty, open and path-connected. The space $M_{\ts}(v)$ is independent with the  generic choice $\ts$ within $\cC_j(v)$.
    \end{enumerate}    
\end{Prop}
\begin{proof}
    Note that each $\cC_i(v)$ and $\cW_j(v)$ as that in \eqref{eq:wall} is invariant under the $\glt$-action, in particular, the action by $[\theta]$. The map $\pi_v$ reduces to a map from  $\Stab_v(\cT)/i\R$ to $\sb^\dag_v(\cT)$.
    
   We first show that the sets on the right hand side of \eqref{eq:wallsb} are disjoint.
   
   Suppose $\pi_v(\sigma)=\pi_v(\tau)$ for some $\sigma\in\cW_i(v)$  and $\tau\in\cC_j(v)$, then by taking a shift $[\theta]$ if necessary, we may assume that $\sigma\sim \tau$ with $\pi_\sim(\sigma)\in (\sb_v(\cT)\setminus \sb_v^{\emptyset}(\cT))^-$. This then contradicts to Lemma \ref{lem:nested} and Proposition \ref{prop:wallchambonstab}.
    
   Suppose $\pi_v(\sigma)=\pi_v(\tau)$ for some $\sigma\in\cC_{j'}(v)$ and $\tau\in\cC_j(v)$, then by taking a shift $[\theta]$ if necessary, we may assume that $\sigma\sim \tau$ with $\pi_\sim(\sigma)\in (\sb_v(\cT)\setminus \sb_v^{\emptyset}(\cT))^-$. In particular, we have $\Re Z_\sigma(v)\cdot\Re Z_\tau(v)>0$. By Proposition \ref{prop:convex}.{(1)}, the path $\gamma(t)=(\cA_\sigma,tZ_\sigma+(1-t)Z_\tau)$ is contained in $\Stab_v(\cT)$ for $0\leq t\leq 1$. Note that $\gamma(t)\sim \sigma$, by the argument above, the path does not intersect any of the walls. It follows that the path is contained in the same chamber, in particular, $\tilde\cC_{j'}(v)=\tilde\cC_j(v)$.

    To sum up, the disjoint union formula \eqref{eq:wallsb} holds.\\

   \noindent {(a)} By the construction of walls, each $\cW_i(v)$ is contained in a numerical wall $\cW(v,w)\coloneq \{\sigma\in\Stab(\cT): \arg Z(w)=\arg Z(v)\}$ for some non-zero character $w\in \Lambda$ and $[w],[v]$ linear independent in $\Lambda_\R$. So the set $\tilde \cW_i(v)$ is  contained in $\pi_v(\cW(v,w))$. It follows that
    \begin{align*}
        \tilde \cW_i(v) & \subseteq \pi_v(\cW(v,w))\\
        & =\{\ts\in \sb_v^\dag(\cT)\mid\arg Z_{\sigma}(w)=\arg Z_{\sigma}(v)\text{ for some representative }\sigma\text{ of }\ts\}\\
        & \subseteq \{\ts\in \sb_v^\dag(\cT)\mid 0=\Im Z_\sigma(w)=B_{\ts}(w)\} = \sb^\dag_v(\cT)\cap \sb_w(\cT).
    \end{align*}
On each open local chart $U\xhookrightarrow{\Forg}\vperp$ of $\sb^\dag_v(\cT)$, the subset $ \Forg(\sb_w(\cT)\cap U)=(\vperp\cap \wperp)\cap \Forg(U)$. As $[w],[v]$ are linear independent in $\Lambda_\R$, the linear subspace $\vperp\cap \wperp$ is of real codimension one in $\vperp$. The second part of the statement holds.

By Proposition \ref{prop:wallchambonstab}.{(b)}, we have $M_\sigma(v)\neq \emptyset$  for every $\sigma$ on the wall $\cW_i(v)$. By Remark \ref{rem:stabv}, each whole wall $\cW_i(v)\subset \Stab_v$. It follows that the wall $\tilde \cW_i(v)=\pi_v(\cW_i(v))$.

Note that $\cW_i(v)$ is with real codimension one and the map $\pi_v$ is with equal dimensional fibers, so $\tilde \cW_i(v)$ is with codimension at most one. As on each open local chart, the wall $\tilde \cW_i(v)$ is contained in a real codimension one linear subspace. It follows that globally $\tilde \cW_i(v)$ is a real codimension one submanifold with boundary.\\

\noindent { (b)} Note that for every chamber $\cC_k(v)\not\subset \Stab_v(\cT)$, by Remark \ref{rem:stabv},  there exists $\sigma\in\cC_k(v)$ such that $M_\sigma(v)=\emptyset$. It follows that $M_{\cC_k(v)}(v)=\emptyset$. 

Therefore, the chamber $\cC_j(v)$ as that in the statement is contained in $\Stab_v(\cT)$. The chamber $\tilde \cC_j(v)$ is non-empty and path-connected. By \eqref{eq:wallsb}, each chamber $\tilde \cC_j(v)$ is open. By Lemma \ref{lem:nested}, the last part of the statement holds.
\end{proof}

\begin{Rem}\label{rem:wcsb}
Let $\Stab^{\text{nd}}(\cT)\coloneq \{\sigma\in \Stab(\cT)\colon \Re Z_\sigma, \Im Z_\sigma\text{ linear independent.}\}$ be the submanifold of non-degenerate stability conditions and $\Stab_v^{\text{nd}}(\cT)\coloneq \Stab^{\text{nd}}(\cT)\cap \Stab_v(\cT)$. Then the image of $\Stab^{\text{nd}}_v(\cT)$ under $\Forg\circ \pi_v$ is in $\vperp\setminus\{0\}$. 
Indeed, if $\Forg(\pi_v(\sigma))=0$, then there exists $\theta\in\R$ such that $\Im(e^{-i\pi \theta}Z_\sigma)=0$. It follows that $\sigma\in\Stab^{\text{nd}}(\cT)$. 

In general, walls on the non-degenerate locus $\Stab^{\text{nd}}(\cT)$ are with the most interests, respectively, walls on $\sb^\dag_v(\cT)\setminus\{\ts:B_{\ts}=0\}$. 

The scaling $\R$-action acts freely on $\sb^\dag_v(\cT)\setminus\{\ts:B_{\ts}=0\}$. Note that each wall and chamber is invariant under the scaling $\R$-action. On every open local chart, we can further projectivize the wall and chamber on $\bP(\vperp)$. Each wall is then the subset of a hyperplane. When $\rk\Lambda= 4$, the wall and chamber structure can be displayed on a plane.
\end{Rem}

When $\rk\Lambda =3$, by the observation above, one can interpret  Proposition \ref{prop:wallchambonsb}.{(a)} as the Bertram nested wall theorem which has been broadly used in the wall-crossing on the stability manifold of a polarized surface.

\begin{Cor}[Bertram nested wall theorem]\label{cor:nested}
    Assume that $\rk\Lambda =3$, then for every non-zero $v\in\Lambda$,  the walls $\cW_i(v)$ on $\Stab^{\text{nd}}(\cT)$ are all disjoint from each other.
\end{Cor}
\begin{proof}
  By assumption, the real linear space $\vperp$ is with dimension two. By Remark \ref{rem:wcsb} and Proposition \ref{prop:wallchambonsb}, on every local chart of $\sb^\dag_v(\cT)$, a projectivized wall is a point on $\bP(\vperp)$. Therefore, the walls are all disjoint from each other.
\end{proof}

\subsection{Remark: Bayer--Macr\`i divisor} As a remark, we explain that the notion of the Cartier divisor class $\ell_{\sigma,\cE}$ on the moduli space $M_\sigma(v)$ as that in \cite[Proposition and Definition 3.2]{BM:projectivity} perfectly matches with the notion of reduced stability condition.

More precisely, one can make the following notion.
 \begin{Def}[Bayer--Macr\`i divisor]\label{def:BMdiv}
Let $X$ be a smooth projective variety over $\C$ and $v\in\Kn(X)$ be a non-zero numerical class. Assume that the lattice factors via the numerical Grothendieck group $\mathrm{K}(X)\to\Kn(X)\twoheadrightarrow\Lambda$. For every reduced stability condition $\ts\in\sb^\dag_v(X)$, assume that there is a family $\cE\in\Db(T\times X)$ of $\ts$-semistable objects of class $v$ parameterized by a proper algebraic space $T$ of finite type over $\C$.

The \emph{Bayer--Macr\`i divisor} $\ell_{\ts,\cE}$ on $T$ is defined as follows:  for every projective integral curve $C\subset T$, we set
\begin{align*}
    \ell_{\ts,\cE}([C])\coloneq B_{\ts} \left((p_X)_*\cE|_{C\times X}\right).
\end{align*}
\end{Def}

 \begin{Thm}[{\cite[Theorem 1.1]{BM:projectivity}} Positivity Lemma] \label{thm:bmdivnef}
     The divisor class $\ell_{\ts,\cE}$ is nef.
 \end{Thm}

Let $\ts$ be a generic reduced stability condition in a chamber $\tilde \cC(v)$, then by \cite[Theorem 1.1]{BM:projectivity}, every reduced stability condition in $\tilde \cC(v)$  associates a nef divisor on $T$. 
\def\BM{\ensuremath{\mathsf{BM}_\cE}}
\begin{align}\label{eq377}
   \BM': \tilde\cC(v)\subset\sb^\dag_v(X) & \to \Nef^0(T)\subset N^1(T)\\
   \ts & \mapsto \ell_{\ts,\cE} \notag
\end{align}
It is clear that the divisor $\ell_{\ts,\cE}$ above only relies on $B_{\ts}$ and $\cE$. Moreover, for every $a,b\in\R$, $B_{\ts}$, and $B_{\ttau}$, we have $\ell_{aB_{\ts}+bB_{\ttau},\cE}([C])=a\ell_{B_{\ts},\cE}([C])+b\ell_{B_{\ttau},\cE}([C])$. So the map extends to a $\R$-linear map on $\vperp$, which can be viewed as an analogue to the Donaldson morphism:

    \begin{align}\label{eq:bm}
        \BM:\vperp & \to N^1(T): f\mapsto \ell_{f,\cE}.
    \end{align}

\begin{Rem} [MMP via wall-crossing]
    To describe the minimal model program of moduli spaces via wall-crossing, one may explore examples for which the following two properties hold. 
    \begin{enumerate}
        \item The map $\BM':\tilde \cC(v)\to \Nef^0(T)$ is an isomorphism.
        \item There are chambers $\tilde \cC_i(v)$ such that the extended map
        \begin{align*}
            \BM': \overline{\coprod \tilde \cC_i(v)}\xrightarrow{\Forg}\vperp\xrightarrow{\BM}\overline{\mathrm{Mov}(T)}
        \end{align*}
        is an isomorphism.  Chambers $\tilde \cC_i(v)$ are one-to-one corresponding to chambers $\cC_i$ of the movable cone of $T$.
    \end{enumerate}
    
    When $X$ is a K3 surface, abelian surface, the projective plane, Enriques surface, etc, one may consider $\Lambda=\kn(X)$, class $v$ with $\dim M(v)\geq4$, and $T=M(v)$. Then for most of the chambers, both properties hold, see \cite{ABCH:MMP,BM:walls,BM:projectivity,LZ:HilbP2,LZ:P2MMP,Wanmin:BM, Nuer:stabEnriques,MYY:wallcrossingK3} for more details of these examples.
\end{Rem}

\begin{Rem}[Strange duality]\label{rem:sd}
Let the lattice $\Lambda=\kn(X)$ and fix an open  subset $U\subset \sb(X)$ on which $\Forg:U\hookrightarrow \lbdd$ is an inclusion.

    Given a non-degenerate quadratic form $Q$ on $\Lambda$, the induced linear map $\tilde Q:\Lambda_\R\to \lbdd: w\mapsto Q(w,-)$  identifies $\Lambda_\R$ with $\lbdd$. For example, when $Q$ is the Euler pairing $\chi(-\otimes -)$, as that proved in \cite[Proposition 4.4]{BM:projectivity}, the map as that in \eqref{eq:bm} can also be expressed as
   \begin{align*}
  \BM=\lambda_\cE\circ \tilde Q^{-1} ,
\end{align*}
    where $\lambda_\cE$ is the Donaldson morphism as that in \cite[Definition 4.3]{BM:projectivity}.

    For a pair of reduced stability conditions $\ts_v,\ts_w\in U$ with $B_{\ts_v}=\tilde Q(v)$ and $B_{\ts_w}=\tilde Q(w)$ satisfying $Q(v,w)=0$, by definition, we have  $\ts_v\in \sb_w(X)$ and $\ts_w\in \sb_v(X)$.   Denote the Bayer--Macr\`i divisor on $M_{\ts_w}(v)$ (resp. $M_{\ts_v}(w)$) as $\ell_{w}$ (resp. $\ell_v$). One may ask under what kind of assumptions  there is an equality $h^0(M_{\ts_w}(v),\ell_w)=h^0(M_{\ts_v}(w),\ell_v)$. 
\end{Rem}

\section{Comparing reduced stability conditions}\label{sec:order}
In this section, we discuss a natural and simple relation $\lsm$ on reduced stability conditions.

\begin{Def}
Given two reduced stability conditions $\tilde{\sigma},\tilde{\tau}\in\sb(\cT)$, we define
\begin{align*}
\ts\lsm \ttau\colon\iff \cA_{\ts}\subset\cP_{\ttau}(<1).
\end{align*}
\end{Def}

We denote by $\ts<\ttau$ if $\ts\lsm\ttau$ and $\ttau\not\lsm\ts$.

Similar notion on hearts also appears in other literature such as \cite[Remark 2.3]{KQ:extquiver}. 

\begin{Lem}\label{lem:posaction}  
Let  $\ts\in\sb(\cT)$, $E\in\cP_{\ts}(1)$ and $F\in\cA_{\ts}[\leq 0]$. Let $\Phi\in\Aut(\cT)$ such that $\ts\lsm\Phi(\ts)$. Then we have $\Hom(\Phi(E),F)=0$.
\end{Lem}
\begin{proof}
    By the assumption, we have  $\Phi(E)\in \cP_{\Phi(\ts)}(1)$. As $\ts\lsm \Phi(\ts)$, by definition, we have $F\in\cP_{\ts}(1)\subset \cA_{\ts}\subset\cP_{\Phi(\ts)}(<1)$. It follows that $\Hom(\Phi(E),F)=0$.
\end{proof}

\begin{Not}
For a reduced stability condition $\ts\in \sb(\cT)$, we denote 
\begin{align}\label{eq:tadef}
    \ta(\ts)\coloneqq\{h\in\lbdd\mid(\cA_{\ts},h+iB_{\ts})\text{ is a representative of }\ts\}.
\end{align}
By Proposition \ref{prop:localisom}, we may let $U$ be an open neighborhood of $\ts$ such that $\Forg|_U:U\to \lbdd$ is homeomorphic onto its image. For $h\in\lbdd$ and $\delta\in \R$ with $|\delta|$ small enough, we denote 
\begin{align*}
    \ts+\delta h\coloneq (\Forg|_U)^{-1}(B_{\ts}+\delta h)
\end{align*}
the deformed reduced stability of $\ts$ along the direction $h$. In particular, when $|\delta|$ is sufficiently small, the reduced stability condition $\ts+\delta h$ is well-defined for all directions $h$ with $||h||\leq 1$ and does not rely on the choice of $U$.
\end{Not}

\begin{Lem}\label{lem:partialorder}
The relation $\lsm$ is transitive. Let $\ts$ and $\ttau$ be two reduced stability conditions. The following statements hold.
    \begin{enumerate}[(1)]
        \item If $\ts\lsm\ttau\lsm\ts$, then $\cA_{\ts}=\cA_{\ttau}$ and $\cP_{\ts}(1)=\cP_{\ttau}(1)=\emptyset$.
        \item Assume that there are representatives $\sigma$ and $\tau$ of $\ts$ and $\ttau$ respectively satisfying $\sigma=\tau\cdot \tilde g$ for some $\tilde g=(g,M)\in \glt$, see Notation \ref{not:glt}, with $g(0)<0$ (resp. $g(0)>0$), then $\ts\lsm \ttau$ (resp. $\ttau\lsm \ts$). In particular, $\pi_\sim(\sigma)\lsm\pi_\sim(\sigma[\theta])$ when $\theta>0$.
        \item Assume that $h\in\ta(\ts)$, then $\ts+\delta h\lsm \ts\lsm\ts-\delta h$ for  $\delta>0$ sufficiently small.
        \item Assume that $\ts$ is non-degenerate, then $\ts<\ts-\delta h$ for  $\delta>0$ sufficiently small.
    \end{enumerate}
\end{Lem}
\begin{proof}
Note that $\cA_{\ts}\subset\cP_{\ttau}(<1)$ is equivalent to the condition that $\cP_{\ts}(\leq 1)\subseteq \cP_{\ttau}(<1)$. So the relation is transitive.

\noindent {(1)} It follows that $\cP_{\ts}(\leq 1)\subseteq \cP_{\ttau}(<1)\subseteq\cP_{\ts}(<1)$. Therefore, we must have $\cP_{\ts}(\leq 1)= \cP_{\ttau}(<1)=\cP_{\ts}(<1)$. So $\cP_{\ts}(1)=\emptyset$, $\cP_{\ts}((0,1))=\cA_{\ts}$, and $\cP_{\ts}((0,1))=\cP_{\ttau}((0,1))$.\\

\noindent {(2)} Assume that $g(0)<0$, then
\begin{align*}
    \cA_{\ts}=\cA_\sigma=\cP_{\tau\cdot \tilde g}((0,1])=\cP_{\ttau}((g(0),g(1)])\subset \cP_{\ttau}(<1).
\end{align*}
The statement holds. 

If $g(0)>0$, then $\tau=\sigma\cdot \tilde g^{-1}$. Note that $\tilde g^{-1}=(M,g')$ for some $g'(0)<0$, the statement holds.\\

\noindent { (3)} Let $\sigma$ be the representative of $\ts$ with central charge $Z_\sigma=h+iB_{\ts}$, then there exists an open neighborhood $U$ of $\sigma$ such that $\Forg|_U$ is homeomorphic onto its image. It follows that 
\begin{align}\label{eq433}
    (\Forg|_U)^{-1}(h+i(B_{\ts}-\delta h))=\sigma \cdot \tilde g,
\end{align}
where $\glt\ni \tilde g=(g,\begin{pmatrix}
        1 & 0 \\ \delta & 1
    \end{pmatrix})$ with $g(0)=\frac{1}{\pi}\arctan(\delta)>0$. 

    When $\delta$ is sufficiently small, by Proposition \ref{prop:localisom}, the reduced stability condition $\pi_\sim( (\Forg|_U)^{-1}(h+i(B_{\ts}-\delta h)))=\ts-\delta h$. By  statement {(2)}, we have $\ts\lsm\ts-\delta h$.\\

    \noindent { (4)} We adopt the notion from the last statement and continue the argument. 
    
    Suppose, for contradiction, that there exists $\delta_0>0$ sufficiently small such that $\ts\not<\ts-\delta_0 h$. Then it must be that $\ts-\delta_0 h\lsm \ts$. By statement {(3)}, we may assume that for all   $0\leq \delta\leq \delta_0$, the following relation holds: $\ts\lsm \ts-\delta h\lsm \ts-\delta_0 h\lsm \ts$. 

    It then follows by statement {(1)} that for every $0\leq \delta\leq \delta_0$, we have  $\cA_{\ts}=\cA_{\ts-\delta h}$ and 
    \begin{align}\label{eq444}
        \cP_{\ts}(1)=\cP_{\ts-\delta h}(1)=\emptyset. 
    \end{align}
    
    For every $t\in(0,\delta_0)$, we denote $\sigma_t\coloneq  (\Forg|_U)^{-1}(h+i(B_{\ts}-t h))$ as the deformed stability conditions of $\sigma$ along the direction $-ih$.
    
   By \eqref{eq433} and \eqref{eq444}, for every $t\in(0,\delta_0)$, there exists $\theta_0(t,\delta_0)>0$ sufficiently small such that for every $|\theta|<\theta_0(t,\delta_0)$, we have
    \begin{align}\label{eq455}
        \cP_{\sigma_t}(\theta)[1]=\cP_{\sigma_t}(1+\theta)=\cP_{\sigma_{\delta}}(1)=\cP_{\ts-\delta h}(1)=\emptyset,
    \end{align}
for some $\delta\in(0,\delta_0)$.

By Lemma \ref{lem:degenerateno1phase}, the reduced stability condition $\pi_\sim(\sigma_t)$ is degenerate for every $t\in(0,\delta_0)$. This contradicts the non-degenerate assumption on $\ts$. So the statement holds.
\end{proof}

\begin{Prop}\label{prop:lifttostab}
     Let $\ts$ be a  reduced stability condition and $0\neq h_0\in\lbdd$. Assume that there is an open neighborhood $ W\subset \lbdd$ of $h_0$ and $\delta>0$ such that for every $h\in W$, we have $\ts+th\lsm \ts$ (resp. $\ts\lsm \ts+th$) for every $0<t<\delta$ (resp. $-\delta<t<0$). Then $h_0\in\ta(\ts)$. 
\end{Prop}
\begin{proof}
    Let the stability condition $\sigma=(\cA,f+i B)$ be a representative of $\ts$. Let $Q$ be a $\Q$-coefficient quadratic form on the lattice $\Lambda$ satisfying the support property for $\sigma$. 

We give the proof according to different cases of the linear relation of $B,h_0$, and $f$. Firstly,  we may assume that $f\neq cB$ for any $c\in \R$, since otherwise $0\in \ta(\ts)$. By Proposition \ref{prop:degeneratesb}, the space $\ta(\ts)=\lbdd$. So $h_0$ is automatically contained in $\ta(\ts)$.\\

\noindent\textbf{Case I}:  We deal with the case that $\dim\mathrm{span}_\R\{h_0,B\}=2$. 

\textbf{Case I.1}: Assume that $h_0\in \mathrm{span}_{\R}\{f,B\}$. Taking account of the $\glt$-action on $\sigma$, we know that $h_0\in\pm\ta(\ts)$. Assume $h_0\in -\ta(\ts)$, then by Lemma \ref{lem:partialorder}.{(3)}, we have $\ts\lsm \ts +\delta h_0$ and $\ts-\delta h_0\lsm \ts$ for $\delta>0$ sufficiently small. Together with the assumption on $h_0$ and Lemma \ref{lem:partialorder}.{(1)}, we have $\cP_{\ts+\delta h_0}(1)=\emptyset$ when $|\delta|$ is sufficiently small. By \eqref{eq455} and Lemma  \ref{lem:degenerateno1phase}, the reduced stability condition $\ts$ is degenerate. By Proposition \ref{prop:degeneratesb}, we have $h_0\in\lbdd=\ta(\ts)$.\\

\textbf{The Main Case I.2}: We may now assume that $\dim \mathrm{span}_\R\{f,h_0,B\}=3$. We may shrink the neighborhood $W$ and $\delta$ if necessary so that 
  \begin{itemize}
      \item $Q|_{\Ker f\cap \Ker (B+th)}$ is negative definite for every $h\in W$ and $|t|<\delta$;
      \item and $W\cap \mathrm{span}_{\R}(f,B)=\emptyset$.
  \end{itemize}
   In particular,  we have $f\in\ta(\ts+t h)$ for every $|t|<\delta$. More precisely, by \cite[Proposition A.5]{BMS:stabCY3s}, see also Proposition and Definition  \ref{propdef:stabQsigma}, there is a connected subspace $S\subset \Stab(Q,\sigma,\cT)$ containing $\sigma$ such that $\Forg|_S$ is homeomorphic onto $\{f+i(B+th)\colon h\in W,|t|<\delta\}$. We denote by $\sigma+ith\coloneq (\Forg|_S)^{-1}\left(f+i(B+th)\right)$.
   
   Denote by
    \begin{align}\label{eq:4.6}
       M_h\coloneq \left(\cup_{a\in(-\delta,\delta)}\Ker(B+ah)\right)\cap \{v\in\Lambda_\R:h(v)f(v)<0\}. 
    \end{align} 
    \begin{Lem}\label{lem:p44}
       Let $E$  be a $(\sigma+ith)$-stable object  for some $t\in(-\delta,\delta)$, then the character $[E]\notin M_h$.
    \end{Lem}
   We postpone the proof of Lemma \ref{lem:p44} after the proof of the proposition. \\

Let $s>0$ be sufficiently small so that $h_0-sf\in W$. We may apply Lemma \ref{lem:B10} by setting $h=B$, $f_1=f$, $f_2=h_0-sf$. There exists $d>0$ sufficiently small so that the set $M_d$ as that in Lemma \ref{lem:B10} is contained in $M_{f_2}$ as that defined in \eqref{eq:4.6}. 

Let $N=\frac{1}{s}$, then by Lemma \ref{lem:B10}, there exists a quadratic form $\tilde Q$ such that 
 \begin{align} \label{eq4499}
       \Ker B\cap \left(\cup_{0\leq t\leq N}\Ker (f+t(h_0-sf))\right)  \subset \nega(\tilde Q)\subset M_{f_2}\cup \nega(Q).
     \end{align}

By Lemma \ref{lem:p44}, there is no $\sigma$-stable object $E$ with character in $M_{f_2}$. So the quadratic form $\tilde Q$ gives the support property for $\sigma$.

By \cite[Proposition A.5]{BMS:stabCY3s} and \eqref{eq4499}, the stability condition $\sigma$ deforms to stability conditions with central charges $f+t(h_0-sf)+iB$, for $0\leq t \leq \tfrac{1}{s}$.  By Proposition \ref{prop:convex}, we have $h_0=f+\tfrac{1}{s}(h_0-sf)\in \ta(\ts)$. The statement holds in this case.\\

 \noindent \textbf{Case II}: We then deal with the remaining case that $\dim\mathrm{span}_\R\{h_0,B\}= 1$.
  
 \textbf{Case II.1}: Assume that $B\neq 0$. We may choose $g\in \lbdd$ linear independent of $B$, the assumption in the proposition holds for $h_0\pm \epsilon g$ when $\epsilon>0$ is sufficiently small. By Case I, we have $h_0\pm\epsilon g\in\ta(\ts)$. By Proposition \ref{prop:convex}, we have $h_0\in \ta(\ts)$. Or actually by Proposition \ref{prop:degeneratesb}, we have $\ta(\ts)=\lbdd$, the reduce stability condition is degenerate in this case.

 \textbf{Case II.2}: Assume that $B=0$. We have $\cA=\cP_{\ts}(1)$ and $f(\cA_{\ts})<0$. For every $h\in \lbdd$, the assumption that $\ts\lsm \ts-th$ for $t$ sufficiently small implies that $\cP_{\ts}(1)\subset \cA_{\ts-th}((0,1))$. In particular, we have $(B-th)(\cA)>0$. It follows that $\sigma_0=(\cA_0=\cP_{\ts}(1),h_0+iB)$ is a pre-stability condition. Note that $h(\cA)<0$ for $h$ in an open neighborhood of $h_0$, so $\sigma_0$ satisfies the support property. Finally, it is clear that $d(\cP_\sigma,\cP_{\sigma_0})=0$, by Proposition \ref{prop:convex}, we have $\sigma\sim \sigma_0$. Therefore, we have $h_0\in\ta(\ts)$.
\end{proof}
\begin{proof}[Proof of Lemma \ref{lem:p44}]
 Suppose there exists a $(\sigma+ith)$-stable object  for some $t\in(-\delta,\delta)$ with character in $M_h$. Then the set 
 \begin{align*}
     S=\{Q(E):[E]\in M_h, E\text{ is $(\sigma+ith)$-stable object  for some $t\in(-\delta,\delta)$}\}
 \end{align*}
 is nonempty.
 Note that the quadratic from $Q$ is with $\Q$-coefficient, the values $\{Q(E): E\in \cT\}$ are discrete. As $Q$ is a quadratic form for the support property of every $\{\sigma+ith:t\in(-\delta,\delta)\}$, we have $s\geq0$ for every $s\in S$. So there exists a minimum value $s_0$ in $S$. \\

Let $E$ be a $(\sigma+it_0h)$-stable object  for some $t_0\in(-\delta,\delta)$ with character $[E]\in M_h$ and $Q(E)=s_0$.
\begin{SubLem}\label{sublem1}
     The object $E$ is $(\sigma+ith)$-stable   for every $t\in(-\delta,\delta)$.
\end{SubLem}
\begin{proof}[Proof of Sublemma \ref{sublem1}]
 If $Q(E)=0$, then by  \cite[Proposition A.8]{BMS:stabCY3s}, the object $E$ is $\tau$-stable for all $\tau\in\Stab(Q,\sigma,\cT)$ and in particular for all $(\sigma+ith)$ with $t\in(-\delta,\delta)$. So we may assume $Q(E)>0$,
 
  Suppose $E$ is not $(\sigma+ith)$-stable  for some $t\in(-\delta,\delta)$, then there exists $s_0\in(-\delta,\delta)$ such that $E$ is strictly $(\sigma+is_0h)$-semistable. Denote by $E_1,\dots,E_m$ the Jordan--H\"older factors of $E$ with respect to $\sigma+is_0h$. In particular, as $Q(E)>0$, by \cite[Lemma 3.9]{BMS:stabCY3s}, we have $Q(E_j)<Q(E)$ for every $1\leq j\leq m$.
 
 As $[E]\in M_h$, we may assume $f([E])>0$ and
 \begin{align}\label{eq41692}
     (B+r_0h)([E])=0
 \end{align} for some $|r_0|<\delta$. 
 
 Assume
 \begin{align}\label{eq41697}
     (B+s_0h+bf)([E])=0
 \end{align}
 for some $b\in \R$, then there exists $g\in\glt$ such that the central charge of $(\sigma+is_0h)\cdot g$ is $f+i(B+s_0h+bf)$. Note that the phases of $E_j$ and $E$ are the same with respect to $(\sigma+is_0h)\cdot g$, we have  \begin{align}\label{eq:47}
     (B+s_0h+bf)(E_j)=0 \text{ and } f(E_j)>0 \text{ for all }1\leq j\leq m.
 \end{align}
As $[E]\in M_h$, we have $h([E])<0$. If $b=0$, then as $[E]=\sum [E_j]$, there exist $E_k$ with $h(E_k)<0$. In particular, together with \eqref{eq:47}, we have $[E_k]\in M_h$, which contradicts to the minimum assumption on $Q(E)$.

Otherwise $b\neq 0$, by \eqref{eq41692} and \eqref{eq41697}, we have 
\begin{align*}
    \frac{b}{r_0-s_0}=\frac{h([E])}{f([E])}<0.
\end{align*}

As $[E]=\sum [E_j]$, there exists $1\leq k\leq m$ such that
\begin{align}\label{eq:48}
    \frac{h([E_k])}{f([E_k])}\leq \frac{h([E])}{f([E])}=\frac{b}{r_0-s_0}<0\implies \frac{r_0-s_0}b\frac{h([E_k])}{f([E_k])}\geq 1.
\end{align}

Together with \eqref{eq:47}, it follows that 
\begin{align*}
   \frac{(B+r_0h)([E_k])}{(B+s_0h)([E_k])}=\frac{\left((r_0-s_0)h-bf\right)([E_k])}{-bf([E_k])}=-\frac{r_0-s_0}b\frac{h([E_k])}{f([E_k])}+1\leq 0.
\end{align*}
   So there exist $t\in [r_0,s_0)$ (or $(s_0,r_0]$) such that $(B+th)([E_k])=0$. Together with \eqref{eq:48}, we have $[E_k]\in M_h$. This contradicts with the minimum assumption on $Q(E)$. So the statement holds.  
\end{proof}

\begin{figure}[h]
    \centering
\scalebox{0.8}{
\begin{tikzpicture}[scale=1.5]

  \draw (-3,0) -- (4.5,0) node[below right] {$\Ker B$};

  \filldraw (-1.5,0) circle (0.05);
  \draw[->] (-1.5,-1) node[below]{$\Ker (B+if)$}--(-1.5,-0.08) ;
  \draw[thick] (-1.5,0) circle (0.5);

  \filldraw (3,0) circle (0.05);
  \node at (3,0.3) {$\Ker B \cap \Ker h$};

  \draw[thick] (-3,0.4) node[above]{$\Ker(B+\delta h)$} -- (3.6,-0.03);

  \draw[thick] (-3,-0.4) node[below]{$\Ker(B-\delta h)$}-- (3.6,0.04);

  \draw[->] (0,-1) node[below]{$[E]$} -- (-0.3,0.15);
\end{tikzpicture}
}
    \caption{The stable character $[E]$ invalidates $\ts+ith\lsm\ts$.}
    \label{fig:3}
\end{figure}
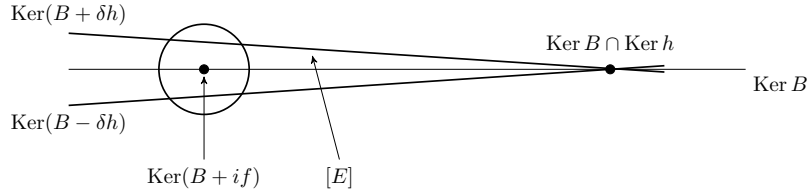

\noindent {\em Back to the proof of Lemma \ref{lem:p44}}: As $[E]\in M_h$, there exists $|r_0|<\delta$ such that $(B+r_0h)(E)=0$ as that in \eqref{eq41692}.  Replacing $E$ by $E[1]$ if necessary, we may assume that $f([E])>0$, then $h([E])<0$. By Sublemma \ref{sublem1} and replacing $E$ by $E[2m]$ if necessary,  we may assume 
\begin{align}\label{eq41723}
    E\in \cP_{\sigma+ir_0h}(0).
\end{align} For every $|t|<\delta$ the object $E\in\cP_{\sigma+ith}(\theta)$ for some $|\theta|<\tfrac{1}{2}$. By \eqref{eq41692}, we have $B(E)>0$ (resp. $<0,=0$) when and only when $r_0>0$ (resp. $<0,=0$). Therefore, the object $E\in \cP_\sigma(\theta)$ for some $\theta>0$ (resp. $<0$) when $B(E)>0$ (resp. $<0$).\\

When $r_0>0$, by \eqref{eq41723}, we have $E[1]\in \cP_{\sigma+ir_0h}(1)\subset \cA_{\sigma+ir_0h}=\cA_{\ts+r_0h}$. On the other hand, we have $E\in\cP_\sigma(\theta)$ for some $\theta>0$, so $E[1]\notin\cA_\sigma[\leq 0]=\cA_{\ts}[\leq 0]$. It follows that $\cA_{\ts+r_0h}\not\subseteq\cA_{\ts}[\leq 0]$, which contradicts to the assumption that $\ts+r_0h\lsm \ts$. 

When $r_0<0$, by \eqref{eq41723}, we have $E[1]\notin\cP_{\sigma+ir_0h}(<1)=\cP_{\ts+r_0h}(<1)$. On the other hand, we have $E[1]\in \cP_\sigma(\theta+1)\subset\cA_{\sigma}=\cA_{\ts}$.  It follows that $\cA_{\ts}\not\subseteq\cP_{\ts+r_0h}(<1)$, which contradicts to the assumption that $\ts\lsm \ts+r_0h$. 

When $r_0=0$, then by Sublemma \ref{sublem1}, for every $0<t<\delta$, we have $E\in \cP_{\sigma+ith}(\theta)$ for some $\theta<0$. It follows that $E[1]\in \cA_{\ts+th}$ and $E[1]\notin \cP_{\ts}(<1)$. Therefore, we have $\cA_{\ts+th}\not\subseteq \cP_{\ts}(<1)$, which contradicts to the assumption that $\ts+th\lsm \ts$.

As a summary, in every case of $r_0$, we get the contradiction. So when $|t|<\delta$, there is no $(\sigma+th)$-stable object with character in $M_h$. The statement hold.
\end{proof}

\noindent By Proposition \ref{prop:lifttostab}, the missing information in $\sb(\cT)$ from $\Stab(\cT)$ can be recovered by the relation $\lsm$. Therefore, the whole stability manifold $\Stab(\cT)$ can be reconstruct from $(\sb(\cT),\lsm)$ as a topological space. More precisely, we make the following notion:
\begin{align}
 \notag   \ta_{\ts}\sb_\Lambda(\cT)&\coloneq \left\{h\in \Hom(\Lambda,\R)\colon\begin{aligned}
        \exists\text{ open }W\ni h\text{ in }\Hom(\Lambda,\R)\text{ and }\delta>0 \text{ such that }\\ \ts+tg\lsm \ts\lsm \ts-tg,\; \forall g\in W\text{ and } 0<t<\delta
    \end{aligned}\right\};\\ \label{eq41226}
    \ta\sb_\Lambda(\cT)&\coloneq \left\{ (\ts,h)\in\sb_\Lambda(\cT)\times \Hom(\Lambda,\R): h\in \ta_{\ts}\sb_\Lambda(\cT)\right\}.
\end{align}

\begin{Cor}\label{cor:recon}
The map $\pi_\sim\times \Forg_{\Re Z}:\Stab(\cT)\to \ta\sb(\cT)$ is a homeomorphism.
\end{Cor}
\begin{proof}
    By Proposition \ref{prop:lifttostab}, Lemma \ref{lem:partialorder}.{(3)} and Lemma \ref{lem:smalldeform}, for every $\ts\in\sb(\cT)$, we have $\ta_{\ts}\sb_\Lambda(\cT)=\ta(\ts)$. By  Proposition \ref{prop:localisom} and the diagram \eqref{diag:1}, the statement holds.
\end{proof}

\begin{Rem}
    There are some natural questions on the notion `$\lsm$' concerning the topology and compactification of the space $\sb(\cT)$. As they are away from the main topic of this paper, we just post two questions here without further comments.
    \begin{enumerate}[(1)]
        \item Let $\ts,\ttau\in\sb(\cT)$ be on the same connected component. Assume that there exist open neighborhoods $U\ni \ts$ and $V\ni \ttau$ such that $\ts'\lsm \ttau'$ for every $\ts'\in U$ and $\ttau'\in V$.
         Does there exists a path $\gamma:[0,1]\to \sb(\cT)$ with $\gamma(0)=\ts$, $\gamma(1)=\ttau$ such that $\gamma(t_1)\lsm\gamma(t_2)$ for every $t_1<t_2$? Are all such paths homotopic to each other?
         \item  Let $\gamma:[0,1)\to \sb(\cT)$ be a `bounded' path of `increasing' (decreasing) reduced stability conditions. In other words, there exists $\ts,\ttau$ such that $\ts\lsm \gamma(t_1)\lsm \gamma(t_2)\lsm \ttau$ for every $t_1<t_2$ (or for every $t_1>t_2$). Then does there exist a limit of weak reduced stability condition $\gamma(1)=(\cA,Z_I)$?
    \end{enumerate}
\end{Rem}

We may also make a similar \emph{binary relation}  `$\lsm$' on $\Stab(\cT)$ as that on $\sb(\cT)$. 
\begin{Def}
For two stability conditions $\sigma,\tau\in\Stab(\cT)$, we define
\begin{align*}
\sigma\lsm \tau & :\iff \cP_{\sigma}(\theta)\subset\cP_{\tau}(<\theta) \text{ for every }\theta\in \R.\\
\sigma\lsmeq \tau & :\iff \cP_{\sigma}(\theta)\subset\cP_{\tau}(\leq\theta) \text{ for every }\theta\in \R.
\end{align*}
\end{Def}
\noindent It is usually difficult for two stability conditions to be comparable. In many cases, in the small neighborhood of a stability conditions $\sigma$, another stability condition $\tau$ satisfies the relation $\sigma\lsm \tau$ when and only when $\sigma=\tau\cdot \tilde g$ for some $\tilde g=(g,M)$ with $g(0)< 0$. 

\noindent It is worth mentioning that the definition makes sense for stability conditions with respect to different lattices as well.

The notion will be useful in setting up the restriction lemma in Section \ref{sec:reslem}. We set up some of its first properties here.
\begin{Lem}\label{lem:eqdefforlsmonstab}
Let $\sigma,\tau\in\Stab(\cT)$ and $\Phi$ an exact autoequivalence on $\cT$. Then the following statements are equivalent:
\begin{enumerate}[(1)]
    \item $\sigma\lsm \tau$.
    \item For every non-zero $E\in \cT$, $\phi^+_\sigma(E)>\phi^+_\tau(E)$ and $\phi^-_\sigma(E)>\phi^-_\tau(E)$.
    \item For every $\sigma$-stable $E\in \cT$, $\phi_\sigma(E)>\phi^+_\tau(E)$. Or for every $\tau$-stable object $E\in\cT$, $\phi_\tau(E)<\phi^-_\sigma(E)$.
    \item $\cP_\tau(\theta)\subset \cP_\sigma(>\theta)$ for every $\theta\in \R$.
    \item $\Phi(\sigma)\lsm\Phi(\tau)$.
    \item $\pi_\sim(\sigma\cdot \tilde g)\lsm \pi_\sim(\tau\cdot \tilde g)$ for every $\tilde g=(g,M)\in\glt$.
    \item $\pi_\sim(\sigma[\theta])\lsm \pi_\sim(\tau[\theta])$ for every $\theta\in(0,1]$.
\end{enumerate}
Same statements  hold for $\lsmeq$ by replacing $>$ with $\geq$. 
\end{Lem}
\begin{proof}
    {(1)}$\iff${(3)} is directly from the definition. It is also clear that {(2)}$\implies${(3)}. 

    {(1)}$\implies${(2)}: Note that $\cP_\sigma(\leq \theta)\subseteq\cP_\tau(<\theta)$, so $\phi^+_\sigma(E)>\phi^+_\tau(E)$. 

    By {(1)}, we have $\phi^-_\sigma(E)=\phi_\sigma(\HN^-_\sigma(E))>\phi^+_\tau(\HN^-_\sigma(E))$. As $\Hom(E,\HN^-_\sigma(E))\neq 0 $, we have $\phi^-_\tau(E)<\phi^+_\tau(\HN^-_\sigma(E))$. Combining these two observations, we have $\phi^-_\sigma(E)>\phi^-_\tau(E)$.

    {(4)}$\iff${(2)} follows the same argument as that for {(1)}$\iff${(2)}.\\
    
   {(1)}$\implies${(5)}: By {(1)}$\iff${(2)}, for every non-zero $E\in \cT$, we have $\phi^\pm_{\Phi(\sigma)}(E)=\phi^\pm_\sigma(\Phi^{-1}(E))>\phi^\pm_\tau(\Phi^{-1}(E))=\phi^\pm_{\Phi(\tau)}(E)$. It follows that $\Phi(\sigma)\lsm \Phi(\tau)$. The other direction {(5)}$\implies${(1)} is by noticing that $\Phi^{-1}$ is an autoequivalence as well.\\

    {(1)}$\iff${(6)}: By definition, we have  
    \begin{align}\label{eq499}
       \pi_\sim(\sigma\cdot\tilde g)\lsm \pi_\sim(\tau\cdot \tilde g)\iff  \cP_{\sigma\cdot \tilde g}((0,1])\subset\cP_{\tau\cdot \tilde g}(<1)\iff  \cP_\sigma((g(0),g(1)])\subset\cP_\tau(<g(1)).
    \end{align}
    Note that {(1)} implies $\cP_\sigma(\leq (\theta-1,\theta])\subseteq\cP_\tau(<\theta)$ for every $\theta\in \R$ and $g(1)$ can be any real number. The statement holds.

    {(1)}$\iff${(7)}: As $\cP(\theta+1)=\cP(\theta)[1]$, we have {(1)}$\iff \cP_{\sigma}(\theta)\subset\cP_{\tau}(<\theta)$ for every $\theta\in(0,1]$. The statement follows by \eqref{eq499}.
\end{proof}

\begin{Lem}\label{lem:lsmeqlem}
    Let $\sigma,\tau\in\Stab(\cT)$. If $\pi_\sim(\sigma[\theta])\lsm \pi_\sim(\tau[\theta])$ for every $\theta\in(0,1]\setminus\{\theta_1,\dots,\theta_n\}$. Then $\sigma\lsmeq\tau$.
\end{Lem}
\begin{proof}
    It is clear that $\cP_{\sigma}(\theta)\subseteq\cP_{\tau}(<\theta)$ for every $\theta\in(0,1]\setminus\{\theta_1,\dots,\theta_n\}$. Note that $\cP_{\sigma}(<\theta_i)\subseteq\cP_{\sigma}(s)\subseteq\cP_{\tau}(s)$ for every $s>\theta_i$, it follows that $\cP_\sigma(<\theta_i)\subseteq\cP_{\tau}(\leq \theta_i)$. The statement holds.
\end{proof}

\section{Reduced stability conditions on curves and polarized surfaces}\label{eg:cs}
From now on, we will focus on the geometric case. Let $X$ be an irreducible smooth projective variety. We will consider (reduced) stability conditions on $\Db(X)$, the bounded derived category of coherent sheaves on  $X$. 

 Denote by $\Stab(X)=\Stab(\Db(X))$ and $\sb(X)=\sb(\Db(X))$ for simplicity.
\begin{Rem}
     In general, it is difficult to know the whole space of $\Stab(X)$ beforehand. In this paper, we always focus on a subset $W$ of $\Stab(X)$ with the property that every fiber of the forgetful map
\begin{center}
    $\Forg:W\to \{$hearts of bounded t-structure on $\Db(X)\}\times \lbdd:\sigma=(\cA,Z)\mapsto(\cA,\Im Z)$
\end{center}
is path-connected.

By Definition \ref{def:sb}, two stability conditions $\sigma$ and $\tau$ in $W$ satisfy $\sigma\sim \tau$ if and only if $\cA_\sigma=\cA_\tau$ and $\Im Z_\sigma=\Im Z_\tau$. In particular, there is no ambiguity to denote a reduced stability condition $\ts$ as $(\cA_{\ts},B_{\ts})$.
\end{Rem} 

In explicit examples, we will frequently use the torsion pair to construct the heart of a bounded t-structure. 

\begin{Not}\label{not:tilt0heart}
Let $\cA$ be the heart of a bounded t-structure and $h:\Lambda\to \R\cup\{+\infty\}$ be a real-valued function, we denote by  \begin{align*}
    \cA^{\sharp 0}_h\coloneq \langle \cA^{> 0}_h,\cA^{\leq 0}_h[1]\rangle
\end{align*} the extension closure of 
 \begin{align*}
        \cA^{> 0}_h&\coloneq \{E\in \cA_{\ts}:h(F)>0\text{ for every }E\twoheadrightarrow F\text{ in }\cA_{\ts}\};\\
        \cA^{\leq 0}_h&\coloneq \{E\in \cA_{\ts}:h(F)\leq 0\text{ for every }F\hookrightarrow E\text{ in }\cA_{\ts}\}.
    \end{align*}
\end{Not}
For example, if $\sigma=(\cA,-h+iB)$ is a stability condition, then $\cA_h^{\sharp 0}$ is the heart of a bounded t-structure. The stability condition $\sigma[\tfrac{1}{2}]$ is given as $\left(\cA^{\sharp 0}_h,B+ih\right)$.

Here we do not require $h$ to be linear so that we may avoid some heavy notions later. For instance, let $\cA$ be the heart of a bounded t-structure, and $Z=-h_0+ig$ be a so-called weak stability function on $\cA$. Then one may define $h(v)=h_0(v)$ when $g(v)\neq0$ and $h(v)=+\infty$ when $g=0$. We get the tilting heart $\cA^{\sharp 0}_h$ as that with respect to $Z$.

\subsection{Reduced stability conditions on curves}
Let $C$ be an irreducible smooth curve with genus $g\geq 1$.
Let the lattice $\Lambda$ be $\Kn(C)$. Then the classical slope stability $\sigma=(\Coh(C),Z=-\deg+i\rk)$ is a stability condition on $\Db(C)$. Moreover, by \cite{Bridgeland:Stab,Macri:curves}, the whole space $\Stab(C)=\sigma\cdot\glt$.

\begin{Ex}[Reduced stability conditions on curves]\label{ex:sbcurve}
    The forgetful map $\Forg: \sb(C)\to \lbdd$ is a universal cover onto the image $\lbdd\setminus\{0\}$. In terms of a parametrized space, we have
    \begin{align*}
        \sb^*(C)&=\left\{\ts_{t}\cdot c=(\cA_t,e^{-c}\vd_{t}):c\in \R,t\in\R\cup\{+\infty\}\right\}\text{ and } \sb(C)=\coprod_{n\in\Z} \sb^*(C)[n].
    \end{align*}
    Here the reduced central charge is given as:
    $\vd_{t}(\rk,\deg)\coloneqq \deg-t\rk $
    when $t\in \R$; and  $\vd_{t}(\rk,\deg)\coloneqq-\rk$ when $t=+\infty$. 
    The heart \begin{align*}
         \cA_t\coloneq \langle \Coh^{>t}(C),\Coh^{\leq t}(C)[1]\rangle=\Coh^{\sharp 0}_{\vd_t}(C)
    \end{align*} when $t\neq+\infty$; and $\cA_t\coloneq \Coh(C)[1]$ when $t=+\infty$. 
\end{Ex}
    It is clear from the definition that $\ts_{s}\lsm \ts_{t}$ when $s<t$. As $g\geq 1$, for every non-zero $v\in\kn(C)$, there exist $\sigma$-semistable objects with character $v$.  We have $\ts_{s}\lsm \ts_{t}$ if and only if $s<t$.

\subsection{Polarized surface}\label{subsec:redsbonsurface}

Let $(S,H)$ be a smooth polarized surface. Fix the $H$-polarized lattice
\begin{align*}
    \lambda_H=H^{2-i}\ch_i:\Kn(S)\to \Lambda_H:[E]\mapsto(H^2\rk(E),H\ch_1(E),\ch_2(E)).
\end{align*}
The $H$-discriminant $\Delta_H=(H\ch_1)^2-2H^2\rk\ch_2$ can be viewed as a quadratic form on $\Lambda_\R\coloneq\Lambda_H\otimes \R$. By the Bogomolov inequality, for every $H$-semistable coherent sheaf $E$ on $S$, we have $\Delta_H(E)\geq 0$.

\noindent We  briefly recall the construction of some stability conditions on $S$. Let $\Coh^{\sharp0}_H(S) \coloneq \langle\Coh^{>0}_H(S),\Coh^{\leq 0}_H(S)[1]\rangle$ be a tilted heart.  Then the pair of datum
\begin{align*}
    \sigma_0\coloneqq(\Coh^{\sharp0}_H(S),Z\coloneqq  -\ch_2+H^2\rk+iH\ch_1)
\end{align*}
is a stability condition on $\Db(S)$.

The quadratic form $\Delta_H$ gives the support property for $\sigma_0$. Indeed, for every $\sigma_0$-semistable object $E$, we have $\Delta_H(E)\geq 0$. The space $\Ker Z$ is spanned by $(1,0,1)$. It is clear that $\Delta_H|_{\Ker Z}$ is negative definite. 

By \cite[Proposition A.5]{BMS:stabCY3s}, see also Proposition and Definition \ref{propdef:stabQsigma}, there is an open family of stability conditions $\Stab(\Delta_H,\sigma_0,\Db(S))$.

\begin{Rem}
 Classically, up to a twist of parameters, the subspace $\Stab(\Delta_H,\sigma_0,\Db(S))$ is by firstly constructing a real $2$-dimensional slice of stability conditions:
\begin{align}\label{eq:zabsurface}
    \sigma_{\alpha,\beta}\coloneq (\Coh^{\sharp \beta}_H(S),Z_{\alpha,\beta}=-\ch_2+\alpha H^2\rk+i(H\ch_1-\beta H^2\rk)),
\end{align}
where $\beta\in \R$ and $\alpha>\frac{\beta^2}{2}$. 

Then by taking the $\glt$-action, we get
\begin{align*}
    \Stab(\Delta_H,\sigma_0,\Db(S))=\{\sigma_{\alpha,\beta}:\beta\in \R,\alpha>\frac{\beta^2}{2}\}\cdot\glt.
\end{align*}

In different contexts, the central charge might be in slightly different format. For instance, it can be given as \begin{align*}
    Z'_{\alpha',\beta}=\ch_2^{\beta-i\alpha'}=\ch^{\beta}_2-\tfrac{\alpha'^2}{2}H^2\rk+i\alpha'H\ch_1^\beta,
\end{align*} 
where $\alpha'>0$ and $\beta\in\R$. After taking the $\glt$-action, they give the same family of stability conditions. Under different parameterizations, the numerical walls of a fixed character are `nested line segments' and `nested semicircles' respectively, see Corollary \ref{cor:nested}.
\end{Rem}
To describe the space of reduced stability conditions from $\Stab(\Delta_H,\sigma_0,\Db(S))$, we set \begin{align}
  \notag  U(\Delta_H)\coloneq &\{B\in\lbdd\colon\Delta_H|_{\Ker B}\text{ is with signature }(1,1)\}\\
     = &\{B\in\lbdd \colon\Delta_H^{-1}(0)\cap\Ker B\text{ is the union of two lines.}\}\label{eq512}
\end{align}
as that in Notation \ref{not:UQ}. By Proposition \ref{prop:sbQsigma}, we may describe a family of reduced stability conditions on $S$ as follows:
\begin{Ex}[Reduced stability conditions on a polarized surface]\label{eg:stabsurface}
    The forgetful map \begin{align*}
        \Forg: \sb(\Delta_H,\ts_0,\Db(S))\to U(\Delta_H)
    \end{align*}
    is a universal cover. In terms of a parametrized space, we may write
    \begin{align}\label{eq5611}
        &\sb^*_H(S)=\left\{\ts_{t_1,t_2}\cdot c=(\cA_{t_1,t_2},e^{-c}\vd_{t_1,t_2})\colon c\in\R, t_1<t_2\in\R\cup\{+\infty\}
        \right\}; \\
    \notag    &\sb(\Delta_H,\ts_0,\Db(S))  =\coprod_{n\in \Z}\sb^*_H(S)[n].
    \end{align}

   When $t_2=+\infty$, the reduced central charge $\vd_{t_1,t_2}=-H\ch_1+t_1H^2\rk$; the heart $\cA_{t_1}\coloneq \Coh_H^{\sharp t_1}(S)[1]$. 
     
   When $t_2\neq +\infty$, the reduced central charge \begin{align*}
        \vd_{t_1,t_2}(H^2\rk,H\ch_1,\ch_2)=\ch_2-\tfrac{1}{2}(t_1+t_2)H\ch_1+\tfrac{1}{2}t_1t_2H^2\rk.
   \end{align*} In \eqref{eq512}, we have $\Delta_H^{-1}(0)\cap \Ker \vd_{t_1,t_2}=\cup_{i=1,2}\R\cdot(1,t_i,t_i^2/2)$. 
    
    The heart $\cA_{t_1,t_2}\coloneq (\cA_t[-1])^{\sharp0}_{B_{t_1,t_2}}$ as that in Notation \ref{not:tilt0heart} for any $t\in(t_1,t_2)$. In particular, it does not rely on the choice of $t$.
\end{Ex}    

\begin{Rem}    
    When $S$ is an abelian surface, or the Albanese map of $S$ is finite, by \cite{FLZ:ab3} and \cite{LahozRojaz}, the space $\sb_H(S)=\sb(\Delta_H,\ts_0,\Db(S))$.
\end{Rem}

\subsection{Bayer Vanishing Lemma}
The parameter $(t_1,t_2)$ is convenient for comparing reduced stability conditions. It helps us to set up the following neat vanishing theorem which is difficult to prove due to subtlety of the heart structure. This vanishing theorem has been used in the $\bP^2$ case, see \cite{LZ:P2MMP,FLLQ:P2}, and then later on for other polarized surfaces.

We give a reprove for the polarized surface case by using the language of reduce stability conditions and the `$\lsm$' relation. 

\begin{Prop}[Bayer Vanishing Lemma]\label{prop:bayerlemsurface}
    Let $(S,H)$ be a polarized surface and $\ts=\ts_{t_1,t_2}$ be a reduced stability condition as that in \eqref{eq5611} with $t_2\neq +\infty$. Then $\ts\lsm\ts\otimes\cO(H)$.
    
    In particular, if $E,F\in\cP_{\ts}(1)$, then $\Hom(E(mH),F)=0$ for every $m>0$.
\end{Prop}

\begin{Lem}\label{lem:comparesurfacesb}
    Let $t_i,s_i\in\R\cup\{+\infty\}$. Then 
    \begin{enumerate}[(1)]
        \item The restricted quadratic form $\Delta_H|_{\Ker B_{t_1,t_2}\cap \Ker B_{s_1,s_2}}$ is negative definite if and only if $t_1<s_1<t_2<s_2$ or $s_1<t_1<s_2<t_2$. When $t_1<s_1<t_2<s_2$, we have $-\vd_{s_1,s_2}\in\ta(\ts_{t_1,t_2})$; when $s_1<t_1<s_2<t_2$, we have $\vd_{s_1,s_2}\in\ta(\ts_{t_1,t_2})$.
        \item If $t_i<s_i$, then $\ts_{t_1,t_2}\lsm \ts_{s_1,s_2}$.
        \item If $t_1<s_2$, then $\ts_{t_1,t_2}\lsm\ts_{s_1,s_2}[1]$.
    \end{enumerate}
\end{Lem}
\begin{proof}
    {(1)}  On the projective plane $\bP(\Lambda_\R)$, the kernel space $\Ker \vd_{a,b}$ corresponds to the line passing through the points $\gamma_2(a)=(1,a,a^2/2)$ and $\gamma_2(b)$. The curve $\Delta_H^{-1}(0)$ is given by the `parabola' $\{\gamma_2(t):t\in \R\cup\{+\infty\}\}$.  A point on $\bP(\Lambda_\R)$ is in  $\nega(\Delta_H)$ if and only if  it is inside the parabola.

    We may assume $(t_1,t_2)\neq (s_1,s_2)$. Then the point $\Ker \vd_{t_1,t_2}\cap\Ker \vd_{s_1,s_2}$ is the intersection of lines through the pairs of points $\gamma_2(t_i)$ and  $\gamma_2(s_i)$, respectively. Drawing this on a plane, it is clear that the intersection point lies inside the parabola if and only if $s_1<t_1<s_2<t_2$ or  $t_1<s_1<t_2<s_2$.
    
    The rest of the statement then follows from Proposition \ref{prop:sbQsigma} and \eqref{eq213}.\\

    \noindent {(2)} If $t_1<s_1<t_2<s_2$, then by {(1)}, we may consider the stability condition $\sigma=(\cA_{t_1,t_2},-B_{s_1,s_2}+iB_{t_1,t_2})$. By definition, we have $\pi_\sim(\sigma)=\ts_{t_1,t_2}$ and $\pi_\sim(\sigma[\tfrac{1}{2}])=\ts_{s_1,s_2}$. By Lemma \ref{lem:partialorder}.{(2)}, we have the relation $\ts_{t_1,t_2}\lsm \ts_{s_1,s_2}$.

    In the general case, there always exists $m_i\in\R$ satisfying $t_1<m_1<t_2<m_2$ and $m_1<s_1<m_2<s_2$. By the first part of the argument, we have $\ts_{t_1,t_2}\lsm\ts_{m_1,m_2}\lsm\ts_{s_1,s_2}$. The statement holds.\\

    \noindent {(3)} By assumption, there exists $s'_i\in\R$ such that
    \begin{align*}
        s'_1<s_1,\; s'_2<s_2, \text{ and } s'_1<t_1<s'_2<t_2.
    \end{align*}
    By {(1)} and  Proposition \ref{prop:sbQsigma}, we may consider the stability condition $\sigma=(\cA_{t_1,t_2},B_{s'_1,s'_2}+iB_{t_1,t_2})$. By definition, we have $\pi_\sim(\sigma)=\ts_{t_1,t_2}$ and  $\pi_\sim(\sigma[\tfrac{1}{2}])=\ts_{s'_1,s'_2}[1]$. By Lemma \ref{lem:partialorder}.{(2)}, we have the relation $\ts_{t_1,t_2}\lsm \ts_{s'_1,s'_2}[1]$. 

    By {(2)}, we have $\ts_{s'_1,s'_2}\lsm\ts_{s_1,s_2}$. The statement holds.
\end{proof}

\begin{proof}[Proof of Proposition \ref{prop:bayerlemsurface}]
    Note that $\ts_{t_1,t_2}\otimes\cO(H)=\ts_{t_1+1,t_2+1}$, the first statement follows from Lemma \ref{lem:comparesurfacesb}.{(2)}. Moreover, we have $\ts_{t_1,t_2}\lsm\ts_{t_1,t_2}\otimes\cO(mH)$ for every positive integer $m$ as well. The second statement follows from Lemma \ref{lem:posaction}.
\end{proof}

\begin{Rem}[Stability conditions to reduced stability conditions]\label{rem:surfacestabtosb}
For a stability condition $\sigma_{\alpha,\beta}$ as that in \eqref{eq:zabsurface}, the reduced stability condition $\pi_\sim(\sigma_{\alpha,\beta})=\ts_{\beta,+\infty}[-1]$. 

 The kernel of its central charge  $\Ker Z_{\alpha,\beta}$ in $\bP(\Lambda_\R)$ is the point $p_{\alpha,\beta}=[1,\beta,\alpha]$. For $\theta\in\R$, the kernel of reduced central charge $\pi_\sim(\sigma_{\alpha,\beta}[\theta])$ in $\bP(\Lambda_\R)$  is a projective line through $p_{\alpha,\beta}$. While $\theta$ is chosen among all values in $(0,1]$, we get the whole pencil of lines through $p_{\alpha,\beta}$.
 
As $\alpha>\frac{\beta^2}2$, when $\theta\notin\Z$, each plane intersects with the parabola $\{\gamma_2(t)\}_{t\in\R}$ at two points $\gamma_2(t_i)$ for some $t_2>t_1$. The reduced stability condition 
\begin{align*}
    \pi_\sim(\sigma_{\alpha,\beta}[\theta])= \ts_{t_1,t_2}[m]\cdot c
\end{align*}
for some $c\in \R$ and $m\in\Z$. While $\theta$ is chosen among all values in $(0,1)$,  we get all the parameters $(t_1,t_2)$ satisfying
\begin{align*}
   \vd_{t_1,t_2}(1,\beta,\alpha)=\alpha-(t_1+t_2)\beta/2+t_1t_2/2=0.
\end{align*}
    One may interpret Proposition \ref{prop:bayerlemsurface} to the classical version with respect to Bridgeland stability conditions as follows.\\

\noindent\textit{Claim:}    Let $E$ and $F$ be  $\sigma_{\alpha,\beta}$-semistable objects satisfying $\phi_{\sigma_{\alpha,\beta}}(E)\geq \phi_{\sigma_{\alpha,\beta}}(F)$ and $H\ch^\beta_1(E)\neq 0$,  then the vanishing \begin{align*}
        \Hom(E(mH),F)=0
    \end{align*} holds for every $m>0$.
\end{Rem}
\begin{proof}
    By the assumption, we may assume that $E\in \cP_{\sigma_{\alpha,\beta}}(\theta)$ for some $\theta\in(0,1)$. Here $\theta\neq 1$ because of $H\ch^\beta_1(E)\neq 0$. So  $E\in \cP_{\ts}(0)$ and $F\in\cP_{\ts}(\leq 0)$, where $\ts=\pi_\sim(\sigma_{\alpha,\beta}[\theta])$. As $\theta\notin \Z$, the reduced stability condition $\ts$ is in the form of $\ts_{t_1,t_2}$ for some $t_2\neq +\infty$. The statement follows by Lemma \ref{lem:posaction}.  
\end{proof}
\section{Restriction theorem}\label{sec:reslem}

\subsection{Heart version}
Let $Y$ be a smooth projective variety and $X\in|D|$ be a smooth subvariety of $Y$ for some divisor $D$ on $Y$. 
Denote by $\iota:X\hookrightarrow Y$ the inclusion morphism, $\iota_*:\Db(X)\to\Db(Y)$ the push-forward functor, $\iota^*$ the derived pull-back functor. The induced map $[\iota_*]:\Kn(X)\to \Kn(Y):[E]\mapsto[\iota_*E]$ is well-defined.

We will use the following two distinguished triangles by adjunction in the arguments later,  see \cite[Corollary 11.4]{Huybrechts:FMtransformation}, \cite[Lemma 2.8]{KP:Serre} or \cite[Proposition 3.4]{Kuz:CYfracCY} for reference.

For every object $E\in \Db(X)$, we have
\begin{align}\label{eq610}
    E\otimes \cO_X(-D)[1]\to \iota^*\iota_*E\xrightarrow{\epsilon_E}E \to E\otimes \cO_X(-D)[2],
\end{align}
where $\epsilon_E$ is the counit morphism of adjunction.

For every object $F\in \Db(Y)$, we have
\begin{align}
 \label{eq612}
    F\otimes \cO_Y(-D)\xrightarrow[]{h_F}F\xrightarrow{\eta_F}\iota_*\iota^*F\to F\otimes \cO_Y(-D)[1],
\end{align}
where $\eta_E$ is the unit morphism of adjunction.

\begin{Lem}\label{lem:63}
Adopt notions as above. Let $\cA$ be the heart of a bounded t-structure on $\Db(Y)$ satisfying 
      \begin{align}\label{eq622}
          \cA\otimes \cO_Y(D)\subset \cA[\leq 1].
      \end{align} 
      Then for every $E,F\in\Db(X)$ with $\iota_*E\in\cA[\geq0]$ and $\iota_*F\in \cA[\leq 0]$, we have $\Hom_X(E[m],F)=0$ for every $m\in\Z_{\geq1}$.
\end{Lem}
\begin{proof}
We make (descending) induction on $m$. Note that $\cA\subset\Coh(Y)[-N,N]$ for some $N$ large enough  and $\iota_*:\Coh(X)\to \Coh(Y)$ is exact. When $m\geq 2N+1$, the object $E[m]\in\Coh(Y)[\geq N+1]$ and $F\in\Coh(Y)[\leq N]$. It follows that $\Hom(E[m],F)=0$. In other words, the statement holds for all $m\geq 2N+1$.\\

Assume that the statement holds for all $m\geq k+1$ for some $k\geq 1$, we are going to prove the statement for $m=k$.

To do so, applying $\Hom_X(-,F)$ to \eqref{eq610}, we get the long exact sequence:
\begin{align}\label{eq641}
    \dots \to \Hom_X(E\otimes \cO_X(-D)[m+2],F)\to \Hom_X(E[m],F)\to \Hom_X(\iota^*\iota_*E[m],F)\to \dots
\end{align}

By the adjointness of functors, we have $\Hom_X(\iota^*\iota_*E[m],F)=\Hom_Y(\iota_*E[m],\iota_*F)=0$ as  $\iota_*E[m]\in\cA[\geq m]$ and $\iota_*F\in\cA[\leq 0]$ with $m\geq 1$ by assumption.

By \eqref{eq641}, to show $\Hom_X(E[m],F)=0$, we only need to show $\Hom_X(E\otimes \cO_X(-D)[m+2],F)=0$.\\
\begin{align}\label{eq:claim65}
    \text{
\noindent \textit{Claim}: The assumption \eqref{eq622} $\cA\otimes \cO_Y(D)\subset \cA[\leq 1]$ $\implies$ $\cA[\geq 1]\otimes \cO_Y(-D)\subset \cA[\geq0]$.}
\end{align}
\begin{proof}[Proof of the claim]
    Let $F\in \cA[\geq 1]$. Then we have the distinguished triangle  $G_+\to F\otimes \cO_Y(-D)\to G_-\xrightarrow{+}$ for some $G_+\in \cA[\geq 0]$ and $G_-\in \cA[\leq -1]$. By \eqref{eq622}, we have $G_-\otimes \cO_Y(D)\in \cA[\leq 0]$. It follows that $\Hom(F\otimes \cO_Y(-D),G_-)=0$. So $G_-=0$, in other words, we have $F\otimes \cO_Y(-D)\in \cA[\geq 0]$.
\end{proof}

Back to the proof of the lemma. By the claim, we have
\begin{align*}
    \iota_*(E\otimes \cO_X(-D)[1])= \iota_*E\otimes \cO_Y(-D)[1]\in \cA[\geq1]\otimes\cO_Y(-D)\subset \cA[\geq0].
\end{align*}

By the induction on $m$, we have $\Hom((E\otimes \cO_X(-D)[1])[m+1],F)=0$. So the statement holds for all $m\geq k$.

Therefore, the statement holds by the descending induction on $m$.
\end{proof}

\begin{Lem}\label{lem:660}
    Adopt the assumptions as that in Lemma \ref{lem:63}. Let $E\in\Db(X)$ and 
    \begin{align}\label{eq6661}
       F^-[-1]\xrightarrow{k}F^+\xrightarrow{f}\iota_*E\xrightarrow{f'}F^-
    \end{align}
    be the distinguished triangle with $F^+\in \cA[\geq m]$ and $F^-\in\cA[\leq m-1]$. Then we have
    \begin{enumerate}[(1)]
        \item $\iota_*\iota^*F^+=F^+\oplus (F^+\otimes \cO_Y(-D)[1])$.
        \item $F^-=\iota_*E^-$ for some $E^-\in \Db(X)$.
    \end{enumerate}
\end{Lem}
\begin{proof}  
{ (1)} In \eqref{eq612}, by the adjunction property, for every $f\in \Hom(F^+,\iota_*E)$, there exists a unique $g\in\Hom(\iota^*F,E)$ commuting the diagram. 
 \begin{equation}\label{eq6767}
         \begin{tikzcd}
              & \iota_*\iota^*F^+ \arrow[dashed,d]{}{\iota_*(g)}\\
              F^+ \arrow{ur}{\eta_{F^+}} \ar{r}{f} & \iota_*E.
         \end{tikzcd}
     \end{equation}

Let $f$ be as that in \eqref{eq6661} and $h_{F^+}$ be as that in \eqref{eq612}. Then we have 
\begin{align*}
    f\circ h_{F^+}=(\iota_*(g)\circ \eta_{F^+})\circ h_{F^+} =\iota_*(g)\circ( \eta_{F^+}\circ h_{F^+}) =0.
\end{align*} 
It follows by \eqref{eq6661} that \begin{align}\label{eq6008}
    h_{F^+}=k\circ g'\text{ for some }g'\in \Hom(F^+\otimes \cO_Y(-D),F^-[-1]).
\end{align} 

By Claim \eqref{eq:claim65}, we have $F^+\otimes \cO_Y(-D)\in\cA[\geq m-1]$. As  ${F^-}[-1]\in\cA[\leq m-2]$, we have
\begin{align}\label{eq667}
    \Hom({F^+}\otimes \cO_Y(-D),F^-[-1])=0.
\end{align} In particular, we have $g'=0$ as that in \eqref{eq6008}. By \eqref{eq6008}, we have $h_{F^+}=0$. 

By \eqref{eq612}, we have $\iota_*\iota^*{F^+} = {F^+}\oplus ({F^+}\otimes \cO_Y(-D)[1])$.\\

\noindent{(2)} By the statement {(1)} and \eqref{eq667}, we have $\Hom(\iota_*\iota^*{F^+},F^-)=0$. Applying the Octahedron Axiom to the composition of morphisms $\iota_*\iota^*{F^+}\xrightarrow{\iota_*(g)}\iota_*E \xrightarrow{f^-}F^-$, we get the following diagram of distinguished triangles (arrows `$\xrightarrow{+}$' are all omitted to simplified the notion):

    \begin{center}
	\begin{tikzcd}
		\iota_*\iota^*F^+ \arrow{r}{t} \arrow[d,equal]
		& F^+ \arrow{d}{f} \arrow{r} &F^+_1\arrow{d}\\
		\iota_*\iota^*F^+ \arrow{r}{\iota_*(g)} \arrow{d}& \iota_*E
		\arrow{d}{f'} \arrow{r}{} & \iota_*E_1 \arrow[d]\\
		0 \arrow{r} & F^-\arrow[r,equal] & F^-.
	\end{tikzcd}
\end{center}
By statement { (1)} and the diagram \eqref{eq6767}, the morphism $t$ is given as $(\id_{F^+},*):\iota_*\iota^*F^+=F^+\oplus (F^+\otimes \cO_Y(-D)[1])\xrightarrow{}F^+$. 

We may assume that $\cH^q_{\Coh(Y)}(F^+)\neq0$ and $\cH^i_{\Coh(Y)}(F^+)=0$ when $i\geq q+1$. It is clear then $\cH^{i}_{\Coh(Y)}(F^+\otimes \cO_Y(-D)[1])=0$ when $i\geq q$. Applying $\cH^{i}_{\Coh(Y)}(-)$ to the distinguished triangle on the top, we have 
\begin{align}
   \dots\to \cH^{q}_{\Coh(Y)}(\iota_*\iota^*F^+)\xrightarrow{\cH^q(t)} \cH^{q}_{\Coh(Y)}(F^+)\to \cH^{q}_{\Coh(Y)}(F^+_1)\to \cH^{q+1}_{\Coh(Y)}(\iota_*\iota^*F^+)=0
\end{align}
As $\cH^{q}_{\Coh(Y)}(\iota_*\iota^*F^+)=\cH^{q}_{\Coh(Y)}(F^+)$, we get $\cH^{i}_{\Coh(Y)}(F^+_1)=0$ when $i\geq q$.\\

\noindent Note that $F_1^+$ fits into the distinguished triangle on the top, we have 
\begin{align*}
    F_1^+&\in\langle F^+,\iota_*\iota^*F^+[1]\rangle\in\langle\cA[\geq m],F^+[1],F^+\otimes \cO_Y(-D)[2]\rangle \\
   & \subset \langle\cA[\geq m],\cA[\geq m+1],\cA[\geq m+1]\rangle=\cA[\geq m].
\end{align*}
Here the `$\subset$' on the second line follows from Claim \eqref{eq:claim65}. It follows that the distinguished triangle $F^+_1\to \iota_*E_1\to F^-\xrightarrow{+}$ also satisfies the assumption as that in \eqref{eq6661} but with $\cH^{i}_{\Coh(Y)}(F^+_1)=0$ when $i\geq q$, decreased by $1$ comparing with that of $F^+$. \\

\noindent We run this whole procedure to get a series of distinguished triangles $F^+_m\to \iota_*E_m\to F^-\xrightarrow{+}$ satisfying the assumption as that in \eqref{eq6661}. In particular, we have $\cH^{i}_{\Coh(Y)}(F^+_m)=0$ when $i\geq q-m+1$, in other words, the object $F_m^+\in \Coh(Y)[m-q]$. 

Assume that $F^-\in \Coh(Y)[-N,N]$ for some $N$, we may let $m>q+N+\dim Y$. In particular, we get $\iota_*E_m=F^+_m\oplus F^-$. As $\iota_*$ commutes with $\cH^{i}_{\Coh}(-)$, we also have 
\begin{align*}
    \iota_*(\cH^{i}_{\Coh(X)}(E_m))=\cH^{i}_{\Coh(Y)}(\iota_*E_m)=0 \text{ when } i\in [N+1, N+\dim Y].
\end{align*}
As $X$ is smooth of dimension $\dim Y-1$, we have $E_m=E^+_m\oplus E^-$ for some $E^+_m\in\Coh(X)[\geq N+\dim Y+1]$ and $E^-\in \Coh(X)[\leq N]$. It follows that $\iota_*E_m=\iota_*E^+_m\oplus \iota_* E^-$. As $\Db(X)$ is Karoubian satisfying the Krull--Schmidt property, see  \cite{ChenLe:Karoubian}, we must have $F^-=\iota_*E^-$.
\end{proof}
\subsection{Restrict stability conditions to a hypersurface}

\begin{Def}\label{def:restrictredstab}
    Let $\cA$ be the heart of a bounded t-structure on $\Db(Y)$, we denote by
    \begin{align*}
        \cA|_{\Db(X)}\coloneq \{E\in\Db(X):\iota_*E\in\cA\}
    \end{align*}
    the full subcategory in $\Db(X)$.
\end{Def}

\begin{Prop}\label{prop:reststab}
     Let $Y$ be a smooth projective variety and $X\in|D|$ be a smooth subvariety of $Y$ for some divisor $D$ on $Y$.  Let $\sigma=(\cA,Z)$ be a stability condition on $\Db(Y)$ satisfying \begin{align*}
         \sigma\otimes \cO_Y(D)\lsm \sigma[1].
     \end{align*} Then 
      \begin{align*}
          \sigma|_{\Db(X)}\coloneq (\cA|_{\Db(X)},Z\circ[\iota_*])
      \end{align*}
      is a stability condition on $\Db(X)$.

      Moreover, an object $E\in \Db(X)$ is $\sigma|_{\Db(X)}$-(semi)stable  if and only if  $\iota_*E$ is $\sigma$-(semi)stable. If $\sigma\otimes \cO_Y(D')\lsm \sigma[1]$ (resp. $\sigma\lsmeq\sigma\otimes \cO_Y(D')$) for some divisor $D'$, then the restricted stability condition also satisfies $\sigma|_{\Db(X)}\otimes \cO_X(D')\lsm \sigma|_{\Db(X)}[1]$ (resp. $\sigma|_{\Db(X)}\lsmeq \sigma|_{\Db(X)}\otimes \cO_X(D')$).
\end{Prop}
\begin{Lem}\label{lem:644}
Adopt the assumptions as that in Proposition \ref{prop:reststab}, then for every $E,F\in\Db(X)$ satisfying $\phi^-_\sigma(\iota_*E)>\phi^+_\sigma(\iota_*F)$, we have $\Hom_X(E,F)=0$.
\end{Lem}
\begin{proof}
    By rotating the stability condition $\sigma$ to $\tau=\sigma[\theta]$ for some $\theta \in \left(\phi^+_\sigma(\iota_*F),\phi^-_\sigma(\iota_*E)\right)$, we have $\iota_*F\in\cA_\tau[\leq -1]$ and $\iota_*E\in\cA_\tau[\geq0]$. Note that the $\glt$ and $\Aut(\cT)$ act on different sides of $\Stab(\cT)$, in particular, they commute with each other. By Lemma \ref{lem:eqdefforlsmonstab}, we have\begin{align*}
         \tau\otimes \cO_Y(D)=\sigma[\theta]\otimes\cO_Y(D)=(\sigma\otimes\cO_Y(D))[\theta]\lsm \sigma[1+\theta]=\tau[1].
    \end{align*}
    The statement then follows from Lemma \ref{lem:63}.
\end{proof}

\begin{Lem}\label{lem:65}
Adopt the assumptions as that in  Proposition \ref{prop:reststab}. Let $E,F\in\Db(X)$ satisfying $\phi^-_\sigma(\iota_*E)\geq \phi^+_\sigma(\iota_*F)$, then the map $\iota_*:\Hom_X(E,F)\to \Hom_Y(\iota_*E,\iota_*F)$ is surjective.
\end{Lem}
\begin{proof}
Apply $\Hom_X(-,F)$ to \eqref{eq610}, we get the long exact sequence:
\begin{align}\label{eq6146}
    \dots\to \Hom_X(E,F)\xrightarrow{-\circ\epsilon_E} \Hom_X(\iota^*\iota_*E,F)\to \Hom_X(E\otimes \cO_X(-D)[1],F) \to \dots
\end{align}
For every $f\in\Hom(E,F)$, as $\epsilon_E$ is a natural transformation, we have \begin{align*}
    f\circ \epsilon_E=\epsilon_F\circ \iota^*(\iota_*(f))=\Phi^{-1}_{\iota_*E,F}(\iota_*(f)),
\end{align*} where $\Phi_{\iota_*E,F}:\Hom(\iota^*\iota_*E,F)\to\Hom(\iota_*E,\iota_*F)$ is the natural isomorphism. So the statement is equivalent to show that $-\circ \epsilon_E$ is surjective. As that in \eqref{eq6146}, it is enough to show $\Hom(E\otimes \cO_X(-D)[1],F)=0$. \\

By the assumption that $\sigma\otimes \cO_Y(D)\lsm \sigma[1]$ and Lemma \ref{lem:eqdefforlsmonstab}.{(2)}, we have
\begin{align*}
    &\phi^-_{\sigma}(\iota_*(E\otimes \cO_X(-D)))=\phi^-_{\sigma}(\iota_*E\otimes \cO_Y(-D))=\phi^-_{\sigma\otimes \cO_Y(D)}(\iota_*E)\\
   > &\phi^-_{\sigma[1]}(\iota_*E)=\phi^-_{\sigma}(\iota_*E[-1])\geq \phi^+_{\sigma}(\iota_*F[-1]).
\end{align*}
By Lemma \ref{lem:644}, we have $\Hom(E\otimes \cO_X(-D)[1],F)=0$. The statement holds.
\end{proof}

\begin{Lem}\label{lem:666}
    Adopt the assumptions as that in Proposition \ref{prop:reststab} and let $E\in\Db(X)$. Then every Harder--Narasimhan factor of $\iota_*E$ with respect to $\sigma$ is $\iota_*E_m$ for some $E_m\in\Db(X)$.
\end{Lem}
\begin{proof}  
Let $F^-=\HN^-_\sigma(\iota_* E)$ be the HN factor of $E$ with minimum phase. By rotating the stability condition $\sigma$ to $\tau=\sigma[\theta]$ with $\theta =\phi^-_\sigma(\iota_*E)$, we get a distinguished triangle 
\begin{align}\label{eq6612}
    F^+\to \iota_*E\xrightarrow{f} F^-\xrightarrow{+}
\end{align} with $F^+\in \cA_{\tau}[\geq 0]$ and $F^-\in\cA_{\tau}[\leq -1]$. By Lemma \ref{lem:660}, we get $F^-=\iota_*E^-$ for some $E^-\in\Db(X)$.

Note that $\phi_\sigma(\iota_*E^-)=\phi^-_\sigma(\iota_*E)$. By Lemma \ref{lem:65}, the morphism $f$ in the HN filtration distinguished triangle as that in \eqref{eq6612} is of the form $\iota_*f_X$ for some $f_X\in\Hom_{\Db(X)}(E,E^-)$. Therefore, the object $F^+$ is also of the form $\iota_*E^+$ for some $E^+\in \Db(X)$. By induction on the number of HN factors, the statement holds.
\end{proof}

Now we can finish the proof for the restriction of stability conditions.
\begin{proof}[Proof of Proposition \ref{prop:reststab}]
By Lemma \ref{lem:65} and \ref{lem:666}, for every $E\in\Db(X)$, the Harder--Narasimhan filtration of $\iota_*E$ with respect to $\sigma$ is in the form of
 \begin{equation*}
        \begin{tikzcd}[column sep=tiny]
0=\iota_*F_0  \arrow[rr]& & \iota_*F_1 \arrow[dl,"\iota_*f_1"] \arrow[rr]& & \iota_*F_2\arrow[r]\arrow [dl,"\iota_*f_2"]& \cdots\arrow[r] &\iota_*F_{m-1}\arrow [rr]& & \iota_* F_m=\iota_*E\arrow[dl,"\iota_*f_m"]\\
& \iota_*E_1 \arrow[ul,dashed] && \iota_*E_2\arrow[ul,dashed] &&&& \iota_*E_m \arrow[ul,dashed]
\end{tikzcd}
    \end{equation*}
for some $E_i$, $F_i$, $f_i$ in the category $\Db(X)$. Together with Lemma \ref{lem:644}, it follows that $\cP|_{\Db(X)}(\theta)\coloneqq\{E\in\Db(X)\colon\iota_*E\in\cP_\sigma(\theta)\}$ is a slicing on $\Db(X)$.   In particular, an object $E\in\Db(X)$ is $\sigma|_{\Db(X)}$-semistable if and only if $\iota_*E$ is $\sigma$-semistable.

It is also clear that if  $\iota_*E$ is $\sigma$-stable then $E$ is $\sigma|_{\Db(X)}$-stable. For the remaining statement, suppose that $E$ is $\sigma|_{\Db(X)}$-stable but $\iota^*E$ is strictly $\sigma$-semistable. Then we get a distinguished triangle $F_1\to \iota_*E\to F_2\xrightarrow{+}$ with $F_i$ $\sigma$-semistable and $\phi_\sigma(F_i)=\phi_\sigma(\iota_*E)$. By the same argument as that in  Lemma \ref{lem:660}, we get $F_2=\iota_* E_2$ for some $E_2\in\Db(X)$. By Lemma \ref{lem:65}, this contradicts with the assumption that $\iota_*E$ is  $\sigma|_{\Db(X)}$-stable. So  an object $E\in\Db(X)$ is $\sigma|_{\Db(X)}$-stable if and only if $\iota_*E$ is $\sigma$-stable.

By Lemma \ref{lem:63} and \ref{lem:666}, the category $\cA|_{\Db(X)}$ is the heart of the bounded t-structure associated with $\cP|_{\Db(X)}$. 

    Denote the lattice of $\sigma$ as $\lambda:K(Y)\to \Lambda_Y$. By the support property of $\sigma$, given a norm $||\bullet||$ on $\Lambda_Y\otimes \R$, there exists a constant $c>0$ such that $|Z(\lambda([F]))|\geq c||\lambda([F])||$ for every $\sigma$-semistable object $F$.
    
    We denote the sublattice $\Lambda_X$ as the image of $\lambda_X\coloneqq\lambda\circ [\iota_*]$ in $\Lambda_Y$. The norm $||\bullet||$  restricts to a norm on $\Lambda_X$. For every $\sigma|_X$-semistable object $E$, we have that the object $\iota_*E$ is $\sigma$-semistable. It follows that $|Z(\lambda_X([E]))|=|Z(\lambda([\iota_*]([E])))|=|Z(\lambda([\iota_*E]))|\geq c||\lambda([\iota_*E])||$. 
    
    So $\sigma|_{\Db(X)}$ admits the support property as well, it is a stability condition on $\Db(X)$.\\

   To see the last statement, we only need to show that  
   \begin{align*}
       \left(\sigma\otimes \cO_Y(D')\right)|_{\Db(X)}=\sigma|_{\Db(X)}\otimes \cO_X(D').
   \end{align*}
   Indeed, an object $E\in \Db(X)$ is in $ \left(\sigma\otimes \cO_Y(D')\right)|_{\Db(X)}$ if and only if 
   \begin{align*}
      & \iota_*E\in\cA_{\sigma\otimes \cO_Y(D')}\iff \iota_*E\otimes \cO_Y(-D)\in\cA_\sigma\iff \iota_*(E\otimes \cO_X(-D'))\in\cA_\sigma\\
      \iff & E\otimes \cO_X(-D')\in \cA_\sigma|_{\Db(X)}\iff E\in \cA_\sigma|_{\Db(X)}\otimes \cO_X(D').
   \end{align*}
   This finishes the claim.
\end{proof}

\begin{Rem}[Autoequivalence as spherical twist]
    Note that $\iota_*:\Db(X)\to \Db(Y)$ is a spherical function associated with the spherical twist $\otimes \cO_Y(D):\Db(Y)\to \Db(Y)$ in the sense of \cite[Definition 2.1]{Segal:sptwist} and \cite{AL:sphericalonDG,ST:sttwist}. In particular, the formula
    \eqref{eq610} and \eqref{eq612}, which plays the essential role in the proof of Proposition \ref{prop:reststab}  also hold for other example spherical functors.

    After completing the proof of Proposition \ref{prop:reststab}, we noticed that the statement also follows from \cite[Corollary 2.2.2]{Polishchuk:families-of-t-structures}. Nevertheless, we include our argument here as it is self-contained and takes a completely different approach from those in \cite{Polishchuk:families-of-t-structures} and \cite{TLS:constructtstructure}.
\end{Rem}

\subsection{The restricted data determines the original one}
\begin{Def}\label{def:geostab}
We call a stability condition $\sigma$ on $\Db(X)$ \emph{geometric} (with respect to $X$) if for each point $p\in X$, the skyscraper sheaf $\cO_p$ is   $\sigma$-stable, and all skyscraper sheaves are of the same phase.
\end{Def}
\begin{Cor}\label{cor:niceimplygeom}
    Let $(X,H)$ be an irreducible smooth variety with a very ample divisor $H$. Let $\sigma$ be a stability condition on $X$ satisfying $\sigma\otimes \cO_X(H)\lsm\sigma[1]$, then $\sigma$ is geometric.
\end{Cor}
\begin{proof}
    By Bertini Theorem, for every closed point $p\in X$, there exists a sequence of varieties $p\in X_1\subset \dots \subset X_n=X$ such that every $X_i$ is smooth with $X_i\in|H_{X_{i+1}}|$.  By Proposition \ref{prop:reststab}, the stability condition $\sigma$ restricts to a stability condition on $X_1$. It further restricts to a stability condition on a zero dimensional subvariety $Z\ni p$ in $|H_{X_1}|$. So every skyscraper sheaf $\cO_p$ is $\sigma$-stable. As any pair of points can be connected by a sequence of such kind of curve $X_1$, their phases are the same.
\end{proof}
\begin{Prop}\label{prop:reconfromrest}Let $Y$ be an irreducible smooth projective variety and $X\in|H|$ be a smooth subvariety of $Y$ for some very ample divisor $H$ on $Y$. Let $\sigma$ and $\tau$ be two stability conditions on $\Db(Y)$ such that:
    \begin{align}\label{eq:6233}
   \sigma\lsmeq \sigma\otimes \cO_Y(H)\lsm \sigma[1], \tau\otimes \cO_Y(H)\lsm \tau[1],\sigma|_{\Db(X)}=\tau|_{\Db(X)}\text{ and } Z_\sigma=Z_\tau.
    \end{align}
    Then $\sigma=\tau$.
\end{Prop}
\begin{proof}
By \cite[Lemma 4.7]{FLZ:ab3}, we only need to show
 $d(\sigma,\tau)\leq 1$.

\noindent We first show that for every $\tau$-stable object $F$, the difference 
\begin{align}\label{eq6156}
    \phi^+_\sigma(F)-\phi_\tau(F)\leq1.
\end{align}

To see this, we consider  the distinguished triangle
\begin{equation}\label{eq6166}
    E\rightarrow F\rightarrow G\xrightarrow{+},  \end{equation}
    
 where $E=\mathrm{HN}^{+}_{\sigma}(F)$ is HN-factor with maximum phase  and $G= \mathrm{HN}^{<\phi^+_\sigma(F)}_{\sigma}(F)$.

By Corollary \ref{cor:niceimplygeom} and the assumption that $\sigma|_{\Db(X)}=\tau|_{\Db(X)}$, all skyscraper sheaves are both $\sigma$ and $\tau$ stable with the same phase. If $\dim\supp(E)=0$, then $E$ must be the extension of skyscraper sheaves with the same homological shift. In particular, we have $\phi_\sigma^+(F)=\phi_\sigma(E)=\phi_\tau(E)\leq \phi_\tau(F)$. The inequality \eqref{eq6156} holds automatically.  We may therefore assume that $F$ is not supported on any $0$-dimensional subscheme. As $H$ is very ample, in particular, we have $\supp(E)\cap X\neq \emptyset$.\\

 As that in \eqref{eq612}, we may consider the distinguished triangle 
 \begin{align*}
     E\otimes \cO_Y(H)[-1]\xrightarrow{\eta} \iota_*\iota^* E\otimes \cO_Y(H)[-1]\xrightarrow{f} E\xrightarrow[]{+}.
 \end{align*} 
Since $\supp(E)\cap X\neq \emptyset$, the morphism $f$ above is not $0$.  In particular, we have
\begin{align}\label{eq6243}
    \Hom(\iota_*\iota^*(E\otimes \cO_Y(H))[-1],E)\neq 0. 
\end{align}
 
 It also follows that 
 \begin{align}\label{eq617}
    & \phi^-_\sigma(\iota_*\iota^*E\otimes \cO_Y(H)[-1])\geq\min\{\phi^-_\sigma(E\otimes \cO_Y(H)[-1]),\phi_\sigma(E)\}\\= &\min\{\phi^-_{\sigma\otimes\cO_Y(-H)}(E)-1,\phi_\sigma(E)\}     
    \geq \phi_\sigma(E)-1>\phi^+_\sigma(G[-1]).\notag
 \end{align}
Here for the `$\geq$' in the second line, we use the assumption that $\sigma\otimes \cO_Y(-H)\lsmeq \sigma$ and Lemma \ref{lem:eqdefforlsmonstab}.

 Applying $\Hom(\iota_*\iota^*E\otimes \cO_Y(H)[-1],-)$ to \eqref{eq6166}, we have an exact sequence 
 \[\dots\rightarrow \Hom(\iota_*\iota^*E(H)[-1],G[-1])\rightarrow\Hom(\iota_*\iota^*E(H)[-1],E)\rightarrow\Hom(\iota_*\iota^*E(H)[-1],F)\rightarrow\dots\]
By \eqref{eq617}, the term $\Hom(\iota_*\iota^*E\otimes \cO_Y(H)[-1],G[-1])=0$. It follows by \eqref{eq6243} that \begin{align*}
    \Hom(\iota_*\iota^*E\otimes \cO_Y(H)[-1],F)\neq 0.
\end{align*}
Therefore, we have
\begin{align*}
    &\phi_\tau(F)\geq \phi^-_\tau(\iota_*\iota^*E\otimes \cO_Y(H)[-1])=\phi^-_{\tau|_{\Db(X)}}(\iota^*E\otimes \cO_X(H)[-1]) \\ 
    =& \phi^-_{\sigma|_{\Db(X)}}(\iota^*E\otimes \cO_X(H)[-1])=\phi^-_\sigma(\iota_*\iota^*E\otimes \cO_Y(H)[-1])\geq \phi_\sigma(E)-1=\phi^+_\sigma(F)-1.
\end{align*}

Here the `$=$' in the first line is due to the assumption that $\tau\otimes \cO_Y(H)\lsm \tau[1]$ and Proposition \ref{prop:reststab}. The first `$=$'in the second line is due to the assumption that $\sigma|_{\Db(X)}=\tau|_{\Db(X)}$. The `$\geq$' is due to \eqref{eq617}. To sum up, the relation \eqref{eq6156} holds.\\

We then show that for every  $\tau$-stable object $F$, the difference  $ \phi_\tau(F)-\phi^-_\sigma(F)\leq1.$ 

To see this, we consider a similar distinguished triangle as that of \eqref{eq6166} requiring $G=\mathrm{HN}^{\phi^-_\sigma(F)}_{\sigma}(F)$  and $E= \mathrm{HN}^{>\phi^-_\sigma(F)}_{\sigma}(F)$.

We consider the  distinguished triangle $G\to\iota_*\iota^*G\to G(-H)[1]$ as that in \eqref{eq610}. Note that  $\Hom(G,\iota_*\iota^*G)\neq 0$. By the same argument as that for \eqref{eq617}, we get $\phi^+_\sigma(\iota_*\iota^*G)\leq \phi_\sigma(G)+1$.

Apply $\Hom(-,\iota_*\iota^*G)$ to \eqref{eq6166}, by the same argument above, it follows that $\Hom(E[1],\iota_*\iota^*G)=0$ and $\Hom(F,\iota_*\iota^*G)\neq0$. Therefore, we get
\begin{align*}
\phi_\tau(F)\leq \phi^+_\tau(\iota_*\iota^*G)=\phi^+_\sigma(\iota_*\iota^*G)\leq\phi_\sigma(G)+1=\phi^-_\sigma(F)+1.
\end{align*}

To sum up, we have $|\phi^\pm_\sigma(F)-\phi_\tau(F)|\leq1$ for every $\tau$-stable object $F$. By \cite[Lemma 6.1]{Bridgeland:Stab}, we have $d(\sigma,\tau)\leq 1$. The statement follows from \cite[Lemma 4.7]{FLZ:ab3}.
\end{proof}
\begin{Rem}\label{rem:differentlattice}
Note that the assumption in Proposition \ref{prop:reconfromrest} does not require $\sigma$ and $\tau$ are with respect to the same lattice in prior. More precisely, in the statement, we may write the central charge $Z_\sigma$ (resp. $Z_\tau$) as the composition $K(Y)\xrightarrow{\lambda_\sigma}\Lambda_\sigma \xrightarrow{Z_{\sigma,\lambda_\sigma}}\C$ (resp. $\lambda_\tau,\Lambda_\tau,Z_{\tau,\lambda_\tau}$). Then Proposition \ref{prop:reconfromrest} says that $\cA_\sigma=\cA_\tau$ and $Z_{\sigma,\lambda_\sigma}\circ \lambda_\sigma=Z_{\tau,\lambda_\tau}\circ \lambda_\tau$.
\end{Rem}

It is known that a geometric stability condition on a surface is determined by its phase on the skyscraper sheaf and central charge, see for example \cite{Bridgeland:K3}. In other words, two geometric stability conditions $\sigma$, $\tau$ on $\Db(S)$ with $\phi_\sigma(\cO_p)=\phi_\tau(\cO_p)$ are  the same if and only if $Z_\sigma=Z_\tau$. This is also true for some higher dimensional varieties, for instance, abelian threefolds. However, there are also examples of different geometric stability conditions $\sigma\neq\tau$ on $\Db(\bP^3)$ with $\phi_\sigma(\cO_p)=\phi_\tau(\cO_p)$ and $Z_\sigma=Z_\tau$. The new assumption that $\sigma\lsmeq\sigma\otimes \cO_X(H)\lsm \sigma[1]$ is a solution to this issue.

\begin{Cor}\label{cor:geodeterbyZ}
    Let $(X,H)$ be an irreducible smooth variety with a very ample divisor $H$. Let  $\sigma$ and $\tau$ be two geometric  stability conditions satisfying
    \begin{enumerate}
        \item $\phi_\sigma(\cO_p)=\phi_\tau(\cO_p)$ \text{ and } $Z_\sigma=Z_\tau$;
        \item $\sigma\lsmeq \sigma\otimes \cO_X(H)\lsm\sigma[1]$ and $\tau\otimes \cO_X(H)\lsm \tau[1]$.
    \end{enumerate}
    Then $\sigma=\tau$.
\end{Cor}
Note that the geometric assumption is implied by { (b)} according to Corollary \ref{cor:niceimplygeom}.
\begin{proof}
    By Bertini Theorem, there is a smooth connected curve $C$ on $X$ cutting out by $\dim X-1$ hyperplanes in $|H|$.  By Proposition \ref{prop:reststab}, there are restricted stability conditions $\sigma|_{\Db(C)}$ and $\tau|_{\Db(C)}$ with the same central charge. Moreover, they are both geometric with the same phase on skyscraper sheaves. It follows that $\sigma|_{\Db(C)}=\tau|_{\Db(C)}$. The statement then follows by Proposition \ref{prop:reconfromrest}.
\end{proof}

\subsection{Example: Polarized surface case}
We may have an immediate application of Proposition \ref{prop:reststab} in the surface case. Here we discuss the polarized surface case. The unpolarized case is discussed in Proposition \ref{prop:D13} and Appendix \ref{sec:appDunpolaized}.

Let $(S,H)$ be a smooth polarized surface and $C\in|dH|$ be a smooth curve. Adopt the notion of stability conditions $\sigma_{\alpha,\beta}=(\cA_\beta,Z_{\alpha,\beta})$ and reduced stability conditions $\ts_{t_1,t_2}=(\cA_{t_1,t_2},B_{t_1,t_2})$ as that in Section \ref{subsec:redsbonsurface}. 
\begin{Ex}
We have the following statement on the restricted stability conditions:
    \begin{enumerate}[(1)]
        \item When $\alpha>\frac{\beta^2}{2}+\frac{d^2}{8}$, we have $\sigma_{\alpha,\beta}\otimes \cO_S(dH)\lsm\sigma_{\alpha,\beta}[1]$. The restricted stability condition $\sigma_{\alpha,\beta}|_{\Db(C)}=(\Coh(C),Z)$ and is equivalent to the slope stability on $C$.
        In other words, a vector bundle $E$ on $C$ is slope (semi)stable if and only if $\iota_*E$ is $\sigma_{\alpha,\beta}$-(semi)stable.
        \item When $t_2-t_1>d$, the heart $\cA_{t_1,t_2}|_{\Db(C)}=\Coh^{\sharp t}(C)$, where $t=\tfrac{1}{2}(d^2+dt_1+dt_2)H^2$.
    \end{enumerate}
\end{Ex}

The second part in statement {(1)}  has been made use in the study of Brill--Noether theory via stability conditions, see \cite{Arend:BN,BL:BN, Soheyla:Mukai1,Soheyla:Mukai2, FL:CliffordK3} for more examples.

\begin{proof}
    {(1)} For every $\theta\in(0,1)$, by Remark \ref{rem:surfacestabtosb}, the reduced stability condition $\pi_\sim(\sigma_{\alpha,\beta}[\theta])=\ts_{t_1,t_2}[m]\cdot c$ for some $t_1<t_2$ satisfying 
    \begin{align*}
        0=2\alpha-(t_1+t_2)\beta+t_1t_2>\frac{d^2}{4}+(\beta-t_1)(\beta-t_2).
    \end{align*}
       It follows that $t_2-t_1=(t_2-\beta)+(\beta-t_1)>2\sqrt{\tfrac{d^2}{4}}=d$. By Lemma \ref{lem:comparesurfacesb}.{(3)}, we have \begin{align*}
           \ts_{t_1,t_2}\otimes \cO(dH)=\ts_{t_1+d,t_2+d}\lsm \ts_{t_1,t_2}[1].
       \end{align*}
 For $\theta=1$, we have $\pi_\sim(\sigma_{\alpha,\beta}[1])=\ts_{\beta,+\infty}[2]$. By Lemma \ref{lem:comparesurfacesb}.{(3)}, we also have $\ts_{\beta,+\infty}\otimes \cO(dH)\lsm \ts_{\beta,+\infty}[1]$.

 By Lemma \ref{lem:eqdefforlsmonstab}, we have  $\sigma_{\alpha,\beta}\otimes \cO(dH)\lsm \sigma_{\alpha,\beta}[1]$.

\noindent By Proposition \ref{prop:reststab},  the restricted 
 stability condition $\sigma_{\alpha,\beta}|_{\Db(C)}=(\Coh_H^{\sharp \beta}(S)|_{\Db(C)},Z_{\alpha,\beta}\circ \iota_*)$.  By Definition \ref{def:restrictredstab}, the heart $\Coh_H^{\sharp \beta}(S)|_{\Db(C)}=\Coh(C)$. By a direct computation, we have \begin{align}\label{eq6161}
     H^{2-\bullet}\ch_i(\iota_*-)=(0, dH^2\rk(-),\deg(-)-\tfrac{1}{2}d^2H^2\rk(-)).
 \end{align}
 The central charge $Z=Z_{\alpha,\beta}\circ\iota_*=-\deg+\tfrac{1}{2}d^2H^2\rk+idH^2\rk$, which is clear the same as $-\deg+i\rk$ up to a linear transformation. \\
       
\noindent    {(2)} By Lemma \ref{lem:comparesurfacesb}, we have stability condition $\tau=(\cA_{t_1,t_2},-\vd_{\frac{t_1+t_2}{2},+\infty}+i\vd_{t_1,t_2})$.

By statement {(1)},  $\tau\otimes \cO_{S}(D)\lsm \tau[1]$. By Proposition \ref{prop:reststab}, the heart $\cA_\tau=\cA_{t_1,t_2}$ restricts to $\Db(C)$.

By \eqref{eq6161}, for every $E\in\Db(C)$,
    \begin{align*}
        \tfrac{1}{t_2-t_1}\vd_{t_1,t_2}(\iota_*E)=\ch_2(\iota_*E)-\tfrac{1}{2}(t_1+t_2)H\ch_1(\iota_*E) =\deg(E)-\tfrac{1}{2}(d^2+dt_1+dt_2)H^2\rk(E).
    \end{align*}
        The statement follows.
\end{proof}

\subsection{Bayer Lemma for non-paralleled divisor}
We have seen in previous sections that the stability condition on $(S,H)$ constructed from the geometric perspective satisfies the Bayer Vanishing property:
\begin{align}\label{eq636}
    \sigma\lsmeq \sigma\otimes \cO_S(H).
\end{align}
By Lemma \ref{lem:eqdefforlsmonstab}, the property \eqref{eq636} implies $\sigma\lsmeq \sigma\otimes \cO_S(mH)$ for every $m\in\Z_{\geq 1}$.

In this section, we strengthen this result to divisors not parallel to $H$.
\begin{Prop}\label{thm:bayerall}
    Let $(X,H)$ be an irreducible smooth polarized variety over $\C$. Then for every divisor $D$ on $X$, there exists an integer $m(D)$ such that for every geometric stability condition $\sigma$ satisfying $\sigma\lsmeq \sigma\otimes \cO_X(H)$, we have
    \begin{align*}
        \sigma\lsmeq \sigma\otimes \cO_X\left(m(D)H+D\right).
    \end{align*}
\end{Prop}

For technical reason, to make induction, we are going to prove the following statement. The theorem follows by the case that $k=0$.
\begin{Prop}\label{prop:bayerall}
    Let $(X,H)$ be an irreducible smooth polarized variety over $\C$. Then for every divisor $D$ on $X$ and $k\in\Z_{\geq 0}$, there exist $m(D,k)\in\Z$ such that for every geometric stability condition $\sigma$ satisfying $\sigma\lsmeq \sigma\otimes \cO_X(H)$, we have
    \begin{align*}
        \sigma\lsmeq \sigma\otimes \cO_X\left(m(D,k)H+D\right)[k].
    \end{align*}
\end{Prop}
\begin{proof}
Firstly, we may  shift $\sigma$ to $\sigma[\theta]$ if necessary so that the phase of all skyscraper sheaves is $1$. As $[\theta]$ commutes with the action of $\otimes \cO_X(-)[-]$,  by Lemma \ref{lem:eqdefforlsmonstab}, we may always assume $\phi_\sigma(\cO_p)=1$.\\

\noindent We make decreasing induction on $k$:

\noindent\textbf{Step 1}: We first deal the case when $k\geq \dim X$.

As $\phi_\sigma(\cO_p)=1$, by \cite[Lemma 2.11]{FLZ:ab3}, we have $\cA_\sigma\subset \Coh(X)[0,n-1]$ and $\Coh(X)\subset\cA_\sigma[1-n,0]$. As $\Coh(X)$ is $\otimes \cO_X(-)$-invariant, for every non-zero object $F\in \cA_\sigma$ and divisor $D'$, we have\begin{align*}
    F\otimes \cO_X(D')[k]\in \Coh(X)[k,k+n-1]\subset \cA_\sigma[k-n+1,k].
\end{align*}  It follows that \begin{align*}
    \phi^+_{\sigma}(F)\leq 1\leq n-k+1<\phi_{\sigma}^-(F\otimes \cO_X(D')[k])=\phi^-_{\sigma\otimes \cO_X(-D')[-k]}(F).
\end{align*} By Lemma \ref{lem:eqdefforlsmonstab}, we have
\begin{align*}
    \sigma\lsm \sigma\otimes \cO_X(D')[k], \text{ when }k\geq \dim X.
\end{align*}
We may set $m(D,k)=-\infty$ when $k\geq \dim X$.\\

\noindent\textbf{Step 2}:    Assume the statement holds for all $s\geq k+1$. 
    
    Let $a\in\Z$ such that $aH+D$ is very ample. By Bertini theorem, we may choose a sequence of smooth varieties:
    \begin{align*}
        X_n\subset X_{n-1}\subset\dots \subset X_2\subset X_1\subset X_0=X
    \end{align*}
    such that each $X_{i+1}$ is a smooth subvariety in $|(aH+D)|_{X_i}|$.
    Denote by $\iota_j:X_j\hookrightarrow X$ the embedding morphism. 
Then for every $0\leq j\leq n-1$ and $F\in \Db(X)$, we have the distinguished triangle
  \begin{align}\label{eq:break}
      \iota_{j*}\iota_j^*F\otimes \cO_X(-aH-D)\to \iota_{j*}\iota_j^*F\to \iota_{j+1*}\iota_{j+1}^*F\xrightarrow{}\iota_{j*}\iota_j^*F\otimes \cO_X(-aH-D)[1].
  \end{align}
  
    Let 
    \begin{align}\label{eqM}
        M=\max\{a+1, m(-sD,k+s)+(s+1)a: s\in\Z_{\geq 1}\}.
    \end{align}
Since $m(-sD,k+s)=-\infty$ when $k+s\geq n$, the number $M$ is well-defined.\\
    
  \noindent \textbf{Step 3}: We will show that $\phi_\sigma(E)\leq \phi^-_\sigma(E\otimes \cO_X(MH+D)[k])$ for every $\sigma$-stable object $E$. The strategy is by dividing $E\otimes \cO_X(MH+D)$ into smaller pieces.

  \begin{Lem}\label{lem:617}
      For every $d\in \Z_{\geq 0}$ and $0\leq t\leq n$, we have
      \begin{align}\label{eq:643}
          \phi_\sigma(E)\leq \phi_\sigma^-\left(\iota_{t*}\iota_t^*E\otimes \cO_X((M-(d+1)a)H-dD)[d+k]\right).
      \end{align}
  \end{Lem}
  \begin{proof}[Proof of Lemma \ref{lem:617}]
\noindent    \textbf{Step 3.1}: We first prove the case when $t=0$.

When $t=0$ and $d=0$, the functor $\iota_{0*}\iota_0^*$ is just the identity. Note that $k\geq 0$. By the assumption \eqref{eq636}, the assumption \eqref{eqM} that $M\geq a+1$ and Lemma \ref{lem:eqdefforlsmonstab}, we have $\phi_\sigma(E)\leq \phi_\sigma^-(E\otimes \cO_X((M-a)H))\leq \phi_\sigma^-(E\otimes \cO_X((M-a)H)[k])$. The statement holds.\\

      When $t=0$ and $d\geq 1$, by the assumption \eqref{eqM} that $M\geq m(-dD,k+d)+(d+1)a$, the induction on $k$, and the assumption \eqref{eq636}, we have
      \begin{align*}
          \sigma\lsmeq \sigma\otimes \cO_X(m(-dD,k+d)H-dD)[d+k]\lsmeq \sigma\otimes \cO_X((M-(d+1)a)H-dD)[d+k]
      \end{align*}
      By Lemma \ref{lem:eqdefforlsmonstab}, we have $\phi_\sigma(E)\leq  \phi_\sigma^-(E\otimes  \cO_X((M-(d+1)a)H-dD)[d+k])$.\\

      \noindent\textbf{Step 3.2}: We make induction on $t$, assume the statement holds for $t-1$.

      Apply \eqref{eq:break} by letting $j=t-1$ and $F=E\otimes \cO_X((M-(d+1)a)H-dD)[d+k]$, we have  the distinguished triangle
      \begin{align*}
      \iota_{t-1*}\iota_{t-1}^*F\to \iota_{t*}\iota_{t}^*F\xrightarrow{}\iota_{t-1*}\iota_{t-1}^*E\otimes \cO_X((M-(d+2)a)H-(d+1)D)[d+1+k]\xrightarrow{+}.
  \end{align*}
  By the induction on $t$, we have $\phi_\sigma(E)\leq \phi_\sigma^-(\iota_{t-1*}\iota_{t-1}^*F)$. Note that $d+1\geq 0$ as well, by the induction on $t$, we also have $\phi_\sigma(E)\leq \phi_\sigma^-(\iota_{t-1*}\iota_{t-1}^*E\otimes \cO_X((M-(d+2)a)H-(d+1)D)[d+1+k])$. 

  It follows that $\phi_\sigma(E)\leq \phi_\sigma^-(\iota_{t*}\iota_{t}^*F)$. The lemma holds by induction.
  \end{proof}
\noindent  \textit{Back to the proof of Proposition \ref{prop:bayerall}}:  Lemma \ref{lem:617} implies that
 \begin{align}\label{eq:6456}
          \phi_\sigma(E)\leq \phi_\sigma^-\left(\iota_{t*}\iota_t^*E\otimes \cO_X((M-a)H)[k]\right).
      \end{align}
      for every $0\leq t\leq n$.\\

\noindent\textbf{Step 4}:      We apply \eqref{eq:break} by letting $F=E\otimes \cO_X(MH+D)$ and $j=0,1,\dots,n-1$. This gives the following distinguished triangles:
        \begin{align*}
      E\otimes \cO_X((M-a)H)\to& F\to \iota_{1*}\iota_{1}^*F\xrightarrow{+}.\\
      \iota_{1*}\iota_1^*E\otimes \cO_X((M-a)H)\to &\iota_{1*}\iota_1^*F\to \iota_{2*}\iota_{2}^*F\xrightarrow{+}.\\
    &  \dots \\
     \iota_{n-1*}\iota_{n-1}^*E\otimes \cO_X((M-a)H)\to& \iota_{n-1*}\iota_{n-1}^*F\to \iota_{n*}\iota_{n}^*F\xrightarrow{+}.\\
  \end{align*}
  It follows that 
  \begin{align*}
      \phi^-_\sigma(F)&\geq \min\{\phi^-_\sigma(\iota_{1*}\iota_{1}^*F),\phi^-_\sigma(\iota_{0*}\iota_0^*E\otimes \cO_X((M-a)H))\}.\\
      &\geq \min\{\phi^-_\sigma(\iota_{2*}\iota_{2}^*F),\phi^-_\sigma(\iota_{j*}\iota_j^*E\otimes \cO_X((M-a)H)):0\leq j\leq 1\}.\\
      &\dots\\
      & \geq \min\{\phi^-_\sigma(\iota_{n*}\iota_{n}^*F),\phi^-_\sigma(\iota_{j*}\iota_j^*E\otimes \cO_X((M-a)H)):0\leq j\leq n-1\}\\
       & = \min\{\phi^-_\sigma(\iota_{j*}\iota_j^*E\otimes \cO_X((M-a)H)):0\leq j\leq n\}
  \end{align*}
  The `$=$' is by noticing that $\iota_{n*}\iota_{n}^*F$ is supported on a zero dimensional subvariety, therefore fixed by taking  tensor of line bundles.

  Substitute this back to \eqref{eq:6456}, we get $\phi_\sigma(E)\leq\phi_\sigma^-( E\otimes \cO_X(MH+D)[k])$. As this holds for all $\sigma$-stable object $E$, by Lemma \ref{lem:eqdefforlsmonstab}, we have $\sigma\lsmeq \sigma\otimes \cO_X(MH+D)[k]$. We may let $m(D,k)=M$. 
  
  The statement holds by induction.
\end{proof}

\section{Threefold cases}\label{sec:3fold}
In this section, we describe a family of reduced stability conditions on a smooth polarized threefold which satisfies the conjecture in \cite{BBMT:Fujita}, \cite{BMT:3folds-BG} or equivalently \cite[Conjecture 4.1]{BMS:stabCY3s}. One goal is to show that \cite[Conjecture 4.1]{BMS:stabCY3s} implies Conjecture \ref{conj:intro} when $X$ is a threefold.

We first briefly recap the construction of stability conditions on a threefold. More details are referred to \cite{BBMT:Fujita}, \cite{BMT:3folds-BG}, \cite{BMS:stabCY3s}, and \cite{DulipToda:DTinv}.
\subsection{Recap: Stability conditions on a polarize threefold} Let $(X,H)$ be a polarized smooth threefold. We fix the $H$-polarized lattice:
\begin{align*}
    \lambda_H: \Kn(X)\to \Lambda_H: [E]\mapsto(H^3\ch_0(E)),H^2\ch_1(E),H\ch_2(E),\ch_3(E)).
\end{align*}
To simplify the notion, we will denote by $\Lambda_\R\coloneq \Lambda_H\otimes \R$.  The twisted Chern characters are denoted by:
\begin{align*}
    \ch^{\beta}_3&=\ch_3-\beta H\ch_2+\frac{\beta^2}{2}H^2\ch_1-\frac{\beta^3}{6}H^3\ch_0;\\
    \ch^{\beta}_2&=\ch_2-\beta H\ch_1+\frac{\beta^2}{2}H^2\ch_0;\;\;\;
    \ch^{\beta}_1=\ch_1-\beta H\ch_0;\;\;\;
    \ch^{\beta}_0=\ch_0. 
\end{align*}

The $H$-discriminant is $\Delta_H=(H^2\ch_1)^2-2H^3\ch_0(H\ch_2)$. Recall the notion of higher discriminant:
\begin{align*}
    \nabla^\beta_H\coloneqq 4(H\ch_2^{\beta })^2-6(H^2\ch_1^{\beta H})\ch_3^{\beta }.
\end{align*}

For every $\beta\in \R$ and $\alpha>0$, we consider the  heart $\Coh^{\sharp \beta}_H(X)$, which admits a slope function given as 
\begin{align*}
    \nu_{\alpha,\beta}\coloneqq\frac{H\ch_2^{\beta }-\tfrac{1}{2}\alpha^2H^3\ch_0}{H^2\ch_1^{\beta }}.  
\end{align*}
Here we set $\nu_{\alpha,\beta}(E)\coloneq+\infty$ if $H^2\ch_1^{\beta }(E)=0$.

Denote \begin{align*}
 \cA_{\alpha,\beta}(X)\coloneqq(\Coh_H^{\sharp\beta})_{\nu_{\alpha,\beta}}^{\sharp 0}.
\end{align*}

Recall the following theorem on the existence of stability conditions on threefolds:
\begin{Thm}[\!{\cite[Theorem 8.2, Lemma 8.3]{BMS:stabCY3s}}]
   Let $(X,H)$ be a polarized smooth threefold satisfying \cite[Conjecture 4.1]{BMS:stabCY3s}. Then there is a slice of stability conditions on $\Db(X)$
   \begin{align}\label{eqstab3}
       \Stb(X)\coloneq\left\{\sigma^{a,b}_{\alpha,\beta}=(\cA_{\alpha,\beta},Z^{a,b}_{\alpha,\beta}):\alpha,\beta,a,b\in \R,\, \alpha>0,\, a>\tfrac{1}{6}\alpha^2+\tfrac{1}{2}|b|\alpha\right\}.
   \end{align}
\end{Thm}
Here the central charge is given as:
\begin{align}\label{eq777}
    Z^{a,b}_{\alpha,\beta}\coloneq\left[-\ch_3^{\beta }+bH\ch_2^{\beta }+aH^2\ch_1^{\beta }\right]+i\left[H\ch_2^{\beta }-\tfrac{1}{2}\alpha^2H^3\ch_0\right].
\end{align}

\begin{Rem}\label{rem:stabX}
    We summarize some other known facts about $\Stb(X)$ that will be useful later.
    \begin{enumerate}[(1)]
        \item Let $E$ be a $\sigma_{\alpha,\beta}^{a,b}$-semistable object, then $Q_{K,\beta}(E)\coloneq K\Delta_H(E)+\nabla^\beta_H(E)\geq 0$ for $K=\tfrac{1}{2}(\alpha^2+6a)$.
        \item Line bundles $\cO_X(mH)$ and skyscraper sheaves are stable with respect to all stability conditions in $\Stb(X)$. The phase of a skyscraper sheaf is always $1$.
    \end{enumerate}
\end{Rem}

\subsection{Reduced stability conditions on the polarized threefold}
Recall the notion $\gamma_3(t)\coloneq(1,t,\frac{t^2}{2},\frac{t^3}{6})$ and  $\vd_{\ut}(v)\coloneq C_{\ut}\det\left(\gamma_3(t_1)^T\;\gamma_3(t_2)^T\;\gamma_3(t_3)^T\;v^T\right)$ for every $\ut=(t_1,t_2,t_3)\in\sbr_3$ as that in \eqref{eq:Bvddefintro}. 

We define
\begin{align*}
\Stab^*_H(X)&\coloneq\Stb(X)\cdot GL_2\text{ and }\tilde {\mathfrak P}_3(X) \coloneq\Stb(X)\cdot \glt=\coprod_{n\in\Z}\Stab_H^*(X)[n], 
\end{align*}
where $GL_2=\{\tilde g=(g,M)\in \glt: g(0)\in(0,1]\}$. The notion $\tilde {\mathfrak P}_3(X)$ adopts from \cite{BMS:stabCY3s}. We will also see that the notion $\Stab_H^*(X)$ is compatible with that stated in Conjecture \ref{conj:intro}.
\begin{Thm}\label{thm:sb3}
     Let $(X,H)$ be a polarized smooth threefold satisfying \cite[Conjecture 4.1]{BMS:stabCY3s}. Then there is a family of reduced stability conditions on $\Db(X)$ given as:
     \begin{align}\label{eq799}
         \sb_H^*(X)\coloneq\left\{\ts_{\ut}\cdot c=(\cA_{\ut},e^{-c}\vd_{\ut}):c\in\R,\ut\in\sbr_3\right\}
     \end{align}
    satisfying the following properties:
    \begin{enumerate}[(1)]
        \item $\sb_H^*(X)=\pi_\sim(\Stab^*_H(X))$. In particular,  $\pi_\sim(\tilde {\mathfrak P}_3(X))=\coprod_{n\in\Z}\sb_H^*(X)[n]$.
        \item For every $m\in \Z$ and $t_1<t_2<t_3\in\R\cup\{+\infty\}$, we have \begin{align}\label{eq7times}
            \ts_{t_1,t_2,t_3}\otimes \cO_X(mH)=\ts_{t_1+m,t_2+m,t_3+m}.
        \end{align}
        
        When $t_3\neq +\infty$, all skyscraper sheaves are in $\cA_{t_1,t_2,t_3}$. The shifted line bundle $\cO(mH)[3-i]\in \cA_{t_1,t_2,t_3}$, when $m\in(t_i,t_{i+1}]$. Here we set $t_0=-\infty$ and $t_4=+\infty$.
         \item Let $E\in\cP_{\ut}(1)$, then its $H$-polarized character \begin{align}\label{eq710}
             \lambda_H(E)= \sum_{i=1}^3(-1)^ia_i\gamma_3(t_i) 
         \end{align}
         for some $a_i\geq 0$.
         \item If $s_1<t_1<s_2<t_2<s_3<t_3$, then $-\vd_{\ut}\in\ta(\ts_{\us})$ and $\vd_{\us}\in\ta(\ts_{\ut})$.
        \item  If $s_i<t_i$ for all $i=1,2,3$, then $\ts_{\us}\lsm \ts_{\ut}$.\\
         If $s_1<t_2$ and $s_2<t_3$, then $\ts_{\us}\lsm \ts_{\ut}[1]$. If $s_1<t_3$, then $\ts_{\us}\lsm \ts_{\ut}[2]$.
    \end{enumerate}
\end{Thm}


\noindent Recall the definition of $\BBB_3\coloneq\{c\vd_{\ut}:c> 0,\ut\in\sbr_3\}\subset \lbdd$  and $\pm\BBB_3\coloneq\BBB_3\cup (-\BBB_3)$ as that in \eqref{eq:defBn}. We denote by $\ut\link\us$ if $t_1<s_1<t_2<\dots<s_3$ or $s_1<t_1\dots<t_3$, see \eqref{eq:defuslinkut}.
\begin{Lem}\label{lem:71}
    $\Forg(\pi_\sim(\tilde{\mathfrak P}_3(X)))=\pm\BBB_3$.
    \end{Lem}
    \begin{proof}
        We first show the `$\subseteq$' direction. The equation \begin{align*}
            \Re Z_{\alpha,\beta}^{a,b}(\gamma_3(t+\beta))=-\tfrac{1}{6}t^3+\tfrac{1}{2}bt^2+at=0
        \end{align*} of $t$  has three distinct roots with order given as:
        \begin{align}\label{eq711}
            \tfrac{1}{2}(3b-\sqrt{9b^2+24a}),\; 0 \text{ and } \tfrac{1}{2}(3b+\sqrt{9b^2+24a}).
        \end{align}
        The equation $\Im Z_{\alpha,\beta}^{a,b}(\gamma_3(t+\beta))=0\cdot t^3+\tfrac{1}{2}t^2-\tfrac{1}{2}\alpha^2=0$ has two distinct roots: $\pm\alpha$.

       Note that $\alpha>0$, the assumption that $a>\tfrac{1}{6}\alpha^2+\tfrac{1}{2}|b|\alpha$ in \eqref{eqstab3} is equivalent to
        \begin{align*}
           & 9b^2+24a>9b^2+4\alpha^2+12|b|\alpha\iff \sqrt{9b^2+24a}>2\alpha\pm 3b\\
            \iff & \tfrac{1}{2}(3b-\sqrt{9b^2+24a})<-\alpha<0<\alpha < \tfrac{1}{2}(3b+\sqrt{9b^2+24a}).
        \end{align*}    
    By Lemma \ref{lem:Epoly}, we have $c_1\Re Z_{\alpha,\beta}^{a,b}+c_2\Im Z_{\alpha,\beta}^{a,b}\in\pm\BBB_3$ for every $c_1,c_2\neq 0$. It follows that\begin{align}\label{eq714}
         \Forg(\pi_\sim(\tilde {\mathfrak P}_3(X)))=\{\Im (zZ_{\alpha,\beta}^{a,b})\mid 0\neq z\in\C, \alpha>0,a>\tfrac{1}{6}\alpha^2+\tfrac{1}{2}|b|\alpha\}\subseteq \pm\BBB_3.
     \end{align} \\

   \noindent  We then show the `$\supseteq$' direction. For every $t_1<t_2\in\R$,  we may consider 
     \begin{align}\label{eq715}
          \alpha=\tfrac{1}{2}(t_2-t_1)>0, \beta=\tfrac{1}{2}(t_2+t_1)
     \end{align}
     and any $a,b\in \R$ satisfying the assumption $a>\tfrac{1}{6}\alpha^2+\tfrac{1}{2}|b|\alpha$. Then the imaginary part  $c\Im Z_{\alpha,\beta}^{a,b}=  \vd_{t_1,t_2,+\infty}$ for some scalar $c\in\R$. In particular, we have $\vd_{t_1,t_2,+\infty}=c\Forg(\pi_\sim(\sab^{a,b}))\in \Forg(\pi_\sim(\tilde{\mathfrak P}_3(X)))$.

     For every $t_1<t_2<t_3\in\R$, we may consider
     \begin{align}\label{eq716}
         \beta=t_2, 3b=t_1+t_3-2t_2, 24a=(t_3-t_1)^2-9b^2=(t_3-t_1)^2-(t_1+t_3-2t_2)^2>0,
     \end{align}
     and $\alpha>0$ sufficiently small so that  $a>\tfrac{1}{6}\alpha^2+\tfrac{1}{2}|b|\alpha$. Then by \eqref{eq711}, we have $\Re Z_{\alpha,\beta}^{a,b}=\vd_{\ut}$ up to a constant. In particular, we have $\vd_{\ut}=c\Forg(\pi_\sim(\sab^{a,b}[\tfrac{1}{2}]))\in \Forg(\pi_\sim(\tilde{\mathfrak P}))$.

     To sum up, we have $\Forg(\pi_\sim(\tilde{\mathfrak P}))\supseteq\pm\BBB_3$. Together with \eqref{eq714}, the statement follows.
    \end{proof}

\begin{Lem}\label{lem:72}
The forgetful map $\Forg':\Stab^*_H(X)\rightarrow\Hom(\Lambda_H,\C): \sigma\mapsto Z_\sigma$ is injective.

For every $\vd_{\ut}\in \BBB_3$,  the fiber  space    $\Forg'(\Stab^*_H(X))\cap (\pi_{\Im})^{-1}(\vd_{\ut})$ is given as $\{-c\vd_{\us}+i\vd_{\ut}:\ut<\us<\ut[1],c>0\}\cup\{c\vd_{\us}+i\vd_{\ut}:\us<\ut<\us[1],c>0\}$ and it is convex in $\lbdd$.
\end{Lem}

\begin{proof} Assume that $Z_\tau=Z_{\tau'}$ for some $\tau,\tau'\in\Stab^*_H(X)$, then by \cite[Theorem 8.2]{BMS:stabCY3s}, $\tau'=\tau\cdot \tilde g$, where $\tilde g=(g,\Id_2)$. As $|g(0)|<2$, we must have $g(0)=0$. It follows that $\tau=\tau'$.\\

    We then show the `$\supseteq$' direction, in other words, the fiber image space contains the central charges as that in the set. 

    For any $\us$ and $\ut$ with $\us\link \ut$, by Lemma \ref{lem:Epoly}, the whole pencil $M\coloneq\{c_1\vd_{\us}+c_2\vd_{\ut}:[c_1,c_2]\in\bP^1\}\subset \pm\BBB_3$, and there exists unique $r_1$ and $r_2$ such that $\vd_{r_1,r_2,r_3=+\infty}\in M$. As that in \eqref{eq715}, we may set
    $ \alpha=\tfrac{1}{2}(r_2-r_1), \beta=\tfrac{1}{2}(r_2+r_1)$.

    Note that $\beta<r_2$, by Lemma \ref{lem:Epoly} again, there exists unique  $q_1<q_2=\beta<q_3$ such that $\vd_{q_1,q_2,q_3}\in M$. As that in \eqref{eq716}, we set $3b=q_1+q_3-2\beta, 24a=(q_3-q_1)^2-9b^2.$
    
     Moreover, the parameters satisfy $q_1<r_1<r_2<q_3$ with gaps $x\coloneq r_1-q_1$, $y=r_2-r_1$, and $z\coloneq q_3-r_2$. It is clear that $\alpha>0$. By a direct computation, we have
     \begin{align*}
         &24a-4\alpha^2-12|b|\alpha \\
         =&(q_3-q_1)^2-(q_1+q_3-r_2-r_1)^2-(r_2-r_1)^2-2|q_1+q_3-r_2-r_1|(r_2-r_1)\\
         =& (x+y+z)^2-(z-x)^2-y^2-2y|x-z|=4xz+2y(x+z-|x-z|)>0.
     \end{align*}
     So $\sigma_{\alpha,\beta}^{a,b}$ is a stability condition in $\Stb(X)$. By the choice of the parameters, we have $\vd_{r_1,r_2,+\infty}=c \Im Z_{\alpha,\beta}^{a,b}$ and $\vd_{q_1,q_2,q_3}=c'\Re Z_{\alpha,\beta}^{a,b}$ for some scalars $c,c'\in\R$. In particular, we have $M=\{c_1\Re Z_{\alpha,\beta}^{a,b}+c_2\Im Z_{\alpha,\beta}^{a,b}:[c_1,c_2]\in \bP^1\}$. 

    \noindent \textit{Claim}: $\Forg'(\sab^{a,b}\cdot GL_2)\supset \{-c\vd_{\us}+i\vd_{\ut}:\ut<\us<\us[1],c>0\}\cup\{c\vd_{\us}+i\vd_{\ut}:\us<\ut<\us[1],c>0\}$.
\begin{proof}[Proof of the claim]
     When $t_3\neq+\infty$, note that $\vd_{\ut}(0,0,0,1)>0$, so there exists $\tilde g\in\glt$ with $g(0)\in(0,1)$ such that the central charge of $\sigma_{\alpha,\beta}^{a,b}\cdot \tilde g$ is of the form $c\vd_{\us}+i\vd_{\ut}$ for some non-zero  $c\in\R$. 
     
     Note that $t_3>r_2$, so $\Im Z_{\alpha,\beta}^{a,b}(\gamma_3(t_3))>0$. As $g(0)\in(0,1)$, the augment of $Z_{\sigma_{\alpha,\beta}^{a,b}\cdot \tilde g}(\gamma_3(t_3))$ cannot be $0$. Note that $\vd_{\ut}(\gamma_3(t_3))=0$ and $\vd_{\us}(\gamma_3(t_3))>0$ (resp. $<0$) when $s_3>t_3$ (resp. $s_3<t_3$). So the coefficient $c<0$ (resp. $c>0$) when $s_3<t_3$ (resp. $s_3>t_3$).

     When $t_3=+\infty$, we have $-\vd_{\ut}= Z_{\alpha,\beta}^{a,b}$ up to a positive scalar, so there exists  $\tilde g\in\glt$ with $g(0)=1$ such that the central charge of $\sigma_{\alpha,\beta}^{a,b}\cdot \tilde g$ is of the form $c\vd_{\us}+i\vd_{\ut}$ for some non-zero  $c\in\R$. 
     Note that  $Z_{\sigma_{\alpha,\beta}^{a,b}\cdot \tilde g}(0,0,0,1)\in\R_{<0}$, so the coefficient $c<0$.   
\end{proof}
To sum up, the `$\supseteq$' direction holds.  \\

     Finally, we show the `$\subseteq$' direction. By Lemma \ref{lem:71}, we have $\Forg'(\Stab^*_H(X))\subseteq \{Z: c_1 \Re Z+ c_2\Im Z\in \pm\BBB_3, \forall\; [c_1,c_2]\in\bP^1\}$. By Lemma \ref{lem:Epoly}, $\Forg'(\Stab^*_H(X))\cap (\pi_{\Im})^{-1}(\vd_{\ut})\subseteq \{c\vd_{\us}+i\vd_{\ut}:\ut\link\us, c\neq 0\}$. By the last paragraphs in the argument for the `$\supseteq$' direction, the sign of the coefficient must be as that in the statement.

     The convexity follows from Lemma \ref{lem:Epolyconvex}. 
 \end{proof}
\begin{proof}[Proof of Theorem \ref{thm:sb3}] Firstly, we  adopt \begin{align*}
    \sb^*_H(X)\coloneq \pi_\sim(\Stab^*_H(X)).
\end{align*} 
as the definition of $\sb^*_H(X)$. Our task is to show that $\sb^*_H(X)$ admits a parametrization as that in \eqref{eq799} and satisfies other properties stated in the theorem.\\

\noindent {(1)} Note that $\tilde {\mathfrak P}_3(X)=\coprod_{n\in\Z} (\Stab^*_H(X))[n]$, it is clear that $\pi_\sim(\tilde{\mathfrak P}_3(X))=\coprod_{n\in\Z}\sb^*_H(X)[n]$. 

By Lemma \ref{lem:72},   the fiber space  $\Forg'(\Stab^*_H(X))\cap (\pi_{\Im})^{-1}(\vd_{\ut})$ is convex and the map $\Forg'|_{\Stab^*_H(X)}$ is injective. Therefore, two stability conditions $\sigma,\tau\in \Stab^*_H(X)$ satisfy $\sigma\sim\tau$ if and only if $\Im Z_\tau=\Im Z_\sigma$. Moreover,  the forgetful map \begin{align}\label{eq723}
    \Forg:\sb^*_H(X)\rightarrow \lbdd
\end{align} is injective. 

By Lemma \ref{lem:71} and $\pi_\sim(\tilde{\mathfrak P}_3(X))=\coprod_{n\in\Z}\sb^*_H(X)[n]$, we have 
\begin{align*}
    \Forg(\sb^*_H(X))= \BBB_3.
\end{align*}

Note that $Z_\sigma(\cO_p)\in \H$, so we have \begin{align*}
    \Forg(\sb^*_H(X))=\{c\vd_{\ut}:c>0, \ut\in\sbr_3\}.
\end{align*}  
Note that $c$ is from the $\R$-action and does not affect the heart structure. We get a parametrizing space for $\sb^*_H(X)$ as that in  \eqref{eq799}.    \\

\noindent {(2)} By the construction of $\sigma_{\alpha,\beta}^{a,b}$, we have $\sigma_{\alpha,\beta}^{a,b}\otimes \cO_X(mH)=\sigma_{\alpha,\beta+m}^{a,b}$. As the $\otimes \cO_X(mH)$- action commutes with the $GL_2$-action, for a reduced stability condition $\ts_{\ut}=\pi_\sim(\sigma_{\alpha,\beta}^{a,b}\cdot g)$, we have \begin{align*}
\ts_{\ut}\otimes \cO_X(mH)=\pi_\sim((\sigma_{\alpha,\beta}^{a,b}\cdot g)\otimes \cO_X(mH))=\pi_\sim(\sigma_{\alpha,\beta+m}^{a,b}\cdot g)\in\sb^*_H(X).   
\end{align*}
Note that $\vd_{\ut}\otimes \cO_X(mH)=\vd_{t_1+m,t_2+m,t_3+m}$, by \eqref{eq723}, we have \eqref{eq7times}.\\

By Remark \ref{rem:stabX}, skyscraper sheaves and lines bundles $\cO_X(mH)$ are in the heart up to a homological shift. Note that $\vd_{\ut}(\gamma_3(x))>0$ when and only when $x\in(t_1,t_2)\cup (t_3+\infty)$. Also note that if $\cO_{X}(mH)\in\cA_{\ut}[s]$, then $\cO_{X}(mH)\in\cA_{\ut'}[s-1,s,s+1]$ for $\ut'$ in a small open neighborhood of $\ut$. The statement holds.\\

\noindent{(4)} follows from  Lemma \ref{lem:72}.

\noindent{(3)} If $E\in\cP_{\ut}(1)$, then $\lambda_H(E)\in\Ker \vd_{\ut}$. Note that $E$ is semistable with respect to every representative $\sigma$ of $\ts_{\ut}$, so in particular $Z_\sigma(\lambda_H(E))\neq 0$. It follows that 
\begin{align*}
    \lambda_H(E)\notin \Ker \vd_{\us}
\end{align*}
for every $s_1<t_1<s_2<t_2<s_3<t_3$. By Lemma \ref{lem:Epoly2}, the character $\lambda_H(E)$ is in the form of \eqref{eq710} for some $a_i$ all $\geq 0$ or all $\leq 0$. 

Note that $\vd_{\us}(\lambda_H(E))>0$, by \eqref{eqE55}, we must have $a_i\geq 0$ for all $i$.\\

\noindent{(5)} By Lemma \ref{lem:72}, Lemma \ref{lem:partialorder},  and the definition of $\ts_{\ut}$, for every $s_1<t_1<s_2<t_2<s_3<t_3$, we have $\ts_{\us}\lsm \ts_{\ut}$.

If $s_i<t_i$, then there exist $a_i$ and $b_i$ such that 
\begin{align*}  s_1<a_1<s_2<a_2<s_3<a_3;\;\;a_1<b_1<a_2<b_2<a_3<b_3;\;\;b_1<t_1<b_2<t_2<b_3<t_3.
\end{align*}
It follows that   $ \ts_{\us}\lsm\ts_{a_1,a_2,a_3}\lsm\ts_{b_1,b_2,b_3}\lsm\ts_{\ut}.$ 

If $s_1<t_2$ and $s_2<t_3$, then there exist $w_i$ such that $w_1<s_1<w_2<s_2<w_3<s_3\text{ and } w_i<t_i.$
It follows that $\ts_{\us}\lsm \ts_{w_1,w_2,w_3}[1]\lsm \ts_{\ut}[1]$.

If $s_1<t_3$, then there exist $u_i$ and $v_i$ such that
\begin{align*}
    u_1<s_1<u_2<s_2<u_3<s_3; \;\;v_1<u_1<v_2<u_2<v_3<u_3;\;\; \text{ and } v_i<t_i.
\end{align*}
It follows that $\ts_{\us}\lsm\ts_{u_1,u_2,u_3}[1]\lsm\ts_{v_1,v_2,v_3}[2]\lsm\ts_{\ut}[2]$.
\end{proof}

\subsection{Bayer Vanishing Lemma and Restriction Theorem}
 By the same argument as that for Proposition \ref{prop:bayerlemsurface} and Remark \ref{rem:surfacestabtosb}, we have the following corollary from Theorem \ref{thm:sb3}.{(2)} and {(5)}.
\begin{Cor}\label{cor:bayer3}
 Let $(X,H)$ be a polarized threefold satisfying \cite[Conjecture 4.1]{BMS:stabCY3s} and $E,F$ be two objects in $\Db(X)$. Then under either of following conditions:
 \begin{enumerate}[(1)]
     \item Assume there exists $\ts_{\ut}$  as that in \eqref{eq799} satisfying $t_3\neq +\infty$ and $E,F\in\cP_{\ts_{\ut}}(1)$. 
     \item Assume there exists $\sigma=\sigma_{\alpha,\beta}^{a,b}$ as that in \eqref{eqstab3} satisfying $\phi_{\sigma}(E)\geq\phi_{\sigma}^+(F)$ and $\Im Z_\sigma(E)\neq 0$.
 \end{enumerate}  
  We have the vanishing $        \Hom(E(mH),F)=0$  for every $m>0$.
\end{Cor}

We may also apply Proposition \ref{prop:reststab} in the threefold case.
\begin{Ex}\label{cor:rest3to2}
   Let $(X,H)$ be a polarized threefold satisfying \cite[Conjecture 4.1]{BMS:stabCY3s}, $S\in|dH|$ be a smooth subvariety of $X$.
   \begin{enumerate}[(1)]
       \item Assume that the parameters $\alpha,a,b$ satisfy
       \begin{align}\label{eq:asp735}
          2\alpha>d \text{ and } \sep(x^3-(3b+c)x^2-6ax+c\alpha^2)>d\text{ for every }c\in\R.
       \end{align}
      Then the stability condition $\sigma_{\alpha,\beta}^{a,b}$ as that in \eqref{eqstab3} restricts to a stability condition on $\Db(S)$.
      
       \item Assume $t_3-t_2>d$ and $t_2-t_1>d$, then $\cA_{\ut}|_{\Db(S)}=\cA_{s_1,s_2}$, which is the heart on $\Db(S)$ as that in \eqref{eq5611} with respect to the polarization $H|_S$. The parameters $s_1,s_2$ are given as
       \begin{align*}
          & \frac{1}{6}\left(2\sum t_i +3d\pm\sqrt{2\sum(t_i-t_j)^2-3d^2}\right), &\text{ when }t_3\neq +\infty;\\
        \text{ and }   & s_1=(t_1+t_2+d)/{2}, s_2=+\infty, & \text{ when }t_3=+\infty.
       \end{align*}
   \end{enumerate}
\end{Ex}
\begin{proof}
    {(1)} By the assumption \eqref{eq:asp735}, for every $\theta\in(0,1]$, we have 
    \begin{align*}
       \pi_\sim( \sigma_{\alpha,\beta}^{a,b}[\theta])=\ts_{\ut}
    \end{align*}
    for some $t_3-t_2,t_2-t_1>d$. 

    By Theorem \ref{thm:sb3}.{(5)}, we have $\ts_{\ut}\otimes \cO_X(dH)=\ts_{t_1+d,t_2+d,t_3+d}\lsm \ts_{\ut}[1]$. By Lemma \ref{lem:eqdefforlsmonstab}, we have $\sigma_{\alpha,\beta}^{a,b}\otimes \cO_X(dH)\lsm \sigma_{\alpha,\beta}^{a,b}[1]$. The statement follows from Proposition \ref{prop:reststab}.\\

    {(2)} By Lemma \ref{lem:Epoly7}, there exists  $\vd_{\um}\in\ta(\ts_{\ut})$ and $\sep(c\vd_{\um}+d\vd_{\ut})>d$ for every $[c:d]\in\bP_{\R}^1$.
    
    Let $\sigma=(\cA_{\ut},\vd_{\um}+i\vd_{\ut})$, then by Theorem \ref{thm:sb3}, for every $\theta\in(0,1]$, the reduced stability condition $\pi_\sim(\sigma[\theta])=\ts_{p_1,p_2,p_3}$ for some $p_i$ such that both $p_3-p_2$ and $p_2-p_1>d$. So $\pi_\sim(\sigma[\theta])\otimes \cO_X(dH)\lsm \pi_\sim(\sigma[\theta])[1]$. By Lemma \ref{lem:eqdefforlsmonstab}, we have $\sigma\otimes \cO_X(dH)\lsm \sigma[1]$.
    
    By Proposition \ref{prop:reststab} and Corollary \ref{cor:geodeterbyZ}, $\sigma|_{\Db(S)}$ is a geometric stability condition and is determined by its central charge. In particular, the heart $\cA_{\ut}|_{\Db(S)}$ is in the form of $\cA_{s_1,s_2}$.

    The parameters $s_i$ can be computed via the property that:
    \begin{align}\label{eq740}
        \vd_{\ut}([\iota_*](\gamma_2(s_i)))=0.
    \end{align}
    Note that $[\iota_*](\gamma_2(s))=\gamma_3(s)-\gamma_3(s-d)$. When $t_3\neq +\infty$, the equation above is given as
    \begin{align*}
        0=\prod(s-t_i)-\prod(s-t_i-d)=3ds^2-(3d^2+2d\sum t_i)s+d\sum t_it_j +d^2\sum t_i +d^3.
    \end{align*}
    When $t_3=+\infty$, the equation \eqref{eq740} is 
    \begin{align*}
        0=(s-t_1)(s-t_2)-(s-t_1-d)(s-t_2-d)=2ds-(d(t_1+t_2)+d^2).
    \end{align*}
    The formula of $s_i$ is by solving these equations.
\end{proof}

\begin{Ex}
Adopt the assumption as that in Example \ref{cor:rest3to2}, we give two examples that condition \eqref{eq:asp735} holds.

    Assume $\alpha>\frac{\sqrt{3}}{3}d$ and $a=\frac{\alpha^2}{2}$, then $\sigma_{\alpha,0}^{a,0}$ restricts to a stability condition on $\Db(S)$. To see this, by Example\ref{cor:rest3to2}, we only need to check \eqref{eq:asp735}. Note that $(x^3-6ax)'=3x^2-6a=3(x^2-\alpha^2)$, the statement follows by Lemma \ref{lem:E3}.

    Assume $\alpha\geq d$ and $a\geq \tfrac{1}{6}(\alpha+d)^2+\tfrac{1}{2}(\alpha+d)|b|$, then $\sigma_{\alpha,\beta}^{a,b}$ restricts to a stability condition on $\Db(S)$. To see this, note that the assumption says that the gap of roots of $x^3-3bx^2-6ax=0$ and $x^2-\alpha^2=0$ are not less than $d$. So $\sep(f)>d$ for every $f$ in the pencil spanned by them.
\end{Ex}

\subsection{Example: Wall-crossing on $\sb^*(\bP^3)$}
For the rest part of this section, we fix the threefold to be the projective space $\bP^3$ and discuss a few more properties about the wall-crossing behavior under the theory of reduced stability conditions. A detailed study of specific examples of moduli spaces will be deferred to future work.\\

\noindent Let $\sb^*(\bP^3)$ be the manifold of reduced stability conditions as that in Theorem \ref{thm:sb3} and $0\neq v\in\Kn(\bP^3)$. Recall the general setup as that in Section \ref{sec:wc}, we may define 
\begin{align*}
    \sb^\dag_v(\bP^3)\coloneq \left\{c\cdot \ts_{\ut}\in\sb^*(\bP^3):\begin{aligned}
         &c>0,\vd_{\ut}(v)=0,\vd_{\us}(v)\neq 0\\
         &\text{ for every } \us<\ut<\us[1].
    \end{aligned}\right\}
\end{align*}

Then by Proposition \ref{prop:wallchambonsb}, see also Definition \ref{def:piv}, the map:
\begin{align*}
    \pi_v:\Stab^*_v(\bP^3)\to \sb^\dag_v(\bP^3): \sigma\mapsto\pi_\sim(\sigma[\theta]),
\end{align*}
where $\theta\in(0,1]$ is the value such that $e^{-i\pi \theta}Z_\sigma(v)\in \R_{\neq 0}$, is well-defined on every chamber in which $M_\sigma(v)\neq \emptyset$. The map preserves all walls and chambers for $M(v)\neq\emptyset$ on $\Stab^*(\bP^3)$.

By Lemma \ref{lem:72}, the forgetful map:
\begin{align*}
    \Forg: \sb^\dag_v(\bP^3)\to \vperp: \ts\mapsto B_{\ts}
\end{align*}
is injective. For each wall $W\subset \sb^\dag_v(\bP^3)$, its image $\Forg(W)$ is contained in a real codimension one linear subspace $\wperp\cap \vperp$ for some $0\neq w\in\Kn(\bP^3)$.

\begin{Ex}(Hilbert scheme of points)\label{eg:wallhilbpts}
    Let $v=(1,0,0,-m)$ for some $m\in\Z_{\geq1}$. When $t_3\neq +\infty$, we have \begin{align}\label{eq738}
        \vd_{\ut}(v)=C(-m-\tfrac{1}{6}t_1t_2t_3),
    \end{align}
    where $C=\prod_{i<j}(t_j-t_i)^{-1}$. The value $\vd_{\ut}(v)$ equals $0$ if and only if $t_1t_2t_3=-6m$. If in addition, both $t_2,t_3>0$, then there always exists $s_1<t_1<s_2<t_2<s_3<t_3$ such that $s_1s_2s_3=-6m$. In particular, such a $\ts_{\ut}$ is in $\Stab^{\emptyset}_v(\bP^3)$, in other words, the space $M_{\ts}(v)=\emptyset$. So one can describe all walls and chambers for $M(v)$ on the following space:
    \begin{align*}
        \sb^\dag_{(1,0,0,-m)}(\bP^3)=\{c\ts_{\ut}:c>0,t_1<t_2<t_3<0,t_1t_2t_3=-6m\}.
    \end{align*}
    Together with Lemma \ref{lem:78} below, one can draw the stage for the wall and chamber of character $(1,0,0,-m)$ as that in Figure \ref{fig:sbp3}.
       \begin{figure}[h]
        \centering
\begin{tikzpicture}
  \begin{axis}[
    axis lines = middle,
    xmin=0, xmax=100,
    ymin=0, ymax=100,
    samples = 300,
    domain = 0.3:10,
    xtick=\empty,
    ytick=\empty,
    xlabel={$-\sum t_i$},
    ylabel={$\sum t_it_j$},
    axis line style={->},
    width=10cm,
    height=10cm,
    tick style={draw=none},
    legend style={fill=none, at={(1.35,0.45)}, anchor=north east}
  ]
    \addplot[
      blue,
      thick, dotted,
    ] 
    ({2*x + 12/(x^2)}, {x^2 + 24/x});
    \addlegendentry{boundary of $\sb^\dag_{v}(\bP^3)$}

    \addplot[red, ultra thick, domain=4:50] 
      ({4 + x}, {3 + 4*x});
    \addlegendentry{boundary wall $t_1=-M$}

    \addplot[green!60!black,  domain=7:80] 
      ({1 + x}, {12 + x});
    \addlegendentry{boundary wall $t_3=-N$}

    \addplot[dashed, domain=4:100] 
      ({2 + x}, {12/2 + 2*x}); 
    \addplot[dashed, domain=5:100] 
      ({3 + x}, {12/3 + 3*x}); 
      \addplot[dashed, domain=9:100] 
      ({0.5 + x}, {12/0.5 + 0.5*x});\addplot[dashed, domain=11:50] 
      ({0.3 + x}, {12/0.3 + 0.3*x});\addplot[dashed, domain=13:70] 
      ({0.2 + x}, {12/0.2 + 0.2*x});
      \addplot[dashed, domain=1.5:100] ({20 + x}, {12/20 + 20*x});
      \addplot[dashed, domain=1.1:100] 
      ({40 + x}, {12/40 + 40*x});\addlegendentry{ potential walls}
  \end{axis}
  \node[below,blue] at (9,5.5) {
$\left(2t + \frac{6m}{t^2},\, t^2 + \frac{12m}{t}\right)$};
\end{tikzpicture}

        \caption{Walls and chamber structures for $v=(1,0,0,-m)$. All walls and chambers for $M(v)$ in  $\Stab^*(\bP^3)$  are described in the region $\sb^\dag_v(\bP^3)$, which lies above the blue curve parametrized by $(2t + \frac{6m}{t^2},\, t^2 + \frac{12m}{t})$. Moreover, the moduli space $M_{\ts}(v)$ is empty whenever $\ts$ lies below the green line or above the red line.}
        \label{fig:sbp3}
    \end{figure}
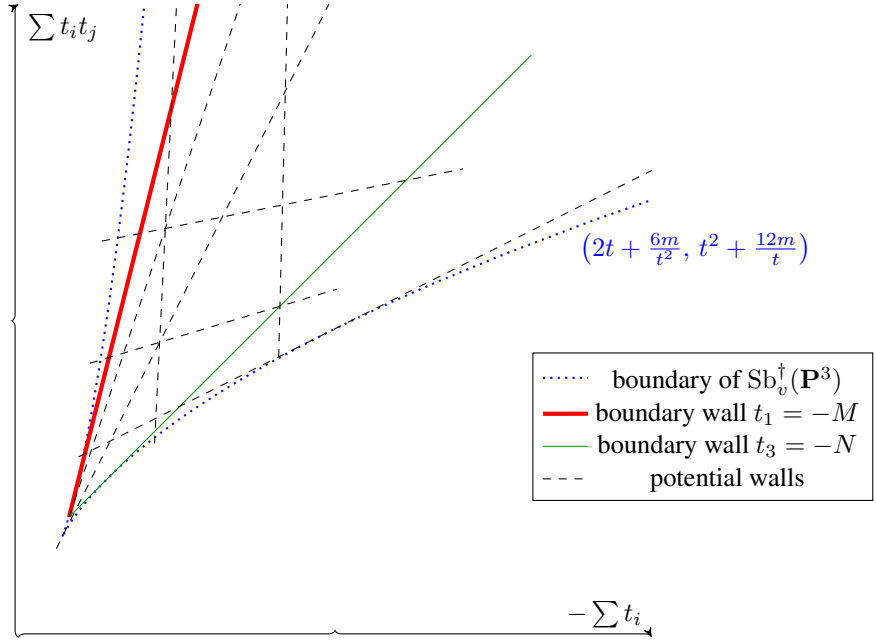
\end{Ex}
\begin{Lem}\label{lem:78}
 There is no $\ts_{\ut}$-semistable object with character $(1,0,0,-m)$ if the parameter satisfies either of the following condition:
 \begin{enumerate}[(1)]
     \item $t_3<-N$, where $N$ is the smallest positive integer satisfying $(N+1)(N+2)(N+3)>6m$.
     \item $t_1>-M$, where $M$ is largest positive integer satisfying $M^2(M-4)<6m$ and $M\leq m+2$.
 \end{enumerate}  
\end{Lem}
\begin{proof}
 {(1)}   Note that $t_3<0$, so  when $s_1<t_1<s_2<t_2<s_3<t_3$, by \eqref{eq738}, we have $\vd_{\us}(1,0,0,-m)>0$. Note that $\vd_{\us}\in \ta(\ts_{\ut})$, so if an object $E$ is $\ts_{\ut}$-semistable with character $(1,0,0,-m)$, then it must be in $\cP_{\ut}(0)$.
 
 Note that $N(N+4)^2>(N+1)(N+2)(N+3)$ for every $N\geq 1$. It follows that $t_2>-N-4$. By Theorem \ref{thm:sb3}.{(2)}, $\cP_{\ut}((0,1))$ contains $\cO(-N)$ and  $\cO(-4-N)[2]$. 

    So for every $E\in\cP_{\ut}(0)$ and $i\geq 0$, we have \[\Hom(\cO(-N)[i],E)=0=\Hom(E,\cO(-4-N)[1-i]).\]  By Serre duality, $\Hom(\cO(-N),E[i])=0$ for every $i\neq 1$.

    It follows that 
    \begin{align*}
        0\leq \hom(\cO(-N),E[1])=-\chi(\cO(-N),E)=m-\tfrac{1}{6}(N+1)(N+2)(N+3)<0,
    \end{align*}
    which leads to the contradiction. So there is no  $\ts_{\ut}$-semistable object with character $(1,0,0,-m)$.\\

\noindent  {(2)}   By the assumption, we have $t_3<-M+4$. By Theorem \ref{thm:sb3}.{(2)}, $\cP_{\ut}((0,1))$ contains the objects $\cO(-M)[3]$ and  $\cO(-M+4)$. 

   Suppose there is a $\ts_{\ut}$-semistable object $E$ with character $(1,0,0,-m)$. Then by Serre duality, we have  $\Hom(\cO(-M+4),E[i])=0$ for every $i\in\Z$. 
    
    It follows that 
    \begin{align*}
        0=\chi(\cO(-M+4),E)=\tfrac{1}{6}(M-3)(M-2)(M-1)-m<0,
    \end{align*}
    which leads to the contradiction. Note that the extra assumption $M\leq m+2$ is just saying that when $m=1$, we set $M=3$. Otherwise, $M=4$ will fail the last inequality.
    
    So there is no  $\ts_{\ut}$-semistable object with character $(1,0,0,-m)$.
\end{proof}

\begin{Ex}\label{eg:vanP31}
    Let $X=\bP^3$ and  $t_i$ be real numbers satisfying $n-4<t_1<t_2<t_3<n$ for some $n\in \Z$. Then $\cP_{\ts_{\ut}}(1)\subseteq \cO(n)^\perp=\{E\in\Db(\bP^3):\RHom(\cO(n),E)=0\}$. If in addition $n-3<t_1<n-2<t_2<n-1<t_3<n$, then $\cP_{\ts_{\ut}}(1)=\emptyset$.
\end{Ex}
\begin{proof}
By Theorem \ref{thm:sb3}.{(2)}, $\cP_{\ut}((0,1))$ contains $\cO(n)$ and  $\cO(n-4)[3]$. So $\Hom(\cO(n)[i],E)=0=\Hom(E,\cO(n-4)[4-i])$ for every $i\geq 1$ and $E\in\cP_{\ut}(1)$. The statement follows by Serre duality.

If in addition $n-3<t_1<n-2<t_2<n-1<t_3<n$, then $\cA_{\ut}$ contains $\cO(n-i)[i]$ for $i=0,\dots,3$. So it must be the Beilinson heart $\langle \cO(n-3)[3],\cO(n-2)[2],\cO(n-1)[1],\cO(n)\rangle$. Note that $\vd_{\ut}(\cO(n-i)[i])>0$, so $\vd_{\ut}(\cA_{\ut})>0$. The statement follows.
\end{proof}

\section{Standard Model}\label{sec:sm}
Throughout this section, we assume that $(X, H)$ is an $n$-dimensional smooth projective irreducible variety over $\mathbb{C}$, equipped with a polarization $H$. We fix the $H$-polarized lattice:
\begin{align*}
    \lambda_H:\Kn(X)\to \Lambda_H:[E]\mapsto(H^n\ch_0(E),H^{n-1}\ch_1(E),\dots,\ch_n(E)).
\end{align*}
We denote by $\Lambda_\R\coloneq \Lambda_H\otimes \R$ and define the $n$-twisted vectors as 
\begin{align*}
\gamma_n:\R\cup\{+\infty\}&\to\Lambda_\R:
    \R\ni t \mapsto (1,t,\frac{t^2}{2!},\dots,\frac{t^n}{n!})\; ;\; +\infty \mapsto (0,\dots,0,1).
\end{align*}
Let $d\geq 0$. We define a subspace of $\lbdd$ as:
\begin{align}
\label{eq84}    \mathfrak B_n^{>d}\coloneq\{c\vd_{\ut}:\ut\in\sbr_n, c>0, \sep(\ut)>d\}.
\end{align}
When $d=0$, we will write $\BBB_n\coloneq \BBB_n^{>0}$ to simplify the notion.
  
Here the parameter space  in \eqref{eq84} is 
\begin{align*}
    \sbr_n&\coloneq \{\ut=(t_1,t_2,\dots,t_n):t_1<t_2<\dots<t_n,t_n\in\R\cup\{+\infty\}\}.\\
    \sep(\ut)&\coloneq \min_{1\leq i\leq n-1}\{t_{i+1}-t_i\}.
\end{align*}
For every $\ut=(t_1,\dots,t_n)\in\sbr_n$ and $\bv=(v_0,\dots,v_n)\in\Lambda_\R$, the reduce central charge is given as 

\begin{align}\label{eq:Bvddef}
        \vd_{\ut}(\bv)\coloneq C_{\ut}\det\begin{vmatrix}
            \gamma_n(t_1) \\ \dots \\ \gamma_n(t_n)\\ \bv
        \end{vmatrix}=C_{\ut}\det\begin{vmatrix}
            1 & t_1 &\dots & \frac{t_1^n}{n!} \\ \dots & \dots & \dots & \dots \\ 1 & t_n &\dots & \frac{t_n^n}{n!}\\ v_0 & v_1 & \dots & v_n
        \end{vmatrix}.
    \end{align}
where the normalizing coefficient $C_{\ut}$ is defined as
\begin{align}\label{eq:BvdC}
    &\prod_{1\leq k\leq n-1}k!\cdot\prod_{1\leq i<j\leq n}(t_j-t_i)^{-1}, &\text{ when }t_n\neq +\infty;\\
    &\prod_{1\leq k\leq n-2}k!\cdot\prod_{1\leq i<j\leq n-1}(t_j-t_i)^{-1}, &\text{ when }t_n= +\infty.\notag
\end{align}
    
  By the property of Vandermonde matrix, when $t_n\neq +\infty$, we have $\vd_{\ut}((0,\dots,0,1))=1$. 
  
  When $t_n=+\infty$, the determinant $\vd_{\ut}(\bv)=-\vd_{t_1,\dots,t_{n-1}}(v_0,v_1,\dots,v_{n-1})$. In particular, we have $\vd_{\ut}((0,\dots,0,1))=0$, and $\vd_{\ut}((0,\dots,0,1,0))=-1$.

We define some more notions on $\sbr_n$ and $\lbdd$ as follows:
\begin{align*}
\ut+a&\coloneq (t_1+a,t_2+a,\dots,t_n+a).\\
    \us<\ut&:\iff s_i<t_i \text{ for every }i=1,\dots,n. \;\;\\
    \us<\ut[k]&:\iff s_i<t_{i+k}\text{ for every }i=1,\dots,n-k.\\
    \us\link\ut&:\iff \us<\ut<\us[1]\text{ or }\ut<\us<\ut[1].\\
    \ell(\us,\ut)&\coloneq \{a\vd_{\us}+b\vd_{\ut}:(a,b)\neq (0,0)\}\subset \lbdd, \text{ when } \us\link\ut.\\
    \sep(\vd_{\ut})&\coloneq\sep(\ut)=\min\{t_{i+1}-t_i:1\leq i\leq n-1\}.\\
    \sep(\ell)&\coloneq \min\{\sep(\vd):\vd\in\ell\}.
\end{align*}

\noindent Note that by Lemma \ref{lem:Epoly}, the whole line $\ell(\us,\ut)\subset \pm\BBB_n$ when and only when $\us\link \ut$. 

  For every $d\geq 0$, we define a subspace of  central charges as follows:
  \begin{align}\label{eq812}
      \UUU_n^{>d}\coloneq \left\{c_1\vd_{\us}+ic_2\vd_{\ut}:\begin{aligned}
          &\ut<\us<\ut[1]\\
          \text{ or }&\us<\ut<\us[1]
      \end{aligned},\;
         \sep(\ell(\us,\ut))>d,\; c_2>0, \text{ and } \begin{aligned}
          c_1<0 \text{ if }\ut<\us;\\\; c_1>0 \text{ if }\us<\ut.
      \end{aligned} \right\}.
  \end{align}

When $d=0$, by Lemma \ref{lem:Epoly}, the condition $\sep(\ell(\us,\ut))>0$ holds automatically and can be dropped. We will write $\UUU_n=\UUU^{>0}_n$ to simplify the notion.
\subsection{Conjectures}
 Let $(X,H)$ be a smooth polarized variety over $\C$ and $d\geq 0$. Our main conjectures are stated as follows.
  
\begin{Con}[$\Stab^d$ Conjecture]\label{conj:stabd}
   There exists a family of stability conditions $\Stab_H^{*>d}(X)$ with respect to the $H$-polarized lattice $\Lambda_H$ satisfying:
    \begin{enumerate}
        \item The forgetful map 
        \begin{align*}
            \Forg: \Stab_H^{*>d}(X)\to \Hom(\Lambda_H,\C): \sigma=(\cA,Z)\mapsto Z
        \end{align*}
        is a homeomorphism onto $\UUU_n^{>d}$.
        \item For any $\sigma\in \Stab_H^{*>d}(X)$, 
        the stability condition $\sigma\otimes \cO_X(H)$ is also in $\Stab_H^{*>d}(X)$.
    \end{enumerate}
\end{Con}

\begin{Con}[$\sb^d$ Conjecture]\label{conj:sbd}    
    There exists a family of reduced stability conditions $\sb_H^{*>d}(X)$ with respect to the $H$-polarized lattice $\Lambda_H$ satisfying:
    \begin{enumerate}
        \item The forgetful map 
        \begin{align*}
            \Forg: \sb_H^{*>d}(X)\to \lbdd: \ts=(\cA,B)\mapsto B
        \end{align*}
        is a homeomorphism onto $\BBB_n^{>d}$. The map $\Forg:\coprod_{n\in\Z}\sb_H^{*>d}(X)[n]\to \pm\BBB_n^{>d}$ is a universal cover.
        
        For every $\vd_{\ut}\in \BBB_n^{>d}$, we may denote its preimage of $\Forg$ as $\ts_{\ut}=(\cA_{\ut},\vd_{\ut})$.
        \item For any $\ts_{\ut}\in \sb_H^{*>d}(X)$,
        the reduced stability condition $\ts_{\ut}\otimes \cO_X(H)$ is also in $\sb_H^{*>d}(X)$.
        
        \item Let $\ts_{\us},\ts_{\ut}\in\sb_H^{*>d}(X)$ satisfying $\us<\ut<\us[1]$ and $\sep(\ell(\us,\ut))>d$, then 
        \begin{align*}
            \ts_{\us}\lsm\ts_{\ut}\lsm\ts_{\us}[1].
        \end{align*}
    \end{enumerate}
\end{Con}

\begin{Rem}\label{rem:onstabmodelconj}
We make some comments on these two conjectures.
\begin{enumerate}[(1)]
    \item 
    By definition, when $d>d'\geq0$, it is clear that $\BBB_n^{>d}\subset \BBB_n^{>d'}$ and $\UUU_n^{>d}\subset \UUU^{>d'}$. Note that for every $\sigma\in \Stab^{>d'}_H(X)$ with central charge in $\UUU_n^{>d}$, the central charge of $\sigma\otimes \cO_X(H)$  is also in $\UUU_n^{>d}$.     So for every $(X,H)$, $\Stab^{d'}$ (resp. $\sb^{d'}$) Conjecture implies $\Stab^{d}$ (resp. $\sb^d$) Conjecture. 
    
    The strongest form is when $d=0$.  We will omit the $>0$ to simplify the notion. Also, by Lemma \ref{lem:Epoly}, when $d=0$, the condition $\sep(\ell(\us,\ut))>0$ in $\sb^0_H$ Conjecture.(c)  can be dropped.

 \item   We have seen in Section \ref{eg:cs} that the $\Stab_H^0$ Conjecture is known to be true when $X$ is a curve or surface. When $X$ is a threefold, by Theorem \ref{thm:sb3}, \cite[Conjecture 4.1]{BMS:stabCY3s} implies the $\Stab_H^0$ Conjecture. 

  \item  The $\Stab^{0}_H$  Conjecture does not hold for all $(X,H)$. In the threefold case, one can find counter-examples such as the blown-up at a point on $\bP^3$. More discussions are referred to \cite{Benjamin:counterexample,BMSZ:Fano3stab}. However, we expect for every $(X,H)$, the $\Stab^{d}_H$  Conjecture holds when $d$ is large enough. We also expect $\Stab^{0}_H$  Conjecture holds for  many important examples such as $\bP^n$ and polarized abelian varieties.
\end{enumerate}
\end{Rem}

\subsection{Properties of $\Stab^*_H$ and $\sb^*_H$}
\begin{Thm}\label{thm:equivalentconjs}
    $\Stab^d$ Conjecture \ref{conj:stabd} holds for $(X,H)$ if and only if $\sb^d$ Conjecture \ref{conj:sbd} holds for $(X,H)$.
\end{Thm}
\begin{proof}
    `$\implies$': Let $\Stab_H^{*>d}(X)$ be a family of stability conditions as that in Conjecture \ref{conj:stabd}. We claim that the image family $\pi_\sim(\Stab_H^{*>d})$ satisfies all properties in Conjecture \ref{conj:sbd}.\\

    { (a)} For every $c_2\vd_{\ut}$, by definition, we have the identification
    \begin{align*}
        \{\Re Z_\sigma \mid \sigma\in \Stab^{*>d}_H(X),\Im Z_\sigma=c_2\vd_{\ut} \}=\{B\mid B+ic_2\vd_{\ut}\in \UUU^{>d}_n\}.
    \end{align*}
 By Lemma \ref{lem:Epolyconvex}, this is a convex subset in $\lbdd$. By Proposition \ref{prop:localisom}, all stability conditions $\sigma=(\cA,Z)\in\Stab_H^{*>d}(X)$ with $\Im Z=c_2\vd_{\ut}$ are with the same $\pi_\sim(\sigma)$. We have the commutative diagram:
      \begin{equation}
        \begin{tikzcd}
            \Stab_H^{*>d}(X)\arrow{r}{\pi_\sim}\ar[d,"\cong" description, "\Forg'"{xshift=3pt}]& \sb_H(X) \ar{d}{\Forg}\\
           \begin{array}{cc}
                 \UUU_n^{>d} \\
               c_1\vd_{\us}+ic_2\vd_{\ut}
           \end{array} \ar{r}{\pi_{\Im}}& \begin{array}{cc}
              \lbdd \\
              c_2\vd_{\ut}
           \end{array}.
        \end{tikzcd}
    \end{equation}
    Moreover, the map $\Forg$ is a homeomorphism from $\pi_\sim(\Stab_H^{*>d}(X))$ onto $\pi_{\Im}(\UUU^{>d}_n)$.
    By Lemma \ref{lem:Epoly7}, we have $\pi_{\Im}(\UUU^{>d}_n)=\BBB^{>d}_n$. So $\pi_\sim(\Stab_H^{*>d}(X))$ satisfies the first part of  $\sb^d$ Conjecture \ref{conj:sbd}.(a).
    \\

   { (b)} For every $\ts\in\pi_\sim(\Stab_H^{*>d}(X))$, we have $\ts\otimes \cO_X(H)=\pi_\sim(\sigma)\otimes \cO_X(H)=\pi_\sim(\sigma\otimes \cO_X(H))\in\pi_\sim(\Stab_H^{*>d}(X))$. The property follows.\\

    { (c)} By assumption, there are $\sigma_{\us,\ut}=(\cA_{\ut},\vd_{\us}+i\vd_{\ut})$ and $\sigma_{-\ut,\us}=(\cA_{\us},-\vd_{\ut}+i\vd_{\us})$  in $\Stab_H^{*>d}(X)$.  Moreover, by the assumption that $\sep(\ell(\us,\ut))>d$, for every $\theta\in[0,\tfrac{1}{2}]$, the central charge of $\sigma_{-\ut,\us}[\theta]$ is in $\UUU_H^{>d}$. 
    
    As the forgetful map $\Forg:\Stab_H(X)\to \Hom(\Lambda_H,\C)$ is locally homeomorphic and $\Stab_H(X)$ is Hausdorff, the curve of central charges $\{Z_{\sigma_{-\ut,\us}[\theta]}:\theta\in[0,\tfrac{1}{2}]\}\subset \UUU_n^{>d}$ uniquely lifts to a curve of stability conditions $\{\sigma_\theta:\theta\in[0,\tfrac{1}{2}]\}$ in $\Stab_H(X)$ with $\sigma_0=\sigma_{-\ut,\us}$. Note that this curve is just $\{\sigma_{-\ut,\us}[\theta]:\theta\in[0,\tfrac{1}{2}]\}$, so it must be contained in $\Stab_H^{*>d}(X)$. In particular, by comparing the central charge,  we have
    \begin{align*}
        \sigma_{{\us},{\ut}}=\sigma_{-{\ut},{\us}}[\tfrac{1}{2}].
    \end{align*}
   According to the previous construction, in $\pi_\sim(\Stab_H^{*>d}(X))$, we have $\ts_{\ut}=\pi_\sim(\sigma_{\us,\ut})$ and $\ts_{\us}=\pi_\sim(\sigma_{-\ut,\us})$. By Lemma \ref{lem:partialorder}, we have
    \begin{align*}
       \ts_{\us}= \pi_\sim(\sigma_{-\ut,\us})\lsm \pi_\sim(\sigma_{-\ut,\us}[\tfrac{1}{2}])= \pi_\sim(\sigma_{\us,\ut})=\ts_{\ut}\lsm\pi_\sim(\sigma_{-\ut,\us}[1])=\ts_{\us}[1].
    \end{align*}
    So the family $\pi_\sim(\Stab_H^{*>d}(X))$ satisfies property in the $\sb^d$ Conjecture \ref{conj:sbd}.(c).\\

    Finally, note that $\sigma_{\us,\ut}[\tfrac{1}{2}]=\sigma_{-\ut,\us}[1]$. So the map $\Forg^{\pm}:\Stab_H^{*>d}(X)\coprod\Stab_H^{*>d}(X)[1]\to \pm \UUU_n^{>d}$ is by `gluing' along the locus on $\Stab_H^{*>d}(X)$ in the form of $(\cA_{\ut},c_1\vd_{\us}+ic_2\vd_{\ut})$ with $t_n=+\infty$ to the boundary of $\Stab_H^{*>d}(X)[1]$ in the form of $(\cA_{\ut'},c_1\vd_{\us'}+ic_2\vd_{\ut'})$ with $t_1\to -\infty$. In particular, the map $\Forg^{\pm}$ is a homeomorphism as well.  It follows by Proposition \ref{prop:localisom} that the induced map $\Forg:\pi_\sim(\Stab_H^{*>d}(X))\coprod\pi_\sim(\Stab_H^{*>d}(X))[1]\to \pm\BBB_n^{>d}$ is also a homeomorphism. As $\pi_1(\pm\BBB_n^{>d})=\Z$, the map $\Forg:\coprod_{n\in\Z}\pi_\sim(\Stab_H^{*>d}(X))[n]\to \pm\BBB_n^{>d}$ is a universal cover.

    To sum up, the family $\pi_\sim(\Stab_H^{*>d}(X))$ satisfies all properties for $\sb_H^{*>d}(X)$ as that in the $\sb^d$ Conjecture \ref{conj:sbd}.\\

    `$\impliedby$': We first apply Proposition \ref{prop:lifttostab} to construct a family of stability conditions with central charges in $\UUU_n^{>d}$.  We will use  $\sb^d$ Conjecture \ref{conj:sbd}.(c) to check the assumption on the tangent direction at each point in $\sb^{*>d}_H(X)$.\\
    
    For every $c_1\vd_{\us}+ic_2\vd_{\ut}\in\UUU_n^{>d}$, by rescaling the coefficient, we may assume $c_2=1$. As $\sep(\ell(\us,\ut))$ is a continuous function on $\UUU_n$, all conditions on the parameters are open. In particular, there exists an open neighborhood $W$ of $c_1\vd_{\us}$ in $\lbdd$ such that $h+i\vd_{\ut}\in\UUU_n^{>d}$ for every $h=c'\vd_{\us'}\in W$.

   When $t_n\neq +\infty$, by  $\sb^d$ Conjecture \ref{conj:sbd}.(a), for every $|b|$ sufficiently small, the reduced stability condition $\ts_{\ut}+bh$ has reduced central charge $\vd_{\ut}+bc'\vd_{\us'}=c_b\vd_{\us_b}$ for some $c_b>0$ and $\us_b\in\sbr_n$. The reduced stability condition is given as $\ts_{\us_b}\cdot c'_b\in \sb_H^{*>d}(X)$.  All such $\us_b$ are on the  line $\ell(\ut,\us')$.
    
    By Lemma \ref{lem:Epoly}, for $b<a$ with $|a|$ sufficiently small, we have $\us_a<\us_b<\us_a[1]$. As the line $\ell(\us_a,\us_b)=\ell(\ut,\us')$, we have $\sep(\ell(\us_a,\us_b))>d$. By $\sb^d$ Conjecture \ref{conj:sbd}.(c), we have 
    \begin{align}\label{eq819}
        \ts_{\ut}+ah=\ts_{\us_a}\cdot c'_a\lsm \ts_{\us_b}\cdot c'_b=\ts_{\ut}+bh.
    \end{align}

    When $t_n=+\infty$, the only difference is that when $b<0$, the reduced central charge $\vd_{\ut}+bc'\vd_{\us'}=c_b\vd_{\us_b}$ is with coefficient $c_b<0$, and is in $-\BBB_n^{>d}$. By the universal cover assumption in Conjecture \ref{conj:sbd}.{(a)}, locally, the reduced stability condition with the reduced central charge $\vd_{\ut}+bc'\vd_{\us'}$ is $(\ts_{\us_b}\cdot c_b')[1]$, which is in $\sb_H^*(X)[1]$.

    For $a$ and $b$ with the same sign, the formula \eqref{eq819} holds in exactly the same way. For $b<0<a$, similarly, we have
     \begin{align}\label{eq820}
        \ts_{\ut}+ah=\ts_{\us_a}\cdot c_a'\lsm (\vd_{\us_b}\cdot c'_b)[1]=\ts_{\ut}+bh.
    \end{align}
    
 { (a)} Now by Proposition \ref{prop:lifttostab}, we have $c_1\vd_{\us}\in\ta( \vd_{\ut})$. In other words, the datum $(\cA_{\ut},c_1\vd_{\us}+ic_2 \vd_{\ut})$ is a stability condition on $X$. Denote the set of such stability conditions as  $T(\sb^{*>d}_H(X))$.  Note that this holds for every $c_1\vd_{\us}+ic_2\vd_{\ut}\in \UUU_n^{>d}$. The forgetful map $\Forg:T(\sb^{*>d}_H(X))\to \UUU_n^{>d}$ is set-theoretically one-to-one.
  
   For every open neighborhood $W$ of $c_1\vd_{\us}$, there is an open neighborhood $W'$ of $c_2\vd_{\ut}$ such that $(\cA_{\ut'},f+ig)$ is a stability condition for every $f\in W$, $g=c'\vd_{\ut'}\in W'$. By Proposition \ref{prop:localisom}, the forgetful map is a homeomorphism.

  { (b)}  By  $\sb^d$ Conjecture \ref{conj:sbd}.{(a)} and {(b)}, we may compare the reduced central charge after taking tensor product with $\cO_X(H)$. It follows that the heart 
   \begin{align*}
       \cA_{\ut}\otimes \cO_X(H)=\cA_{\ut+1}.
   \end{align*} 

   Therefore, we have $\sigma\otimes \cO_X(H)=(\cA_{\ut+1},c_1\vd_{\us+1}+ic_2\vd_{\ut+1})$ and it is in $T(\sb_H^{*>d}(X))$.

   To sum up the family $T(\sb_H^{*>d}(X))$ of stability conditions satisfies both properties assumed in the $\Stab^d$ Conjecture \ref{conj:stabd}. The statement holds.
\end{proof}

\begin{Prop}\label{prop:stabdprop}
   Let $(X,H)$ be a smooth polarized variety satisfying $\Stab^d$ Conjecture \ref{conj:stabd}. Then the following properties hold.
   \begin{enumerate}[(1)]
      \item Let $\ts_{\ut}\in \sb_H^{*>d}(X)$ with $t_n\neq +\infty$, then for every $m\in\R_{>0}$, we have $\ts_{\ut}\lsm \ts_{\ut+m}$. In particular, we have
      \begin{align}\label{eq8221}
          \ts_{\ut}\lsm \ts_{\ut}\otimes \cO_X(H).
      \end{align}
      \item Let $m\in\R_{>0}$, assume $\sep(\ut)>m$ and $\sep(\ell(\ut,\ut+m))>d$, (here when $t_n=+\infty$, we set $\ut'=(t_1,t_2,\dots,t_{n-1})\in\sbr_{n-1}$ and assume $\sep(\ell(\ut',\ut'+m))>d$, same convention applies to {(4)}), then 
      \begin{align}\label{eq823}
          \ts_{\ut+m}\lsm \ts_{\ut}[1].
      \end{align}
      In particular, if $\sep(\ut)>\max\{m+d,2d\}$, then \eqref{eq823} holds.
       \item Let $\sigma\in\Stab_H^{*>d}(X)$, then $\sigma\lsmeq\sigma\otimes \cO_X(H)$.
       
       Let $E$ and $F$ be  $\sigma$-stable objects  not with character in the form of $(0,\dots,0,*)$. Assume that   $\phi_\sigma(E)=\phi_\sigma(F)$, then $\Hom(E\otimes \cO(mH),F)=0$ for every $m\in\Z_{\geq1}$.

       \item Let $\sigma=(\cA_{\ut},c_1\vd_{\us}+ic_2\vd_{\ut})\in\Stab_H^{*>d}(X)$ and $m\in\Z_{\geq1}$. Assume $\sep(\ell(\us,\ut))>m$ and $\sep(\ell(\ut',\ut'+m))>d$ for every $\vd_{\ut'}\in\ell(\us,\ut)$, then 
       \begin{align}\label{eq824}
           \sigma\otimes \cO_X(mH)\lsm \sigma[1].
       \end{align}
       In particular, if $\sep(\ell(\us,\ut))>\max\{m+d,2d\}$, then \eqref{eq824} holds.
   \end{enumerate}
\end{Prop}
\begin{proof}
    {(1)} By Lemma \ref{lem:E4}.{(1)}, there exist $\delta_0>0$ sufficiently small such that for every $0<\delta<\delta_0$ we have $\ut+\delta<\ut[1]$ and $\sep(\ell(\ut,\ut+\delta))>d$. By $\sb^d$ Conjecture \ref{conj:sbd}.(c), we have
    \begin{align*}
        \ts_{\ut}\lsm\ts_{\ut+\delta}.
    \end{align*}
    We may choose $\delta$ so that $m=M\delta$ for some $M\in\Z_{\geq1}$. 

    Note that for every $a\in\R$, it is clear that $\sep(\ell(\ut+a,\ut+a+\delta))=\sep(\ell(\ut,\ut+\delta))>d$. By  $\sb^d$ Conjecture \ref{conj:sbd}.(c), we have
    \begin{align*}
        \ts_{\ut+k\delta}\lsm \ts_{\ut+(k+1)\delta}\text{ for every }k\in\Z.
    \end{align*}
    It follows that $\ts_{\ut}\lsm\ts_{\ut+\delta}\lsm\ts_{\ut+2\delta}\lsm \dots\lsm\ts_{\ut+M\delta}=\ts_{\ut+m}$.

    By $\sb^d$ Conjecture \ref{conj:sbd}.(b), comparing the reduced central charge, we must have $\ts\otimes \cO_X(H)=\ts_{\ut+1}$. The relation \eqref{eq8221} holds.\\

   \noindent {(2)} When $t_n\neq+\infty$, by the assumption that $\sep(\ut)>m$, we have $\ut<\ut+m<\ut[1]$. By  $\sb^d$ Conjecture \ref{conj:sbd}.(c), the relation \eqref{eq823} holds. \\
 
When $t_n=+\infty$, then we may choose $\vd_{\us'}\in\ell(\ut',\ut'+m)$ with $\ut'+m<\us'$ in $\sbr_{n-1}$. By Lemma \ref{lem:E6}, there exists $N$ sufficiently large such that if we let $\us=(-N,s'_1,\dots,s'_{n-1})$ then
\begin{align*}
  \us<\ut<\ut+m<\us[1]\text{ and }  \sep(\ell(\ut,\us)),\sep(\ell(\ut+m,\us))>d.
\end{align*}

By $\sb^d$ Conjecture \ref{conj:sbd}.(c), we have
    \begin{align*}
        \ts_{\ut+m}\lsm\ts_{\us}[1]\lsm\ts_{\ut}[1].
    \end{align*}
    So the relation \eqref{eq823} holds.\\
    
    For the second part of statement, let $m'=\max\{m,d\}$, then by Lemma \ref{lem:E5}, 
    \begin{align}\label{eq8301}
        \sep(\ell(\ut,\ut+m))>\min\{m',\sep(f)-m'\}\geq d.
    \end{align}
    By statement {(1)} and the first part of the statement, we have
    \begin{align*}
        \ts_{\ut+m}\lsmeq\ts_{\ut+m'}\lsm \ts_{\ut}[1].
    \end{align*}
    Note that here we use the adhoc notion $\ts_{\ut+m}\lsmeq\ts_{\ut+m'}$ to mean $\cA_{\ut+m}\subset \cA_{\ut+m'}[\leq 0]$. When $t_n=+\infty$, the formula \eqref{eq8301} still holds by replacing $t$ with $t'$.\\

  \noindent  {(3)} Note that there is exactly one $\theta\in(0,1]$ such that $\pi_\sim(\sigma[\theta])=c\ts_{\ut}[0\text{ or }1]$ for some $\ut$ with $t_n=+\infty$. By \eqref{eq8221},  for all but only one $\theta\in(0,1]$, we have 
    \begin{align*}
        \pi_\sim(\sigma[\theta])\lsm \pi_{\sim}(\sigma[\theta])\otimes \cO_X(H)= \pi_{\sim}((\sigma\otimes \cO_X(H))[\theta]).
    \end{align*}
       By Lemma \ref{lem:lsmeqlem}, we have $\sigma\lsmeq \sigma\otimes \cO_X(H)$.\\

\noindent       For the second part of the statement, by taking $\sigma[\theta]$ instead if necessary, we may assume $\phi_\sigma(E)=\phi_\sigma(F)=1$. Assume that $\Im Z_\sigma= c \vd_{\ut}$ for some scalar $c\in\R$, if $t_n=+\infty$, then since the characters of $E$ and $F$ are not in the form of $(0,\dots,0,*)$, we may deforming $\ut$ to some $\ut'$ so that $t'_n\neq +\infty$ and $\vd_{\ut'}(E)=\vd_{\ut'}(F)=0$. In particular, both $E$ and $F$ are in $\cP_{\ut'}(1)$. By statement {(1)}, we have $F\in\cP_{\ut'+m}(<1)$. As $E\otimes \cO_X(mH)\in\cP_{\ut'+m}(1)$, we have $\Hom(E\otimes \cO_X(H),F)=0$.\\

  \noindent     {(4)} For every $\theta\in(0,1]$, $\pi_\sim(\sigma[\theta])=c\ts_{\ut'}[0\text{ or }1]$ for some $\vd_{\ut'}\in\ell(\ut,\us)$. By the assumption, we have $\sep(\ut')>m$ and $\sep(\ell(\ut',\ut'+m))>d$. By statement {(2)}, we have $\ts_{\ut'}\otimes \cO_X(mH)=\ts_{\ut'+m}\lsm \ts_{\ut'}[1]$. It follows that $\pi_\sim((\sigma\otimes\cO_X(H))[\theta])=\pi_\sim(\sigma[\theta])\otimes \cO_X(H)\lsm \pi_\sim(\sigma[\theta])[1]$ for every $\theta\in(0,1]$. By Lemma \ref{lem:eqdefforlsmonstab} the relation \eqref{eq824} holds.\\

 Assume that $\sep(\ell(\us,\ut))>\max\{m+d,2d\}$,  then by definition  $\sep(\ut')\geq \sep(\ell(\us,\ut))>\max\{m+d,2d\}$. It follows from the second part of statement {(2)} that $\ts_{\ut'+m}\lsm \ts_{\ut'}[1]$. So the relation \eqref{eq824} holds. 
\end{proof}

\begin{Cor}[Stability of skyscraper sheaves]\label{cor:stabdpt1}
     Let $d\geq0$ and  $(X,H)$ be with $mH$ very ample. Assume  $\Stab^d$ Conjecture  for $(X,H)$, then all skyscraper sheaves are $\sigma$-stable  for every $\sigma\in\Stab^{*>\max\{m+d,2d\}}_H(X)$.
\end{Cor}
\begin{proof}
    The statement follows from Corollary \ref{cor:niceimplygeom} and Proposition \ref{prop:stabdprop}.
\end{proof}

\begin{Cor}[Uniqueness of $\Stab^*$]\label{cor:stabdunique}
     Assume  $\Stab^d$ Conjecture \ref{conj:stabd}  for $(X,H)$, then the family of stability conditions $\Stab_H^{*>d}(X)$ (and $\sb_H^{*>d}(X)$) as that described in the conjecture is unique up to a homological shift.
\end{Cor}
\begin{proof}
    Let $m\in\Z$ be sufficiently large so that $mH$ is very ample. By Lemma \ref{lem:Epoly7}, let $\us<\ut<\us[1]$ with $\ell(\us,\ut)>\{m+d,2d\}$. By Corollary \ref{cor:stabdpt1} and taking a homological shift, we may assume that the phase of all skyscraper points is in $[0,1)$. Then by Proposition \ref{prop:stabdprop} and Corollary \ref{cor:geodeterbyZ}, such a stability condition does not rely on the choice of the family $\Stab_H^{*>d}(X)$. Note that $\UUU^{>d}_n$ is connected, so the whole family $\Stab_H^{*>d}(X)$ is unique. By the construction of $T(\sb_H^{*>d}(X))$ as that in the proof of Theorem \ref{thm:equivalentconjs}, the space $\sb_H^{*>d}(X)$ must also be unique.
\end{proof}

\begin{Cor}[Restriction of stability conditions]\label{cor:restrictionall}
     Assume $\Stab^d$ Conjecture \ref{conj:stabd} for $(X,H)$. Let $Y$ be an irreducible smooth subvariety of $X$. Then there exists $M$ such that every stability condition $\sigma\in \Stab_H^{*>M}(X)$ restricts to $\Db(Y)$. In particular, the space $\Stab_{H|_Y}(Y)\neq \emptyset$.
\end{Cor}
\begin{proof}
   We may choose a sequence of irreducible smooth varieties
   \begin{align*}
       Y=Y_0\subset Y_{1}\subset \dots\subset Y_s= X
   \end{align*}
   such that each $Y_{i-1}$ is in $|D_i|$ for some divisor $D_i$ in $Y_i$. Denote by $H_i$ the restricted divisor $H|_{Y_i}$. Let $m_i\coloneq m(D_i)$ be as that in Proposition \ref{thm:bayerall}. 

   By Proposition \ref{thm:bayerall} and Lemma \ref{lem:eqdefforlsmonstab}, for every geometric stability condition $\sigma_i$ on $Y_i$ with $\sigma_i\lsmeq\sigma_i\otimes \cO_{Y_i}(H_i)$, we have $\sigma_i\otimes  \cO_X(D)\lsmeq \sigma_i \otimes \cO_X(m_iH_i)$.

  Assume that $aH$ is very ample for some $a\in \Z_{\geq1}$. We may let $M=\max\{2d,a,m_i+d:1\leq i\leq s\}$. Then for every $\sigma\in \Stab_H^{*>M}(X)$, by Proposition \ref{prop:stabdprop}, we have 
  \begin{align*}
      \sigma\otimes \cO_X(D_s)\lsmeq\sigma\otimes \cO_X(m_sH)\lsm \sigma[1].
  \end{align*}
  By Corollary \ref{cor:niceimplygeom}, the stability condition $\sigma$ is geometric. By Proposition \ref{prop:reststab}, the stability condition $\sigma$ restricts to $\sigma_{s-1}=\sigma|_{\Db(Y_{s-1})}$ on $Y_{s-1}$ and inherits the following properties
  \begin{align*}
      \sigma_{s-1}\lsmeq \sigma_{s-1}\otimes \cO_{Y_{s-1}}(H_{s-1})\text{ and }\sigma_{s-1}\otimes \cO_{Y_{s-1}}(MH_{s-1})\lsm \sigma_{s-1}[1].
  \end{align*}
Moreover, all skyscraper sheaves are $\sigma_{s-1}$-stable with the same phase.  By Proposition \ref{thm:bayerall}, Lemma \ref{lem:eqdefforlsmonstab} and the choice of $M$, we have 
\begin{align*}
      \sigma_{s-1}\otimes \cO_{Y_{s-1}}(D_{s-1})\lsmeq\sigma_{s-1}\otimes \cO_{Y_{s-1}}(m_{s-1}H_{s-1})\lsm \sigma_{s-1}[1].
  \end{align*}
   By Proposition \ref{prop:reststab}, the stability condition $\sigma_{s-1}$ restricts to $Y_{s-2}$ and inherits the corresponding properties. Repeat this procedure, the stability condition $\sigma$ restricts to $Y$.
\end{proof}

\def\DF{\ensuremath{\Xi}}

\begin{Rem}[Central charge of restricted stability conditions]\label{rem:restcentralcharge}
For every $m>0$ and $\ut\in\sbr_n$ with $\sep(\ut)>m$, we denote by $\DF_m(\ut)$ the roots of \[\prod_{i}(x-t_i)-\prod_{i}(x-t_i-m)=0\] (drop the terms $x-t_n$ and $x-t_n-m$ when $t_n=+\infty$).

By Lemma \ref{lem:rootsandgap},  as $\sep(\ut)>m$, we have $\DF(\ut)\in\sbr_{n-1}$ with $\sep(\DF_m(\ut))>m$ (when $t_n=+\infty$, we set $\DF_m(\ut)_{n-1}=+\infty$).

For a collection of numbers $\um=(m_1,\dots,m_d)$ with $m=\max\{m_i\}$ and $\ut\in\sbr_n$ with $\sep(\ut)>m$, we define 
$\DF_{\um}(\ut)\coloneq (\DF_{m_1}\circ\DF_{m_2}\circ\dots\circ \DF_{m_d})(\ut).$
Note that $\DF_{m_i}(\DF_{m_j}(\ut))=\DF_{m_j}(\DF_{m_i}(\ut))$ for every $i$ and $j$, the definition of $\DF_{\um}$ does not rely on the order of $m_i$.\\

\noindent   When $Y$ is a smooth complete intersection in $X$, we can give a more accurate description for the restricted stability conditions as that in the threefold case Example \ref{cor:rest3to2}.
    
    Let  $Y\in|mH|$ be a smooth projective subvariety in $X$ with $\iota:Y\to X$ the inclusion map. Let $\sigma_{\us,\ut}$ be a stability condition on $\Db(X)$ as that in the Stab$^{d}$ Conjecture \ref{conj:stabd}. Assume $\sigma_{\us,\ut}\otimes \cO_X(mH)\lsm \sigma_{\us,\ut}[1]$, then by Proposition \ref{prop:reststab}, the restricted central charge on $\Db(Y)$ is given as $Z_{\us,\ut}\circ\iota_*$. In particular, it factors via the lattice: $\lambda_h:\Kn(Y)\to \Lambda_h:(h^{n-1}\rk,h^{n-2}\ch_1,\dots,\ch_{n-1})$, where $h=H|_{Y}$. The central charge is determined by its image of $\gamma_{n-1}(-)$, which can be computed as
    \begin{align*}
        Z_{\us,\ut}(\iota_*[\gamma_{n-1}(x)])&= Z_{\us,\ut}(\gamma_{n}(x))-Z_{\us,\ut}(\gamma_{n}(x-m))=\left(Z_{\us,\ut}-Z_{\us+m,\ut+m}\right)(\gamma_n(x)).
    \end{align*}
    In particular, $\Re Z(\iota_*[\gamma_{n-1}(x)])=0$ (resp. $\Im Z=0$) if and only if $x\in\DF_m(\us)$ (resp. $\DF_m(\ut)$). So up to scalars on the real and imaginary part, the restricted central charge is:
    \begin{align*}
     Z_{\us,\ut}\circ\iota_*=c_1\vd_{\DF_m(\us)}+ic_2\vd_{\DF_m(\ut)}.
    \end{align*}  
   When the dimension $n$ of $X$ is less than or equal to $4$, the map $\DF_m:\sbr_n^{>m}\to \sbr_{n-1}^{>m}$ is surjective. So assuming the $\Stab^0$ Conjecture \ref{conj:stabd} for $(X,H)$, we can get the whole family of reduced stability conditions with central charges in $\BBB^{*>m}_{n-1}$. By Theorem \ref{thm:equivalentconjs}, the $\Stab^{>m}$ Conjecture \ref{conj:stabd} holds for $(Y,H|_Y)$.
   
   However, when the dimension $n$ of $X$ is greater than or equal to $5$, the map $\DF_m:\sbr_n\to\sbr_{n-1}$ is not surjective anymore. We cannot get the whole family $\Stab^{*>m}_{H|_{Y}}(Y)$ by restriction. 

   For the complete intersection $Y=Y_1\cap\dots\cap Y_d$, the restricted central charge of $\sigma_{\us,\ut}$ is given as $\vd_{\DF_{\um}(\us)}+i\vd_{\DF_{\um}(\ut)}$ up to scalars on the real and imaginary parts.
\end{Rem}

\begin{Rem}\label{rem:abnrest}
    To restrict stability conditions from higher dimensional varieties, instead of $\bP^n$, one may also consider $X=E\times E\times \dots\times E$ for an elliptic curve $E$. The category $\Db(X)$ admits stability conditions  by \cite{yucheng:stabprodvar}.  It is worthwhile to study whether the $\sigma\otimes\cO(D)\lsm \sigma[1]$ assumption can be proved inductively in these cases. If so, one can then pull back the stability condition by  \'etale covers to get stability conditions on $X$ satisfying $\sigma\otimes \cO(mD)\lsm \sigma[1]$ for $m$ arbitrarily large. This will gives the existence of stability conditions on all smooth subvarieties in $X$.
\end{Rem}

\subsection{Bounds for the numerical characters}\label{sec:sub:bounds}
In this section, we describe the bounds for the numerical characters of stable objects assuming the Stab$^0$ Conjecture \ref{conj:stabd}.

\begin{Prop}\label{prop:boundsforsbd}
     Let $(X,H)$ be smooth polarized variety satisfying $\Stab^0$ Conjecture \ref{conj:stabd} and $\sb^0$ Conjecture \ref{conj:sbd}. Let $E\in\Db(X)$ be a $\ts_{\ut}$-semistable object, then 
     \begin{align}\label{eq837}
         \lambda_H(E)= \sum_{i=1}^n (-1)^ia_i\gamma_n(t_i)
     \end{align}
     with  coefficients $a_i\geq 0$ for all $i$ or $a_i\leq 0$ for all $i$.
\end{Prop}
\begin{proof}
    By definition, $\vd_{\ut}(\lambda_H(E))=0$, in other words, $\lambda_H(E)\in\Ker(\vd_{\ut})$.
    
    By Lemma \ref{lem:nested} and the construction of Theorem \ref{thm:equivalentconjs}, the object $E$ is $\sigma_{\us,\ut}$-semistable for all $\us<\ut<\us[1]$. In particular, the $H$-polarized character 
    \begin{align*}
        \lambda_H(E)\not\in \Ker Z_{\us,\ut}= \Ker \vd_{\us}\cap\Ker \vd_{\ut}\implies \lambda_H(E)\not\in\Ker\vd_{\us}
    \end{align*}
    for all $\us<\ut<\us[1]$. By Lemma \ref{lem:Epoly2}, the statement holds.
\end{proof}

\begin{Rem}\label{rem:boundstab}
    Assume that $E$ is $\sigma_{\us,\ut}$-semistable, then $\vd_{\ut}(E)\vd_{\us}-\vd_{\us}(E)\vd_{\ut}=c\vd_{\um}$ for some $\um\in\sbr_n$ and $c\in\R$. The $H$-polarized character $\lambda_H(E)= \sum_{i=1}^n (-1)^ia_i\gamma_n(m_i)$
     with  coefficients $a_i\geq 0$ for all $i$ or $a_i\leq 0$ for all $i$.

     For each $\ut\in\sbr_n$ and $1\leq i\leq n$, we denote $\hut_i\coloneq (t_1,\dots,t_{i-1},t_{i+1},\dots,t_n)\in\sbr_{n-1}$. Denote $\lambda_{H,n-1}(E)=(H^n\rk(E),\dots,H\ch_{n-1}(E))$ the truncated $H$-polarized character. Then the condition \eqref{eq837} can be equivalently described as $\vd_{\hut_i}(\lambda_{H,n-1}(E))\vd_{\hut_j}(\lambda_{H,n-1}(E))\geq0$ for all $i,j$.
     
     For example, when $n=2$, (ignoring the $H$ to reduce heavy notions) given that 
     \begin{align*}
         0=\vd_{t_1,t_2}(E)=2\ch_2(E)-(t_1+t_2)\ch_1(E)+t_1t_2\rk(E),
     \end{align*} then \eqref{eq837} says that
     \begin{align*}
         (\ch_1(E))^2-(t_1+t_2)\ch_1(E)\rk(E)+t_1t_2(\rk(E))^2=\vd_{t_1}(\ch_{\leq 1}(E))\vd_{t_2}(\ch_{\leq 1}(E))\geq 0.
     \end{align*} Combine them together, this is just always equivalent to the polarized Bogomolov inequality: $\Delta_H(E)=(H\ch_1(E))^2-2H^2\rk(E)\ch_2(E)\geq 0$.

     When $n\geq 3$,  the bound  cannot be summarized as a single quadratic form, but needs a family of quadratic forms as that in \cite[Theorem 8.7]{BMS:stabCY3s}, see also \eqref{eq117}. We may generalize this to higher dimensions as well. One application is that skyscraper sheaves are stable on $\Stab^*_H(X)$. 
\end{Rem}

\begin{Prop}   \label{prop:quadonstab0}
 Let $(X,H)$ be smooth polarized variety satisfying $\Stab^0$ Conjecture \ref{conj:stabd}. Then for every $\sigma_{\us,\ut}\in\Stab_H^*(X)$, there exists a (family) of  quadratic form(s) $\tilde\vq_{\us,\ut}$ on $\Lambda_\R$ giving the support property for $\sigma_{\us,\ut}$ such that
\begin{align}\label{eq842}
   & \tilde\vq_{\us,\ut}(\gamma_n(x),\gamma_n(x))=0, \;\forall\;x\in\R\cup\{+\infty\}.
\end{align}
\end{Prop}
\begin{proof}
    Let $\tilde\vq_{\us,\ut}$ be $\tilde \vq_{\ell(\us,\ut)}$ as that in Proposition \ref{prop:Equadratic}. Then by Proposition \ref{prop:Equadratic}.{(a)}, the formula \eqref{eq842} holds. By Proposition \ref{prop:Equadratic}.{(c)}, the quadratic form $\tilde\vq_{\us,\ut}$ is negatively definite on $\Ker \ell(\us,\ut)=\Ker \vd_{\us}\cap \Ker \vd_{\ut}=\Ker Z_{\us,\ut}$.
    
    By Proposition \ref{prop:boundsforsbd}, for every $\sigma_{\us,\ut}$-semistable object $E$, we have $\lambda_H(E)\in\sc(\ell(\us,\ut))$ as that in \eqref{eq:defSCl}. By Proposition \ref{prop:Equadratic}.{(b)}, we have $\tilde\vq_{\us,\ut}(E)\geq0$. The statement holds.
\end{proof}

\begin{Prop}[Stability of points]\label{prop:stabdpt}
     Let $(X,H)$ be an irreducible smooth polarized variety satisfying  $\Stab^0$ Conjecture \ref{conj:sbd}.  Then for every $\sigma\in\Stab^*_H(X)$, an object $E$ with $\lambda_H(E)=(0,0,\dots,0,c)$ is $\sigma$-stable if and only if $E$ is a skyscraper sheaf up to a homological shift.
\end{Prop}
\begin{proof}
Let  $E$ be a $\tau$-stable object with character $\lambda_H(E)\in\{c\gamma_n(t):t\in\R\cup\{+\infty\}\}$ for some $\tau\in\Stab^*_H(X)$. We first show that all $E\otimes \cO_X(mH)$ are stable with respect to evert stability condition in $\Stab^*_H(X)$.

By Proposition \ref{prop:quadonstab0}, for every $\sigma\in \Stab^*_H(X)$, there exists a quadratic form $\tilde\vq_\sigma$ with $\tilde\vq_\sigma(\lambda_H(E))=0$ offering the support property for $\sigma$. By \cite[Proposition A.8]{BMS:stabCY3s}, if $E$ is stable with respect to one stability condition in $\Stab(\tilde\vq_\sigma,\sigma)$, then it is stable with respect to every stability condition in $\Stab(\tilde\vq_\sigma,\sigma)$. As the path connected set $\Stab_H^*(X)$ can be covered by such $\Stab(\tilde\vq_\sigma,\sigma)$, if $E$ is stable with respect to one stability condition $\sigma$ in $\Stab^*_H(X)$, then it is stable with respect to every stability condition in $\Stab_H^*(X)$. Moreover, note that $E\otimes \cO_X(H)$ is $\sigma\otimes \cO_X(H)$-stable. By the assumption  $\Stab^0$ Conjecture \ref{conj:stabd}.(2), $\sigma\otimes \cO_X(H)\in\Stab_H^*(X)$. So the object $E\otimes \cO_X(H)$ is  stable with respect to every stability condition in $\Stab^*_H(X)$.\\

  `$\impliedby$':  By Corollary \ref{cor:stabdpt1}, all skyscraper sheaves are $\sigma$-stable for some $\sigma\in\Stab^*_H(X)$. So they are stable with respect to every stability condition in $\Stab^*_H(X)$.\\

  `$\implies$': Let $E$ be a $\sigma$-stable object with $\lambda_H(E)=\gamma_n(+\infty)$, then by the observation in the first paragraph, the object  $E\otimes \cO_X(mH)$ is $\sigma$-stable for every $m\in\Z$. Note that $\lambda_H(E\otimes \cO_X(mH))=\lambda_H(E)$, and $\phi^\pm(E\otimes \cO_X(mH))$ is bounded, there there exists $m>0$ with $\phi_\sigma(E)=\phi_\sigma(E\otimes \cO_X(mH))$. It follows that $\phi_\sigma(E)=\phi_\sigma(E\otimes \cO_X(mnH))$ for every $n\in\Z$.  Note that $\Hom(E,E\otimes \cO_X(mnH))\neq0$ when $mnH$ is very ample, we must have $E\cong E\otimes \cO_X(mnH)$, so $E$ is supported on some points. As $E$ is $\sigma$-stable, it is supported on one point $p$, therefore extended by $\cO_p[i]$'s. As $\cO_p[i]$ is $\sigma$-stable, $E$ can only be $\cO_p[i]$ for some $i\in\Z$.
\end{proof}
\begin{Rem}
    The construction of $\tilde \vq_{\ell(\ut,\us)}=\tilde \vq_\ell$ in Proposition \ref{prop:Equadratic} is by induction from $\tilde \vq_{\pi(\ell)}$. 
    \begin{align*}
     \vq_\ell&=\vd_{\ell}\widetilde{\vd_{\pi(\ell)}}-\vd_{\pi(\ell)}\widetilde{\vd_\ell};\\
        \tilde \vq_\ell &= \alpha \vq_\ell +\tilde\vq_{\pi(\ell)}\text{ for some }\alpha\in(0,\alpha(\ell)).
    \end{align*}
   The leading terms in the expressions for $\vd$ are given by
    \begin{align*}
        \vd_\ell=&-H\ch_{n-1}+ a_2 H^2\ch_{n-2} + a_3H^3\ch_{n-3}+\dots\\
        \widetilde{\vd_\ell}=& -n\ch_{n}+  (n-1)a_2 H\ch_{n-1} + (n-2)a_3H^2\ch_{n-2}+\dots\\
        \vd_{\pi(\ell)}=& - H^2\ch_{n-2} + bH^3\ch_{n-3}+\dots\\
        \widetilde{\vd_{\pi(\ell)}}=&-(n-1)H\ch_{n-1}+  b H^2\ch_{n-2} +\dots.
    \end{align*}

One can readily verify that for $n \leq 3$, the expression $\tilde \vq_\ell$ coincides with the classical formulas.
    
When  $n=1$, we have $\vq_\ell =0$. When $n=2$, $\vq_\ell= \Delta_H=(H\ch_1)^2-2H^2\ch_0\ch_2$. 
When $n=3$, the computation yields
\begin{align*}
    \vq_\ell&=2(H\ch_2)^2-3(H^2\ch_1)\ch_3-bH^2\ch_1H\ch_2+3bH^3\ch_0\ch_3+(a_3+ba_2)\Delta_H\\
    &=\tfrac{1}{2}\nabla^b_H+(a_3+ba_2-\tfrac{b^2}{2})\Delta_H. 
\end{align*}

matching the expression appearing in \cite[Conjecture 4.1]{BMS:stabCY3s}. 
\end{Rem}

\appendix
\section{Degenerate Loci}
\label{sec:degenloci}
\subsection{Reduced stability conditions with a given heart}
In general we are interested in under what assumption a reduced stability condition $\ts$ can be determined by the data $(\cA_{\ts},B_{\ts})$, or even just by the heart $\cA_{\ts}$. This could potentially  lead to alternative definitions for reduced stability conditions that are independent of the stability condition.

Unfortunately, we are currently unable to provide satisfactory answers to either question. Regarding the data $(\cA_{\ts},B_{\ts})$, we believe that examples may exist where distinct reduced stability conditions share the same such data. However, we are not yet able to construct any explicit examples.

As for the heart $\cA_{\ts}$, one can roughly distinguish two types (with possible intermediate cases) of stability conditions or hearts of bounded t-structures. The first is the algebraic type, such as those arising from quiver representations. In this case, the image of stable characters under the central charge has a `gap', see \cite{Takeda:fukaya, Dongjian:JH}. Intuitively, as the kernel of the reduced charge deforms within this gap, the heart remains unchanged. This gives rise to a natural wall-and-chamber structure on $\sb(\cT)$ for such algebraic-type hearts. Moreover, as we will see in Corollary \ref{cor:stabdegen}, at an interior point of each chamber, the reduced stability condition $\ts$ is determined by $(\cA_{\ts}, B_{\ts})$.

The second is the geometric type, such as stability conditions $\sigma$ on $\Db(X)$ satisfying $\sigma \otimes \cO_X(H) \lsm \sigma[1]$. In this case, the reduced stability manifold behaves more like a parameter space for the hearts of bounded t-structures. That is, the reduced stability condition $\ts$ is expected to be determined solely by $\cA_{\ts}$. However, we are not yet able to give a rigorous proof of this intuitive expectation.

In this appendix, we focus on the algebraic type setting, establishing some foundational properties that may be useful in future developments.
\begin{Def}\label{def:sbA}
    Let $\cA$ be the  heart of a bounded t-structure of $\cT$, we denote \begin{align*}
    \sb^{}(\cA) & \coloneqq\{\ts\in\sb(\cT)\mid \cA_{\ts}=\cA\};\\
    \sbr(\cA) & \coloneqq \{f\in\lbdd\mid f(E)\geq 0\text{ for all }E\in \cA\}.
\end{align*}
\end{Def}

We denote by $\sbr^\circ(\cA)\coloneq \{f\in \sbr(\cA):\forall g\in \sbr(\cA), \exists \;\epsilon>0 \text{ such that }f-\epsilon g\in \sbr(\cA)\}$, the interior of $\sbr(\cA)$ in $\bP\lbdd$. Let $\sb^\circ(\cA)\coloneq (\Forg|_{\sbr(\cA)})^{-1}(\sbr^\circ(\cA))$ be the subset of reduced stability conditions with heart $\cA$ whose reduced central charge lies in $\sbr^\circ(\cA)$.

The following lemma shows that the reduced stability condition with heart $\cA$ can always deform to $\sb^\circ(\cA)$. 

\begin{Lem}\label{lem:sbrA}
     Let $\sigma=(\cA,f+ig)$ be a stability condition  and $h\in \sbr^\circ(\cA)$. Then $(\cA,f+i((1-t)g+th))$ is a stability condition for every $t\in[0,1]$.
\end{Lem}
\begin{proof}
Let $\{x_i\}$ be a basis for $\lbdd$. Then there exists $K\gg 1$ such that the quadratic form $Q_1=K(f^2+g^2) -\sum x_i^2$ gives the support property for $\sigma$.

By the assumption that $h\in\sbr^\circ(\cA)$, there exist $\epsilon>0$ such that $h-\epsilon g\in\sbr(\cA)$. 
    
    Let $Q=Q_1+ N(h-\epsilon g)g$ for some $N\gg K/\epsilon$. Then for every $\sigma$-semistable object $E\in \cA$, as $h-\epsilon g,g\in\sbr(\cA)$, we have $(h(E)-\epsilon g(E))g(E)\geq 0$. It follows that $Q(E)\geq Q_1(E)\geq 0$. 

   Note that $g(\Ker(f+ig))=0$, it is clear that $Q|_{\Ker(f+ig)}=Q_1|_{\Ker(f+ig)}$ is negative definite. So $Q$ gives the support property for $\sigma$ as well.\\
    
    For every $s\geq 0$, the restricted form $Q|_{\Ker(f+i(sg+h))}=(K-N\epsilon-Ns)g^2-\sum x_i^2$ is negative definite by the choice of $N$.

    By \cite[Proposition A.5]{BMS:stabCY3s}, there is a family of stability conditions $\{(\cA_t,f+i((1-t)g+th))\}_{t\in[0,1]}$. The only remaining task is to show that the heart structures $\cA_t$ are all the same.
    
    For every non-zero object $E\in\cA$, as $g,h\in\sbr(\cA)$, we have $((1-t)g+th)(E)\geq 0$. When $((1-t)g+th)(E)= 0$, since $h-\epsilon g\in \sbr(\cA)$, we must have $g(E)=0$, which implies $f(E)<0$. Therefore, $(f+i((1-t)g+th))(\cA\setminus\{0\})\subset \H$. 
    
    By Lemma \ref{lem:StabQconstantheart}, the heart structures $\cA_t=\cA$ for all $t\in[0,1]$.  So the statement holds.
\end{proof}

\begin{Lem}\label{lem:StabQconstantheart}
   Let $\gamma:[0,1]\to \Stab(\cT)$ be a path. Assume that $Z_{\gamma(t)}(\cA_{\gamma(0)}\setminus\{0\})\subset \H$ for every $t\in[0,1]$, then $\cA_{\gamma(t)}=\cA_{\gamma(0)}$ for every $t\in[0,1]$.
\end{Lem}
\begin{proof}
    By cutting the path into pieces if necessary, we may assume that $d(  {\gamma(t_1)},  {\gamma(t_2)})<\tfrac{1}{4}$.

    Let $E$ be a non-zero object in $\cA_{\gamma(0)}$. Suppose $E\not\in \cA_{\gamma(t)}$, then we have $\phi^-_{  {\gamma(t)}}(E)\leq 0$ or $\phi^+_{  {\gamma(t)}}(E)>1$. 
    
    Assume $\phi^-_{  {\gamma(t)}}(E)\leq 0$, then as $d(  {\gamma(t_1)},  {\gamma(t_2)})<\tfrac{1}{4}$, we have $-\tfrac{1}{4}<\phi^-_{  {\gamma(t)}}(E)\leq 0$. In particular, $Z_{\gamma(t)}(\HN^-_{  {\gamma(t)}}(E))\not\in\H$. 

    By the assumption that $Z_{\gamma(t)}(\cA_{\gamma(0)})\subset \H$, the object $\HN^-_{{\gamma(t)}}(E)\not\in \cA_{{\gamma(0)}}$. Since $d(  {\gamma(t_1)},  {\gamma(t_2)})<\tfrac{1}{4}$ and $-\tfrac{1}{4}<\phi^-_{  {\gamma(t)}}(E)\leq 0$,  we have $\phi^\pm_{\gamma(0)}(\HN^-_{{\gamma(t)}}(E))\in (-\tfrac{1}{2},\tfrac{1}{4})$. It follows that $\phi^-_{\gamma(0)}(\HN^-_{{\gamma(t)}}(E))\in(-\tfrac{1}{2},0)$. In particular, the object $G\coloneq \HN^-_{\gamma(0)}(\HN^-_{{\gamma(t)}}(E))\in \cA_{\gamma(0)}[-1]$ satisfies 
    \begin{align}\label{eq:B22}
        \Hom(\HN^-_{{\gamma(t)}}(E),G)\neq 0.
    \end{align}
    Apply $\Hom(-,G)$ to the distinguished triangle $F\to E\xrightarrow{\ev}\HN^-_{{\gamma(t)}}(E)\to F[1]$, we get the long exact sequence 
    \begin{align*}
       \dots \to \Hom( F[1],G)\to\Hom(\HN^-_{{\gamma(t)}}(E),G)\to \Hom(E,G)\to \dots
    \end{align*}
    Note that $\HN^-_{\gamma(t)}(F)>\HN^-_{{\gamma(t)}}(E)>-\tfrac{1}{4}$ and $d(\gamma(0),\gamma(t))<\tfrac{1}{4}$, we have $\HN^-_{\gamma(0)}(F)>-\tfrac{1}{2}$. It follows that $\Hom(F[1],G)=0$. As $E\in\cA_{\gamma(0)}$, we have $\Hom(E,G)=0$. This leads to the contradiction with \eqref{eq:B22}.

    So we must have $E\in \cA_{\gamma(t)}$. It follows $\cA_{\gamma(0)}\subseteq \cA_{\gamma(t)}$. As both of them are bounded heart structures, we have $\cA_{\gamma(0)}=\cA_{\gamma(t)}$.
\end{proof}

\begin{Lem}\label{lem:smalldeform}
    The space $\ta(\ts)$ is open convex in $\lbdd$ and homeomorphic to the fiber $(\pi_\sim)^{-1}(\ts)$. 
    
    For every $f\in \ta(\ts)$, $c>0$, and $d\in\R$, the element $cf+dB_{\ts}\in \ta(\ts)$. 
    
    For every compact subset $S\subset \ta(\ts)$, there exists an open neighborhood $U\ni \ts$ such that $S\subset \sb(\ttau)$ for every $\ttau\in U$.
\end{Lem}
\begin{proof}
For every $f\in \ta(\ts)$, as $\Forg:\Stab(\cT)\to \Hom(\Lambda,\C)$ is local homeomorphism,  there exists an open path-connected neighborhood $U$ of $f$ such that the stability condition $(\cA_{\ts},f+iB_{\ts}$ deforms with real part of central charge in $U$. By Lemmas \ref{lem:nested} and \ref{lem:sameImimpliesheart}, the heart $\cA_{\ts}$ is unchanged. So $\ta(\ts)$ contains $U$. It is therefore open. By Proposition \ref{prop:convex}, the space $\ta(\ts)$ is convex.

Note that a stability condition is determined by $(\cA,Z)$, so the map $\Forg:\pi^{-1}_\sim(\ts)\to \lbdd$ is injective. It follows that  $\ta(\ts)$ is homeomorphic to $\pi^{-1}(\ts)$.\\
    
     By taking the $\glt$-action,  we have $\sigma_{c,d}=\left(\cA_{\ts},(cf+dB_{\ts})+iB_{\ts}\right)\in \Stab(\cT)$. By definition, $\tilde \sigma_{c,d}=\ts$. So $cf+dB_{\ts}\in \ta(\ts)$.

 For every $f\in S$, as $\Forg:\Stab(\cT)\to \Hom(\Lambda,\C)$ is local homeomorphic, there exists an open neighborhood $U'$ of $(\cA_{\ts},f+iB_{\ts})$ in $\Stab(\cT)$ on which $\Forg$ is homeomorphic. Assume that $U'=V_f\times W_f$ for some open connected subsets $V_f,W_f\subset \lbdd$ under the decomposition $\Hom(\Lambda,\C)=\lbdd\times i\lbdd$. U'y Proposition \ref{prop:localisom}, there exists an open neighborhood $W'_f\ni\ts$ such that $V_f\subset \ta(\ttau)$ for every $\ttau\in W'_f$. As $S$ is assumed to be compact, it can be covered by finitely many such $V_f$'s. Let $U$ be the intersection of $W_f'$'s of such $f$'s, the statement holds.
\end{proof}

\def\SbA{\ensuremath{\mathcal S}}

\begin{Prop}\label{prop:sbrA}
    Let $\cA$ be the  heart of a bounded t-structure on $\cT$. Then $\sbr(\cA)$ is a closed convex cone which does not contain any line in $\lbdd$.  
    
    Let  $\SbA$ be a connected component of $\sb(\cA)$. Then
    \begin{align}\label{eqA44}
        \sbr^\circ(\cA)\subseteq \Forg(\SbA)\subseteq \sbr(\cA).
    \end{align} 
    Moreover, for every $\ts,\ttau\in \SbA$, if $\Forg(\ttau)\in \sbr^\circ(\cA)$, then $\ta(\ts)\subseteq \ta(\ttau)$. The map $\Forg$ is homeomorphic from $(\Forg|_{\SbA})^{-1}(\sbr^\circ(\cA))$ to $\sbr^\circ(\cA)$. 
\end{Prop}
\begin{proof}
For every $f,g\in\sbr(\cA)$ and $a,b\geq0$, it is clear that $(af+bg)(E)\geq 0$ for every object $E\in \cA$. Note that $\sbr(\cA)=\cap_{E\in\cA}\{f:f(E)\geq0\}$ is the union of closed subsets, so $\sbr(\cA)$ is closed. Note that the objects in $\cA$ generate the whole category $\cT$, so $\mathrm{span}_{\R}\{[E]:E\in\cA\}=\Lambda_\R$. Therefore, there is no line in  $\sbr(\cA)$. \\

\noindent    Let $\ts\in\SbA$, then $B_{\ts}(E)\geq 0$ by definition. So $\Forg(\SbA)\subseteq\sbr(\cA)$. 
    
    For every $f\in \ta(\ts)$ and $h\in \sbr^\circ(\cA)$, by Lemma \ref{lem:sbrA}, the reduced stability condition $\pi_\sim(\cA,f+ih)$ is in $\SbA$. The relation \eqref{eqA44} holds. \\

\noindent    For every $g\in\sbr(\cA)$, we denote by $T(g)\coloneq \coprod_{\ts\in\SbA,B_{\ts}=g}\ta(\ts) $ the subset in $\lbdd$. By Proposition \ref{prop:convex} and Lemma \ref{lem:smalldeform}, this is indeed a disjoint union of open convex subsets. For every $h\in \sbr^\circ(\cA)$, by Lemma \ref{lem:sbrA}, we have $T(g)\subseteq T(h)$.
    
    Note that $\pi_\sim^{-1}(\SbA)$ is path-connected and homeomorphic to $\cup_{g\in \sbr(\cA)}T(g)$, so there is exactly one connected component of $T(h)$ for every $h\in\sbr^\circ(\cA)$. In other words, there is exactly one $\ttau\in\SbA$ with $B_{\ttau}=h$. 
    
    So for every $\ts,\ttau\in\SbA$ with $B_{\ttau}\in\sbr^\circ(\cA)$, we have $\ta(\ts)\subseteq T(B_{\ts})\subseteq T(B_{\ttau})=\ta(\ttau)$. The map $\Forg$ is homeomorphic from $(\Forg|_{\SbA})^{-1}(\sbr^\circ(\cA))$ to $\sbr^\circ(\cA)$.
\end{proof}

\subsection{Degenerate reduced stability conditions}
\begin{Prop}\label{prop:degeneratesb}
    Let $\ts\in\sb(\cT)$, then the following statements are equivalent:
    \begin{enumerate}[(1)]
        \item $0\in \ta(\ts)$;
        \item $\ta(\ts)=\lbdd$;
        \item There exists a quadratic form $Q$ on $\Lambda_\R$ that is negative definite on $\Ker B_{\ts}$ and $Q(E)\geq 0$ for every object $E\in \cA$;
        \item There exists an open neighborhood $U$ of $\ts$ in $\sb(\cT)$ such that for every $\ttau\in U$, $\cA_{\ttau}=\cA$.
        \item There exists an open subset $U\subset \lbdd$ such that $B_{\ts}\in U\subset \sbr(\cA)$.
    \end{enumerate}
\end{Prop}
\begin{proof}
    {(1)}$\implies${(3)}: Let $Q$ be a quadratic form for the support property of $\sigma=(\cA,iZ_I)$. Note that  all non-zero objects $ E\in \cA$ are with phase $\tfrac{1}{2}$, so they are all $\sigma$-semistable. It follows that $Q(E)\geq 0$ for every object $E\in \cA$.

    {(3)}$\implies${(2)}: Note that for every central charge $Z=Z_R+iZ_I$, the quadratic form $Q$ is negative on $\Ker Z$. By \cite[Proposition A.5]{BMS:stabCY3s},  we have deformed stability conditions $(*,Z_R+iZ_I)$ for all $Z_R\in\lbdd$. By Lemma \ref{lem:nested} and \ref{lem:d<1}, the heart structure is constantly $\cA$.

    {(2)}$\implies${(1)}: This is obvious by letting $Z_R=0$.\\

    {(3)}$\implies${(4)}: There is an open neighborhood $U'$ of $Z_I$ in $\lbdd$ such that for every $Z\in U'$, $Q|_{\Ker Z}$ is negative definite. By \cite[Proposition A.5]{BMS:stabCY3s} and Lemma \ref{lem:StabQconstantheart}, the datum $(\cA,Z)$ is a stability condition for every $Z\in U'$. By Proposition \ref{prop:localisom}, the statement holds.

    {(4)}$\implies${(3)}: By Proposition \ref{prop:localisom}, there is an open neighborhood $U'$ of $Z_I$ in $\lbdd$ such that for every $f\in U'$ and $0\neq E\in\cA$, we have $f(E)\geq 0$. There exists a quadratic form $Q^*$ with signature $(1,\rho-1)$ on $\lbdd$ such that
    \begin{itemize}
        \item $Q^*(Z_I)>0$ and
        \item $Q^*(g)<0$ for every $g\not\in \Cone(U')\coloneq  \{f': f'=cf$, $f\in U', c\in \R\}$.
    \end{itemize}
    Then the dual form $Q$ of $Q^*$ on $\Lambda_\R$ is negative definite on $\Ker Z_I$. For every $0\neq E\in \cA$, as $\lambda(E)^*\cap\Cone(U')=\{0\}$, the form $Q^*$ is negative definite on $\lambda(E)^*$. It follows that $Q(E)>0$. The statement holds.

    {(4)}$\iff${(5)} follows from Proposition \ref{prop:sbrA}.
\end{proof}

\begin{Def}\label{def:degeneratesb}
We call a reduced stability condition \emph{degenerate} if it is of the form as that in Proposition \ref{prop:degeneratesb}. Denote the subset of all degenerate reduced stability conditions as $\sb^{\text{degen}}(\cT)$. We call a bounded heart structure $\cA$ of $\cT$ \emph{degenerate} if it is the heart structure of some degenerate reduced stability conditions.
 
 We call a reduced stability condition \emph{non-degenerate} if it is not in the closure of $\sb^{\text{degen}}(\cT)$. 
\end{Def}

\begin{Cor}\label{cor:stabdegen} 
Let $\cA$ be a degenerate heart structure.  Then the map $\Forg:\sb^\circ(\cA)\to \sbr^\circ(\cA)$ is a homeomorphism, with $\sbr^\circ(\cA)$ being an open convex cone in \lbdd. 

The space 
    \begin{align}\label{eq2171}
        \sb^{\text{degen}}(\cT)=\coprod_{\cA\text{ degenerate}}\sb^\circ(\cA).
    \end{align} 
\end{Cor}
\begin{proof}
By Propositions \ref{prop:degeneratesb}.{(2)} and \ref{prop:sbrA}, the space $\sb(\cA)$ is path-connected. By Proposition \ref{prop:sbrA} again, the map $\Forg$ restricted on the $\sb^\circ(\cA)$ is a homeomorphism. By Proposition \ref{prop:degeneratesb}.{(5)}, the subset $\sbr^\circ(\cA)$ is open in \lbdd.

By Proposition \ref{prop:degeneratesb}.{(5)}, a reduced stability condition $\ts\in\sb(\cA)$ is degenerate if and only if $\ts\in\sb^\circ(\cA)$. The formula \eqref{eq2171} then follows from  Proposition \ref{prop:degeneratesb}.{(4)}.
\end{proof}

\begin{Lem}\label{lem:degenerateno1phase}
    Let $\sigma$ be a stability condition such that $\cP_\sigma((-\theta,\theta))=\emptyset$ for some $\theta>0$. Then $\pi_\sim(\sigma)$ is degenerate.
\end{Lem}
\begin{proof}
Denote the central charge of $\sigma$ as $g+ih$. Then by the assumption, there exists $\theta_1>0$ such that every $\sigma$-semistable object $E$ satisfies the inequality \begin{align*}
    Q_1(E)\coloneq h^2(E)-\theta_1g^2(E)\geq 0.
\end{align*}
By choosing a suitable basis $\{g,h,f_1,\dots,f_{\rho-2}\}$ for $\lbdd$ and $K>0$ sufficiently large, we may assume the quadratic form $Q_2=K(g^2+h^2)-\sum f_i^2$ satisfies the support property with respect to $\sigma$.

Let $N>K/\theta_1$ be sufficiently large and consider $Q\coloneq Q_2 + N Q_1$. Then $Q$ is with negative definite on $\Ker h=\Ker Z_\sigma$ and $Q(E)\geq 0$ for every $\sigma$-semistable object $E$. By Proposition \ref{prop:degeneratesb}.{(3)}, the reduced stability condition $\pi_\sim(\sigma)$ is degenerate.
\end{proof}  

\subsection{Example: Space of reduced stability conditions on \texorpdfstring{$\bP^1$}{Lg}}\label{sec:stabp1}
 As that studied in \cite{Okada:P1}, the heart structure of a stability condition on $\Db(\bP^1)$ is either $\Coh(\bP^1)$ or  one of the following forms up to a homological shift:
\begin{align*}
    \cA_{m,k}\coloneq \langle \cO(m)[k],\cO(m+1)\rangle,
\end{align*}
where $m,k\in \Z$ with $k\geq1$. 

 It follows that the space 
\begin{align*}
    \sb(\bP^1)=\left(\coprod_{m,n\in\Z,k\in\Z_{\geq1}}\sb(\cA_{m,k})[n]\right)\coprod\left(\coprod_{n\in\Z}\sb(\Coh(\bP^1))[n]\right).
\end{align*}
Each space $\sb(\cA_{m,k})$ is mapped   homeomorphically onto $\sbr(\cA_{m,k})\subset \lbdd\cong \R^2$ by the forgetful map. More precisely, we have
\begin{align}\label{eq566}
    \sb(\cA_{m,k})=\left\{\ts^{m,k}_{c_1,c_2}=\left(\cA_{m,k},c_1(-1)^k((m+1)\rk-\deg)+c_2(\deg-m\rk\right):c_i{\geq0}\right\}.
\end{align}
The heart $\Coh(\bP^1)$ only admits the rank function up to a positive scalar as its reduced central charge.\begin{align*}
     \sb(\Coh(\bP^1))=\{(\Coh(\bP^1),c\rk):c>0\}. 
\end{align*}

To compare with the space of reduced stability conditions as that in Example \ref{ex:sbcurve}, for every $m\in\Z$ and $t\in[m,m+1)$, the heart $\cA_t=\cA_{m,1}$. More precisely, we have following decomposition for $\sb^*(\bP^1)$:
\begin{align}\label{eq577}
    \sb^*(\bP^1)=\left(\coprod_{m\in\Z}\sb^{c_1>0}(\cA_{m,1})\right)\coprod\sb(\Coh(\bP^1)[1])
\end{align}
Here we denote $\sb^{c_i>0}(\cA_{m,k})$ for the reduced stability conditions with $c_i>0$ as that in \eqref{eq566}.
\begin{Rem}[$\sb(\bP^1)$ is Non-Hausdorff]\label{rem:sbp1nonhausd}
 The non-Hausdorff locus of $\sb(\bP^1)$ is contained in the boundaries of $\sb(\cA_{m,k})$. Each $\sb(\cA_{m,k})$ has three types of boundary points, namely, $\{\ts^{m,k}_{c_1,0}:c_1\neq0\}$, $\{\ts^{m,k}_{0,c_2}:c_2\neq0\}$, and $\{\ts^{m,k}_{0,0}\}$.

For the first type, each $\ts^{m,k}_{c_1,0}$ has a small open neighborhood $U_{m,k,\mathrm{I}}$ where $\cO(m)[k]$ is in the heart and either $\cO(m+1)$ or $\cO(m+1)[-1]$ is in the heart. The heart structure is therefore uniquely determined as $\cA_{m,k}$ or $\cA_{m,k+1}[-1]$. It follows that the open neighborhood $U_{m,k,\mathrm{I}}\subset \sb(\cA_{m,k})\coprod\sb(\cA_{m,k+1})[-1]$. In particular, it cannot be separated from $\ts^{m,k+1}_{c_1,0}[-1]$  in $\sb(\cA_{m,k+1}[-1])$.

 For the second type,  each $\ts^{m,k}_{0,c_2}$ has a small open neighborhood $U_{m,k,\mathrm{II}}$ where $\cO(m+1)$ is in the heart and either $\cO(m)[k]$ or $\cO(m)[k-1]$ is in the heart. 
 When $k\geq 2$, the open neighborhood $U_{m,k,\mathrm{II}}\subset \sb(\cA_{m,k})\coprod\sb(\cA_{m,k-1})$. In particular,  it cannot be separated from $\ts^{m,k-1}_{0,c_2}$.
 
 When $k=1$, the reduced stability condition $\ts^{m,k}_{0,c_2}\in\sb^*(\bP^1)$. 

 For the third type, let $U_{m,k,0}$ be the open neighborhood of $\ts^{m,k}_{0,0}$ that is mapped homeomorphically onto $\lbdd$. Then in $U_{m,k,0}$, either $\cO(m+1)$ or $\cO(m+1)[-1]$ is in the heart and either $\cO(m)[k]$ or $\cO(m)[k-1]$ is in the heart. 
 When $k\geq 2$, the open neighborhood 
 \begin{align*}
     U_{m,k,0}= \sb(\cA_{m,k})\coprod\sb^\circ(\cA_{m,k})[-1]\coprod\sb^{c_1>0}(\cA_{m,k-1}) \coprod\sb^{c_2>0}(\cA_{m,k+1})[-1].
 \end{align*}
In particular, it cannot be separated from $\ts_{0,0}^{m,k}[-1]$, $\ts_{0,0}^{m,k-1}$, and $\ts_{0,0}^{m,k+1}[-1]$.

 When 
 $k=1$, if both $\cO(m)$ and $\cO(m+1)$ are in the heart, then the reduced stability condition is in $\sb^*(\bP^1)$ or $\sb^*(\bP^1)[-1]$. Together with  \eqref{eq577}, we have
  \begin{align*}
     U_{m,1,0}=& \sb(\cA_{m,1})\coprod\sb^\circ(\cA_{m,1})[-1]\coprod\sb^{c_2>0}(\cA_{m,2})[-1] \coprod \\
    & \left(\coprod_{ n\in\Z_{\leq m-1}}\sb^{c_1>0}(\cA_{n,1})\right)\coprod \sb(\Coh(\bP^1))\coprod\left(\coprod_{ n\in\Z_{\geq m+1}}\sb^{c_1>0}(\cA_{n,1})[-1]\right). 
 \end{align*}
\end{Rem}
So $\ts^{m,1}_{0,0}$ cannot be separated from  $\ts_{0,0}^{m,1}[-1]$, and $\ts_{0,0}^{m,2}[-1]$.

\section{Local chart on reduced stability space}
In \cite{Arend:shortproof} and \cite[Appendix A]{BMS:stabCY3s}, an effective version of Bridgeland's deformation theorem for stability conditions is developed. We have made use of this result at several points in the main body of the paper.

In this section, we establish an analogous result for reduced stability conditions. As an application, we use it to compare reduced stability conditions on an unpolarized abelian surface, and to prove both the Bayer Vanishing Lemma and a restriction theorem in this context.

\subsection{Local chart given by quadratic form} \label{sec:qudform}
We begin by recalling the effective deformation theorem for stability conditions.

Let $\sigma$ be a non-degenerate stability condition on $\cT$, and let $Q$ be a quadratic form on $\Lambda_\R$ with signature $(2, \rho - 2)$ that provides a support property for $\sigma$. The effective deformation theorem asserts that the central charge of $\sigma$ can be deformed to any other $Z \in \Hom(\Lambda, \C)$ for which the restriction $Q|_{\Ker Z}$ is negative definite. More precisely, one can define an open neighborhood of $\sigma$ as in \cite[Appendix A]{BMS:stabCY3s} using this condition on the central charge.

\begin{PropDef}[{\!\cite[Proposition A.5]{BMS:stabCY3s}}] \label{propdef:stabQsigma}
     Consider the open subset of $\Hom(\Lambda, \C)$ consisting of central charges whose kernels are negative definite with respect to $Q$, and let $W=W(Q,Z_\sigma)$ be the connected component of this subset containing the central charge $Z_\sigma$ of $\sigma$.
     
     Let $\Stab(Q,\sigma,\cT)\subset \Stab(\cT)$ denote the connected component of the preimage $\Forg^{-1}(W)$ that contains $\sigma$. Then the following properties hold.
     \begin{enumerate}
         \item The map $\Forg|_{\Stab(Q,\sigma,\cT)}:\Stab(Q,\sigma,\cT)\to W$ is a universal cover.
         \item Any stability condition $\sigma'\in \Stab(Q,\sigma,\cT)$ satisfies the support property with respect to the same quadratic form $Q$.
     \end{enumerate}
\end{PropDef}

\begin{Rem}
    The only difference with \cite[Proposition A.5]{BMS:stabCY3s} is that we state in ({a}) that the covering map $\Forg|_{\Stab(Q,\sigma,\cT)}$ is universal. This is by noticing the space $W/\GL^+(2,\R)$ is contractible as $Q$ is with signature $(2,\rho-2)$, see the argument in Lemma \ref{lem:quadraticform}. 
\end{Rem}

For a non-degenerate reduced stability condition, we may define a similar neighborhood as the  image of $\Stab(Q,\sigma,\cT)$. Moreover, we can make the following corresponding notions.

\begin{Not}\label{not:UQ}
Let $\sigma$ be a non-degenerate stability condition with $\ts=\pi_\sim(\sigma)$ and $Q$ be a quadratic form on $\Lambda_\R$ with signature $(2,\rho-2)$ offering the support property for $\sigma$. We denote
\begin{align*}
    \sb(Q,\ts,\cT)&\coloneq \pi_\sim\left(\Stab(Q,\sigma,\cT)\right)\\
 U(Q)&\coloneq \{f\in\lbdd\;:\; Q|_{\Ker f}\text{ is with signature }(1,\rho-2)\}.
\end{align*} 

For a subset $S\subset \lbdd$, we denote  \begin{align*}\tgr(2,S)\coloneq \{(f,g)\;|\; f, g\text{ linear independent and } \mathrm{span}_\R\{f,g\}\subset S\cup\{0\}\}\subset \lbdd\times \lbdd.   
\end{align*}
\end{Not}

\begin{Prop}\label{prop:sbQsigma}
   Let $\ts$ be a reduce stability condition with $0\notin \ta(\ts)$. Then for every non-degenerate representative stability condition $\sigma$ with quadratic form $Q$ giving the support property, the following statements hold:
   \begin{enumerate}
       \item  The map $\Forg|_{\sb(Q,\ts,\cT)}:\sb(Q,\ts,\cT)\to U(Q)$ is a universal cover.
       \item  The following diagram commutes.
   \begin{center}
	\begin{tikzcd}
		\Stab(Q,\sigma,\cT) \arrow{d}{\Forg'} \arrow{r}{\Pi}[swap]{\cong}
		& \Pi(\sb(Q,\ts,\cT)) \arrow{d}{\Forg\times \Forg} \arrow[r, phantom,"\subset"] & \sb(Q,\ts,\cT)\times \sb(Q,\ts,\cT)\arrow{d}{\Forg\times \Forg}\\
		W(Q,Z_\sigma) \arrow{r}{(\Im,-\Re)}[swap]{\cong} & \tgr(2,U(Q))_+
		 \arrow[r, phantom,"\subset"]  & U(Q)\times U(Q).
	\end{tikzcd}
\end{center}
   \end{enumerate}
  
\end{Prop} 

Here the map $\Pi=(\pi_\sim,\pi_\sim\circ[\tfrac{1}{2}]):\Stab\to \sb\times \sb$. The space $\tgr(U(Q))_+$ is the connected component of $\tgr(U(Q))$ that contains the image of $Z_\sigma$. As $0\notin\ta(\ts)$, by Proposition \ref{prop:convex}.{(1)}, the maps $Z_\sigma$ and $-Z_\sigma$ cannot both be the central charge on $\cA$. So $\tgr(U(Q))_+$ is determined by $\ts$.

The proposition follows from Proposition and Definition \ref{propdef:stabQsigma} and the following properties of linear algebra.
\begin{Lem}\label{lem:quadraticform}
    Let $Q$ be a bilinear form on $\Lambda_\R$ with signature $(2,\rho-2)$ and $Z_0\in\Hom(\Lambda,\C)$ with $Q|_{\Ker Z_0}$ being negative definite. Then $U(Q)$ is  connected open and equal to the following subsets:
    \begin{align}
        U(Q)
     \label{eq211}   =& \{f\in\lbdd\;|\; \exists Z\in W(Q,Z_0)\text{ such that } \Ker f\supset \Ker Z\}\\
      \label{eq212}  =&\{\Im Z\;|\; Z\in W(Q,Z_0)\}.
    \end{align}
    Moreover, \begin{align}
       W(Q)\coloneq  W(Q,Z_0)\coprod W(Q,\overline Z_0)&=\{g+if\;|\; f,g\in U(Q), Q|_{\Ker f\cap \Ker g}\text{ is negative definite}\} \label{eq213}\\
      \label{eq214} & = \{g+if\;|\; (f,g)\in \tgr(2,U(Q))\};\\
      \label{eq2166} & = \{g+if\;|\; (Q^*(f,g))^2<Q^*(f)Q^*(g), \;Q^*(f)>0\}
    \end{align}
    where $Q^*$ is dual form of $Q$ on $\lbdd$.
\end{Lem}
\begin{proof}
  We may assume that there exists a basis $\be_1,\dots \be_\rho$ under which $Q(x,y,z_1,\dots,z_{\rho-2})=x^2+y^2-z_1^2-\dots-z_{\rho-2}^2$. Denote the dual basis as $\be_1^*,\dots,\be_\rho^*$. It follows that \begin{align*}
      U(Q)=\{a\be_1^*+b\be_2^*-c_1\be_3^*-\dots-c_{\rho-2}\be^*_\rho\;|\;Q(a,b,c_1,\dots,c_{\rho-2})>0\}=\{f\in\lbdd:Q^*(f)>0\}
  \end{align*} which is clearly a connected open subset. \\

 \noindent We then prove \eqref{eq213}.  Denote by $W'(Q)$ the set on the right hand side as that in \eqref{eq213}.  Then by taking $\Ker(g+if)$, the quotient space $W'(Q)/\GL(2,\R)$ is  identified as $\{V\subset \lbdd\mid \dim V=\rho-2, Q|_V$ is negative definite$\}$. 
    
    For every $V\in W'(Q)/\GL(2,\R)$, by the choice of the basis, the projection of $V$ onto the subspace $V_0\coloneq \{(0,0,*,\dots,*)\}$ is isomorphism since otherwise $V$ is contained in a codimension one linear subspace with signature $(2,\rho-3)$. It follows that $V_t\coloneq \{(ta,tb,c_1,\dots,c_{\rho-2}): (a,b,c_1,\dots,c_{\rho-2})\in V\}$ is in $W'(Q)/\GL(2,\R)$ when $t\in[0,1]$. So the space $W'(Q)/\GL(2,\R)$ contracts to $V_0$. Therefore, $W'(Q)$ has two connected components corresponding to $\GL^\pm(2,\R)$. So \eqref{eq213} holds.\\

\noindent For every $f\in U(Q)$, as $Q^*$ is with signature $(2,\rho-2)$ and $Q^*(f)>0$, there exists $g\in \lbdd$ linear independent with $f$ such that $Q^*$ is positive definite on $\mathrm{span}_\R\{f,g\}$. By \eqref{eq213}, we have $\pm g+if\in W(Q,Z_0)$, the formula \eqref{eq211} and \eqref{eq212} hold.

 Note that: $Q|_{\Ker f\cap \Ker g}$ is negative definite $\iff$ $\dim(\Ker f\cap \Ker g)=\rho-2$ and $Q^*|_{(\Ker f\cap \Ker g)^*}$ is positive definite $\iff$ $f,g$ are linear independent and $Q^*(af+bg)>0$ for every $af+bg\neq 0$ $\iff$ $(f,g)\in \tgr(2,U(Q))$. The formula \eqref{eq214} holds.

 Note that: $f, g$ are linear independent and $Q^*(af+bg)>0$ for every $af+bg\neq 0$ $\iff$ $a^2Q^*(f)+2abQ^*(f,g)+b^2Q^*(g)>0$ for every $[a:b]\in \bP_\R^1$ $\iff$ $(Q^*(f,g))^2<Q^*(f)Q^*(g)$ and $Q^*(f)>0$. The formula \eqref{eq2166} holds.
 \end{proof}

\subsection{Example: Reduced stability conditions on an unpolarized abelian surface} \label{sec:appDunpolaized}  
In general, the space $\sb^*(S)$ of reduced stability conditions on an unpolarized surface is difficult to describe as there might exist a curve $C$ with negative self-intersection, in other words, the discriminant $\Delta(\cO_C)<0$. One needs a modified version of quadratic form for the support property of the stability condition.

As the paper is not on this topic, we just study one simple case, the abelian surface case, when this issue does not involve. 
\begin{Asp}
In this section, we always let $S$ be a smooth abelian surface and the lattice $\Lambda=\kn(S)$, the full numerical Grothendieck group.
\end{Asp}

One particular advantage of the abelian surface is that   every semistable (in whatever sense) object $E$ in $\Db(S)$ satisfies the Bogomolov inequality $\Delta(E)\geq 0$, see Remark \ref{lem:negdefonZ}. By \cite[Theorem 15.2]{Bridgeland:K3}, a connected component of the stability manifold is constructed. By \cite{hannah:scfreequotients}, see also \cite{HMS:Orientation,FLZ:ab3,Nick:stab}, this is the only component of the whole manifold. We first briefly recap its construction as follows.

Let $\Coh^{\sharp0}_H(S) \coloneq \langle\Coh^{>0}_H(S),\Coh^{\leq 0}_H(S)[1]\rangle$ be the heart of a bounded t-structure and the central charge be $Z\coloneq -\ch_2+\rk+iH\ch_1$. Then $\sigma_0\coloneq (\Coh^{\sharp0}_H(S),Z )$ is a stability condition on $\Db(S)$, see \cite[Corollary 2.1]{AB:Ktrivial} for reference.

 The discriminant $\Delta$ is a quadratic form on $\Lambda_\R$, more precisely, for every $v=(r,D,s)$ and $v'=(r',D',s')\in \Lambda_\R$, the form is given as\begin{align*}
     \Delta(v,v')=DD'-rs'-r's.
 \end{align*}
 By Hodge Index Theorem, the signature of $\Delta$ is $(2,\rho)$, where $\rho$ is the rank of the N\'eron--Severi group of $S$. 
 \begin{Rem}\label{lem:negdefonZ}
     The discriminant $\Delta$ gives the support property for $\sigma_0$.
 \end{Rem}
\begin{proof}
 We first show that the restricted quadratic form $\Delta|_{\Ker Z}$ is negative definite. For every non-zero character $v=(r,D,s)\in\Kn(S)$ in  $\Ker Z$ as above, we have 
\begin{align*}
    s=r \text{ and } DH=0
\end{align*}
It follows that 
\begin{align*}
    H^2\Delta(v)=H^2D^2-2rsH^2\leq (HD)^2-2r^2H^2\leq0
\end{align*}
Note that if the second inequality holds, then $r=s=0$. It follows that $D\neq 0$, so the first inequality must be strict. In other words, we have $\Delta(v)<0$. The quadratic form $\Delta$ is negative definite on $\Ker Z$.\\

\noindent    Let $E\in\Db(S)$ be a $\sigma_{0}$-stable object. We show that $\Delta(E)\geq0$.

When $\dim\supp(E)\neq0$, there exists $\cL\in\Pic^0(S)$ such that $F\coloneq E\otimes \cL\not\cong E$. Otherwise, we may choose a non-zero automorphism $g$ of the abelian surface such that $F\coloneq g^*E\not\cong E$. In any case, there exists $F$ satisfying
    \begin{align*}
        F\not\cong E, [F]_{\mathrm{num}}=[E]_{\mathrm{num}},\text{ and }F\text{ is $\sigma_0$-stable.}
    \end{align*}
    If follows that 
    \begin{align*}
        0=\hom(E,F)+\hom(F,E)=\hom(E,F)+\hom(E,F[2])\geq \chi(E,F)=\chi(E,E)=-\Delta(E).
    \end{align*}
  As $\Delta$ is with signature $(2,\rho)$ and $\Delta|_{\Ker Z}$ is negative definite, by \cite[Appendix A]{BMS:stabCY3s}, when  $E$ is $\sigma_0$-semistable, we also have $\Delta(E)\geq 0$.
\end{proof}

As that in Proposition and Definition \ref{propdef:stabQsigma}, we have the space $\Stab(\Delta,\sigma_0,\Db(S))$. By  \cite{Bridgeland:K3,hannah:scfreequotients}, the space  $\Stab(S)=\Stab(\Delta,\sigma_0,\Db(S))$.\\

Denote by \begin{align*}
    U(\Delta)\coloneq \{B\in\lbdd:\Delta|_{\Ker B}\text{ is with signature }(1,\rho)\}=\{B:\Delta^*(B)>0\}
\end{align*}
as that in Notation \ref{not:UQ}. By Proposition \ref{prop:sbQsigma}, we may describe the space of reduced stability conditions on $S$ as follows:
\begin{Not}[Reduced stability conditions on abelian surfaces]\label{eg:sbnpsurface}
    The forgetful map \begin{align*}
        \Forg: \sb(S)\to U(\Delta)
    \end{align*}
    is a universal cover. In terms of a parametrized space, we may write
    \begin{align}\label{eqD11}
        &\sb^*(S)=\left\{\ts_{(r,D,s)}\;\middle|\;\begin{aligned}
            r\in \R_{\geq0},s\in\R,D\in\NS_R(S), D^2-2rs>0; \text{ when } r=0, D\in\overline\Eff(S)
        \end{aligned}\right\} \\ 
        &\sb(S)  =\coprod_{n\in \Z}\sb^*(S)[n].\notag
    \end{align}
    The reduced central charge of $\ts_{(r,D,s)}$ is given as $B_{(r,D,s)}=r\ch_2-D\ch_1+s\rk$. When $r>0$, the heart $\cA_{(r,D,s)}$ contains all skyscraper sheaves.
    When $r=0$, all skyscraper sheaves are in $\cP_{\ts}(0)$.
 \end{Not}   
 \begin{Prop}\label{prop:deltresneg}
    Let $S$ be an abelian surface, $v=(r,D,s)$ and $v'=(r',D',s')$ be two parameters  as that in \eqref{eqD11}. Then the restricted quadratic form $\Delta|_{\Ker B_v\cap \Ker B_{v'}}$ is negative definite if and only if  \begin{align}\label{eqD601}
       (\Delta(v,v'))^2<\Delta(v)\Delta(v').
   \end{align}
  In particular, this always implies $(rD'-r'D)^2>0$. If $rD'-r'D$ is effective, then $-B_{v'}\in\ta(\ts_v)$ and $B_v\in\ta(\ts_{v'})$.
\end{Prop}
\begin{proof}
    The criterion \eqref{eqD601} follows from \eqref{eq2166} in Lemma \ref{lem:quadraticform} immediately.

    By Proposition \ref{prop:sbQsigma}, \eqref{eq213}, and Lemma \ref{lem:partialorder}, either $B_{v'}$ or $-B_{v'}$ is in $\ta(\ts_v)$ depending on whether $\ts_{v'}\lsm \ts_v$ or $\ts_{v}\lsm \ts_{v'}$. Note that a line bundle $\cO(E)[1]\in \cA_{\ts_v}$ when $B_v(E)<0$. A line bundle $\cO(F)\in \cA_{\ts_v}$ if $B_v(F)>0$ and $F-E$ is effective for some $\cO(E)[1]\in\cA_{\ts_v}$. So when $rD'-r'D$ is effective, we can only have $\ts_{v}\lsm\ts_{v'}$. In particular, we must have $-B_{v'}\in\ta(\ts_v)$.
\end{proof}

Now we can set up a more general version of Bayer Lemma for a surface without polarization.
\begin{Prop}\label{prop:bayersurf}
   Let $v=(r,D,s)$  be a parameter  as that in \eqref{eqD11} with $r\neq 0$. Then 
    \begin{align*}
        \ts_{v}\lsm \ts_{v}\otimes \cO_S(H)
    \end{align*}
    for every ample divisor $H$. 
\end{Prop}
\begin{proof}   
Given $v=(r,D,s)$ and $G\in \NS_\R(S)$, we denote by 
\[v\cdot e^G\coloneq (r,D+rG,s+DG+\tfrac{1}{2}rG^2).\]
Then we have the following simple properties:
\begin{itemize}
    \item $(v\cdot e^{G_1})\cdot e^{G_2}=v\cdot e^{G_1+G_2}$ for every $G_i\in\NS_\R(S)$.
    \item $\Delta(v\cdot e^G)=(D+rG)^2-2r(s+DG+\tfrac{1}{2}rG^2)=\Delta(v)$.
    \item For every divisor $H$, we have $\ts_{v}\otimes \cO_S(H)=\ts_{v\cdot e^G}$.
\end{itemize}
Substitute $(r',D',s')=v^G$ into \eqref{eqD601}, the difference between the right hand side and left hand side is 
\begin{align*}
    & \Delta(v)\Delta(v\cdot e^G)-(D(D+rG)-rs-r(s+DG+\tfrac{1}{2}rG^2))\\
    =& (D^2-2rs)^2-(D^2-2rs-\tfrac{1}{2}r^2G^2)^2=r^2G^2(\Delta(v)-\tfrac{1}{4}r^2G^2).
\end{align*}
This is positive when and only when $r\neq 0$ and
\begin{align}\label{eqD17}
    0<r^2G^2<4\Delta(v).
\end{align}

\noindent Back to the proof of the proposition. We may let $m\in\Z_{\geq1}$ be large enough so that $4m^2\Delta(v)>H^2>0$. Then by Proposition \ref{prop:deltresneg} and the observation above, we have
\begin{align*}
    \ts_{v}\lsm \ts_{v\cdot e^{\frac{H}{m}}}\lsm \ts_{v\cdot e^{\frac{H}{m}}\cdot e^{\frac{H}{m}}}=\ts_{v\cdot e^{\frac{2H}{m}}}\lsm \dots\lsm \ts_{v\cdot e^H}=\ts_v\otimes \cO_S(H).
\end{align*}
The statement holds.
\end{proof}

The Hom vanishing version follows immediately.
\begin{Cor}\label{cor:bayersurf2}
     Let $v=(r,D,s)$  be a parameter  as that in \eqref{eqD11} with $r\neq 0$ and $H$ be an effective divisor with $H^2>0$. Then for any objects $E_1,E_2\in \cP_{\ts_v}(1)$, we have the vanishing
    \begin{align*}
        \Hom(E_1\otimes \cO_S(H),E_2)=0.
    \end{align*}
\end{Cor}
We can also state this with respect to stability conditions as follows.

\begin{Not}\label{not:stabonsurf}
Let $v=(1,D_1,s_1)$ and $w=(0,D_2,s_2)$ be parameters as that in  \eqref{eqD11} satisfying \eqref{eqD601}. More precisely, 
\begin{align*}
    D^2_1-2s_1>0, D_2^2>0, D_2\in\overline{\Eff}(S)\text{  and } (D_1D_2-s_2)^2<(D_1^2-2s_1)D_2^2.
\end{align*}
By Proposition \ref{prop:deltresneg}, there is a stability condition 
\begin{align*}
    \sigma_{v,w}=(\cA_w,B_v+iB_w).
\end{align*}    
\end{Not}
Moreover,  every stability condition on $\Db(S)$ is of the form $\sigma_{v,w}\cdot \tilde g$ for some $\tilde g\in\glt$. 

\begin{Cor}\label{cor:bayersurf3}
    Let $\sigma_{v,w}$ be a stability condition as above, and $E_1,E_2\in\Db(S)$ be $\sigma_{v,w}$-semistable objects with $\phi_{\sigma_{v,w}}(E_1)\geq\phi_{\sigma_{v,w}}(E_2)$ and $B_w(E_i)\neq 0$. Then for every ample divisor $H$,  we have the vanishing        $\Hom(E_1\otimes \cO_S(H),E_2)=0$.
\end{Cor}

\begin{Lem}\label{lem:surfacesbrest}
    Let $v=(r,D,s)$ be a parameter as that in \eqref{eqD11} and $H\in\NS(S)$ with $r^2H^2<4\Delta(v)$. Then $\ts_v\otimes \cO(H)\lsm \ts_v[1]$.
\end{Lem}
\begin{proof}
    When $r\neq 0$, by \eqref{eqD17} and Proposition \ref{prop:deltresneg}, the restricted quadratic form $\Delta|_{\Ker B_v\cap \Ker B_{v\cdot e^H}}$ is negative. By Proposition \ref{prop:sbQsigma}, \eqref{eq213} and Lemma \ref{lem:partialorder}.{(2)}, we have $\ts_{v\cdot e^H}\lsm \ts_v[1]$.

   If $r=0$, then $v\cdot e^H=(0,D,t)$ for some $t\in \R$. There exists $s_0<0$ such that $s_0^2D^2>\max\{(DT-t)^2,(DT-s)^2\}$. We may let $w=(1,0,s_0)$, then by Proposition \ref{prop:deltresneg}, we have $-B_v,-B_{v\cdot e^H}\in\ta(\ts_w)$. It follows that $\ts_{v\cdot e^H}\lsm \ts_w[1]\lsm \ts_v[1] $. The statement follows.
\end{proof}
\begin{Prop}\label{prop:D13}
    Let $\sigma_{v,w}$ be a stability condition as that in Notation \ref{not:stabonsurf}. Let $C\subset S$ be a smooth curve with $C\in |H|$ for some divisor $H\in\NS(S)$ such that \begin{align}\label{eqD22}
        D_2^2H^2+(D_1D_2-s_2)^2<(D_1^2-2s_1)D^2_2.
    \end{align}
    Then $\sigma_{v,w}\otimes \cO_S(H)\lsm \sigma_{v,w}[1]$. The stability restricts to $\sigma_{v,w}|_{\Db(C)}$. A vector bundle $E$ on $C$ is slope stable if and only if $\iota_*E$ is $\sigma_{v,w}|_{\Db(C)}$-stable.
\end{Prop}
\begin{proof}
    For the first statement, by Lemma \ref{lem:eqdefforlsmonstab}, we only need to show that
    \begin{align}\label{eqD23}
       \pi_\sim(\sigma_{v,w}[\theta])\otimes \cO_S(H)= \pi_\sim((\sigma_{v,w}\otimes \cO_S(H))[\theta])\lsm \pi_\sim(\sigma_{v,w}[\theta+1])
    \end{align}
    for every $\theta\in(-1,0]$. 
    
    When $\theta=0$, $\pi_\sim(\sigma_{v,w}[\theta])=\ts_w$, By Lemma \ref{lem:surfacesbrest}, the formula \eqref{eqD23} holds.

    When $\theta\neq 0$, $\pi_\sim(\sigma_{v,w}[\theta])=\ts_{v+tw}\cdot c$ for some $t,c\in \R$. Note that \begin{align*}
        &\Delta(v+tw)=(D_1+tD_2)^2-2s_1-2ts_2\\
        =& D_1^2+2tD_1D_2+t^2D_2^2-2s_1-2ts_2 = \Delta(v)+2t(D_1D_2-s_2)+t^2D_2^2\\
        \geq & \Delta(v)-\frac{(D_1D_2-s_2)^2}{D_2^2}>H^2.
    \end{align*}
    Here the `$\geq$' in the last line is by substituting $t=\frac{s_2-D_1D_2}{D_2^2}$. The `$>$' is by \eqref{eqD22}.

    The rest of the statement follows from Proposition \ref{prop:reststab}.
\end{proof}

\section{Basic algebra: polynomial with distinct real roots}\label{sec:basicalgeb}In this section, we study the space of real polynomials with distinct real roots, which serves as a parameter space for certain reduced stability conditions. Although many of the properties discussed here may be known in the literature (see, for example, \cite{fisk:interlaced}), we were unable to find explicit statements of these results in the form we require. For the sake of completeness, we provide detailed and self-contained proofs using elementary methods.

Fix a positive integer $n$, in this section, we  denote by 
\begin{align*}
    P_n&\coloneq \{f(x)\in\R[x]\;|\;\deg f(x)=n\text{ and }f(x)=0\text{ has } n \text{ distinct real roots}\}.\\
    B_n&\coloneq P_n\bigcup P_{n-1}.
\end{align*}

We regard each polynomial in $B_n$ by its coefficients, thus embedding $B_n$ as a subset of $\R^{n+1}$. Endowed with the induced Euclidean topology, $B_n$ forms an open cone in $\R^{n+1}$. In particular, the projective space $\bP(B_n)$ is well-defined.

We denote by $\sbr_n$ the complement of the big diagonal in $\mathrm{Sym}^n(\bP^1_\R)$, in other words, the space of ordered $n$-tuples of distinct points in $\bP^1_\R$. More explicitly,
\begin{align*}
    \sbr_n\coloneq \{(s_1<s_2\dots<s_n):s_i\in\bP_\R^1\}.
\end{align*}

Here the order $<$ is by viewing $\bP_\R^1=\R\cup\{+\infty\}$ and define $s<+\infty=[1:0]$ for every $s\in \R$.

For every $f(x)\in B_n$, the ordered set of its roots $\Root(f)\coloneq \{s_1<s_2<\dots<s_n:f(s_i)=0\}$ is in $\sbr_n$. Here when $f(x)\in P_{n-1}$, we set $s_n=+\infty$. It is clear that map $\Root$ is well-defined from $\bP(B_n)$ to $\sbr_n$.  On the other hand, for every $\us\in \sbr_n$, we have $\Psi(\us)\coloneq f_{\us}(x)\coloneq \prod (x-s_i) \in B_n$. Here when $s_n=+\infty$, the product is from $i=1$ to $i=n-1$.

This gives us a homeomorphism between $\bP(B_n)$ and $\sbr_n$:
\begin{align*}
    \Root:\bP(B_n)&\overset{\simeq}{\longleftrightarrow}\sbr_n: \Psi\\
    f& \;\longmapsto \text{ ordered roots of $f$}\\
    \prod (x-s_i) & \;\longmapsfrom\;\; \us
\end{align*}

We define the following relation on elements in $\sbr_n$:
\begin{align}
  \notag  \us<\ut&:\iff s_i<t_i \text{ for every }i=1,\dots,n. \;\;\\
   \notag \us<\ut[1]&:\iff s_i<t_{i+1}\text{ for every }i=1,\dots,n-1.\\
  \label{eq:defuslinkut}  \us\link\ut&:\iff \us<\ut<\us[1]\text{ or }\ut<\us<\ut[1].
\end{align}

\subsection{Lines on $\bP(B_n)$}
\begin{Lem}\label{lem:rootsandgap}
    Let $f,g\in B_n$, then the following two statements are equivalent:
    \begin{enumerate}[(1)]
     \item $af+bg\in B_n$ for every $[a:b]\in \mathbf P_\R^1$.
        \item $\Root(f)\link\Root(g)$, in other words, $f$ and $g$ have strict interlaced roots.
    \end{enumerate}
\end{Lem}
\begin{proof}
    {(1)}$\implies${(2)}: Suppose the roots of $f$ and $g$ are not strictly interlaced, then there exist neighbor roots $s_i<s_{i+1}\in\Root(f)$ such that on the interval $(s_i,s_{i+1})$, the polynomial $g(x)\neq 0$. If $g(s_i)=0$ or $g(s_{i+1})=0$, then there exists $[a:b]\in\bP_\R^1$ such that $af'(s_i)+bg'(s_i)=0$ (resp. $s_{i+1}$). In particular, $af+bg$ has a double root at $s_i$ (resp. $s_{i+1}$) and cannot be in $B_n$.

    Therefore, the function $h(x)\coloneq \frac{f(x)}{g(x)}$ is well-defined on the interval $[s_i,s_{i+1}]$ with $h(s_i)=h(s_{i+1})=0$. It follows that $h'(t)=0$ for some $t\in(s_i,s_{i+1})$. Let $a=g(t)$ and $b=-f(t)$, then $(af+bg)(t)=0$ and $(af+bg)'(t)=h'(t)g(t)^2=0$. So $t$ is a double root of $af+bg$, which leads to the contradiction.\\

    {(2)}$\implies${(1)}: Without loss of generality, we may assume that $\Root(f)=\us<\Root(g)=\ut<\us[1]$ and both $f$ and $g$ are monic, then $(-1)^{n-i}f(t_i)>0$ for every $i$. So for every $a\neq 0$, the polynomial $(af+bg)(x)$ has at least one root in the interval $(t_i,t_{i+1})$ for every $1\leq i\leq n-1$. Counting the multiplicity of roots, the number of roots of $(af+bg)(x)$ in the interval $(t_i,t_{i+1})$ is odd. Therefore, the polynomial $(af+bg)(x)$ has exactly one single root in each interval $(t_i,t_{i+1})$ for every $1\leq i\leq n-1$. A polynomial with degree at most $n$ and at least $n-1$ single real roots must be in $B_n$. 
\end{proof}

\begin{Not}
    For every pair of polynomials $f,g\in B_n$ satisfying the properties in Lemma \ref{lem:rootsandgap}, we will write $f\link g$ and denote by
    \begin{align*}
        \ell(f,g)\coloneq \{[af+bg]\in\bP(B_n):[a:b]\in\bP_\R^1\}
    \end{align*}
    the projective line in $\bP(B_n)$. 
    
   When $f\link g$,  the pair induces an $n$-to-$1$ `real \'etale map' $\tfrac{f}{g}$ from $\bP_\R^1$ to $\bP_\R^1$.

Let $\ell$ be a projective line contained in $\bP(B_n)$. Then for any different points $[f],[g]\in\ell$, we have $f\link g$. For every $t\in\R$, there is a unique monic polynomial $h$ with $[h]\in\ell$ and $h(t)=0$.
    
    There is a unique monic polynomial in $\ell$ with degree $n-1$. We denote this polynomial as $f_{\ell}(x)$.
\end{Not}

\begin{Lem}\label{lem:Epi}
    Let $\ell\subset \bP(B_n)$ be a projective line and  $f(x)$ be a monic polynomial with degree $n$ and $[f(x)]\in\ell$. Then $f(x)-xf_{\ell}(x)\in B_{n-1}$. 
    
    Moreover, $f(x)-xf_{\ell}(x)\link f_{\ell}(x)$ in $B_{n-1}$. The projective line $\ell(f(x)-xf_{\ell}(x),f_{\ell}(x))$ in $\bP(B_{n-1})$ does not depend on the choice of $f(x)$.
\end{Lem}
\begin{proof}
    Let $\Root(f_{\ell})=(s_1<s_2<\dots<s_{n-1}<s_n=+\infty)$, then by Lemma \ref{lem:rootsandgap}, we have $(-1)^{n-i}f(s_i)>0$ for every $1\leq i\leq n-1$. It follows that $(-1)^{n-i}(f-xf_{\ell})(s_i))>0$. So for each $1\leq i\leq n-2$, in the interval $(s_{i},s_{i+1})$, the polynomial $f-xf_{\ell}$ has at least one root, and counting multiplicity, the number of roots of $f-xf_{\ell}$ is odd. Note that $\deg(f-xf_{\ell})\leq n-1$. Therefore, for each $1\leq i\leq n-2$, in the interval $(s_{i},s_{i+1})$, the polynomial $f-xf_{\ell}$ has exactly one single root. A polynomial with degree at most $n-1$ and at least $n-2$ single real roots must be in $B_{n-1}$. \\

    For every $c\in\R$, $f+cf_{\ell}$ is a monic polynomial in $\ell$ with degree $n$. By the first part of the statement, the polynomial \begin{align*}
      (f-xf_{\ell})+cf_{\ell}= (f+cf_{\ell})-xf_{\ell}\in B_{n-1}.
    \end{align*}
    As $f_{\ell}\in B_{n-1}$, for every $[a:b]\in\bP_\R^1$, we have $a(f-xf_{\ell})+bf_{\ell}\in B_{n-1}$. Therefore, the relation $f-xf_{\ell}\link f_{\ell}$ holds in $B_{n-1}$.\\
    
    Let $g(x)$ be another monic polynomial with degree $n$ in $\ell$, then $g(x)=f(x)+cf_{\ell}(x)$ for some $c\in\R$. It is clear that $g-xf_{\ell}=(f-xf_{\ell})+cf_{\ell}\in\ell(f-xf_{\ell},f_{\ell})$. So $\ell(g-xf_{\ell},f_{\ell})=\ell(f-xf_{\ell},f_{\ell})$.
\end{proof}

\begin{Not}\label{not:E4}
    We denote by $\pi(\ell)$ the line $\ell(f-xf_{\ell},f_{\ell})$ in $\bP(B_{n-1})$ as that in Lemma \ref{lem:Epi}. In particular, $f_{\pi(\ell)}(x)$ is the unique monic polynomial in $\pi(\ell)$ with degree $n-2$.
\end{Not}

\begin{Lem}\label{lem:E10}
    Let $\ell\subset \bP(B_n)$ be a projective line and denote by
    \begin{align*}
        Q_\ell(x,y)\coloneq (y-x)\left(f_{\ell}(x)f_{\pi(\ell)}(y)-f_{\ell}(y)f_{\pi(\ell)}(x)\right)
    \end{align*}
    Then for every $[f_{\ut}]\in \ell$ with $t_n\neq+\infty$ and $1\leq i\neq j\leq n$, we have $(-1)^{i+j}Q_\ell({t_i,t_j})>0$.
\end{Lem}
\begin{proof}
    By Lemma \ref{lem:Epi}, $f_{\pi(\ell)}(x)=a(f_{\ut}(x)-xf_{\ell}(x))+bf_{\ell}(x)$ for some $[a:b]\in\bP_\R^1$. Let $s_{n-1}$ be the greatest root of $f_\ell$, then it is greater than all roots of $f_{\pi(\ell)}$. As $f_{\pi(\ell)}$ is monic, we have $0<f_{\pi(\ell)}(s_{n-1})=a f_{\ut}(s_{n-1})$. As $t_{n-1}<s_{n-1}<t_n$ and $f_{\ut}$ is monic, $f_{\ut}(s_{n-1})<0$. It follows that $a<0$.
    
    Substitute $x=t_i$ and $y=t_j$ into $Q_\ell$, we get \begin{align*}
      Q_\ell(t_i,t_j)=  &(t_j-t_i)\left(f_{\ell}(t_i)(a(f_{\ut}(t_j)-t_jf_{\ell}(t_j))+bf_{\ell}(t_j))-f_{\ell}(t_j)(a(f_{\ut}(t_i)-t_if_{\ell}(t_i))+bf_{\ell}(t_i))\right)\\
       = &-a(t_i-t_j)^2f_{\ell}(t_i)f_{\ell}(t_j)
    \end{align*}
    Note that $(-1)^{n+i}f_{\ell}(t_i)>0$, the statement follows.
\end{proof}

\subsection{Roots separation}
\begin{Not}
    For every $f\in\R[x]$, the roots separation of $f$ is defined as
\begin{align*}
    \sep(f)\coloneq \min\{|s-t|:s\neq t,\;f(s)=f(t)=0\}.
\end{align*}
For every  $\ell\subset \bP(B_n)$, we define its root separation as
\begin{align*}
    \sep(\ell)\coloneq \min\{\sep(f):f\in\ell\}.
\end{align*}
Note that $\ell$ is compact and $\sep$ is a continuous function on $\bP(B_n)$, we have $\sep(\ell)=\sep(f)$ for some $f\in\ell$.

For every $d\geq 0$, we denote $B_n^{>d}\coloneq \{f\in B_n: \sep(f)>d\}$.
\end{Not}

\begin{Lem}\label{lem:E5}
    Let $f\in B_n$ with degree $n$ and $0<m<\sep(f)$, then $f(x)\link f(x+m)$ and 
    \begin{align*}
        \sep(\ell(f(x),f(x+m)))>\min\{m,\sep(f)-m\}.
    \end{align*}
\end{Lem}
\begin{proof}
   Let $\Root(f)=(t_1<\dots<t_n)$, then $\Root(f(x+m))=(t_1-m<\dots<t_n-m)$. It is clear that $\Root(f(x+m))<\Root(f)<\Root(f(x+m))[1]$, so $f(x)\link f(x+m)$. 
    
    Let $\Root(af(x)+bf(x+m))=(r_1<\dots<r_n)$, then by Lemma \ref{lem:rootsandgap}, we have either
    \begin{align*}
        \text{}&t_1-m<r_1 <t_1<t_2-m<r_2<t_2<\dots<t_n-m<r_n<t_n;\\
        \text{or }& r_1<t_1-m<t_1<r_2<t_2-m<t_2<\dots<r_n<t_n-m<t_n;\\
        \text{or }& t_1-m<t_1<r_1<t_2-m<t_2<r_2<\dots<t_n-m<t_n<r_n.
    \end{align*}
    The statement follows.
\end{proof}
\begin{Lem}\label{lem:E6}
    Let $f,g\in B_{n}$ with degree $n$ and $\Root(f)<\Root(g)<\Root(f)[1]$.  Then for every $d<\sep(\ell(f,g))$, there exists $N$ sufficiently large such that
    \begin{align*}
        \sep\left(\ell(f(x),(x+N)g(x))\right)>d.
    \end{align*}
\end{Lem}
\begin{proof}
 Viewing $f(x)$ as an element in $B_{n+1}$, it is clear that when $N$ is sufficiently large, we have $\Root((x+N)g(x))<\Root(f(x))<\Root((x+N)g(x))[1]$. By Lemma \ref{lem:rootsandgap}, the line $\ell(f(x),(x+N)g(x))$ is well defined. It is clear that $\ell(f(x),(x+N)g(x))=\ell(f(x),(1+\tfrac{x}{N})g(x))$.

For every $[a:b]\in\bP_\R^1$, we have $\lim_{N\to +\infty}af(x)+b(1+\frac{x}{N})g(x)=af(x)+bg(x)$. So there exists $N_{a,b}$ and an open neighborhood $U$ of $[a:b]$ in $\bP_\R^1$ such that for every $N>N_{a,b}$ and $(a_0,b_0)\in U$, we have $\sep(a_0f(x)+b_0(1+\frac{x}{N})g(x))>d$. (Note that this is also the case when $b=0$.) As $\bP_\R^1$ is compact, the statement holds.
\end{proof}
\begin{Lem}\label{lem:E3}
    Let $f(x)\in B_n$, then 
    \begin{align}\label{eq:sep}
        \sep(\ell(f(x),f'(x)))\geq \sep(f(x)).
    \end{align}
\end{Lem}
\begin{proof}
   The statement does not depend on the degree of $f$ and we may assume that $\deg f=n$.
   
   Let  $0<d<\sep(f)$, we claim that 
   \begin{align}\label{eqE4}
      g(x)\coloneq  f'(x)f(x+d)-f(x)f'(x+d)> 0,\; \forall x\in\R.
   \end{align}

   To see this,  we  denote the roots $\Root(f)=(t_1<\dots<t_n)$. Dividing \eqref{eqE4} by $f(x)f(x+d)$, the function becomes:
   \begin{align}\label{eqE5}
    h(x)\coloneq   - \frac{1}{x-(t_1-d)}+\frac{1}{x-t_1}-\frac{1}{x-(t_2-d)}+\frac{1}{x-t_2}-\dots -\frac{1}{x-(t_n-d)}+\frac{1}{x-t_n}.
   \end{align}
   By the assumption that $\sep(f)>d$, we have 
   \begin{align*}
       t_1-d<t_1<t_2-d<t_2<\dots t_n-d<t_n.
   \end{align*}
   It is then clear that 
   \begin{align*}
      & h(x)>0\text{ and }f(x)f(x+d)>0, \text{ when }x\in(t_i,t_{i+1}-d);\\
      & h(x)<0 \text{ and }  f(x)f(x+d)<0, \text{ when }x\in(t_i-d,t_{i});
   \end{align*}
  for every $1\leq i\leq n$. Here we set $t_0=-\infty$ and $t_{n+1}=+\infty$.
   
  It follows that $g(x)>0$ when $x\neq t_i,t_i-d$. For $x=t_i$, as $d<\sep(f)$, we have $f(t_i+d)\neq0$. As $f(x)\in U$, we have $f'(t_i)\neq0$. For $x=t_i-d$, we have $f(t_i-d)\neq0$ and $f'((t_i-d)+d)\neq0$. So $g(x)\neq 0$ for all $x\in \R$. In particular, $g(x)>0$.

  It follows that $\tfrac{f'(t)}{f(t)}\neq \tfrac{f'(t+d)}{f'(t+d)}$ for any $t\in\R$. So for every $[a:b]\in\bP_\R^1$, there is no $t\in\R$ and $0<d<\sep(f)$ satisfying 
  \begin{align*}
      af(t)+bf'(t)=af(t+d)+bf'(t+d)=0.
  \end{align*}
If follows that $\sep(af+bf')\neq d$. Note that $\sep(-)$ is a continuous function on $\ell(f,f'))$ and $\sep(f)>d$, we must have $\sep(\ell(f,f'))>d$. The statement holds.
\end{proof}

\begin{Cor}\label{cor:E4}
    Let $f\in B_n$ with $\sep (f)>d$, then 
    \begin{enumerate}[(1)]
        \item there exists $\delta>0$, such that for every $0\neq c<|\delta|$, we have $\sep(\ell(f(x),f(x+c)))>d$.
        \item there exists $g\link f$ in $B_n$ such that $\sep(\ell(f,g))>d$.
    \end{enumerate}
\end{Cor}
\begin{proof}
{(1)} The statement does not depend on the degree of $f$ and we may assume that $\deg f=n$. Note that $\sep(-)$ is a continuous function on the set $\{\ell\mid \ell\subset \bP(B_n)\}$ with respect to the Euclidean topology on $\Gr(2,\R^{n+1})$. By Lemma \ref{lem:E3}, there exists an open neighborhood $W_d$ of $f'(x)$ in $\bP(B_n)$ such that for every line $\ell$ satisfying  $f(x)\in\ell$ and  $\ell\cap W_d\neq \emptyset$, we have $\sep (\ell)>d$.

    Note that $f(x)-f(x+c)=cf'(x)+O(c^2)$, so there exists $\delta>0$ small enough such that for every $0\neq c<\delta$, the polynomial $(f(x)-f(x+c))/c$ is in $W_d$. The statement holds.\\

\noindent    {(2)} When $\deg f=n$, the statement follows from Lemma \ref{lem:E3}.
    
    When $\deg f=n-1$, we may choose $g\in \ell(f,f')$ in $B_{n-1}$ such that $f\link h$ and $\deg h=n-1$. By Lemma \ref{lem:E6}, there exists $N$ such that $\sep(\ell(f(x),(x+N)h(x)))>d$. Let $g(x)=(x+N)h(x)$, then $f\link g$ in $B_n$, the statement holds.
\end{proof}

For every $f,g\in B_n$,  we denote by $f\linkless  g$ if $f\link g$ and $g(t_n)<0$, where $t_n$ is the largest root of $f$; when $\deg f=n-1$, $g(t_n)$ is set to be the leading coefficient of $g$.

\begin{Lem}\label{lem:E4}
    Let $f,g,h\in B_n$ and $d\geq0$ such that $f\linkless g,f\linkless h$ and $\sep(\ell(f,g)),\sep(\ell(f,h))>d$. Then $g+h\in B_n$, $f\linkless g+h$ and $\sep(\ell(f,g+h))>d$.
\end{Lem}
\begin{proof}
    Let $\Root(f)=(t_1<t_2<\dots<t_n)$. Then by the assumption that $f\linkless g,h$ and Lemma \ref{lem:rootsandgap}, we have  $(-1)^{n-i}(g+h)(t_i)<0$ for every $1\leq i\leq n$. It follows that for each $1\leq i\leq n-1$, in the interval $(t_i,t_{i+1})$, counting the multiplicity of the roots, the polynomial $(g+h)(x)$ has odd number of roots. As the degree of $g+h$ is not greater than $n$, the polynomial $g+h$ has exactly one single root on each of the interval $(t_i,t_{i+1})$. Therefore, the polynomial $g+h$ is in $B_n$ and $f\linkless (g+h)$.\\

    We then show that $\sep(g+h)>d$. Let $p\in\R$ be a root of $g+h$, then $f(p)\neq 0$ since $(-1)^{n-i}(g+h)(t_i)<0$ for every $1\leq i\leq n$. We modify the two polynomials as:
    \begin{align*}
        G(x)\coloneq g(x)-\frac{g(p)}{f(p)}f(x)\text{ and }H(x)\coloneq h(x)-\frac{h(p)}{f(p)}f(x)=h(x)+\frac{g(p)}{f(p)}f(x).
    \end{align*}
    Then as $G\in\ell(f,g)$, we have $G\linkless f$ and $\sep(G)>d$. Similar properties hold for $H$. It is also clear that $G(p)=H(p)=0$ and  $G+H=g+h$.
    
    Let $q$ (resp. $q_G,q_H$) be the smallest number $>p$ satisfying $(G+H)(q)=0$ (resp. $G(q_G)=0$, $H(q_H)=0$). Then $q_G-p\geq \sep(G)>d$ and $q_H-p>d$. By Lemma \ref{lem:rootsandgap}, there exists a unique root $t_i$ such that $p<t_i<q_G,q_H$. Since $G\linkless f$ and $H\linkless f$, by Lemma \ref{lem:rootsandgap}, we have $(-1)^{n-i}G(t_i)<0$ and $(-1)^{n-i}H(t_i)<0$, the polynomials $G$ and $H$ have the same sign in the interval $(p,\min\{q_G,q_H\})$. Therefore, the polynomial $G+H=g+h$ has no root in the interval $(p,\min\{q_G,q_H\})\supset(p,p+d)$.

    As we can choose any root $p$ of $g+h$, it follows that $\sep(g+h)>d$.\\

    Finally, we show that  $\sep(\ell(f,g+h))>d$. We only need to show $\sep(af+b(g+h))>d$ for every $[a:b]\in\bP_\R^1$ with $b\geq 0$. When $b=0$, it is clear that $\sep(af)=\sep(f)>d$ by the assumption. When $b>0$, as $f\linkless g$, by Lemma \ref{lem:rootsandgap}, we have $f\linkless  af+bg$ and $\ell(f,g)=\ell(f,af+bg)$. Also, we have $f\linkless  bh$ and $\ell(f,h)=\ell(f,bh)$. By the second part of the proof, we have $\sep((af+bg)+bh)>d$.  The statement holds.
\end{proof}

\subsection{Reduced central charge}
Let $n\in\Z_{\geq1}$ and $\Lambda_n\cong \R^{n+1}$ be a real $(n+1)$-dimensional space. Fix a basis $\{\be^*_0,\be^*_1,\dots,\be^*_n\}$ for the dual space $\Lambda_n^*$
We denote by \begin{align*}
    \BBB_n\coloneq \{c\vd_{\ut}:c> 0,\ut\in\sbr_n\}\subset\Lambda_n^*\text{ and }\pm\BBB_n\coloneq \{c\vd_{\ut}:c\neq 0,\ut\in\sbr_n\},
\end{align*} where $\vd_{\ut}$ is defined as that in \eqref{eq:Bvddef}.

\begin{Rem}\label{rem:BvdB}
    The correspondence
    \begin{align*}
        B_n\subset \R[x]_{\leq n} &\longleftrightarrow \Lambda_n^*\supset \pm\BBB_n\\
        a_nx^n+a_{n-1}x^{n-1}+\dots +a_0& \longleftrightarrow n!a_n\be_n^*+(n-1)!a_{n-1}\be_{n-1}^*+\dots +a_0\be_{0}^*
    \end{align*}
    is a linear isomorphism and identifies $B_n$ with $\pm\BBB_n$. In particular, it identifies the projective spaces $\bP(B_n)$ and $\bP(\pm\BBB_n)$. By abuse of notions, we will denote $\bP(\BBB_n)$ instead of $\bP(\pm\BBB_n)$ for simplicity.
    
 By the property of Vandermonde determinant, for every $t\in\R$ and $\us\in\sbr_n$ with $s_n\neq+\infty$, we have 
 \begin{align}\label{eq:E23}
     f_{\us}(t)=\prod_{1\leq i\leq n} (t-s_i)=n!\vd_{\us}(\gamma_n(t));
 \end{align}
  and $ f_{\us}(t)=-(n-1)!\vd_{\us}(\gamma_n(t))$ when $s_n=+\infty$. 
 
 The identification is compatible with the parameter in $\sbr_n$. In other words, the following diagram commutes:
     \[
        \begin{tikzcd}
            \bP(B_n) \arrow[rr,leftrightarrow]\arrow[dr,leftrightarrow]  && \bP(\BBB_n) \arrow[dl,leftrightarrow]\\
           & \sbr_n
        \end{tikzcd}\;\;
        \begin{tikzcd}
            \text{$[$}f_{\us}\text{$]$} \arrow[rr,leftrightarrow]\arrow[dr,leftrightarrow]  && \text{$[$}\vd_{\us}\text{$]$}\arrow[dl,leftrightarrow]\\
           & \us
        \end{tikzcd}
    \]
\end{Rem}

Statements and notations on elements and lines in $\bP(B_n)$ can be interpreted as those in $\bP(\BBB_n)$. In particular, Lemma \ref{lem:rootsandgap} can be restated as follows.
\begin{Lem}\label{lem:Epoly}
    Let $\us\neq\ut\in\sbr_n$, then $[a\vd_{\us}+b\vd_{\ut}]\in \bP(\mathfrak B_n)$ for all $[a:b]\in\bP^1_\R$ if and only if $\us\link \ut$.

    Every projective line $\ell\subset \bP(\BBB_n)$ contains a unique point $[\vd_{\underline r}]$ with $r_n=+\infty$.  For every $r\in\R$, the line $\ell$ contains a unique point $[B_{\underline q}]$ with $q_i=r$, where $i$ is index such that $r_{i-1}<r\leq r_i$.
\end{Lem}

\begin{Not}
We will write $\vd_{\us}\link\vd_{\ut}$ if $\us\link \ut$. 
For $\vd_{\us}\link\vd_{\ut}$, we denote by $\ell(\us,\ut)$ the projective line through $[\vd_{\us}]$ and $[\vd_{\ut}]$ in $\bP(\BBB_n)$. We denote $\sep(c\vd_{\ut})\coloneq \min\{t_{i+1}-t_i\}$.
\end{Not}

\begin{Lem}\label{lem:Epoly7}
    Let  $\vd_{\ut}\in\mathfrak B_n$ with $\sep(\vd_{\ut})>d$, then there exists a line $\ell\subset \bP(\mathfrak B_n)$ containing $[\vd_{\ut}]$ such that $\sep(\ell)>d$.
\end{Lem}
\begin{proof}
    By Remark \ref{rem:BvdB}, the statement follows from Corollary \ref{cor:E4}.
\end{proof}
Lemma \ref{lem:E4} can be restated as follows. 
\begin{Lem}\label{lem:Epolyconvex}
For every $\ut\in\sbr_n$ and $d\geq 0$, the space
\begin{align*}
    \ta^{>d}(\ut)\coloneq \{c\vd_{\us}:\us<\ut<\us[1],\sep(\ell(\us,\ut))>d,c>0\}\cup\{c\vd_{\us}:\ut<\us<\ut[1],\sep(\ell(\us,\ut))>d,c<0\} 
\end{align*}
is a convex subset in $\Lambda_n^*$.
\end{Lem}

\begin{Lem}\label{lem:Epoly2}
    Let $\ut\in\sbr_n$, then the space $\Ker \vd_{\ut}$ in $\Lambda_n$ is spanned by $\gamma(t_i)$, $1\leq i\leq n$. Let $\bv=\sum_{i=1}^n (-1)^ia_i\gamma(t_i)$ be a non-zero vector in $\Ker \vd_{\ut}$, then \begin{align*}
        \bv\notin \bigcup_{\us<\ut<\us[1]}\Ker \vd_{\us}\iff \bv\notin \bigcup_{\us\link\ut}\Ker \vd_{\us}\iff \text{either } a_i\geq 0\text{ for all }i \text{ or }  a_i\leq 0\text{ for all }i.
    \end{align*}
\end{Lem}
\begin{proof}
    Note that for every $\ut<\us<\ut[1]$, there exist $\us'<\ut<\us'[1]$ so that $\vd_{\us'}\in\ell(\ut,\us)$. So we have
    \begin{align*}
        \Ker \vd_{\ut}\cap (\bigcup_{\us<\ut<\us[1]}\Ker \vd_{\us})=  \Ker \vd_{\ut}\cap (\bigcup_{\us\link\ut}\Ker \vd_{\us}).
    \end{align*}
The first `$\iff$' holds.

 \noindent   By \eqref{eq:E23}, when $t_n\neq +\infty$, we have
    \begin{align}\label{eqE55}
        \vd_{\us}(\bv)=\sum_{i=1}^n(-1)^i a_i\vd_{\us}(\gamma(t_i)) = n!\sum_{i=1}^n(-1)^i a_i\prod_{j=1}^n(t_i-s_j).
    \end{align}
    when $t_n= +\infty$, we have
    \begin{align}\label{eqE552}
        \vd_{\us}(\bv)=\sum_{i=1}^n(-1)^i a_i\vd_{\us}(\gamma(t_i)) = (-1)^n a_n+n!\sum_{i=1}^{n-1}(-1)^i a_i\prod_{j=1}^n(t_i-s_j).
    \end{align}

\noindent  `$\impliedby$': When $\us<\ut<\us[1]$, each term $(-1)^{n-i} \prod_{j=1}^n(t_i-s_j)> 0$. Note that $\bv$ is assumed to be non-zero, so if all $a_i\geq0$ or all $a_i\leq 0$, the formula \eqref{eqE55} (or \eqref{eqE552}) is always non-zero.\\

\noindent `$\implies$': Assume that $\bv\notin \bigcup_{\us<\ut<\us[1]}\Ker \vd_{\us}$.
 
 \noindent When $t_n\neq +\infty$, suppose $a_k\cdot a_l<0$ for some $1\leq k,l\leq n$.  Let $s_q=t_q-\epsilon$ for all $q\neq k$ and $s_k=t_{k-1}+\epsilon$, where $\epsilon>0$ is sufficiently small. Then $\us<\ut<\us[1]$ and \eqref{eqE55} is equal to
   \begin{align*}
       n!\cdot\left( (-1)^ka_k\prod_{j=1}^n(t_k-s_j)+\sum_{i\neq k}(-1)^k a_i\epsilon\cdot\prod_{j\neq i}(t_j-s_i)\right),
   \end{align*}
   which has the same signature of $(-1)^na_k$.

   Similarly, we may also let $\us'$ be with $s'_q=t_q-\epsilon$ for all $q\neq l$ and $s'_l=t_{l-1}+\epsilon$ for some $\epsilon>0$ sufficiently small. Then  $\vd_{\us'}(\bv)$ has the same signature of $(-1)^na_l$.
   
   As $\vd_{\us}$ is a continuous function with respect to $\us$ and the set $\{\us:\us<\ut<\us[1]\}\cong \R^n$ is connected, there exist $\us$ with $\us<\ut<\us[1]$ and $\vd_{\us}(\bv)=0$, which leads to the contradiction.\\

 \noindent  When $t_n=+\infty$, if $a_n=0$, then the statement follows the same argument as that for $t_n\neq +\infty$. Otherwise, we may assume $(-1)^na_n>0$.  Let $s_q=t_q-\epsilon$ for all $q\neq n$ and fix an $s_n>t_{n-1}$. Then $\vd_{\us}(\bv)>0$ when $\epsilon>0$ is sufficiently small. 

   Suppose $(-1)^na_k<0$ for some $1\leq k\leq n-1$, then we may let $s'_q=t_q-\epsilon$ for all $q\neq k,n$, $s'_k=t_{k-1}+\epsilon$, and $s'_n=1/\epsilon$. Then when $\epsilon>0$ is sufficiently small, we have $\vd_{\us'}(\bv)<0$. By the continuity of $\vd_{\us}(\bv)$ with respect to $\us$, we get the contradiction.
\end{proof}
\subsection{Quadratic form}
For every projective line $\ell\subset \bP(\BBB_n)$,  we set $\vd_{\ell}\coloneq \vd_{\underline r}$, where  $[\vd_{\underline r}]$ is the unique point on $\ell$ satisfying $r_n=+\infty$. We denote by 
  \begin{align*}
      \Ker \ell\coloneq \{p\in \Lambda_n:\forall\ut\in\ell, \vd_{\ut}(p)=0\}
  \end{align*}
  the codimension $2$ subspace in $\Lambda_n$. We denote by  $\sc(\ell)$ for the set of vectors as that in Lemma \ref{lem:Epoly2}:
  \begin{align}\label{eq:defSCl}
      \sc(\ell)\coloneq \left\{p\in\Lambda_n:p\in \Ker\vd_{\ut}\setminus\left(\cup_{\us\link\ut}\Ker\vd_{\us}\right) \text{ for some }\ut\in\ell\right\}.
  \end{align}

  Note that $\Ker \ell\cap\sc(\ell)=\emptyset$. For every $p\notin \Ker\ell$, there is a unique $\ut\in\ell$ such that $\vd_{\ut}(p)=0$. By Lemma \ref{lem:Epoly2}, $p$ is in the form of $\sum(-1)^ia_i\gamma_n(t_i)$ for some $a_i$ all $\geq0$ or $\leq0$.\\
  
\noindent  When $n\geq 2$, by Remark \ref{rem:BvdB}, there is a corresponding projective line $\pi(\ell)\subset \bP(B_{n-1})$. It corresponds to a line $\pi(\ell)\subset \bP(\BBB_{n-1})$. Here $\Lambda_{n-1}^*$ is spanned by $\{\be_{n-1}^*,\dots,\be_0^*\}$. By fixing this basis,  one can view $\Lambda_{n-1}^*$ as a subspace of $\Lambda_n^*$. The element $\vd_{\pi(\ell)}$ in $\Lambda_{n-1}^*$ becomes a function on $\Lambda_n$. By abuse of notion, we still denote it as $\vd_{\pi(\ell)}$, which is with `leading term' $-\be_{n-2}^*$. In particular, for every $t\in\R$, we have 
  \begin{align}\label{eq:E45}
     & \vd_\ell(\gamma_n(t))=-\frac{1}{(n-1)!}f_\ell(t)\text{ and }\vd_{\pi(\ell)}(\gamma_n(t))=-\frac{1}{(n-2)!}f_{\pi(\ell)}(t).
  \end{align}
  
\begin{Not}   
For a linear map $\vd=\sum_{i\geq0}^na_i\be_i^*$, we set $\widetilde{\vd}\coloneq \sum_{k\geq1}^nka_{k-1}\be_k^*$. In particular, when $a_n=0$, for every $t\in\R$, we have 
\begin{align}\label{eqE30}
    \widetilde \vd(\gamma_n(t))=t\vd(\gamma_n(t))
\end{align}
For $t=+\infty$, we have
\begin{align}\label{eqE29}
  \vd_\ell(\gamma_n(+\infty))=\vd_{\pi(\ell)}(\gamma_n(+\infty))=\widetilde{\vd_{\pi(\ell)}}(\gamma_n(+\infty))=0 \text{ and }\widetilde{\vd_{\ell}}(\gamma_n(+\infty))=-n.
\end{align}
\end{Not}

For every $\ell\subset\bP(\BBB_n)$, we define the quadratic form on $\Lambda_n$ as 
\begin{align}\label{eq:CdefQ}
    \vq_\ell\coloneq \vd_{\ell}\widetilde{\vd_{\pi(\ell)}}-\vd_{\pi(\ell)}\widetilde{\vd_\ell}
\end{align}
For every $s,r\in\R$, by \eqref{eq:E45} and \eqref{eqE30}, the value of $\vq_\ell(\gamma_n(s),\gamma_n(r))$ is equal to
    \begin{align}
       \notag &\vd_{\ell}(\gamma(s))\widetilde{\vd_{\pi(\ell)}}(\gamma(r))+\vd_{\ell}(\gamma(r))\widetilde{\vd_{\pi(\ell)}}(\gamma(s))-\vd_{\pi(\ell)}(\gamma(s))\widetilde{\vd_\ell}(\gamma(r))-\vd_{\pi(\ell)}(\gamma(r))\widetilde{\vd_\ell}(\gamma(s))\\
     \notag   =& \frac{1}{(n-1)!(n-2)!}\left(f_{\ell}(s)rf_{\pi(\ell)}(r)+f_{\ell}(r)sf_{\pi(\ell)}(s)-f_{\pi(\ell)}(s)rf_\ell(r)-f_{\pi(\ell)}(r)sf_\ell(s)\right)\\
        =&\frac{1}{(n-1)!(n-2)!}Q_\ell(s,r).\label{eqE33}
    \end{align}

\begin{Lem}\label{lem:Eprequadric}
    Let $\ell$ be a projective line in $\bP(\BBB_n)$. Then the quadratic form satisfies the following properties:
    \begin{enumerate}[(1)]
        \item $\vq_\ell(\gamma_n(t))=0$, $\forall$ $t\in\R\cup\{+\infty\}$.
        \item For every $\bv\in\sc(\ell)$, we have $\vq_\ell(\bv)\geq0$. The `$=$' holds when and only when $\bv\in\Ker \vd_\ell\cap \Ker \widetilde{\vd_\ell}$ or is $\gamma_n(t)$ for some $t\in\R\cup\{+\infty\}$ up to a scalar.
        \item  For every $\bv\in\Ker\ell$, we have $\vq_\ell(\bv)\leq 0$. The `$=$' holds when and only when $\bv\in \Ker \widetilde{\vd_\ell}$.
    \end{enumerate}
\end{Lem}
\begin{proof}
    {(1)} When $t\in\R$, the statement follows from \eqref{eqE33}. When $t=+\infty$, the statement follows from \eqref{eqE29}.

    \noindent {(2)} Let $\ut$ be the unique $\ut\in\ell$ satisfying $\vd_{\ut}(\bv)=0$, then by Lemma \ref{lem:Epoly2}, we may assume $\bv=\sum_{i=1}^n(-1)^ia_i\gamma(t_i)$ for some $a_i\geq 0$. 

    When $t_n\neq +\infty$, by Lemma \ref{lem:E10} and \eqref{eq:E45}, we have
    \begin{align}\label{eqC31}
       \vq_\ell(\bv)=\sum_{i,j=1}^n \vq_\ell\left((-1)^ia_i\gamma(t_i),(-1)^ja_j\gamma(t_j)\right)
      = \frac{1}{ (n-1)!(n-2)!}\sum_{i,j=1}^n (-1)^{i+j}a_ia_j Q_\ell\left(t_i,t_j\right)\geq 0
    \end{align}
    By Lemma \ref{lem:E10}, the `$=$' holds when and only when there is exactly one $a_i\neq0$. In other words, the vector $\bv$ equals to $\gamma(t_i)$ up to a scalar.

    When $t_n=+\infty$, $\bv\in \Ker \vd_{\ut}=\Ker \vd_{\ell}$. By \eqref{eq:E45} and \eqref{eqE30}, 
    \begin{align}
      \notag  \vq_\ell(\bv)&=-\sum_{i,j=1}^n (\vd_{\pi(\ell)}\widetilde{\vd_\ell})\left((-1)^ia_i\gamma(t_i),(-1)^ja_j\gamma(t_j)\right)=-\sum_{i=1}^{n-1} (-1)^{i+n}a_ia_n\vd_{\pi(\ell)}(\gamma(t_i))\widetilde{\vd_\ell}(\gamma(\infty))\\
        &=\frac{-na_n}{(n-2)!}\sum_{i=1}^{n-1}(-1)^{i+n}a_i f_{\pi(\ell)}(t_i)\geq0 \label{eqE37}
    \end{align}
    The `$\geq$' is due to the observation that $(-1)^{n-1+i}f_{\pi(\ell)}(t_i)>0$ for every $1\leq i\leq n-1$. The `$=$' holds when and only when $a_n=0$ or $a_1=\dots=a_{n-1}=0$. By \eqref{eqE30}, $a_n=0$ if and only if $\widetilde{\vd_\ell}(\bv)=0$. The statement follows.\\

    \noindent {(3)} By Lemma \ref{lem:Epi}, Notation \ref{not:E4}, \eqref{eq:E45}, \eqref{eqE30} and the first paragraph in the proof for Lemma \ref{lem:E10}, there exists a unique $\ut\in\ell$ with $t_n\neq+\infty$ such that 
    \begin{align*}
        -(n-2)!\vd_{\pi(\ell)}=a'n!\vd_{\ut}+a(n-1)!\widetilde{\vd_\ell}
    \end{align*}for some $a<0$. Note that $\bv\in\Ker\ell=\Ker\vd_\ell\cap\Ker \vd_{\ut}$. Therefore, we have
    \begin{align}\label{eqE38}
        \vq_\ell(\bv)=-2\vd_{\pi(\ell)}(\bv)\widetilde {\vd_{\ell}}(\bv)=2a(n-1)\widetilde {\vd_{\ell}}(\bv)^2\leq 0.
    \end{align}
    The `$=$' holds if and only if $\bv\in\Ker \widetilde {\vd_{\ell}}$.
\end{proof}
\begin{Prop}\label{prop:Equadratic}
     Let $\ell$ be a projective line in $\bP(\BBB_n)$. Then there exists a (family of) quadratic form(s) $\tilde \vq_{\ell}$ on $\Lambda_n$ satisfying:
    \begin{enumerate}
        \item $\tilde \vq_\ell(\gamma_n(t))=0$, $\forall$ $t\in\R\cup\{+\infty\}$.
        \item For every $\bv\in\sc(\ell)$, we have $\tilde \vq_\ell(\bv)\geq0$. The `$=$' holds when and only when $\bv=c\gamma_n(t)$ for some $t\in\R\cup\{+\infty\}$ and $c\in\R$.
        \item  $\tilde \vq_\ell$ is negatively definite on $\Ker \ell$.
    \end{enumerate}
\end{Prop}
\begin{proof}
    We prove the statement by induction on $n$. When $n=1$, we may just let $\tilde \vq_\ell=0$. 
    
    Assume the statement holds for the lower dimensional case, then there exists a quadratic form $\tilde \vq_{\pi(\ell)}$ on $\Lambda_{n-1}$, which is the kernel space of $\be_n^*$, satisfying properties {(a)}, {(b)} and {(c)}. The quadratic form extends to $\Lambda_n$ by setting $\tilde \vq_{\pi(\ell)}(\be_n,\Lambda_n)=0$.

    We may consider the quadratic forms \begin{align}\label{eq:CdefQa}
        \vq_{\alpha}\coloneq \alpha\vq_\ell+\tilde \vq_{\pi(\ell)}
    \end{align}  for some $\alpha>0$. We show that $\vq_{\alpha,\beta}$ satisfies properties {(a)}, {(b)}, and {(c)} when $\alpha$ is sufficiently large. \\
    
  \noindent  {(a)} For every $t\in\R$, $\tilde \vq_{\pi(\ell)}(\gamma_n(t))=\tilde \vq_{\pi(\ell)}(\gamma_{n-1}(t))=0$. When $t=+\infty$, $\tilde \vq_{\pi(\ell)}(\be_n)=0$. The property then follows from Lemma \ref{lem:Eprequadric}.\\

  \noindent  {(c)} On the space $\Ker \ell$, by \eqref{eqE38}, the quadratic form
  \begin{align*}
      \vq_\alpha|_{\Ker \ell}=-\frac{2\alpha}{a(n-1)}(\vd_{\pi(\ell)}|_{\Ker \ell})^2+\tilde \vq_{\pi(\ell)}|_{\Ker \ell}.
  \end{align*} By induction, $\tilde \vq_{\pi(\ell)}|_{\Ker \ell}$ is negative definite on ${\Ker \ell}\cap\Ker\pi(\ell)={\Ker \ell}\cap \Ker \vd_{\pi(\ell)}$. So $\vq_\alpha|_{\Ker \ell}$ is negative definite when $\alpha$ is sufficiently large. More precisely, let $\be\in\Ker \ell$ such that $\vq_{\pi(\ell)}(\be,\Ker \pi(\ell))=0$ and $\vd_{\pi(\ell)}(\be)=1$, then when $\alpha>\frac{a(n-1)}{2}\tilde \vq(\be,\be)$, the form $\vq_\alpha$ is negative definite on $\Ker \ell$.\\

   \noindent {(b)} We will show that when $\alpha$ is sufficiently large,  the inequality  $(-1)^{i+j}\vq_\alpha(\gamma_n(t_i),\gamma_n(t_j))>0$ holds for  every $\ut\in\ell$ and $i\neq j$. Once this is proved, the statement then follows from Lemma \ref{lem:Epoly2}.\\
   
We first deal with the case when $t_n=+\infty$.

 \noindent   For $\us\in\ell$ with $s_n=+\infty$, we have $\tilde \vq_{\pi(\ell)}(\gamma_n(s_n),-)=0$.  By \eqref{eqE37}, we have $(-1)^{i+n}\vq_\alpha(\gamma_n(s_n),\gamma_n(s_i))>0$ for every $i\neq n$ and $\alpha>0$.
 
 For $1\leq i\neq j\leq n-1$,  viewing the vector $\bv_{ij}\coloneq (-1)^i\gamma_{n-1}(s_i)+(-1)^j\gamma_{n-1}(s_j)\in \Ker \vd_\ell$  as an element in $\Lambda_{n-1}$, by Lemma \ref{lem:Epoly2}, we have $\bv_{ij}\in \sc(\pi(\ell))$. By induction, we have\begin{align}\label{eqC35}
       0< \tilde \vq_{\pi(\ell)}(\bv_{ij})= \tilde \vq_{\pi(\ell)}\left((-1)^i\gamma_{n}(s_i)+(-1)^j\gamma_{n}(s_j)\right)=2(-1)^{i+j}\tilde \vq_{\pi(\ell)}\left(\gamma_{n-1}(s_i),\gamma_{n-1}(s_j)\right).
    \end{align}
   By  \eqref{eq:E45}, \eqref{eqE30} and \eqref{eq:CdefQ}, we get $\vq_\ell(\gamma_n(s_i),\gamma_n(s_j))=0$ when $i\neq j\leq n-1$.  By \eqref{eqC35} and \eqref{eq:CdefQa}, we have $(-1)^{i+j}\vq_\alpha(\gamma_n(s_i),\gamma_n(s_j))>0$ for every $i\neq j\leq n-1$ and $\alpha>0$.\\

We then deal with the case that $\ut$ is in the neighborhood of the $\us$ above.

\noindent   For $\ut\in\ell$, when $t_n$ tends to  $+\infty$, the value $(-1)^{i+n}\frac{n!}{t_n^n}\vq_\alpha(\gamma_n(t_n),\gamma_n(t_i))$ tends to $ (-1)^{i+n}\vq_\alpha(\be_n,\gamma_n(s_i))\allowbreak>0$ for every $i\neq n$; and $(-1)^{i+j}\vq_\alpha(\gamma_n(t_i),\gamma_n(t_j))$ tends to $ (-1)^{i+j}\vq_\alpha(\gamma_n(s_i),\gamma_n(s_j))>0$ for every $1\leq i\neq j\leq n-1$. When $t_1$ tends to  $-\infty$, the value $(-1)^{i+1}\frac{n!}{t_1^n}\vq_\alpha(\gamma_n(t_1),\gamma_n(t_i))$ tends to $ (-1)^{i+n}\vq_\alpha(\be_n,\gamma_n(s_{i-1}))>0$ for every $i\neq 1$; and $(-1)^{i+j}\allowbreak\vq_\alpha(\gamma_n(t_i),\gamma_n(t_j))$ tends to $ (-1)^{i+j}\vq_\alpha\allowbreak(\gamma_n(s_{i-1}),\allowbreak\gamma_n(s_{j-1}))>0$ for every $2\leq i\neq j\leq n$.

    So for every $\alpha_0>0$, there exist $N_{\alpha_0}$ such that when $\ut\in\ell$, $t_n>N_{\alpha_0}$ or $t_1<-N_{\alpha_0}$,  the inequality  $(-1)^{i+j}\vq_\alpha(\gamma_n(t_i),\gamma_n(t_j))>0$ holds for  every $i\neq j$. \\

Finally, we deal with the rest cases of $\ut$, which form a compact set.

 \noindent   For $\ut\in\ell$ satisfying $t_1\geq -N_{\alpha_0}$ and $t_n\leq N_{\alpha_0}$, by Lemma \ref{lem:Eprequadric}.{(2)} or more precisely the inequality \eqref{eqC31},  we have $(-1)^{i+j}\vq_\ell(\gamma_n(t_i),\gamma_n(t_j))\allowbreak>0$ for every $i\neq j$. As the all functions are continuous when $t_n\neq +\infty$ and the region $\{\ut\in\ell:t_1\geq -N_{\alpha_0}, t_n\leq N_{\alpha_0}\}$ is compact, there exists $M_{\alpha_0}>0$ such that $(-1)^{i+j}M_{\alpha_0}\vq_\ell(\gamma_n(t_i),\gamma_n(t_j))>(-1)^{i+j+1}\tilde \vq_{\pi(\ell)}(\gamma_n(t_i),\gamma_n(t_j))$ for every $i\neq j$.

    As a summary, when $\alpha>\max\{\alpha_0,M_{\alpha_0}\}$, the inequality  $(-1)^{i+j}\vq_\alpha(\gamma_n(t_i),\gamma_n(t_j))>0$ holds for  every $\ut\in\ell$ and $i\neq j$. Property {(b)} holds.
\end{proof}

\bibliography{rstab}                      
\bibliographystyle{halpha}  
\end{document}